\documentclass[11pt]{amsart}
\usepackage{color,amssymb,amsthm,amstext,latexsym,amsmath,amscd,amsfonts}
\usepackage{enumitem}
\usepackage[mathscr]{eucal}
\usepackage{tensor}
\usepackage{wrapfig} 
\usepackage[all]{xy}
\usepackage{float}
\usepackage[pdftex]{graphicx}
\usepackage{longtable}
\usepackage{mathtools}
\usepackage{accents}
\usepackage[normalem]{ulem}
\usepackage{cancel}
\usepackage{upgreek}
\usepackage{tikz}
\usepackage{graphicx}
\usepackage{hyphenat}
\usepackage{lmodern}
\usepackage{tcolorbox}
\usepackage{tikz}
\usepackage{tikz-cd}
\usepackage{courier}
\usepackage{hyperref}
\usepackage[normalem]{ulem}

\input xy
\xyoption{all}
\usetikzlibrary{calc}

\theoremstyle{plain}
\newtheorem{thm}{Theorem}
\newtheorem{theorem}[thm]{Theorem}
\newtheorem{corollary}[thm]{Corollary}
\newtheorem{lemma}[thm]{Lemma}

\newtheorem{proposition}[thm]{Proposition}
\newtheorem{mr}{Main Result}

\newtheoremstyle{exm}
{9pt}{9pt}{}{}{\bfseries}{}{.5em}{}
\theoremstyle{exm}

\newtheoremstyle{rmk}
{9pt}{9pt}{}{}{\bfseries}{}{.5em}{}
\theoremstyle{rmk}

\newtheoremstyle{question}
{9pt}{9pt}{}{}{\bfseries}{}{.5em}{}
\theoremstyle{question}

\numberwithin{equation}{section}
\numberwithin{thm}{section}
\numberwithin{figure}{section}

\theoremstyle{definition}
\newtheorem{definition}[thm]{Definition}

\newtheorem{example}[thm]{Example}
\newtheorem{remark}[thm]{Remark}

\newtheorem*{claim}{Claim}

\newcommand{\R}{\mathbb{R}}
\newcommand{\C}{\mathbb{C}}
\newcommand{\Z}{\mathbb{Z}}
\newcommand{\Q}{\mathbb{Q}}

\newcommand{\oc}{(\widehat{U},\Gamma,\varphi)} % orbifold chart
\newcommand{\otheroc}{(\widehat{V},\Gamma',\psi)} % other orbifold chart
\newcommand{\thirdoc}{(\widehat{W},\Gamma'',\chi)} % third orbifold chart

\title[Classification of Hamiltonian $S^1$-spaces]{Classification of
  Hamiltonian $S^1$-actions on compact symplectic
  orbifolds with isolated cyclic singular points in dimension four}

\author[L. Godinho]{Leonor Godinho}
\address{Departamento de Matem\'atica, Centro de An\'alise Matem\'atica, Geometria e Sistemas din\^a\-micos, Instituto Superior
T\'ecnico, Universidade de Lisboa, Av. Rovisco Pais, 1049-001 Lisbon, Portugal}

\email{lgodin@math.tecnico.ulisboa.pt}

\author[G. T. Mwakyoma-Oliveira]{Grace T. Mwakyoma-Oliveira}
\email{gracetmwakyoma@gmail.com}

\author[D. Sepe]{Daniele Sepe}
\address{Escuela de Matem\'aticas, Universidad Nacional de Colombia sede
  Medell\'in, Colombia, and Instituto de Matem\'atica e Estat\'istica, Departamento de
  Matem\'atica Aplicada (GMA), Universidade Federal Fluminense,
  Brazil.}
\email{dsepe@unal.edu.co}

\date{\today}

\begin{document}

% 53D20: momentum map, reduction theory
% 53D35: global theory of symplectic and contact manifolds
\begin{abstract}
In this paper, we classify Hamiltonian $S^1$-actions on compact, four dimensional
symplectic orbifolds that have isolated singular
points with cyclic orbifold structure groups, thus extending the
classification due to Karshon to the orbifold setting. To such a
space, we associated a combinatorial invariant, a labeled multigraph,
that determines the isomorphism type of the space. Moreover, we show
that any such space can be obtained by applying finitely many
equivariant weighted blow-ups to a minimal space, i.e., one on which
no equivariant weighted blow-down can be applied. 
\end{abstract}

\maketitle

\tableofcontents

\section*{Introduction}\label{sec:introduction}
In this paper we study Hamiltonian $S^1$-actions on compact, four dimensional
symplectic orbifolds that have isolated singular
points with cyclic orbifold structure groups. Intuitively, an {\bf orbifold} is a space that is locally
modeled on a quotient of Euclidean space by an effective action of a finite group (see Definition
\ref{defn:orbifold}). First introduced by Satake
in \cite{satake}, not only are they intrinsically interesting since
they are examples of `singular spaces', but in fact play an important
role in complex algebraic and symplectic geometry, and beyond. In the
former, they represent examples of `mild' singular behavior and have
been extensively studied in relation to {\bf desingularizations},
i.e., ways to resolve the singular locus to obtain a smooth complex
variety (see, for instance, \cite{reid}). In the latter, they arise naturally in the context of
symplectic reduction at a regular value of a moment map of a
Hamiltonian action. Thus, as stated in \cite{LermanTolman},
\begin{center}
  {\em `we need to understand symplectic orbifolds even if we only
    want to understand symplectic manifolds.'}
\end{center}
In this direction, in recent years there has been plenty of interest in
desingularizing symplectic orbifolds to construct
interesting examples of symplectic manifolds (see,
among others, \cite{fer_mun,fine_panov}). Finally, symplectic
orbifolds also arise naturally when studying the so-called almost
regular (or Besse) contact manifolds, as the latter are principal
$S^1$-orbi-bundles over the former (see \cite{kegel_lange}).

Let $M$ be an orbifold of dimension $n$ and let $x \in M$ be a point. Associated to $x$
is its {\bf orbifold structure group} $\Gamma$, which, using the above
intuitive definition, corresponds to the isotropy of the orbit of the
finite group action that $x$ represents (see Definition
\ref{defn:orbifold_str_group}). We say that $x$ is {\bf singular} if $\Gamma$
is not trivial and that $M$ has {\bf isolated singular points} if the
set of singular points of $M$ is discrete. In this paper, we are only
concerned with orbifolds with isolated singular points, the main
reason being that many of the subtleties of the differential topology
of orbifolds are made significantly simpler by imposing this
condition. For instance, for such orbifolds, the standard differential
topological notions of vector fields, differential forms, etc. can be
defined `naively'. Hence, the definitions of a symplectic manifold
and of a Hamiltonian action are the usual ones (see Definitions
\ref{defn:symplectic} and \ref{defn:hamiltonian}). 

Let $(M,\omega)$ be a compact, four dimensional
symplectic orbifold that has isolated singular
points and suppose that $H \colon M \to \R$ is the moment map for an
effective Hamiltonian $S^1$-action on $M$. In this paper, we study
such objects under the extra condition that all orbifold structure groups be {\bf cyclic}. This turns out to be a mild restriction, as
explained below. Let $x \in M$ be a singular
point with non-trivial orbifold structure group $\Gamma$. Since $M$ has only
isolated singular points, it follows that $x$ is a fixed point of the
action (see Corollary \ref{cor:action_preserves_structure_group}). Moreover, since $(M,\omega)$ is
symplectic and has dimension four, the group $\Gamma$ is (isomorphic
to) a subgroup of $\mathrm{U}(2)$ -- see Theorem
\ref{thm:darboux_orbifolds}. If $\Gamma$ is not abelian, by Lemma
\ref{lem::orbiweights} and Remark
\ref{rmk:non-abelian_local_extremum}, it follows that $x$ is an
isolated fixed point that is a local (hence global) extremum. Otherwise, since the singular set of $M$ is
isolated, it follows that $\Gamma$ has to be cyclic (see Lemma
\ref{lemma:cyclic_rep_isolated_sing}). In other words, by restricting
the orbifold structure groups to be cyclic, we are only excluding isolated
extrema with non-abelian structure groups. A Hamiltonian $S^1$-action
on $(M,\omega)$ with the above hypotheses is called a {\bf Hamiltonian
  $S^1$-space} and denoted by $(M,\omega,H)$. \\

The aim of this paper is to classify Hamiltonian $S^1$-spaces up to
the natural notion of isomorphism (see Definition
\ref{defn:hamiltonian}), thus generalizing the main results in
\cite{karshon} to the setting of orbifolds. Given a Hamiltonian $S^1$-space $(M,\omega,H)$, we associate a
{\bf labeled multigraph} to it that, when $M$ is a manifold, coincides with
the invariant introduced in \cite[Section 2.1]{karshon}. However,
the invariant that we introduce is necessarily a
multigraph even when all fixed points of the action
are isolated and edges that are incident to a vertex need not have
coprime labels (see, for instance, Example \ref{tswEx} and Figure \ref{fig::typeBspace1}).

\renewcommand{\thefigure}{A}
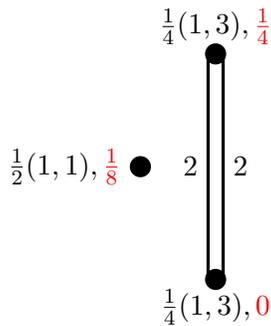
\begin{figure}[h!]
	\centering
	\begin{tikzpicture}	
		\coordinate (Q1) at (7,3);
		\fill (Q1) circle (4pt);
		\coordinate (Q2) at (7,0);
		\fill (Q2) circle (4pt);	
		
		\coordinate (Q0) at (6,1.5);
		\fill (Q0) circle (4pt);	
	
		\coordinate[label=left:$2$] (BC) at ($(6.9,1.5)$);	
		\coordinate[label=right:$2$] (BC) at ($(7.1,1.5)$);
		
		\coordinate[label=below:${\frac{1}{4}(1,3), \color{red}{0}}$] (BC) at ($(Q2)$);
		\coordinate[label=above:${\frac{1}{4}(1,3), \color{red}{\frac{1}{4}}}$] (BC) at ($(Q1)$);
		\coordinate[label=left:${\frac{1}{2}(1,1),\color{red} \frac{1}{8}}\,\,$] (BC) at ($(Q0)$);
		
		\draw[very thick](7.1,3) -- (7.1,0);
		\draw[very thick](6.9,3) -- (6.9,0);
	\end{tikzpicture}
\caption{Example of a labeled multigraph of a Hamiltonian $S^1$-space
  (see Example \ref{tswEx}).}
\label{fig::typeBspace1}
\end{figure}
Moreover, there are more
labels than in the setting of manifolds, reflecting the richness of
the (local) orbifold structure. Our first main
result states that the labeled multigraph is a complete invariant up
to isomorphism (see Lemma
\ref{lemma:trivial_direction} and Theorem \ref{thm:uniqueness}).

\begin{mr}\label{mr1}
  Two Hamiltonian $S^1$-spaces are isomorphic if and only if the
  associated labeled multigraphs are equal.
\end{mr}

Intuitively, our strategy to prove Main Result \ref{mr1} is to use an equivariant
desingularization via {\bf weighted blow-ups} to reduce the problem to
one between smooth manifolds, so that we can invoke \cite[Theorem
4.1]{karshon}. We believe that the existence of an equivariant desingularization
(see Theorem \ref{thm:desingularization}) should be of interest in its own
right, especially in light of its relation to Hirzebruch-Jung
continued fractions and, therefore, to the well-known Hirzebruch-Jung
resolution of cyclic quotient singularities in complex algebraic
geometry (see \cite{reid}).

Our second main result is concerned with determining the so-called
{\bf minimal spaces}, i.e., those Hamiltonian $S^1$-spaces on which no
equivariant weighted blow-down can be performed. Together with
equivariant weighted blow-ups, these can be thought of
as the fundamental building blocks of Hamiltonian $S^1$-spaces. More
precisely, the following holds (see Lemmas \ref{lemma:existsurface1}
and \ref{lemma:only_iso}, and Theorems
\ref{thm:existsurface} and \ref{thm:fixed}).

\begin{mr}\label{mr2}
  Any Hamiltonian $S^1$-space can be obtained by applying finitely
  many equivariant weighted blow-ups to a minimal space.
\end{mr}

We split the proof of Main Result \ref{mr2} in two cases, depending on
whether the given Hamiltonian $S^1$-space has at least one fixed orbi-surface (see Theorem
\ref{thm:existsurface}), or not (see Theorem \ref{thm:fixed}). In each
case, our strategy for the proof is
the same and consists of two steps. First, we use the
localization formula for orbifolds (see Theorem
\ref{prop:localization}) to determine the possible labeled multigraph
of the minimal models (see Theorems \ref{fig::minimalfixedsurface} and
\ref{thm:existisolated}). Second, we construct
Hamiltonian $S^1$-spaces that have as labeled multigraphs those
determined in the first step (see Sections \ref{MSF} and
\ref{MSI}). In all but two cases we construct the spaces explicitly; in the remaining two, we
obtain them from quotients of weighted projective spaces after
applying finitely many equivariant weighted blow-ups and
blow-downs. As a simple consequence of Main Result \ref{mr2}, we
exhibit minimal spaces with isolated fixed points for which
the $S^1$-action cannot be extended to a toric one (see Lemma
\ref{lemma:notextendTorus}). These include an example studied in
\cite{counterex}; moreover, this phenomenon should be contrasted with
\cite[Theorem 5.1]{karshon} in the case of manifolds. \\

We view this paper as a step towards studying
`large' Hamiltonian actions on symplectic orbifolds, extending the
work in the completely integrable case carried out in
\cite{LermanTolman}. As such, there are several potential extensions
of the present paper that we intend to work on in the future. These
include removing the (easier) assumption on the orbifold structure
group and the (harder) assumption on the singular set of the orbifold,
as well as computing the equivariant cohomology of these
spaces. Further ahead, there may be scope to extend the classification
of tall complexity one Hamiltonian torus actions in
\cite{karshon,kt2,karshontolman3} to the realm of orbifolds.

\subsection*{Structure of the paper} In Section \ref{ChapterPreliminaries}
we review standard differential topological
constructions for orbifolds, ranging from the definition and examples
of theses spaces
to maps between them, orbi-bundles, quotients, group actions and the study of
symplectic geometry for orbifolds. While most of the material in this
section is probably well-known to experts, we set the notation and
construct some fundamental examples that we use throughout the
paper. In Section \ref{sec:orbifolds} we
define orbifolds using local charts (but see also
\ref{rmk:groupoids}). We introduce the orbifold structure group of a 
point on an orbifold. Depending on whether this group is trivial or
not, the point is regular or singular respectively. We discuss maps
between orbifolds, coming to the notion of good maps introduced in
\cite[Section 4]{chen_ruan}. We conclude this section by recalling the
classification of compact, orientable two dimensional orbifolds with
isolated singular points, which is a simple extension of the analogous
result for manifolds. In Section \ref{sec:orbi-bundles} we discuss the
analog of fiber bundles in the setting of orbifolds, which are known
as orbi-bundles. We define these objects using local charts and
construct some examples of vector and principal orbi-bundles that are
used throughout the paper. Special attention is given to Seifert
fibrations, i.e., principal
$S^1$-orbi-bundles over compact, two dimensional orbifolds with
isolated singular points. In particular, we discuss the construction of the
associated Seifert invariant (see Definition
\ref{defn:seifert_invariant}). To conclude this section, we consider
orbifold covering maps which generalize topological covering maps. In
Section \ref{sec:suborbifolds} we introduce suborbifolds, placing
particular emphasis on full (or strong) embedded suborbifolds (see Definition
\ref{defn:suborbifold}). These can be thought of as being the
appropriate analog of embedded submanifolds in the setting of
orbifolds. For instance, such suborbifolds have a normal orbi-bundle
and there is a tubular neighborhood theorem (see Theorem
\ref{thm:tub_neigh_orb}). We discuss group actions on orbifolds in
Section \ref{sec:group-acti-orbif}. To start, we define them and
construct several examples that are useful for our purposes. As in the
case of manifolds, there are linearization results near fixed points
(see Corollary \ref{cor:linearization_fixed_point}). To conclude the
section, we state that finite quotients of orbifolds are again
orbifolds (see Corollary \ref{cor:quotient_orbifold}), and use this
result to describe cyclic quotients of weighted projective
planes. Once endowed with a suitable symplectic form and Hamiltonian
$S^1$-action, these underlie some of the minimal spaces that
we obtain. In Section \ref{sec:sympl-orbif-hamilt} we discuss aspects
of symplectic geometry on orbifolds, with a view to introduce Hamiltonian
group actions. These are very similar to the better known versions for
manifolds. We start by defining and giving examples of symplectic
orbifolds. The connection between symplectic forms and compatible
almost complex structures is discussed and then the Darboux theorem is
stated (see Theorem \ref{thm:darboux_orbifolds}). We use this result
to classify open neighborhoods of isolated singular points with cyclic
orbifold structure group in a four
dimensional symplectic orbifold (see Corollary
\ref{cor:sing_point_dim_4}). For our purposes, it is also useful to
recall the classification of compact symplectic surfaces with isolated
singular points (see Theorem
\ref{thm:classification_symplectic_orbi-surfaces}), as these appear in
our spaces as connected components of the fixed point set. We conclude
this section by defining Hamiltonian group actions, giving examples,
stating the equivariant Darboux and symplectic neighborhood theorems
(see Theorems \ref{thm:darboux} and \ref{thm:esnt}), and by stating
the connectedness result for fibers of the moment map of Hamiltonian
torus actions on a compact symplectic manifold (see Theorem
\ref{thm:connected_fibers}). 

In Section \ref{ChapterLocalForms} we take a closer look at the local
behavior of Hamiltonian $S^1$-actions near fixed points. We start in
Section \ref{sec:unit-repr-s1} by discussing the orbi-weights of a
fixed point of a Hamiltonian $S^1$-action. First, we prove a general
result about finite extensions of tori (see Theorem
\ref{thm::groupext}), which may be of independent interest and of use
in the study of Hamiltonian torus actions on orbifolds. The crucial
result that we need to define orbi-weights is Lemma
\ref{lem::orbiweights}. In Section \ref{sec:local-class-near} we prove
the local normal form for an isolated fixed point with cyclic orbifold
structure group of a Hamiltonian $S^1$-action on a four dimensional
symplectic orbifold (see Theorem \ref{cor::localformpt}). Moreover,
this section contains several technical results that we use in the
rest of the paper.

In Section \ref{sec:labelled-multigraph} we introduce the complete
combinatorial invariant of a Hamiltonian $S^1$-space, the 
labeled multigraph. As a first step, in Section
\ref{sec:isotr-orbi-spher} we discuss properties of the connected
component of the fixed point set of such a space, as well as the
so-called isotropy orbi-spheres, which are the closure of a connected
component of the set of points with finite stabilizer. Then we proceed
to define the desired invariant in Section
\ref{sec:label-mult-hamilt}. We also discuss some crucial examples
that, eventually, turn out to be minimal spaces (see Examples
\ref{ex:projective2} and \ref{tswEx}). Moreover, we prove the `easy'
direction of Main Result \ref{mr1} in Lemma
\ref{lemma:trivial_direction}. As a first partial converse to this
result, in Section \ref{sec:label-mult-determ} we show that the
labeled multigraph of a Hamiltonian $S^1$-space allows one to
reconstruct invariant open neighborhoods of all of the connected
components of the fixed point set (see Lemma
\ref{lemma:gamma_thin_vertex} and Proposition
\ref{prop:graph_seifert}). In the proof of Proposition
\ref{prop:graph_seifert} we make use of the localization formula (see
Theorem \ref{prop:localization}) and so we also give
a short introduction to equivariant cohomology for orbifolds.

Section \ref{ChapterEquivariantOperations} is the technical heart of
our paper, where we discuss the effects of equivariant weighted
blow-ups and blow-downs on the labeled multigraphs of Hamiltonian
$S^1$-spaces. In Section \ref{sec:sympl-weight-blow} we deal with
equivariant weighted blow-ups. We start by recalling the original
(equivariant) definition of this operation introduced in \cite{lgodinho,MS}, as well as
establishing some basic properties (see, for instance,
\ref{lemma:blow_up_preserves}). Then we turn to the two fundamental
technical results of the paper, Propositions \ref{prop::desing} and
\ref{prop::desing2}, which describe the effects of certain equivariant
weighted blow-ups on the labeled multigraph of a Hamiltonian
$S^1$-space. Using Proposition \ref{prop::desing}, we prove the key
result that every Hamiltonian $S^1$-space may be desingularized after
finitely many equivariant
weighted blow-ups (see Theorem \ref{thm:desingularization}). In Section
\ref{sec:wbd} we discuss equivariant weighted
blow-downs. This is a fundamental tool that we use to prove Main
Result \ref{mr2}. In particular, in this section we prove several
technical results that culminate in Corollary
\ref{cor::index2gradientsph}, which gives a sufficient condition to
blow-down an $S^1$-invariant orbi-sphere.

We prove Main Results \ref{mr1} and \ref{mr2} in Section
\ref{ChapterClassification}. The `hard' part of Main Result \ref{mr1}
is proved in Section \ref{uniqueness}, while Section
\ref{sec:existence} deals with Main Result \ref{mr2}. To prove the
latter, we need to introduce chains of $S^1$-invariant orbi-spheres,
which are entirely analogous to those considered in \cite[Section
3.1]{karshon}. Finally, we prove Main Result \ref{mr2} by showing a
sequence of results (see Lemma \ref{lemma:existsurface1}, Theorem
\ref{fig::minimalfixedsurface}, Lemma \ref{lemma:only_iso}, Theorem
\ref{thm:existisolated}, as well as Sections \ref{MSF} and
\ref{MSI}). We conclude this section by showing that certain minimal
models with only isolated fixed points are such that the $S^1$-action
cannot be extended to a toric one (see Lemma
\ref{lemma:notextendTorus}).

In the final section of this paper, Section \ref{sec:final-remarks},
we prove some additional results regarding Hamiltonian $S^1$-spaces,
namely that they are Kähler (see Theorem \ref{thm:kahler}), and that
some of them are birationally equivalent to a certain minimal model
(see Theorem \ref{thm:birationalequiv}), thus answering a question
posed to us by Margaret Symington. 

\subsection*{Conventions}\label{sec:conventions}
Given a positive integer $c > 1$, we denote the
cyclic group of order $c$ by $\Z_c$ and we identify it as the subgroup
of $S^1 = \{z \in \C \mid |z| = 1\}$ of $c$-th roots of unity. A
generator of $\Z_c$ is denoted by $\xi_c$. 
\subsection*{Acknowledgments} We would like to thank Margaret
Symington for interesting discussions. L.G. was partially funded by
FCT/Portugal through UID/MAT/04459/2020 and PTDC/MATPUR/29447/2017. G.T.M.O. was supported by FCT/Portugal through PD/BD/128420/2017. 
 D.S. was partially supported by FAPERJ
grant JCNE E-26/202.913/2019 and by a CAPES/Alexander von Humboldt
Fellowship for Experienced Researchers
88881.512955/2020-01. D.S. would like to thank Instituto Superior
Técnico for the kind hospitality during several short visits. This
study was financed in part by the Coordenação de Aperfeiçoamento de Pessoal de Nível Superior -- Brazil
(CAPES) -- Finance code 001.

\section{A primer on (symplectic) orbifolds, orbi-bundles and group actions}\label{ChapterPreliminaries}

In this section we review some background material on (symplectic)
orbifolds and group actions (see for example
\cite{adem,chen_ruan,HaefligerSalem,LermanTolman,Thurston} for
details). 

\subsection{Orbifolds}\label{sec:orbifolds}

\subsubsection{Definition and first examples}\label{sec:basic-definitions}

Throughout this section, we denote a topological vector space with
underlying set $M$ by $|M|$ and every subset of $|M|$ is endowed with
the relative topology.

\begin{definition}\label{defn:luc}
  Let $U \subseteq |M|$ be an open set. A {\bf local uniformizing chart} (l.u.c. for short) on $|M|$ is a triple
  $\oc$, where $\widehat{U}$ is a connected smooth manifold, $\Gamma
  \subset \mathrm{Diff}(\widehat{U})$ is a finite group 
  acting effectively, and $\varphi :
  \widehat{U} \to U$ is a  $\Gamma$-invariant continuous
  map such that the map $\bar{\varphi} : \widehat{U}/\Gamma \to
  U$ given by $\hat{\varphi}([\hat{x}]) = \varphi(\hat{x})$ is a homeomorphism. We say that $U$ is {\bf
    uniformized} by $\oc$.
\end{definition}

The above definition is equivalent to the standard one, in
which $\widehat{U}$ is assumed to be a connected open subset of
$\mathbb{R}^n$ (see \cite[Lemma 4.1.1]{chen_ruan}). Like charts for (topological) manifolds, l.u.c.'s need not be
unique. For this reason, it is useful to introduce the following
notions.

\begin{definition}\label{defn:embedding_isomorphism}
  Given two l.u.c.'s $\oc$ and $\otheroc$ on $|M|$,
  \begin{itemize}[leftmargin=*]
  \item an {\bf embedding} 
    $$\lambda : \oc \hookrightarrow \otheroc$$ 
    is a smooth open
    embedding $\lambda : \widehat{U} \to \widehat{V}$ such that $\psi \circ
    \lambda = \varphi$, and
  \item we say that $\oc$ and $\otheroc$ are {\bf isomorphic} if
    there exists an embedding $\lambda : \oc \hookrightarrow \otheroc$
    that is a diffeomorphism. 
  \end{itemize}
\end{definition}

\begin{remark}\label{rmk:embeddings_induce_homom}
  An embedding $\lambda : \oc \hookrightarrow \otheroc$ of l.u.c.'s
  induces an injective group homomorphism $\Gamma \to \Gamma'$ (which we
  also denote by $\lambda$), satisfying the following property:
  $$ \lambda(\gamma \cdot \hat{x}) = \lambda(\gamma) \cdot
  \lambda(\hat{x}) \text{ for all } \hat{x} \in \widehat{U} \text{ and
    for all } \gamma \in \Gamma. $$
  Moroever, the image of $\lambda : \Gamma \to \Gamma'$ can be
  characterized as follows:
  \begin{equation}
    \label{eq:9}
     \gamma' \in \mathrm{im}\, \lambda \Leftrightarrow
     \lambda(\widehat{U}) \cap \gamma' \cdot \lambda(\widehat{U}) \neq
     \emptyset, 
  \end{equation}
  (see \cite[Section 1.1]{adem}). In
  particular, if $\lambda$ is an isomorphism, then $\lambda: \Gamma
  \to \Gamma'$ is also an isomorphism. As a
  consequence, the above notion of isomorphism is equivalent to the notion of
  isomorphism of l.u.c.'s given in \cite[Section 4.1]{chen_ruan}.
\end{remark}

Given a l.u.c. $\oc$ on $|M|$ and $\gamma_0 \in \Gamma$,
the map $\lambda_{\gamma_0} : \hat{U} \to \hat{U}$ given by $\hat{x}
\mapsto \gamma_0 \cdot \hat{x}$ is an automorphism of $\oc$. In fact,
{\em any} automorphism of $\oc$ is of this form. This is the content
of the following result, stated below without proof (see \cite[Lemma 1]{satake_gauss} and
\cite[Proposition A.1]{mo_pr}).

\begin{lemma}\label{lemma:auto_luc}
  Let $\oc$ be a l.u.c. on $|M|$. Given an automorphism $\lambda : \oc \to \oc$, there exists a unique $\gamma_0 \in \Gamma$ such that
  $\lambda = \lambda_{\gamma_0}$.
\end{lemma}

\begin{remark}\label{rmk:induced_luc}
  The property of being uniformized behaves well under restriction to
  connected subsets. More precisely, let $U \subseteq |M|$ be an open
  set that is uniformized by the l.u.c. $\oc$ and let $U' \subseteq U$
  be a connected subset. If we set $\widehat{U}'$ to be a connected
  component of $\varphi^{-1}(U')$, $\Gamma'$ to be the subgroup that
  fixes $\widehat{U}'$ and $\varphi':= \varphi|_{\widehat{U}'}$, then
  $(\widehat{U}', \Gamma', \varphi')$ is a l.u.c. that uniformizes
  $U'$. Moreover, the isomorphism class of
  $(\widehat{U}', \Gamma', \varphi')$ does not depend on the choices
  made (see \cite[Lemma 4.1.1]{chen_ruan}). We say that
  $(\widehat{U}', \Gamma', \varphi')$ is the l.u.c. on
  $U' \subseteq U$ {\bf induced} by $\oc$.
\end{remark}

\begin{definition}\cite[Definition 1.1]{adem} \label{defn:orbifold}
    Let $\lvert M \rvert$ be a topological space. 
  \begin{enumerate}[label = (\arabic*), ref = (\arabic*), leftmargin=*]
  \item \label{item:3} An {\bf orbifold atlas} on $\lvert M \rvert$ is a collection
    $\mathcal{A} = \{\oc\}$ of l.u.c.'s such that
    $\{\varphi(\widehat{U})\}$ covers $\lvert M \rvert$ and
    satisfies the following compatibility condition: For every $\oc,\otheroc
    \in \mathcal{A}$ such that $U \cap V \neq \emptyset$ (where $U = \varphi(\widehat{U})$ and $V =
    \psi(\widehat{V})$) and every $x \in U \cap V$, there exist an open
    subset $W \subset U \cap V$ containing $x$, a l.u.c.
    $\thirdoc \in \mathcal{A}$ with $W = \chi(\widehat{W})$ and embeddings
    $\thirdoc \hookrightarrow \oc$ and $\thirdoc \hookrightarrow
    \otheroc$.
  \item \label{item:4} An orbifold atlas $\mathcal{B}$ on $\lvert M \rvert$ {\bf
      refines} the orbifold atlas $\mathcal{A}$ if every
    l.u.c. in $\mathcal{B}$ admits an embedding in some
    l.u.c. in $\mathcal{A}$. 
  \item \label{item:5} Two orbifold atlases $\mathcal{A}$ and
    $\mathcal{A}'$ of $\lvert M \rvert$ are {\bf equivalent} if they admit a common
    refinement. 
  \item \label{item:6} An {\bf (effective) orbifold} is a paracompact, Hausdorff
    topological space $\lvert M \rvert$ endowed with the equivalence class of an
    orbifold atlas. To simplify notation, we denote an orbifold by $M$.
  \end{enumerate}
\end{definition}

\begin{remark}\label{rmk:equivalent_definition_orbifold}
  Using the above notions of isomorphism of l.u.c.'s and of induced
  l.u.c.'s, it follows that Definition \ref{defn:orbifold} is
  equivalent to \cite[Definition 4.1.2]{chen_ruan}. In particular,
  given an orbifold $M$ and a point $x \in M$, there exists an open
  neighborhood $U$ of $x$ that is uniformized by some l.u.c. $\oc$;
  moreover the isomorphism class of the germ of $\oc$ at $x$ is unique.
\end{remark}

\begin{remark}\label{rmk:dim_orbifold}
  In analogy with the case of manifolds, if $M$ is an orbifold and
  $|M|$ is connected, then there is a well-defined notion of {\bf
    dimension} of $M$, which is given by the dimension of the domain
  of any l.u.c. on $|M|$. 
\end{remark}

The `simplest' families of orbifolds are illustrated in the following
two examples.

\begin{example}\label{exm:manifolds_as_orbifolds}
  Any (paracompact, Hausdorff) manifold is an orbifold: Given a smooth atlas
  $\mathcal{A}$ on a
  manifold $M$, we endow each
  coordinate chart $(U,\varphi)$ with the action of the trivial
  subgroup of $\mathrm{Diff}(U)$. \hfill$\Diamond$ 
\end{example}

\begin{example}[Orbi-vector spaces]\label{exm:orbi_vector_space}
  Let $\widehat{V}$ be a finite dimensional vector space over $\mathbb{K} \in
  \{\R, \C\}$ and let $\Gamma < \mathrm{GL}(\widehat{V})$ be a finite
  subgroup. The global quotient $\widehat{V}/\Gamma$ is an orbifold that is
  known as an {\bf orbi-vector space}. \hfill$\Diamond$ 
\end{example}

Example \ref{exm:orbi_vector_space} can be greatly generalized as follows. Let $G$ be a Lie group acting on (the left on) a manifold
$\widehat{M}$. We say that the action is  {\bf proper} if the
map $\phi:G \times \widehat{M} \to \widehat{M} \times \widehat{M}$ given by $\phi(g,p)= g \cdot p$ is proper, and
 that it is {\bf locally free} if all stabilizers are finite. If the
 $G$-action on $\widehat{M}$ is proper, then the quotient space $\widehat{M}/G$ is paracompact and Hausdorff (see \cite[Section
 2.5]{dk}). If the action is also locally free, then the Slice Theorem
 (see \cite[Theorem 2.4.1]{dk}),
 implies that $\widehat{M}/G$ can naturally be endowed with
 the structure of an orbifold (see \cite[Section 1.1]{adem}): A l.u.c. centered at a point $[p] \in \widehat{M}/G$ is given by the triple $(\widehat{U}, \Gamma, \varphi)$,
 where
 \begin{itemize}[leftmargin=*]
 \item $\widehat{U} \subseteq T_p\widehat{M}$ is an open neighborhood of
   the origin that is invariant under the the linear action of the isotropy
   group $G_p$,
 \item $\Gamma = \rho(G_p)$, where
   $\rho : G_p \to \mathrm{GL}(T_pM)$ is the slice
   representation\footnote{It is a consequence of the Slice Theorem
     that $\rho$ is injective if the $G$-action is effective.}, and
 \item $\varphi : \widehat{U} \to \widehat{M}/G$ is the
   composite of the quotient map $\widehat{U} \to
   \widehat{U}/\Gamma$ and the inclusion $
   \widehat{U}/\Gamma \hookrightarrow \widehat{M}/G$ arising from the
   Tube Theorem.
 \end{itemize}

Following \cite[Definitions 1.7 and 1.8]{adem}, we call
$\widehat{M}/G$ a(n {\bf effective}) {\bf quotient
  orbifold}; moreover, if $G$ is finite, we say that $
\widehat{M}/G$ is a(n {\bf effective}) {\bf global quotient}.
 
\begin{example}[Weighted projective spaces]\label{ex_weightedprj}	
  Fix an integer $n \geq 1$ and positive
  integers $m_0,\dots, m_{n}$ satisfying $\gcd(m_0,\dots,m_n)=1$. Consider the smooth function $H :
  \C^{n+1} \to \R$ given by
  $$H(z_0,\dots,z_{n})=\frac{1}{2} \sum\limits_{j=0}^nm_j \lvert z_j \rvert ^2.$$
  \noindent
  The level set $E:=H^{-1}(\frac{1}{2}\Pi_{j=0}^n \,\, m_j)$ is regular (and
  diffeomorphic to a $2n+1$ dimensional sphere). Moreover it is invariant under the smooth $S^1$-action on $\C^{n+1}$ given by
  \begin{equation}
    \label{eq:8}
    \begin{split}
      S^1\times\mathbb{C}^{n+1}&\rightarrow\mathbb{C}^{n+1}\\
    \left(\lambda,(z_0,\dots,z_{n})\right)&\mapsto(\lambda^{m_0}z_0,\dots,\lambda^{m_n}z_n).
    \end{split}
  \end{equation}
  \noindent
  The induced $S^1$-action on $E$ is proper,
  effective and locally free. Hence, the quotient 
  $$
  \mathbb{C}P^n(m_0,\dots,m_n):= E / S^1  $$ 
is an
orbifold, known in the literature as a {\bf weighted
  projective space} of (complex) dimension $n$.
\hfill$\Diamond$

\renewcommand{\thefigure}{\thesection.\arabic{figure}}
\begin{figure}[h]
	\centering
	\includegraphics[width=2.5cm]{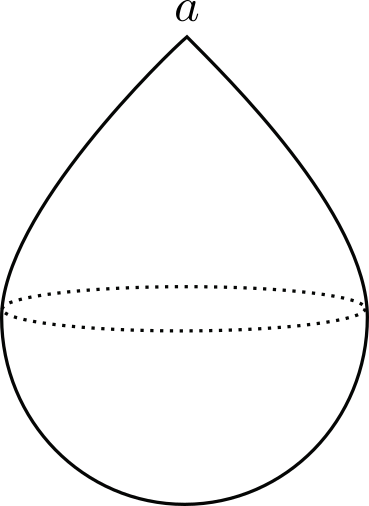}
	\caption{$\mathbb{C}P^1(1,a)$, also called a teardrop. }
	\label{fig:teardrop}
\end{figure}
\end{example}

\begin{example}\label{tswEx_1}
  The following example was constructed by Singer, Talvacchia and
  Watson in \cite{counterex}, and is carried out
  throughout the text. The \color{black}
  effective action of
  $\mathbb{Z}_4$  on $\widehat{M}:=\mathbb{C}P^1\times\mathbb{C}P^1$
  uniquely determined by 
  \begin{equation}
    \label{eq:29}
    e^{\frac{\pi \imath}{2}} \cdot \left([z_0:z_1],[w_0:w_1]\right) = \left([-w_0:w_1],[z_0:z_1]\right),
  \end{equation}
  is proper and locally free. Hence, the
  quotient $M:= \widehat{M}/\mathbb{Z}_4$ inherits the
  structure of an orbifold. 	\hfill$\Diamond$ 

  % It has two fixed points\color{red}, namely $\left([0:1],[0:1]\right)\quad
  % \text{and} \quad \left([1:0],[1:0]\right)$, an orbit given by
  % $\{ \left([1:0],[0:1]\right), \left([0:1],[1:0]\right)\}$ with
  % stabilizer equal to $\mathbb{Z}_2$, and all other points have
  % trivial stabilizers. \color{black}
  % \sout{Moreover, the points $\left([1:0],[0:1]\right)$ and
  %   $\left([0:1],[1:0]\right)$ are fixed by the subgroup
  %   $\mathbb{Z}_2<\mathbb{Z}_4$. These are the only points
  %   $\widetilde{M}$ that have non-trivial stabilizer.}  Therefore, with three isolated singularities: two points
  % with structure group of  order 4 and one of order 2. 
\end{example}

\begin{example}\label{ex:quotient_cp1xcp1_finite}
  Let $c > 1$ be an integer and fix $1 \leq l < c$ such that
  $\gcd(l,c)=1$. The effective action of $\mathbb{Z}_c$ on
  $\mathbb{C}P^1\times\mathbb{C}P^1$ given by  
  \begin{equation}
    \label{eq:30}
    \lambda \cdot \left([z_0: z_1],[w_0 : w_1]\right)=\left([z_0: \lambda z_1], [ w_0:  \lambda^{-l} w_1]\right), 
  \end{equation}
  is proper and locally free. Hence, the quotient $M_c:=
  \mathbb{C}P^1\times\mathbb{C}P^1/\mathbb{Z}_c$ inherits the structure of an
  orbifold. In order to stress the dependence on $c$, we denote the
  elements of $M_c$ by $[[z_0:z_1],[w_0,w_1]]_c$. \hfill $\Diamond$
\end{example}

To conclude this section, we introduce the notion of orientation and
orientability for orbifolds. To this end, we fix once and for all the
standard orientation on $\R^n$. 

\begin{definition}\label{defn:orientability}
  Let $|M|$ be a topological space.
  \begin{itemize}[leftmargin=*]
  \item A l.u.c. $\oc$ on $|M|$ is {\bf oriented} if $\Gamma$ is a
    subgroup of the group of orientation-preserving diffeomorphisms of
    $\R^n$.
  \item An embedding $\lambda : \oc \to \otheroc$ between oriented l.u.c.'s on
    $|M|$ is {\bf orientation-preserving} if $\lambda : \widehat{U}
    \to \widehat{V}$ is orientation-preserving.
  \item An orbifold atlas $\mathcal{A} =\{\oc\}$ on $|M|$ is {\bf oriented} if
    each l.u.c. $\oc$ in $\mathcal{A}$ is oriented and the embeddings
    of part \ref{item:3} of Definition \ref{defn:orbifold} are
    orientation-preserving.
  \item An orbifold $M$ is {\bf orientable} if it admits an oriented
    orbifold atlas. In this case, an {\bf orientation} on $M$ is a
    choice of (equivalence class of) such an orbifold atlas.
  \end{itemize}
\end{definition}

\begin{remark}\label{rmk:orient}
  \mbox{}
  \begin{itemize}[leftmargin=*]
  \item A manifold is orientable as an orbifold if and only if it is
    orientable as a manifold.
  \item An orbi-vector space $\widehat{V}/\Gamma$ is orientable if and only if
    $\Gamma$ is a subgroup of $\mathrm{SL}(\widehat{V})$.
  \item Let $\widehat{M}$ be an oriented manifold and let $G$ be a Lie
    group acting on $\widehat{M}$ properly, in a locally free fashion,
    and by orientation-preserving diffeomorphisms. The quotient orbifold
    $\widehat{M}/G$ inherits a natural orientation. 
  \end{itemize}
\end{remark}

\begin{remark}\label{rmk:product_orbifolds}
  If $M_1$ and $M_2$ are orbifolds, there is a natural equivalence
  class of orbifold atlases on the product $|M_1| \times |M_2| =|M_1
  \times M_2|$: If
  $\mathcal{A}_1 = \{(\widehat{U}_1,\Gamma_1,\varphi_1)\}$ and
  $\mathcal{A}_2 = \{(\widehat{U}_2,\Gamma_2,\varphi_2)\}$ are
  orbifold atlases representing the orbifold structures on $|M_1|$ and
  $|M_2|$ respectively, the orbifold structure on $|M_1 \times M_2|$
  is represented by the `product' atlas
  $$\mathcal{A}_1 \times
  \mathcal{A}_2 = \{(\widehat{U}_1 \times \widehat{U}_2, \Gamma_1
  \times \Gamma_2, \varphi_1 \times \varphi_2)\}.$$
  The resulting orbifold is denoted by $M_1 \times M_2$. 
  Moreover, if both $M_1$ and $M_2$ are orientable, so is $M_1 \times
  M_2$. 
\end{remark}

% \begin{remark}\label{rmk:all_orbifolds_quotient}
%   \color{red}
%   In analogy with the case of manifolds, an orbifold $M$ admits a {\em
%     frame bundle} $\mathrm{Fr}(M)$ (see \cite[Definition 1.22]{adem}); this is a {\em
%     manifold} that is endowed with a locally free action of a compact
%   group $G$ with the property that $M$ is diffeomorphic to the quotient
%   $G \backslash \mathrm{Fr}(M)$ (see \cite[Theorem 1.23]{adem}). In
%   particular, every orbifold is, up to diffeomorphism, a global quotient.
  
%   \color{black}
%   \sout{The frame bundle $\mathcal{F}(M)$ of an orbifold $M$ is a smooth
%   manifold equipped with a smooth, effective and locally free action of
%   $O(n)$. Then $M$ can be naturally identified with the  quotient orbifold $\mathcal{F}(M)/O(n)$ and so 
%   every orbifold is a quotient orbifold. Note that the manifold and the group action from which we can obtain a given orbifold as a quotient are not unique.}
% \end{remark}

\subsubsection{Orbifold structure group}\label{sec:orbif-struct-group}
Let $M$ be an orbifold, fix a point $x \in \lvert M \rvert$ and let $\oc$ be an
l.u.c. such that $x \in U = \varphi(\widehat{U})$. The stabilizers
of any two points $\hat{x},\hat{x}' \in \widehat{U}$ such that
$\varphi(\hat{x}) = x = \varphi(\hat{x}')$ are conjugate
as subgroups of $\Gamma$. We denote the underlying abstract group by $\Gamma_x\oc$. If $\otheroc$ is an orbifold
chart satisfying the above properties, then
$\Gamma_x \oc \cong \Gamma_x \otheroc$ -- for a proof,
see the discussion preceding \cite[Definition 1.5]{adem}. 

\begin{definition}\label{defn:orbifold_str_group}
  Let $M$ be an orbifold. Given a point $x \in \lvert M
  \rvert$, let $\oc$ be an
  l.u.c. such that $x \in \varphi(\widehat{U})$. We call
  $\Gamma_x \oc$  the {\bf orbifold structure group of
    $\boldsymbol{x}$} and denote it  simply by $\Gamma_x$. 
\end{definition}

\begin{example}(Orbi-vector spaces, continued)\label{exm:orbi-vector_space_origin}
  Let $\widehat{V}/\Gamma$ be an orbi-vector space in the sense of Example
  \ref{exm:orbi_vector_space}. The
  orbifold structure group of $[0] \in \widehat{V}/\Gamma$ is $\Gamma$. \hfill$\Diamond$ 
\end{example}

\begin{remark}\label{rmk:product_orbifold_structure_group}
  For $i=1,2$, let $M_i$ be an orbifold and let $x_i \in M_i$ be a point with
  orbifold structure group $\Gamma_i$. The point $(x_1,x_2) \in M_1
  \times M_2$ has orbifold structure group given by $\Gamma_1 \times \Gamma_2$.
\end{remark}

\begin{remark}\label{rmk:centred}
  By the Slice Theorem \color{black} (see \cite[Theorem
  2.4.1]{dk}) and since l.u.c.'s behave well with respect
  to restriction (see Remark \ref{rmk:induced_luc}), for every $x \in \lvert M \rvert$, there exists
  a l.u.c. $\oc$ with the
  following properties
  \begin{itemize}[leftmargin=*]
  \item $\Gamma$ is the orbifold structure group of $x$ and $\Gamma <
    \mathrm{O}(n)$,
  \item the open subset $\widehat{U}$ of $\R^n$ is
    invariant under the standard linear action of $\mathrm{O}(n)$, and
  \item the point $0 \in \widehat{U}$ satisfies $\varphi(0) = x$.
  \end{itemize}
  In this case, we say that $\oc$ is a l.u.c. {\bf centered
    at $\boldsymbol{x}$}. If, in addition, $M$ is
  oriented, then $\Gamma$ is a subgroup of $\mathrm{SO}(n)$.
\end{remark}

\begin{definition}\label{defn:singular_point}
  Let $M$ be an effective orbifold. We say that a point $x \in |M|$ is
  {\bf singular} if $\Gamma_x \neq 1$, and {\bf regular} otherwise. 
\end{definition}

\begin{remark}\label{rmk:smooth_dense}
  Since the group action in any l.u.c. is assumed to be effective, it follows that the set of regular points of
  an effective orbifold is necessarily open and dense.
\end{remark}

\begin{remark}\label{rmk:orbifolds_no_singular_points}
  Remark \ref{rmk:centred} implies that any orbifold that has no
  singular point can be viewed as a manifold, the smooth atlas being
  given by (the equivalence class of) the union of the l.u.c.'s
  centered at each point.
\end{remark}

\begin{definition}\label{defn:isolated}
  We say that an orbifold $M$ has {\bf isolated singular
    points} if the set of singular points is discrete.
\end{definition}

\begin{example}[Weighted projective spaces, continued]\label{exm:weighted_isolated}
  If, in Example \ref{ex_weightedprj}, the positive integers
  $m_0,\ldots, m_n$ are chosen to be pairwise
  coprime, then the weighted
  projective space $\mathbb{C}P^n(m_0,\dots,m_n)$ has isolated
  singular points. These are $[(2\prod\limits_{j \neq 0} m_j)^{1/2}:0:\ldots:0] ,[0: (2\prod\limits_{j \neq 1} m_j)^{1/2}:\ldots:0], \ldots,
  [0:0:\ldots: (2\prod\limits_{j \neq n} m_j)^{1/2}]$ with orbifold structure groups given by
  $\mathbb{Z}_{m_0},\ldots, \mathbb{Z}_{m_n}$ respectively.  \hfill$\Diamond$ 
\end{example}

\begin{example}[Example \ref{tswEx_1}, continued]\label{exm:singer}
  The orbifold of Example \ref{tswEx_1} has isolated singular
  points. These are $[[0:1],[0:1]], [[1:0],[1:0]]$ and $[[0:1],[1:0]]$
  with orbifold structure groups given by $\mathbb{Z}_4$,
  $\mathbb{Z}_4$ and $\mathbb{Z}_2$ respectively. 	\hfill$\Diamond$ 
\end{example}

\begin{example}[Example \ref{ex:quotient_cp1xcp1_finite}, continued]\label{exm:finite_quotient_cp1xcp1}
  The orbifold of Example \ref{ex:quotient_cp1xcp1_finite} has
  isolated singular points. These are $[[1:0], [1:0]]_c, [[1:0], [0:1]]_c,
  [[0:1], [1:0]]_c$ and $[[0:1], [0:1]]_c$ and all have orbifold structure
  group given by $\mathbb{Z}_c$. \hfill$\Diamond$ 
\end{example}

\subsubsection{Maps between orbifolds}\label{sec:maps-betw-orbif}

In this section, we define a certain class of maps between
orbifolds. We start by considering what happens for fixed l.u.c.'s. 

\begin{definition}\label{defn:c_infty_lift_luc}
  Let $|M_i|$ be a topological space and let $U_i \subseteq |M_i|$ be an
  open set that is uniformized by the l.u.c. $(\widehat{U}_i,
  \Gamma_i,\varphi_i)$ for $i=1,2$. Let $f: U_1 \to U_2$ be a
  continuous map.
  \begin{itemize}[leftmargin=*]
  \item A {\bf $\boldsymbol{C^{\infty}}$ lift} of $f$ is a
    $C^{\infty}$ map $\widehat{f} : \widehat{U}_1 \to \widehat{U}_2$
    such that:
    \begin{itemize}[leftmargin=*]
    \item $\varphi_2 \circ \widehat{f} = f \circ \varphi_1$, and
    \item for all $\gamma_1 \in \Gamma_1$, there exists a
      $\gamma_2 \in \Gamma_2$ such that
      $\widehat{f}(\gamma_1 \cdot \hat{x}) = \gamma_2 \cdot \widehat{f}(\hat{x})$
      for all $\hat{x} \in \widehat{U}_1$.
    \end{itemize}
  \item A {\bf good $\boldsymbol{C^{\infty}}$ lift} of $f$ is a pair
    $(\widehat{f},\Theta)$ such that:
    \begin{itemize}[leftmargin=*]
    \item $\widehat{f} : \widehat{U}_1 \to \widehat{U}_2$
      is a $C^{\infty}$ map with $\varphi_2 \circ \widehat{f} = f \circ \varphi_1$, and
    \item $\Theta : \Gamma_1 \to \Gamma_2$ is a homomorphism making 
      $\hat{f}$ is $\Theta$-equivariant. 
    \end{itemize}
  \end{itemize}
\end{definition}

\begin{remark}\label{rmk:maps_not_equivalent}
  It is immediate from Definition \ref{defn:c_infty_lift_luc} that the
  $C^{\infty}$ map of a good $C^{\infty}$ lift is, in fact, a $C^{\infty}$ lift. However,
  the converse need not hold (see \cite[Example 4.4.2a]{chen_ruan}).
\end{remark}

Just as l.u.c.'s are not unique, neither are (good) $C^{\infty}$ lifts of
continuous maps between l.u.c.'s (whenever they exist).

\begin{remark}\label{rmk:lifts_not_unique}
  Let $|M_i|$ be a topological space and let $U_i \subseteq |M_i|$ be an
  open set that is uniformized by the l.u.c. $(\widehat{U}_i,
  \Gamma_i,\varphi_i)$ for $i=1,2$. Let $f: U_1 \to U_2$ be a
  continuous map and let $\widehat{f}$ be a $C^{\infty}$ lift of $f$. If $\lambda:
  (\widehat{U}'_1, \Gamma'_1, \varphi'_1) \hookrightarrow (\widehat{U}_1,
  \Gamma_1,\varphi_1)$ (respectively $\lambda:
  (\widehat{U}'_2, \Gamma'_2, \varphi'_2) \hookrightarrow (\widehat{U}_2,
  \Gamma_2,\varphi_2)$) is an isomorphism in the sense of Definition
  \ref{defn:embedding_isomorphism}, it follows from Remark
  \ref{rmk:embeddings_induce_homom} that $\widehat{f} \circ \lambda$
  (respectively $\lambda^{-1} \circ \widehat{f}$) is a
  $C^{\infty}$ lift of $f$. Analogously, if $(\widehat{f},\Theta)$ is
  a good $C^{\infty}$ lift of $f$, then by Remark
  \ref{rmk:embeddings_induce_homom} (and, in particular,
  \eqref{eq:9}), $(\widehat{f} \circ \lambda, \Theta \circ \lambda)$
  (respectively $(\lambda^{-1} \circ \widehat{f}, \lambda^{-1} \circ
  \Theta)$) is a good $C^{\infty}$ lift of $f$.
\end{remark}

In light of Remark \ref{rmk:lifts_not_unique}, we introduce the following equivalence relation between (good)
$C^{\infty}$ lifts of
continuous maps between l.u.c.'s.

\begin{definition}\label{defn:isomorphic_lifts}
  Let $|M_i|$ be a topological space and let $U_i \subseteq |M_i|$ be an
  open set that is uniformized by the l.u.c.'s $(\widehat{U}_i,
  \Gamma_i,\varphi_i)$ and $(\widehat{U}'_i,
  \Gamma'_i,\varphi'_i)$ for $i=1,2$. Let $f :U_1 \to U_2$ be a
  continuous map. We say that 
  \begin{itemize}[leftmargin=*]
  \item two $C^{\infty}$ lifts $\widehat{f} :
    \widehat{U}_1 \to \widehat{U}_2$ and  $\widehat{f}' :
    \widehat{U}'_1 \to \widehat{U}'_2$ of $f$ are {\bf isomorphic} if there exist isomorphisms $\lambda_1 : (\widehat{U}_1,
    \Gamma_1,\varphi_1) \to (\widehat{U}'_1,
    \Gamma'_1,\varphi'_1)$ and $\lambda_2 : (\widehat{U}_2,
    \Gamma_2,\varphi_2) \to (\widehat{U}'_2,
    \Gamma'_2,\varphi'_2)$ of l.u.c.'s such that $\widehat{f}' \circ \lambda_1 =
    \lambda_2 \circ \widehat{f}$, and
  \item two good  $C^{\infty}$ lifts $(\widehat{f}, \Theta: \Gamma_1
    \to \Gamma_2)$ and $(\widehat{f}', \Theta': \Gamma'_1
    \to \Gamma'_2)$ are {\bf isomorphic} if there exist isomorphisms $\lambda_1 : (\widehat{U}_1,
    \Gamma_1,\varphi_1) \to (\widehat{U}'_1,
    \Gamma'_1,\varphi'_1)$ and $\lambda_2 : (\widehat{U}_2,
    \Gamma_2,\varphi_2) \to (\widehat{U}'_2,
    \Gamma'_2,\varphi'_2)$ of l.u.c.'s such that $\widehat{f}' \circ \lambda_1 =
    \lambda_2 \circ \widehat{f}$ and $\Theta' \circ \lambda_1 =
    \lambda_2 \circ \Theta$.
  \end{itemize}
\end{definition}

\begin{remark}\label{rmk:restriction_lifts}
  Whenever they exist, (good) $C^{\infty}$ lifts behave well with respect to
  restriction to connected open subsets. More precisely, let $|M_i|$ be a topological space
  and let $V_i \subseteq |M_i|$ be an open set that is uniformized by
  the l.u.c. $(\widehat{V}_i, \Gamma'_i,\psi_i)$ for $i=1,2$. Let
  $f: V_1 \to V_2$ be a continuous map that has a $C^{\infty}$ lift
  $\widehat{f} : \widehat{V}_1 \to \widehat{V}_2$. Let $U_i \subseteq V_i$ be an
  open set that is uniformized by $(\widehat{U}_i, \Gamma_i,
  \varphi_i)$ for $i=1,2$ such that $f(U_1) \subseteq U_2$.  Suppose that we are given an embedding
  $\lambda_1 : (\widehat{U}_1, \Gamma_1, \varphi_1) \hookrightarrow
  (\hat{V}_1, \Gamma'_1,\psi_1)$. By construction,
  $\psi_2(\hat{f} \circ \lambda_1(\widehat{U}_1)) \subseteq
  U_2$. Therefore, there exists an
  embedding
  $\lambda_2 : (\widehat{U}_2,\Gamma_2,\varphi_2) \hookrightarrow
  (\hat{V}_2, \Gamma'_2,\psi_2)$ such that
  $\hat{f} \circ \lambda_1(\widehat{U}_1) \subseteq \lambda_2(\widehat{U}_2)$
  (see \cite[Section 4.1]{chen_ruan}). Using Remark
  \ref{rmk:embeddings_induce_homom} (and, in particular,
  \eqref{eq:9}), it follows that the map
  $\lambda_2^{-1} \circ \widehat{f} \circ \lambda_1 : \widehat{U}_1 \to
  \widehat{U}_2$ is a $C^{\infty}$ lift of the restriction of $f$ to
  $U_1$. Analogously, if $(\widehat{f}, \Theta)$ is a good
  $C^{\infty}$ lift of $f$, then $(\lambda_2^{-1} \circ \widehat{f}
  \circ \lambda_1, \lambda_2^{-1} \circ \Theta \circ \lambda_1)$ is a
  good $C^{\infty}$ lift of the restriction of $f$ to
  $U_1$. We refer to the $C^{\infty}$ lift (respectively good $C^{\infty}$ lift) constructed above as the {\bf restriction} of
  $\widehat{f}$ (respectively $(\widehat{f}, \Theta)$) to $\widehat{U}_1$.
\end{remark}

Following \cite[Section 4]{chen_ruan}, we break the definition of the
maps of our interest in two parts. First, we introduce $C^{\infty}$
lifts of maps between orbifolds. 

\begin{definition}\label{defn:c_infty_lift}
  Let $M_1, M_2$ be orbifolds and let $f : |M_1| \to |M_2|$ be a
  continuous map. A {\bf $\boldsymbol{C}^{\infty}$ lift} of $f$
  consists of the following data:
  \begin{itemize}[leftmargin=*]
  \item \label{item:1} orbifold atlases $\mathcal{A}_1 =
    \{(\widehat{U}_1,\Gamma_1,\varphi_1)\}$ and $\mathcal{A}_2 =
    \{(\widehat{U}_2,\Gamma_2,\varphi_2)\}$ such that, for all
    $(\widehat{U}_1,\Gamma_1,\varphi_1) \in \mathcal{A}_1$, there
    exists $(\widehat{U}_2,\Gamma_2,\varphi_2) \in \mathcal{A}_2$ with
    $f(\varphi(\widehat{U}_1)) \subseteq \varphi_2(\widehat{U}_2)$,
  \item for all  $(\widehat{U}_1,\Gamma_1,\varphi_1) \in
    \mathcal{A}_1$ and $(\widehat{U}_2,\Gamma_2,\varphi_2) \in
    \mathcal{A}_2$ such that
    $f(\varphi(\widehat{U}_1)) \subseteq \varphi_2(\widehat{U}_2)$, a
    $C^{\infty}$ lift of the restriction of $f$ to $\varphi_1
    (\widehat{U}_1)$,
  \end{itemize}
  satisfying the following compatibility condition:

  \noindent
  Given $(\widehat{U}_1,\Gamma_1,\varphi_1), (\widehat{V}_1,\Gamma'_1,\psi_1)
  \in \mathcal{A}_1$ with $U_1 \cap V_1 \neq \emptyset$ (where $U_1 = \varphi(\widehat{U_1})$ and $V_1 =
  \psi(\widehat{V_1})$), and given a point $x \in U_1 \cap V_1$,
  there exist
  \begin{itemize}[leftmargin=*]
  \item an open subset $W_1 \subseteq U_1 \cap V_1$ 
    uniformized by a l.u.c. $(\widehat{W}_1,\Gamma''_1,\chi_1) \in
    \mathcal{A}_1$ containing $x$,
  \item an open subset $W_2 \subseteq |M_2|$ uniformized by
    $(\widehat{W}_2, \Gamma_2'', \chi_2)$ such that
    $f(W_1) \subseteq W_2$, and
  \item embeddings $(\widehat{W}_1,\Gamma''_1,\chi_1) \hookrightarrow
    (\widehat{U}_1,\Gamma_1,\varphi_1)$ and
    $(\widehat{W}_1,\Gamma''_1,\chi_1) \hookrightarrow
    (\widehat{V}_1,\Gamma'_1,\psi_1)$, 
  \end{itemize}
  such that the $C^{\infty}$ lift of the restriction
  of $f$ to $W_1$ is isomorphic to the restrictions of the
  $C^{\infty}$ lifts of $f|_{U_1}$ and $f|_{V_1}$ to $W_1$.
  We denote a $C^{\infty}$ lift of $f$ by the triple
  $(\mathcal{A}_1,\mathcal{A}_2,\{\widehat{f}_{\widehat{U}_1} : \widehat{U}_1 \to
  \widehat{U}_2\})$. 
\end{definition}

In order to make Definition \ref{defn:c_infty_lift} independent of
choices, we introduce the following notion.

\begin{definition}\label{defn:equivalent_lifts}
  Let $M_1, M_2$ be orbifolds and let $f : |M_1| \to |M_2|$ be a
  continuous map. We say that the $C^{\infty}$ lifts $(\mathcal{A}_1,\mathcal{A}_2,\{\widehat{f}_{\widehat{U}_1} : \widehat{U}_1 \to
  \widehat{U}_2\})$ and $(\mathcal{A}'_1,\mathcal{A}'_2,\{\widehat{f}'_{\widehat{U}'_1} : \widehat{U}'_1 \to
  \widehat{U}'_2\})$ of $f$ are {\bf equivalent} if there exists a
  $C^{\infty}$ lift $(\mathcal{B}_1,\mathcal{B}_2,\{\widehat{f}_{\widehat{V}_1} : \widehat{V}_1 \to
  \widehat{V}_2\})$ such that
  \begin{itemize}[leftmargin=*]
  \item the orbifold atlas $\mathcal{B}_i$ is a common refinement of
    $\mathcal{A}_i$ and $\mathcal{A}_i'$ for $i=1,2$, and
  \item for all l.u.c.'s $(\widehat{V}_1,
    \tilde{\Gamma}_1,\tilde{\varphi}_1)$ and for all embeddings $(\widehat{V}_1,
    \tilde{\Gamma}_1,\tilde{\varphi}_1) \hookrightarrow
    (\widehat{U}_1,\Gamma_1,\varphi_1)$ and $(\widehat{V}_1,
    \tilde{\Gamma}_1,\tilde{\varphi}_1) \hookrightarrow
    (\widehat{U}'_1,\Gamma'_1,\varphi'_1)$ , the $C^{\infty}$ lift
    $\widehat{f}_{\widehat{V}_1}$ of the restriction of $f$ to $V_1 =
    \tilde{\varphi}_1(\widehat{V}_1)$ is equivalent to the restrictions of
    $\widehat{f}_{\widehat{U}_1}$ and $\widehat{f}'_{\widehat{U}'_1}$
    to $V_1$. 
  \end{itemize}
\end{definition}

Using Definitions \ref{defn:c_infty_lift} and
\ref{defn:equivalent_lifts}, we can introduce smooth maps between
orbifolds.

\begin{definition}\label{defn:Cinfty_maps}
  Let $M_1$ and $M_2$ be orbifolds. A {\bf $\boldsymbol{C^{\infty}}$
    map} between $M_1$ and $M_2$ is a continuous map $f : |M_1| \to
  |M_2|$ together with an equivalence class of $C^{\infty}$ lifts of
  $f$. For simplicity, we denote a $C^{\infty}$ map by $f : M_1 \to
  M_2$. 
\end{definition}

Unfortunately, $C^{\infty}$ maps in the sense of Definition
\ref{defn:Cinfty_maps} do not have good geometric properties (for
instance, in general they do not pull vector bundles back to vector bundles). To
rectify this issue, following \cite[Section 4.4]{chen_ruan}, we consider those $C^{\infty}$ maps that have the
following type of representative in the equivalence class of $C^{\infty}$ lifts.

\begin{definition}\label{defn:compatible_sys}
   Let $M_1, M_2$ be orbifolds and let $f : |M_1| \to |M_2|$ be a
  continuous map. A $C^{\infty}$ lift $(\mathcal{A}_1,\mathcal{A}_2,\{\widehat{f}_{\widehat{U}_1} : \widehat{U}_1 \to
  \widehat{U}_2\})$ of $f$ is {\bf good} if
  \begin{itemize}[leftmargin=*]
  \item for all $(\widehat{U}_1,\Gamma_1,\varphi_1) \in
    \mathcal{A}_1$, there exists a homomorphism $\Theta_{\widehat{U}_1} :
    \Gamma_1 \to \Gamma_2$ such that $(\widehat{f}_{\widehat{U}_1},
    \Theta_{\widehat{U}_1})$ is a good $C^{\infty}$ lift of the
    restriction of $f$ to $U_1$,
  \item if $\lambda_1 :
    (\widehat{U}'_1,\Gamma'_1,\varphi'_1) \hookrightarrow
    (\widehat{U}_1,\Gamma_1,\varphi_1)$ is an embedding between
    l.u.c.'s in $\mathcal{A}_1$, then there
    exists an embedding $\lambda_2 : (\widehat{U}'_2,\Gamma'_2,\varphi'_2) \hookrightarrow
    (\widehat{U}_2,\Gamma_2,\varphi_2)$ between the corresponding
    l.u.c.'s in $\mathcal{A}_2$ such that the restriction of $(\widehat{f}_{\widehat{U}_1},
    \Theta_{\widehat{U}_1})$ to $\widehat{U}'_1$ equals $(\widehat{f}_{\widehat{U}'_1},
    \Theta_{\widehat{U}'_1})$. 
  \end{itemize}
\end{definition}

\begin{definition}\label{defn:good_map}
  Let $M_1$ and $M_2$ be orbifolds.
  \begin{itemize}[leftmargin=*]
  \item A $C^{\infty}$ map $f : M_1 \to M_2$ is {\bf good} if there
    exists a good $C^{\infty}$ lift in the equivalence class of $f$.
  \item A {\bf diffeomorphism} between $M_1$ and $M_2$ is a good
    $C^{\infty}$ map $f : M_1 \to M_2$ that admits an inverse $g : M_2
    \to M_1$ that is also a good $C^{\infty}$ map.
  \end{itemize}
\end{definition}

\begin{remark}[Good maps via Lie groupoids]\label{rmk:groupoids}
  The above definition of good $C^{\infty}$ maps is somewhat
  cumbersome and can be made more transparent using the approach to
  orbifolds using {\em Lie groupoids} (see \cite[Section 1.4]{adem}
  and references therein for more
  details). Loosely speaking, an orbifold $M$ is an
  equivalence class of certain Lie groupoids, where the topological
  space $|M|$ is homeomorphic to the orbit space of any representative
  in the equivalence class. (The equivalence relation is under the
  so-called {\em Morita equivalences}, see \cite[Definition
  1.43]{adem}.) In the case of a global quotient
  $\widehat{M}/G$, a Lie groupoid representing $\widehat{M}/G$ is the
  action groupoid $G \ltimes \widehat{M}$. The orbifold structure
  group at a point $x \in M$ corresponds to the isotropy Lie group of
  the orbit representing $x$ with respect to any Lie groupoid
  representing $M$.

  From this perspective, if
  $M_1,M_2$ are orbifolds, a
  continuous map $f : |M_1| \to |M_2|$ is a good $C^{\infty}$ map if
  and only if there exist Lie groupoids $\mathcal{G}_1$ and
  $\mathcal{G}_2$ representing $M_1$ and $M_2$ respectively, and a Lie
  groupoid homomorphism $\mathcal{F} : \mathcal{G}_1 \to
  \mathcal{G}_2$, such that the induced map on the orbit spaces equals
  $f$ -- once the homeomorphisms identifying the orbit spaces of
  $\mathcal{G}_1$ and of $\mathcal{G}_2$ with $|M_1|$ and $|M_2|$
  respectively are
  taken into account -- (see \cite[Proposition 5.1.7]{lup_ur}).

  The perspective on orbifolds via Lie groupoids goes beyond the scope
  of this paper. However, it is used sparingly to give
  quick sketches of proofs (see, for instance, Lemma \ref{lemma:obvious_good_maps}
  below).
\end{remark}

In what follows we give families of $C^{\infty}$ maps between
orbifolds that are good. We start by introducing the following notion (see
\cite[Definition 4.4.10]{chen_ruan}).

\begin{definition}\label{definition:regular_map}
  A $C^{\infty}$ map $f : M_1 \to M_2$ between orbifolds is {\bf
    regular} if the preimage of the subset of $M_2$ consisting of
  regular points under $f$ is open, dense and connected in $M_1$. 
\end{definition}

The following result, stated below without proof, is proved in
\cite[Lemma 4.4.11]{chen_ruan}.

\begin{lemma}\label{lemma:regular_good}
  Regular $C^{\infty}$ maps are good.
\end{lemma}

The next result provides another large class of good $C^{\infty}$
maps.

\begin{lemma}\label{lemma:obvious_good_maps}
  Let $\widehat{M}_i$ be a manifold and let $G_i$ be a Lie group
  acting on $\widehat{M}_i$ so that the action is proper and locally
  free for $i=1,2$. If $\widehat{f} : \widehat{M}_1 \to \widehat{M}_2$
  is a smooth map and if $\Theta : G_1 \to G_2$ is a homomorphism of
  Lie groups such that $\widehat{f}$ is $\Theta$-equivariant, then the
  continuous map $f : \widehat{M}_1/G_1 \to \widehat{M}_2/G_2$ given
  by $[\widehat{x}] \mapsto [\widehat{f}(\widehat{x})]$ is a good
  $C^{\infty}$ map.  
\end{lemma}

\begin{proof}[Sketch of proof]
  We give two strategies to prove the above result. The first is to
  use the l.u.c.'s for quotients given in Section
  \ref{sec:basic-definitions}, i.e., to check the result in `local
  coordinates'. It can be checked directly that $\widehat{f}$ and
  $\Theta$ induce a good $C^{\infty}$ lift of the restriction of $f$
  to any open neighborhood that is uniformized by a l.u.c., and that
  the compatibility condition of Definition \ref{defn:compatible_sys}
  holds.

  The second (more direct) strategy is to use the Lie groupoid approach to orbifolds
  (see Remark \ref{rmk:groupoids}). The map $\widehat{f}$ and the
  homomorphism $\Theta$ can be used to construct a Lie groupoid
  homomorphism $\mathcal{F} : G_1 \ltimes M_1 \to G_2 \ltimes M_2$
  with the property that the induced map on orbit spaces coincides
  with $f$. The result then follows at once from \cite[Proposition 5.1.7]{lup_ur}.
\end{proof}

In the following example, we illustrate an application of Lemma
\ref{lemma:obvious_good_maps}.

\begin{example}[Weighted projective spaces, continued]\label{exm:weighted_proj_standard}
  As in Example \ref{ex_weightedprj}, we fix an integer $n \geq 1$ and positive
  integers $m_0,\dots, m_{n}$ satisfying $\gcd(m_0,\dots,m_n)=1$. We
  consider the smooth $\C^*$-action on $\C^{n+1} \smallsetminus
  \{0\}$ given by
  \begin{equation}
    \label{eq:11}
    \begin{split}
      \C^*\times\mathbb{C}^{n+1} \smallsetminus \{0\}
      &\rightarrow \mathbb{C}^{n+1} \smallsetminus \{0\}\\
    \left(w,(z_0,\dots,z_{n})\right)&\mapsto(w^{m_0}z_0,\dots,w^{m_n}z_n),
    \end{split}
  \end{equation}
  (cf. \eqref{eq:8}). This action is proper and locally free,
  so that the quotient $(\C^{n+1} \smallsetminus
  \{0\})/\C^*$ inherits the structure of an orbifold. We denote an
  element in $(\C^{n+1} \smallsetminus
  \{0\})/\C^*$ by $[z_0:\ldots :z_n]$. We remark that, just as for
  projective spaces, we may exploit the equivalence relation to carry
  out calculations in `local coordinates' that we refer to as {\bf
    homogeneous coordinates} (see, for instance, Example
  \ref{ex::orbifold_quotient} below).

  Let $E \subset \C^{n+1} \smallsetminus \{0\}$ be as in Example
  \ref{ex_weightedprj}. If $\widehat{f} : E \to \C^{n+1}
  \smallsetminus \{0\}$ and $\Theta : S^1 \to \C^*$ denote the
  inclusions, then, by Lemma \ref{lemma:obvious_good_maps}, the induced continuous map $f : \C P^n(m_0,\ldots,
  m_n) \to  (\C^{n+1} \smallsetminus
  \{0\})/\C^*$ is a good $C^\infty$ map. In fact, it is a
  diffeomorphism and the inverse is given by the `standard' projection
  maps $\C^{n+1}
  \smallsetminus \{0\} \to E$ and $\C^* \to S^1$. Throughout
  this paper, we use this diffeomorphism to identify $(\C^{n+1} \smallsetminus
  \{0\})/\C^*$ with $\C P^n(m_0,\ldots,
  m_n)$ tacitly, trusting that this does not cause confusion. \hfill$\Diamond$ 
\end{example}

To conclude this section, we give a sketch of a proof of the following
result regarding good $C^{\infty}$ maps.

\begin{lemma}\label{lemma:good_homomorphism}
  Let $f : M_1 \to M_2$ be a good $C^{\infty}$ map. For any point $x
  \in M_1$, there is an induced conjugacy class of homomorphisms
  $\Gamma_x \to \Gamma_{f(x)}$. 
\end{lemma}

\begin{proof}[Sketch of proof]
  As in the proof of Lemma \ref{lemma:obvious_good_maps}, we give two
  strategies. The first is to check the result directly as in
  \cite[Definition 4.4.9]{chen_ruan}. The second (more direct) strategy is to use the Lie groupoid approach to orbifolds
  (see Remark \ref{rmk:groupoids}): With the latter, the result is
  immediate as a Lie groupoid morphism induces homomorphisms between
  isotropy groups and Morita equivalences preserve the conjugacy class
  of these homomorphisms.
\end{proof}

\subsubsection{Compact, connected, orientable orbi-surfaces with isolated
  singular points}\label{ss::classification_orbispheres}
In this section, we state a classification result for a certain family
of two-dimensional orbifolds, which we call {\bf orbi-surfaces}. This
result should be seen as the analog of the well-known classification of closed
two-dimensional manifolds for orbifolds.

Let $\Sigma$ be an orientable orbi-surface, let $x \in
\Sigma$ be a point, and let $\oc$ be a l.u.c. centered at $x$ as in Remark
\ref{rmk:centred}. The orbifold structure group $\Gamma$ of $x$ is a
subgroup of $\mathrm{SO}(2)$. Hence, if $\Gamma$ is not trivial, then it is
a finite cyclic subgroup $\Z_c$ for some natural number $c \geq
2$. If, in addition, $\Sigma$ is compact and has only isolated singular
points, then it has finitely many such points. The following result,
stated below without proof, provides a complete classification of such orbi-surfaces.

\begin{theorem}\cite[Theorem
  9.5.2]{alvaro}\label{thm::classorbisurface}
  \mbox{}
  \begin{enumerate}[label=(\arabic*),ref=(\arabic*),leftmargin=*]
  \item \label{item:11} If $\Sigma$ be a compact, connected, orientable orbi-surface with isolated
    singular points, then there exists a non-negative integer $g$ such
    that $|\Sigma|$ is homeomorphic to the compact orientable surface of
    genus $g$.
  \item \label{item:12} For $i=1,2$, let $\Sigma_i$ be a compact, connected, orientable orbi-surface with isolated
    singular points, let $g_i$ be the genus of $|\Sigma_i|$, let
    $S_i$ be the set of singular points of $\Sigma_i$, and let $c_i : S_i
    \to \mathbb{N}$ be the map that sends a singular point to the order of its orbifold structure group. The two orbifolds $\Sigma_1$ and $\Sigma_2$ are diffeomorphic if
    and only if $g_1 = g_2$ and there exists a bijection $\sigma : S_1
    \to S_2$ such that $c_2 \circ \sigma = c_1$.
  \item \label{item:13} For any non-negative integers $g, n$ and for any map $c :
    \{1,\ldots, n\} \to \mathbb{N}$ such that $c(j) \geq 2$ for all
    $j=1,\ldots, n$, there exists a compact, connected, orientable orbi-surface with isolated
    singular points $\Sigma$ such that $|\Sigma|$ is homeomorphic to the compact orientable surface of
    genus $g$, the set of singular points $S$ of $\Sigma$ has cardinality
    $n$ and there exists a bijection $\sigma : S \to \{1,\ldots, n\}$
    with the property that $c(\sigma(p))$ is the order of the orbifold
    structure group of $p$, for any $p \in S$.
  \end{enumerate}
\end{theorem}

\subsection{Orbi-bundles}\label{sec:orbi-bundles}
\subsubsection{Definition and Examples} \label{sec:definition-examples}
In this section, we follow mainly \cite{furuta,lgodinho}.

\begin{definition}\label{defn:trivial_orbi-bundle}
  Let $|M_i|$ be a topological space for $i=1,2$, let $\mathrm{pr}: |M_1| \to |M_2|$ be a
  surjective continuous map, let $F$ be a connected smooth manifold and let $G$
  be a Lie group acting on $F$. We say that $\mathrm{pr}: |M_1| \to
  |M_2|$ admits a {\bf local trivializing l.u.c.}
  over the open set $U_2 \subseteq |M_2|$ with generic fiber $F$ and
  structure group $G$ if there exist
  \begin{itemize}[leftmargin=*]
  \item a l.u.c. $(\widehat{U}_2, \Gamma_2,\varphi_2)$ uniformizing
    $U_2$,
  \item a  l.u.c. $(\widehat{U}_1, \Gamma_1,\varphi_1)$ uniformizing
    $U_1:=\mathrm{pr}^{-1}(U_2)$,
  \item a smooth map $\rho : \Gamma_2 \times \widehat{U}_2\to G$, and
  % \item an action of $\Gamma_2$ on $\widehat{U}_2 \times
  %   F$ with respect to which the projection onto the first component $\widehat{\mathrm{pr}}_1 : \widehat{U}_2 \times F \to
  %   \widehat{U}_2$ is $\Gamma_2$-equivariant, and
  % \item a $\Gamma_2$-invariant continuous map $\varphi'_1
  %   : \widehat{U}_2 \times  F \to U_1$, 
  \item a diffeomorphism $\lambda : \widehat{U}_2 \times F \to
    \widehat{U}_1$,
  \end{itemize}
  such that
  \begin{itemize}[leftmargin=*]
  \item the map
    \begin{equation*}
      \begin{split}
        \Gamma_2 \times \widehat{U}_2 \times F &\to \widehat{U}_2
        \times F \\
        (\gamma,\widehat{x},y) &\mapsto (\gamma \cdot \widehat{x}, \rho(\gamma,\widehat{x})(y))
      \end{split}
    \end{equation*}
    is an action of $\Gamma_2$ on $\widehat{U}_2 \times F$, where
    $\cdot$ denotes the $\Gamma_2$-action on $\widehat{U}_2$,
  \item the open set $U_1$ is uniformized by the l.u.c.
    $(\widehat{U}_2 \times F, \Gamma_2, \varphi_1 \circ \lambda)$, and
  \item the pair $(\widehat{\mathrm{pr}}_1,
    \mathrm{id})$ is a good $C^{\infty}$ lift of $\mathrm{pr}$.
  \end{itemize}
  We denote a local trivializing l.u.c. for $\mathrm{pr}: |M_1| \to
  |M_2|$ over $U_2$ by
  $$ (\widehat{U}_1 \to \widehat{U}_2, \Gamma_2,
  \varphi_1,\varphi_2,\rho, \lambda). $$
\end{definition}

\begin{remark}\label{rmk:obs_defn_triv_orbi-bundle}
  The conditions in Definition \ref{defn:trivial_orbi-bundle} imply that $\lambda : (\widehat{U}_2 \times F,
  \Gamma_2, \lambda \circ \varphi_1) \hookrightarrow (\widehat{U}_1,
  \Gamma_1,\varphi_1)$ is an isomorphism of l.u.c.'s and that
  $(\widehat{\mathrm{pr}}_1 \circ \lambda^{-1},\lambda^{-1})$ is a
  good $C^{\infty}$ lift of $\mathrm{pr}$. This is summarized in the
  following commutative diagram:
  \begin{equation*}
    \xymatrix{\widehat{U}_2 \times F \ar[r]_-{\simeq}^-{\lambda}
      \ar[dr]_-{\widehat{\mathrm{pr}}_1} & \widehat{U}_1
      \ar[r]^-{\varphi_1} \ar[d]^-{\widehat{\mathrm{pr}}_1 \circ
        \lambda^{-1}} & U_1 \ar[d]^-{\mathrm{pr}} \\
       & \widehat{U}_2 \ar[r]_-{\varphi_2}& U_2.}
  \end{equation*}
\end{remark}

As the terminology suggests, local trivializing l.u.c.'s should be
thought of as being analogous to l.u.c.'s and local trivializations for
fiber bundles with a given structure group. In this sense, the following notion is the
analog of Definition \ref{defn:embedding_isomorphism}.

\begin{definition}\label{defn:embedding_bundle}
  Let $|M_i|$ be a topological space for $i=1,2$, let $\mathrm{pr}: |M_1| \to |M_2|$ be a
  surjective continuous map, let $F$ be a smooth manifold, let $G$ be
  a Lie group acting effectively on $F$, and let
  $U_2, V_2 \subseteq |M_2|$ be open sets. Given local trivializing
  l.u.c.'s $(\widehat{U}_1 \to \widehat{U}_2, \Gamma_2,
  \varphi_1,\varphi_2,\lambda)$ and $(\widehat{V}_1 \to \widehat{V}_2, \Gamma'_2,
  \psi_1,\psi_2,\mu)$ for $\mathrm{pr}: |M_1| \to |M_2|$ with generic
  fiber $F$ and structure group $G$ over $U_2$
  and $V_2$ respectively, an {\bf embedding} of $(\widehat{U}_1 \to \widehat{U}_2, \Gamma_2,
  \varphi_1,\varphi_2,\lambda)$ into $(\widehat{V}_1 \to \widehat{V}_2, \Gamma'_2,
  \psi_1,\psi_2,\mu)$ is an embedding $\Lambda : (\widehat{U}_2,
  \Gamma_2, \varphi_2) \hookrightarrow  (\widehat{V}_2,
  \Gamma'_2, \psi_2)$ together with a map $g: \widehat{U}_2 \to
  G$ such that the map
  \begin{equation*}
    \begin{split}
      \widehat{U}_2 \times F &\to \widehat{V}_2 \times F \\
      (\widehat{x},y) &\mapsto (\Lambda(x), g(x)(y))
    \end{split}
  \end{equation*}
  is an embedding between $(\widehat{U}_2 \times F, \Gamma_2,
  \varphi_1 \circ \lambda)$ and $(\widehat{V}_2 \times F, \Gamma'_2,
  \psi_1 \circ \lambda)$. We denote such an embedding by
  $(\Lambda, g)$ and say that $(\Lambda, g)$ {\bf extends} $\Lambda$.

  We say that $(\widehat{U}_1 \to \widehat{U}_2, \Gamma_2,
  \varphi_1,\varphi_2,\lambda)$ and $(\widehat{V}_1 \to \widehat{V}_2, \Gamma'_2,
  \psi_1,\psi_2,\mu)$ are {\bf isomorphic} if there exists an
  embedding  $(\Lambda, g)$ of $(\widehat{U}_1 \to \widehat{U}_2, \Gamma_2,
  \varphi_1,\varphi_2,\lambda)$ into $(\widehat{V}_1 \to \widehat{V}_2, \Gamma'_2,
  \psi_1,\psi_2,\mu)$ such that $\Lambda: (\widehat{U}_2,
  \Gamma_2, \varphi_2) \hookrightarrow  (\widehat{V}_2,
  \Gamma'_2, \psi_2)$ is an isomorphism.
\end{definition}

\begin{remark}\label{rmk:analog_homom_induced}
  It can be checked that the analogs of Remarks
  \ref{rmk:embeddings_induce_homom} and \ref{rmk:induced_luc} hold for
  local trivializing l.u.c.'s. 
\end{remark}

\begin{definition}\label{defn:locally_trivial_orbi-bundle_atlas}
  Let $|M_i|$ be a topological space for $i=1,2$,  let $\mathrm{pr} : |M_1| \to |M_2|$ be a surjective
  continuous map, let $F$ be a smooth manifold, and let $G$ be a Lie
  group acting on $F$. We say that $\mathrm{pr} : |M_1| \to |M_2|$ admits an {\bf orbi-bundle atlas}
  with generic fiber $F$ and structure group $G$ if there exists an orbifold atlas
  $\mathcal{A}_2$ on $|M_2|$ such that
  \begin{itemize}[leftmargin=*]
  \item for each $(\widehat{U}_2,\Gamma_2,\varphi_2) \in
    \mathcal{A}_2$, the map $\mathrm{pr} : |M_1| \to |M_2|$ admits a
    local trivializing l.u.c. over $U_2 = \varphi_2(\widehat{U}_2)$ with generic fiber $F$ and structure group $G$ of
    the form $(\widehat{U}_1 \to \widehat{U}_2, \Gamma_2,
    \varphi_1,\varphi_2,\lambda)$, and
  \item for every pair $(\widehat{U}_2,\Gamma_2,\varphi_2),
    (\widehat{V}_2,\Gamma'_2,\psi_2) \in
    \mathcal{A}_2$ with $U_2 \cap V_2 \neq \emptyset$, for
    every $x \in U_2 \cap V_2$, and for every
    $(\widehat{W}_2,\Gamma''_2,\chi_2) \in \mathcal{A}_2$ with $x \in
    W_2$ and with embeddings $(\widehat{W}_2,\Gamma''_2,\chi_2)
    \hookrightarrow (\widehat{U}_2,\Gamma_2,\varphi_2)$ and $(\widehat{W}_2,\Gamma''_2,\chi_2)
    \hookrightarrow  (\widehat{V}_2,\Gamma'_2,\psi_2)$, there exist
    embeddings of $(\widehat{W}_1 \to \widehat{W}_2, \Gamma''_2,
    \chi_1,\chi_2,\nu)$ into $(\widehat{U}_1 \to \widehat{U}_2, \Gamma_2,
    \varphi_1,\varphi_2,\lambda)$ and $(\widehat{V}_1 \to \widehat{V}_2, \Gamma'_2,
    \psi_1,\psi_2,\mu)$ that extend the
    above embeddings of l.u.c.'s. 
  \end{itemize}
\end{definition}

\begin{remark}\label{rmk:locally_trivial_orbi-bundle_atlas}
  Given an orbi-bundle atlas as in Definition
  \ref{defn:locally_trivial_orbi-bundle_atlas}, the collection
  $\mathcal{A}_1=\{(\widehat{U}_1, \Gamma_1, \varphi_1)\}$ defines an orbifold atlas
  on $|M_1|$. Hence, {\em a posteriori}, $M_1$ is an orbifold. The
  second bullet point in Definition \ref{defn:locally_trivial_orbi-bundle_atlas}
  is a compatibility condition between $\mathcal{A}_1$ and $\mathcal{A}_2$. For this reason, we denote an orbi-bundle atlas
  by $(\mathcal{A}_1,\mathcal{A}_2)$ and we refer to the local
  trivializing l.u.c.'s of Definition
  \ref{defn:locally_trivial_orbi-bundle_atlas} as {\em lying in} $(\mathcal{A}_1,\mathcal{A}_2)$. 
\end{remark}

\begin{definition}\label{defn:locally_trivial_orbi-bundle}
  Let $|M_i|$ be a topological space for $i=1,2$,  let $\mathrm{pr} : |M_1| \to |M_2|$ be a surjective
  continuous map, let $F$ be a smooth manifold and let $G$ be a Lie
  group acting on $F$. Let $(\mathcal{A}_1,\mathcal{A}_2)$ and
  $(\mathcal{B}_1,\mathcal{B}_2)$ be orbi-bundle atlases for
  $\mathrm{pr} : |M_1| \to |M_2|$ with generic fiber $F$ and structure
  group $G$.
  \begin{enumerate}[label=(\arabic*), ref=(\arabic*), leftmargin=*]
  \item \label{item:2} The orbi-bundle atlas
    $(\mathcal{B}_1,\mathcal{B}_2)$ {\bf refines}
    $(\mathcal{A}_1,\mathcal{A}_2)$ if every local trivializing
    l.u.c. lying in $(\mathcal{B}_1,\mathcal{B}_2)$ admits an
    embedding into some local trivializing
    l.u.c. lying in $(\mathcal{A}_1,\mathcal{A}_2)$.
  \item \label{item:7} The orbi-bundle atlases
    $(\mathcal{A}_1,\mathcal{A}_2)$ and
    $(\mathcal{B}_1,\mathcal{B}_2)$ are {\bf equivalent} if they admit
    a common refinement.
  \item \label{item:8} An {\bf orbi-bundle} with generic fiber $F$ and
    structure group $G$ is a
    surjective continuous map $\mathrm{pr} : |M_1| \to |M_2|$ endowed
    with an equivalence class of orbi-bundle
    atlases with generic fiber $F$ and
    structure group $G$. To simplify notation, we denote an
    orbi-bundle by $\mathrm{pr} : M_1 \to M_2$. 
  \end{enumerate}
\end{definition}

\begin{remark}\label{rmk:transition_maps}
  By Remark \ref{rmk:analog_homom_induced}, Definition
  \ref{defn:locally_trivial_orbi-bundle} agrees with the notion of
  orbi-bundles in \cite[Section 4.1]{chen_ruan} and \cite[Section
  2.1]{lgodinho}. Moreover, an orbi-bundle atlas can be understood in
  terms of the analog of transition maps (see \cite[page
  123]{lgodinho}): this comes from unraveling the compatibility
  condition between the atlases $\mathcal{A}_1$ and $\mathcal{A}_2$ in
  Definition \ref{defn:locally_trivial_orbi-bundle_atlas}.
\end{remark}

\begin{remark}\label{rmk:projection_good}
  If $\mathrm{pr} : M_1 \to M_2$ is an orbi-bundle, then $\mathrm{pr}$
  is a good $C^{\infty}$ map (see \cite[Remark 4.4.12b]{chen_ruan}). 
\end{remark}

In analogy with the case of manifolds, if $F$ is a vector space and
$G$ is (a subgroup of) $\mathrm{GL}(F)$, we call an orbi-bundle with
generic fiber $F$ and structure group $G$ a {\bf vector orbi-bundle}.

\begin{remark}\label{rmk:tangent_orbi-bundle}
  If $M$ is an orbifold, one can construct the {\bf tangent}
  orbi-bundle $TM \to M$ in a natural fashion (see, for instance,
  \cite[Section 1.3]{adem}). The fiber $T_x M$ over $x \in
  M$ can be identified with the following orbi-vector space: If $\oc$
  is a l.u.c. centered at $x$ and $\widehat{U} \subseteq \R^n$ (see Remark \ref{rmk:centred}), then $T_x
  M \simeq \R^n/\Gamma$. If $f : M_1 \to M_2$ is a good $C^{\infty}$
  map, then there is a well-defined notion of {\bf derivative} of $f$
  at $x \in M_1$, i.e., a $\Gamma$-equivariant linear map $d_x f : T_x
  M_1 \to T_{f(x)} M_2$, where $\Gamma$ is the orbifold structure
  group of $x$ (see \cite[Section 2.1]{kleiner_lott}).
\end{remark}

\begin{example}[Standard vector orbi-bundles
  over orbifolds]\label{exm:standard_bundles}
  All the standard operations on
  vector bundles over manifolds extend to vector orbi-bundles over
  orbifolds (see \cite[Section 1.3]{adem}). In particular, given an
  orbifold $M$, one can construct the {\bf cotangent}
  orbi-bundle $T^*M \to M$ and all tensor bundles over $M$.  \hfill$\Diamond$   
\end{example}

% \begin{definition}\cite[Section 2]{LermanTolman} Let $(\widetilde{U},\Gamma,\varphi)$ be a  l.u.c. for a neighborhood $U$ of $x\in  M$  with $\Gamma$ the structure group of $x$,  and let 
% 	 $\tilde{x}\in\widetilde{U}$ be such that $\varphi(\tilde{x})=x$. 
%  The tangent space $T_{\tilde{x}}\widetilde{U}$ is called the \textbf{uniformized tangent space} at $x$ and is denoted by $\widetilde{T}_xM$. The \textbf{orbi-tangent space} of  $x\in M$ is defined to be the quotient $T_xM:=\widetilde{T}_{x}M/\Gamma$. 
% \end{definition}

The construction of orbifolds using locally free and proper actions
described in Section \ref{sec:basic-definitions} can be generalized to
construct orbi-bundles. For our purposes, we need only the following
special case: Let $\widehat{M}, F$ be smooth manifolds, let
$G$ be a Lie group acting on $F$ and let $H$ be a Lie group. We
suppose that
\begin{itemize}[leftmargin=*]
\item $H$ acts on $\widehat{M}$,
\item there exists a smooth map $\rho : H \times \widehat{M} \to G$
  such that the map
  \begin{equation*}
    \begin{split}
      H \times \widehat{M} \times F &\to \widehat{M} \times F \\
      (h,\widehat{x},y) &\mapsto (h \cdot x, \rho(h,\widehat{x})(y))
    \end{split}
  \end{equation*}
  defines an $H$-action on $\widehat{M} \times F$, where $\cdot$
  denotes the $H$-action on $\widehat{M}$,
\item the $H$-action on $\widehat{M}$ is locally free, and
\item the $H$-actions on $\widehat{M}$ and on  $\widehat{M} \times F$
  are proper. 
\end{itemize}
Since the $H$-actions on $\widehat{M}$ and $\widehat{M} \times F$ are
locally free and proper, we obtain orbifolds $\widehat{M}/H$ and
$(\widehat{M} \times F)/H$. Moreover, there is a surjective continuous
map $(\widehat{M} \times F)/H \to \widehat{M}/H$ such that the
following diagram commutes
\begin{equation*}
  \xymatrix{\widehat{M} \times F \ar[d] \ar[r] & (\widehat{M} \times
    F)/H \ar[d] \\
    \widehat{M} \ar[r] & \widehat{M}/H,}
\end{equation*}
where the horizontal arrows are the quotient maps and the vertical
arrow on the left is projection onto the first component. Using the
standard l.u.c.'s on $\widehat{M}/H$ and
$(\widehat{M} \times F)/H$ (described in Section \ref{sec:basic-definitions}), it can be checked directly that
$(\widehat{M} \times F)/H \to \widehat{M}/H$ is an orbi-bundle with
generic fiber $F$ and structure group $G$. The following example gives
a concrete application of this construction.

\begin{example}[Weighted projective spaces, continued]\label{exm:weighted_O}
  Fix an integer $n \geq 1$ and positive
  integers $m_0,\dots, m_{n}$ satisfying
  $\gcd(m_0,\dots,m_n)=1$. Fix an integer $k$ and consider the
  $\C^*$-action on $(\C^{n+1} \smallsetminus \{0\}) \times\mathbb{C}$ given by
  \begin{align*}
    \C^* \times (\C^{n+1} \smallsetminus \{0\})\times\mathbb{C}&\rightarrow (\C^{n+1} \smallsetminus \{0\})\times\mathbb{C}\\
    (w,z_0,\dots,z_n,v)&\mapsto(w^{m_0}z_0,\dots,w^{m_n}z_n,\lambda^kv),
  \end{align*}
  (cf. \eqref{eq:11}). Together with the $\C^*$-action on $\C^{n+1} \smallsetminus \{0\}$ given in Example
  \ref{exm:weighted_proj_standard}, these two $\C^*$-action satisfy the above
  hypotheses, where $G = \C^* = \mathrm{GL}(\C)$. The resulting complex vector
  orbi-bundle
  $$(  (\C^{n+1} \smallsetminus \{0\}) \times \C)/\C^* \to 
  \mathbb{C}P^{n}(m_0,\dots,m_n)$$
  of rank 1 is denoted by
  $\mathcal{O}_{m_0,\dots,m_n}(k)$. Finally, we observe that it is
  also possible to construct $\mathcal{O}_{m_0,\dots,m_n}(k)$ as an
  appropriate $S^1$-quotient of the product $E \times \C$, where $E$
  is as in Example \ref{ex_weightedprj}. We leave the details of this
  latter construction to the reader (see Example
  \ref{exm:weighted_proj_standard}). \hfill$\Diamond$   
\end{example}

% If, in addition, $k >0$ and $\gcd(m_0,\dots,m_n, k)=1$, then the total space of the orbi-bundle $\mathcal{O}_{q_0,\dots,q_n}(k)$, can be identified with  
% $$\mathbb{C}P^{n+1}(q_0,\dots,q_n,k)\setminus[0:\dots:(2\prod\limits_{j}
% q_j)^{1/2}].$$ 
% Therefore, for $k>0$, we can think of
% $\mathcal{O}_{q_0,\dots,q_n}(k)$ as the normal orbi-bundle of
% $\mathbb{C}P^n(q_0,\dots,q_n)$ inside
% $\mathbb{C}P^{n+1}(q_0,\dots,q_n,k)$.

\begin{remark}\label{rmk:section}
  Given an orbi-bundle $\mathrm{pr} : M_1 \to M_2$, there is a natural
  notion of a {\bf $C^{\infty}$ section} of $\mathrm{pr} : M_1 \to
  M_2$ (see \cite[Section 4.1]{chen_ruan}): This is a $C^{\infty}$ map
  $s : M_2 \to M_1$ that is a set-theoretic section of $\mathrm{pr}$
  such that the $C^{\infty}$ lifts of $s$ in local trivializing
  l.u.c.'s for $\mathrm{pr} : M_1 \to
  M_2$ are equivariant sections. By \cite[Remark 4.4.12b]{chen_ruan},
  any $C^{\infty}$ section of $\mathrm{pr} : M_1 \to M_2$ is good. 

  In particular, it makes sense to define {\bf tensors} on an orbifold
  $M$ as $C^{\infty}$ sections of the appropriate vector
  orbi-bundle constructed in Remark \ref{exm:standard_bundles} (see,
  for instance, \cite[Definition 1.27]{adem}). Moreover, the {\bf de
    Rham differential} $d : \Omega^{\bullet}(M) \to \Omega^{\bullet
    +1}(M)$ can be defined in the same way as for manifolds. Hence, it
  also makes sense to consider closed and exact forms, as well as de
  Rham cohomology (see \cite[Chapter 2.1]{adem} for an approach
  involving Lie groupoids). Finally, {\bf integration} of differential
  forms can be defined using the standard formulas.
\end{remark}

\begin{remark}\label{rmk:flow}
  By Remark \ref{rmk:section}, we may consider a vector field $\xi$ on an
  orbifold $M$. The naive definition of the {\bf flow} of $\xi$ on $M$
  makes sense and satisfies the standard properties (see
  \cite{hepworth} for a more sophisticated approach using Lie
  groupoids). If $M$ has isolated singular points, this is
  particularly simple using l.u.c.'s and their compatibility. In particular, it can be shown that the standard Cartan
  calculus holds for orbifolds. 
\end{remark}

To conclude this section, we introduce {\bf principal
  $G$-orbi-bundles}: These are orbi-bundles in the sense of Definition
\ref{defn:locally_trivial_orbi-bundle}, where $F = G$ and $G$ acts on
itself by left multiplication. In fact, just as for manifolds, it is
possible to go from orbi-bundles over $M$ with typical fiber $F$ and
structure group $G$ to principal $G$-orbi-bundles by means of the
standard construction and vice versa (see, for instance, \cite[Section
2.1]{lgodinho}). In particular, this establishes the `standard' correspondence
between complex line orbi-bundles and principal $S^1$-orbi-bundles over orbisurfaces (see, for example, \cite{furuta}).

%\begin{remark}
%Note that since $\Gamma$ is finite and $\widetilde{U}$ is Hausdorff for all l.u.c. $(\widetilde{U},\Gamma,\varphi)$, the action of $\Gamma$ on $\widetilde{U}$  is proper and one can pick a $\Gamma$-invariant Riemannian metric on $\widetilde{U}$, such that $\Gamma$ acts by isometries on $\widetilde{U}$. Hence, there is a representation $\Gamma\rightarrow O(n)$ and, by effectiveness, $\Gamma<O(n)$.	
%\hfill$\diamond$
%\end{remark}

\subsubsection{Seifert manifolds}\label{sec:seifert-manifolds}
To start, we define certain 3-manifolds endowed with an
$S^1$-action, first introduced by Seifert in \cite{seifert}.

\begin{definition}\label{defn:Seifert_manifold}
  A {\bf Seifert manifold} is a 3-dimensional compact, connected oriented manifold
  $M$ endowed with an effective, locally free,
  orientation-preserving $S^1$-action. We call the
  orbit map $\pi: M\rightarrow M/S^1$ a \textbf{Seifert
    fibration}.  
\end{definition} 
 
\begin{remark}\label{rmk:base_seifert_fibration}
  Given a Seifert fibration $\pi : M \to M/S^1$, the construction
  illustrated in Section \ref{sec:basic-definitions} yields that
  orbit space $\Sigma:=M/S^1$ is a compact, connected oriented
  orbi-surface.
\end{remark}

Seifert manifolds and fibrations have been classified by
Raymond \cite{raymond}. For completeness, in what follows we
illustrate this classification (see also \cite[Section I.3.c]{audin}
and \cite{jankinsNeuman,orlik}). Let $\pi:M\rightarrow \Sigma$ be a
Seifert fibration. Since $S^1$ is abelian, it makes sense to consider
the stabilizer of an orbit $S^1\cdot x$, where $x \in M$, which we
denote by $\Gamma$. Since the $S^1$-action on $M$ is effective, the
only possible stabilizers are the trivial subgroup and the cyclic
subgroups $\Z_{\alpha}$ given by $\alpha$-th roots of unity, where
$\alpha > 1$ is an integer. In particular, every orbit is
diffeomorphic to $S^1$. We say that an orbit $S^1 \cdot x$ is {\bf
  principal} (respectively {\bf exceptional}) if its stabilizer is
trivial (respectively $\Z_{\alpha}$ for some $\alpha >1$).

The Slice
Theorem (see \cite[Theorem
2.4.1]{dk}) provides a description of $\pi : M \to \Sigma$ near any
orbit, which we illustrate below. To this end, we fix $x \in M$ and let
$D^2\subset \R^2 \simeq \mathbb{C}$ be an open disc centred at the
origin. Here, $\R^2$ is identified with the normal space to the orbit
$S^1 \cdot x$ at $x$. Since the $S^1$-action is
orientation-preserving, the image of the isotropy representation $\rho
(\Gamma)$ lies in $\mathrm{SO}(2) \simeq \mathrm{U}(1)$, so that
$\rho$ is completely determined by $\nu \in \Z$ such that $\rho
(\gamma) = \gamma^{\nu}$, for all $\gamma \in \Gamma
\subset S^1$. We consider the $\Gamma$-action on $S^1 \times D^2$
given by
\begin{equation}
  \label{eq:10}
  \gamma \cdot (\lambda, w) = (\lambda \gamma^{-1}, \gamma^{\nu}z).
\end{equation}
This action is free and proper so that the quotient $(S^1 \times
D^2)/\Gamma$ is a smooth manifold. We endow $(S^1 \times
D^2)/\Gamma$ with the $S^1$-action given by 
$$e^{i\theta}\cdot[\lambda,w]=[e^{i\theta}\lambda,w].$$
By the Slice Theorem, an
$S^1$-invariant neighborhood of $S^1 \cdot x$ is $S^1$-equivariantly
isomorphic to $(S^1 \times
D^2)/\Gamma$. In particular, since the $S^1$-action on $M$ is
effective, it follows that $\nu \neq 0$ and, if $\Gamma = \Z_{\alpha}$
for some $\alpha >1$, then $\gcd(\alpha, \nu)=1$. In the latter case
(i.e., if the orbit is exceptional),
there exists $\gamma \in \Gamma$ such that $\gamma^{\nu} =
e^{\frac{2\pi i}{\alpha}}$ and $\beta \in \Z$ with $\gcd(\alpha ,
\beta)=1$ such that \eqref{eq:10} becomes
$$ e^{\frac{2\pi i}{\alpha}} \cdot (\lambda, w) = \left(\lambda
  e^{-\frac{2\pi i \beta}{\alpha}}, e^{\frac{2\pi i}{\alpha}}w
\right). $$
If we choose $0 < \beta < \alpha$, we obtain the \textbf{Seifert
  invariant} of an exceptional orbit, which we denote by $(\alpha,
\beta)$. 

\begin{remark}\label{rmk:Seifert_principal}
  A closer look at the above argument yields the following extra
  results:
  \begin{enumerate}[label=(\arabic*),ref= (\arabic*),leftmargin=*]
  \item \label{item:9} exceptional orbits are isolated in $\Sigma$, and
  \item \label{item:10} locally near an orbit with stabilizer $\Gamma$ (and isotropy
    representation determined by the integer $\nu$), the Seifert
    fibration $\pi : M \to \Sigma$ is isomorphic to the map $ (S^1 \times
    D^2)/\Gamma \to D^2/\Gamma$ sending $[\lambda, w] \to [w]$, where
    the $\Gamma$-action on $D^2$ is given by the second component of
    \eqref{eq:10}. 
  \end{enumerate}
  In particular, since exceptional orbits are precisely the singular points
  of the orbifold $\Sigma$, by property \ref{item:9} the orbifold structure on $B$ has only isolated
  singular points. In particular, by Remark
  \ref{rmk:base_seifert_fibration} and part \ref{item:11} of Theorem
  \ref{thm::classorbisurface}, there exists a non-negative integer $g$
  such that $|B|$ is homeomorphic to the compact orientable surface of
  genus $g$. Finally, by property \ref{item:10}, $\pi : M \to
  \Sigma$ is a principal $S^1$-orbi-bundle.
\end{remark}

\begin{remark}\label{rmk:neigh_exceptional_orbit}
  Let $S^1 \cdot x$ be an exceptional orbit with stabilizer $\Gamma$
  and Seifert invariant
  $(\alpha, \beta)$. It can be shown that the quotient $(S^1 \times
  D^2)/\Gamma$ is, in fact, $S^1$-equivariantly
  isomorphic to a solid torus $S^1 \times
  D^2$ endowed with an $S^1$-action that is not trivial on both components -- for
  more details, see \cite[Section I.3.c]{audin}. Moreover, this identification can be extended to an identification
  between $(S^1 \times \overline{D}^2)/\Gamma$ and $S^1 \times
  \overline{D}^2$. Under this identification, the exceptional orbit corresponds to $S^1
  \times \{0\}$ and a principal orbit lying on the boundary $S^1
  \times S^1$ intersects the `meridian' $\{1\} \times S^1$ in precisely
  $\alpha$ points.
\end{remark}

By Remark \ref{rmk:Seifert_principal}, there exists a positive integer
$n$ such that $\Sigma$ has precisely $n$ singular points, that we denote by
$[x_1],\dots,[x_n]$ for some $x_1,\ldots, x_n \in
M$. Let $(\alpha_i,\beta_i)$ denote the Seifert invariant of $S^1
\cdot x_i$ for $i=1,\ldots, n$. The remaining invariant of the Seifert
fibration $\pi : M \to \Sigma$ is encoded by how the neighborhoods of the
exceptional orbits are `glued in'. To make this precise, for each
$i=1,\ldots, n$, let $\Gamma_i$ be the stabilizer of $x_i$ and let $\mathcal{V}_i \subset M$ be an $S^1$-invariant
open neighborhood of $[x]$ that is $S^1$-equivariantly
diffeomorphic to $(S^1 \times D^2)/\Gamma_i$. Moreover, we fix a point
$x_0 \in M$ with trivial stabilizer and let $\mathcal{V}_0 \subset M$
be an open neighborhood of $[x_0]$ in $M$. (The reason for
this extra neighborhood becomes clear in what follows.) Finally, we choose
$\mathcal{V}_0, \mathcal{V}_1,\ldots , \mathcal{V}_n$ sufficiently
small so that their closures are pairwise disjoint. For $i=0,\ldots, n$, we set $D_i:= \pi
(\mathcal{V}_i) \subset \Sigma$; by construction, these are homeomorphic to
open discs centered at the respective orbits that have closures that are pairwise
disjoint. By Remark \ref{rmk:Seifert_principal}, the complement $\Sigma
\smallsetminus\bigcup_{i=0}^n \overline{D}_i$ is a surface with vanishing
second cohomology. Moreover, the restriction
\begin{equation}
  \label{eq:12}
  \pi : M \smallsetminus \bigcup_{i=0}^n \overline{\mathcal{V}}_i \to \Sigma
  \smallsetminus\bigcup_{i=0}^n \overline{D}_i 
\end{equation}
is a principal $S^1$-bundle. Hence, it is a trivial principal
$S^1$-bundle. In particular, by Remark
\ref{rmk:neigh_exceptional_orbit}, in order to reconstruct the
manifold $M$ from the non-negative integers $g, n$ and the Seifert invariants
$(\alpha_1,\beta_1),\ldots, (\alpha_n,\beta_n)$, it suffices to
understand how the closed solid tori $\overline{\mathcal{V}}_0,
\ldots, \overline{\mathcal{V}}_n$ are glued to $M \smallsetminus
\bigcup_{i=0}^n \mathcal{V}_i$.

To understand the gluing maps, we start by observing that, since
\eqref{eq:12} is trivial, there exists a smooth section
$\sigma: \Sigma
\smallsetminus\bigcup_{i=0}^n \overline{D}_i \to 
M \smallsetminus \bigcup_{i=0}^n \overline{\mathcal{V}}_i$. Moreover, $\sigma$ extends to
the closure $\Sigma
\smallsetminus\bigcup_{i=0}^n D_i $ of $\Sigma
\smallsetminus\bigcup_{i=0}^n \overline{D}_i$; by a slight abuse of
notation, we denote the extension also by $\sigma$. We set $\partial_i
\sigma:=
\sigma(\partial \overline{D}_i)$ and let $b_i \in \partial
\overline{\mathcal{V}}_i$ be a principal orbit (see Remark
\ref{rmk:neigh_exceptional_orbit}), for $i=0,\ldots, n$. Since
$\sigma$ is a section, the homology classes of $\partial_i \sigma$ and $b_i$ form a basis of $H_1(\partial
\overline{\mathcal{V}}_i)$. A different basis of  $H_1(\partial
\overline{\mathcal{V}}_i)$ is given by the homology classes of a choice of meridian $a_i$ and
parallel $\ell_i$, which are the `standard' ones for
$\overline{\mathcal{V}}_0$ and the ones given by the images of $\{1\} \times S^1$ and $S^1
\times \{1\}$ under the identification of Remark
\ref{rmk:neigh_exceptional_orbit} for $i=1,\ldots, n$. (In what
follows we abuse notation and denote a curve and its homology class by
the same symbol.) Since $\partial
\overline{\mathcal{V}}_i \simeq S^1 \times S^1$, the isotopy
class of the gluing map of $\overline{\mathcal{V}}_i$ is determined by
the induced automorphism on $H_1(\partial\overline{\mathcal{V}}_i)$, which, in turn, is codified by how the basis
$a_i, \ell_i$ is related to the basis $\partial_i \sigma$ and
$b_i$. It can be shown (see \cite[Section I.3.c]{audin}), that there
exists $\beta_0 \in \Z$ such that the section $\sigma$ can be chosen to
satisfy\footnote{In some sense, the appearance of $\beta_0$ is so that
\eqref{eq:13} holds and can be thought of as a `balancing' constant. This justifies the removal of an extra solid
torus in the beginning of the discussion.}
\begin{equation}
  \label{eq:13}
  \begin{split}
    a_0 &= \partial \sigma_0 - \beta_0 b \\
    a_i& =\alpha_i\partial\sigma_i-\beta_ib_i\quad\text{for}\quad 1 \leq
    i\le n.
  \end{split}
\end{equation}
Moreover, together with the identification of Remark \ref{rmk:neigh_exceptional_orbit}, \eqref{eq:13} determines how $a_i, \ell_i$ is related to the basis $\partial_i \sigma$ and
$b_i$ for $i=0,\ldots, n$. 

\begin{definition}\label{defn:seifert_invariant}
  Let $\pi : M \to \Sigma$ be a Seifert fibration. Its {\bf Seifert
    invariant} is given by
  $$ (g;\beta_0, (\alpha_1,\beta_1),\dots,(\alpha_n,\beta_n)). $$
\end{definition}

\begin{remark}\label{rmk:seifert_invariant}
  If two Seifert fibrations have equal Seifert invariants, then the
  total spaces are diffeomorphic. Conversely, given any suitable
  abstract data $
  (g;\beta_0, (\alpha_1,\beta_1),\dots,(\alpha_n,\beta_n)) $, there
  exists a Seifert fibration with the above as its Seifert invariant
  (see \cite[Theorem I.3.7]{audin2}).
\end{remark}

\begin{remark}\label{rmk:seifert_base}
  If a Seifert fibration $\pi : M \to \Sigma$ has Seifert invariant $
  (g;\beta_0, (\alpha_1,\beta_1),\dots,(\alpha_n,\beta_n)) $, then the
  compact, connected orientable orbi-surface $\Sigma$ only has isolated
  singular points (see Remark \ref{rmk:Seifert_principal}). Hence, by
  Theorem \ref{thm::classorbisurface}, $\Sigma$
  is classified by $g$,
  the non-negative integer $n$ and the pairs  $(\alpha_1,\beta_1),\dots,(\alpha_n,\beta_n)$.
\end{remark}

Following \cite[Section 3]{bs}, we state a result that shows how the Seifert
invariant of a Seifert fibration $\pi : M \to \Sigma$ determines the Euler
class of $\pi : M \to \Sigma$. First, we need to define the latter
notion. Instead of using a conceptually clearer but more involved
approach (as in, e.g., \cite{Haefliger}), we follow a simpler, {\it ad
  hoc} path as
in \cite[Section I.3.d]{audin} or \cite[Section 3]{bs}. Let $\pi : M
\to \Sigma$ be a Seifert fibration with Seifert invariant $(g;\beta_0,
(\alpha_1,\beta_1),\dots,(\alpha_n,\beta_n))$. Let $\alpha =
\mathrm{lcm}(\alpha_1,\ldots, \alpha_n)$, where $\alpha :=1$ if
$n=0$. We set $M':= M/\Z_{\alpha}$, where $\Z_{\alpha}$ acts on $M$ via
the inclusion in $S^1$. By \cite[Proposition I.3.8]{audin}, the
quotient $M'$ can be endowed with a structure of smooth manifold such
that 
\begin{itemize}[leftmargin=*]
\item the quotient map $q: M \to M'$ is smooth, 
\item there is a free and proper smooth action of $S^1/\Z_{\alpha}$ on
  $M'$, and
\item the quotient map is equivariant with respect to the projection $S^1 \to
  S^1/\Z_{\alpha}$. 
\end{itemize}
In the presence of exceptional orbits, the quotient map $q$ fails to
be a covering; in fact, it is a {\em branched} covering along the
exceptional orbits. (It is what is known as an {\em orbifold covering
  map}, see Definition \ref{defn:orbifold_covering_map} below.) Together with the second bullet point above, this
implies that the base space of the principal $S^1$-bundle $\pi' : M'
\to M'/(S^1/\Z_{\alpha})$ can be identified
with $|\Sigma|$. Abusing notation, we denote $|\Sigma|$ with the resulting
manifold structure by $B$. If $e' \in H^2(|\Sigma|;\Z) \simeq \Z$ denotes the Euler
class of $\pi' : M' \to \Sigma$, then we define the {\bf Euler class} of
$\pi : M \to \Sigma$ to be $e:= e'/\alpha \in H^2(|\Sigma|;\Q) \simeq \Q$
(cf. \cite[Proposition I.3.9]{audin}). 

\begin{proposition}[Section 4 in \cite{bs}]\label{def::EulerOrbibundle}
  The Euler class $e$ of a Seifert fibration with Seifert invariant $ (g;\beta_0, (\alpha_1,\beta_1),\dots,(\alpha_n,\beta_n))$
  is given by
  \begin{equation}
    \label{eq:23}
    e=\beta_0+\sum_{i=1}^n\frac{\beta_i}{\alpha_i}.
  \end{equation}
\end{proposition}

\begin{proof}[Sketch of proof]
  The idea is to calculate $e$ using standard obstruction theory as in
  the manifold case, using \cite[Section 3]{bs}. We fix the notation
  as in the derivation of the Seifert invariant of a Seifert
  fibration. Since all meridians are homologous to zero in $M$,
  \eqref{eq:13} implies that
  \begin{equation}
    \label{eq:14}
    \partial \sigma_0 = \beta_0 b_0 \text{ and } \partial \sigma_i =
    \frac{\beta_i}{\alpha_i}b_i \text{ for } 1 \leq i \leq n. 
  \end{equation}
  Any two principal orbits are homologous in $M$: To see this, first one can
  prove it for any two principal orbits that are contained in an
  $S^1$-invariant open neighborhood of a principal orbit and then use
  the fact that the open subset of principal orbits is path-connected
  in $M$. Hence, using \eqref{eq:14} and the fact that $\partial
  \sigma = \sum\limits_{i=0}^n \partial \sigma_i$, we obtain that
  $$ \partial \sigma = \left(\beta_0 + \sum\limits_{i=1}^n
    \frac{\beta_i}{\alpha_i}\right) b_0. $$
  Unraveling the obstruction-theoretic definition of $e$ given in
  \cite[Section 3]{bs}, the desired result follows.
\end{proof}

\begin{remark}\label{rmk:chern_class}
  Under the correspondence between complex line orbi-bundles and principal
  $S^1$-orbi-bundles over (compact, connected orientable)
  orbi-surfaces (see Section \ref{sec:definition-examples}), the Euler
  class of a Seifert fibration is equal to the {\bf first Chern class
    (or degree)} of the corresponding complex line orbi-bundle. (For
  our purposes, we take this to be a {\em definition} of the latter,
  which can be alternatively defined intrinsically using the ideas in \cite{Haefliger}.)
\end{remark}

\color{black}
\begin{remark}[Seifert orbifolds]\label{rmk:seifert-orbifolds}
  The above discussion of Seifert manifolds can be generalized to {\bf
    Seifert orbifolds} (see \cite[Sections 2 and
  3]{bs}): the former are 3-dimensional, compact connected oriented
  orbifolds that are the total spaces of principal $S^1$-orbi-bundles
  over compact, connected oriented orbi-surfaces that have only
  isolated singular points. Such a principal $S^1$-orbi-bundle $\pi :
  M \to \Sigma$ is called a {\bf Seifert orbi-fibration}. Associated
  to $\pi : M \to \Sigma$ is a Seifert invariant given as in
  Definition \ref{defn:seifert_invariant} and the Euler class of $\pi
  : M \to \Sigma$ can be obtained from the Seifert invariant by
  \eqref{eq:23}. 
\end{remark}

To conclude this section, we use Seifert manifolds to construct
certain four-dimensional orbifolds that play an important role in our
paper.

\begin{example}[From Seifert fibrations to projectivized plane bundles]\label{exm:projectivized}
  Let $\pi : M \to \Sigma$ be a Seifert fibration
  and let $\mathrm{pr} : L \to \Sigma$ be the corresponding complex
  line orbi-bundle. The {\bf projectivization} $\mathbb{P}(L \oplus
  \C) \to \Sigma$ is an orbifold that is the total space of an
  orbi-bundle with generic fiber $\C P^1$. To see this, we consider
  the manifold (!) $M \times \C P^1$ endowed with the following
  $S^1$-action:
  \begin{equation}
    \label{eq:16}
    \lambda \cdot (x,[z_0:z_1]) = (\lambda x, [z_0:\lambda z_1]). 
  \end{equation}
  Since the $S^1$-action on $M$ is locally free, so is this
  $S^1$-action on $M \times \C P^1$. Hence, the quotient $(M \times \C
  P^1)/S^1$ is an orbifold. In fact, since the exceptional orbits in
  $M$ are isolated, it follows that $(M \times \C
  P^1)/S^1$ has isolated singular points. The set of singular points
  is 
  \begin{equation*}
    \begin{split}
      & \{[x,[0:1]] \in (M \times \C
      P^1)/S^1 \mid [x] \text{ is singular in } \Sigma \} \\
      \cup & \{[x,[1:0]] \in (M \times \C
      P^1)/S^1 \mid [x] \text{ is singular in } \Sigma \}.
    \end{split}
  \end{equation*}
  Moreover, since the above $S^1$-action preserves the projection $M
  \times \C P^1 \to \Sigma$, there is a $C^{\infty}$ map $(M \times \C
  P^1)/S^1 \to \Sigma$ that can be checked to be an orbi-bundle with
  generic fiber $\C P^1$. It remains to show that the total space of
  this orbi-bundle can be identified with  $\mathbb{P}(L \oplus
  \C)$. To this end, we use the idea of \cite[Lemma VIII.2.4]{audin}:
  the map $(M \times \C
  P^1)/S^1 \to \mathbb{P}(L \oplus
  \C)$ that sends $[x,[z_0:z_1]]$ to $[xz_0:z_1]$ is the desired
  identification. \hfill$\Diamond$   
\end{example}

\subsubsection{Orbifold covering maps}\label{sec:orbif-cover-maps}
For our purposes, it is also useful to introduce orbi-bundles with
zero dimensional {\em disconnected} fibers. To this end, we follow
\cite[Section 1]{bs} and \cite[Section 13.2]{Thurston}.

\begin{definition}\label{defn:orbifold_covering_map}
  Let $M_1$ and $M_2$ be orbifolds. A surjective continuous map
  $\mathrm{pr}: |M_1| \to |M_2|$ is an {\bf orbifold covering map} if,
  for any $x \in M_2$, there exists an open neighborhood $U \subseteq
  M_2$ of $x$ and a l.u.c. $\oc$ uniformizing $U$ such that, for any
  connected component $V$ of $\mathrm{pr}^{-1}(U)$, there exist a
  subgroup $\Gamma'$ of $\Gamma$ and a
  l.u.c. $(\widehat{U},\Gamma', \psi)$ uniformizing $V$ with $\varphi
  = \mathrm{pr} \circ \psi$. We denote such a map by $\mathrm{pr}: M_1 \to M_2$.
\end{definition}

In other words, orbifold covering maps are locally modeled on taking
quotients of subgroups. Hence, in general, an orbifold covering
map fails to be a topological covering map, but it is one away from
singular points. Moreover, any orbifold covering map is automatically
a good $C^{\infty}$ map.

If $\mathrm{pr}: M_1 \to M_2$ is an orbifold covering map and the set
of singular points in $M_2$ has codimension at least two (as in the
case that is of most interest to us), then there exists $r \in
\Z_{\geq 0} \cup \{\infty\}$ such that, for any regular $x \in M_2$,
the cardinality of $\mathrm{pr}^{-1}(x)$ equals $r$. Following \cite[Section 1]{bs}, we call $r$ the
{\bf geometric degree} of $\mathrm{pr}: M_1 \to M_2$. 

\begin{example}[An orbifold covering map for the weighted projective line]\label{exm:cp1_covering_cp1pq}
  Let $p, q$ be positive integers that are coprime. We claim that the
  map
  \begin{equation}
    \label{eq:19}
    \begin{split}
      \mathrm{pr}: \C P^1 & \to \C P^1(p,q) \\
      [z_0:z_1] &\mapsto [z^p_0:z^q_1]
    \end{split}
  \end{equation}
  is an orbifold covering map of geometric degree $pq$. To see this,
  we observe that $\mathrm{pr}$ is induced by the smooth map
  \begin{equation}
    \label{eq:20}
    \begin{split}
      \widehat{\mathrm{pr}}: \C^2 \smallsetminus \{0\} & \to \C^2 \smallsetminus
      \{0\} \\
      (z_0,z_1) &\mapsto (z^p_0,z^q_1).
    \end{split}
  \end{equation}
  This map of \eqref{eq:20} is surjective; hence so is $\mathrm{pr}$. Moreover, it
  is a covering map when restricted to
  $(\C^*)^2$. Hence, the restriction of $\mathrm{pr}$ to the
  complement of the singular points in $\C P^1(p,q)$ is an (orbifold)
  covering map. Finally, by considering local charts near $[1:0]$ and
  $[0:1]$ in $\C P^1$ and the standard l.u.c.'s near $[1:0]$ and
  $[0:1]$ in $\C P^1(p,q)$, it can be checked that $\mathrm{pr} : \C P^1 \to \C P^1(p,q)$
  satisfies the condition of Definition
  \ref{defn:orbifold_covering_map} near the singular points of
  $\C
  P^1(p,q)$. Finally, the geometric degree of $\mathrm{pr} : \C P^1 \to \C
  P^1(p,q)$ is $pq$.
\end{example}

To conclude this section, we use Example \ref{exm:cp1_covering_cp1pq}
to compute the degree of the complex line
orbi-bundle $\mathcal{O}_{p,q}(k) \to \C P^1(p,q)$ (see Example
\ref{exm:weighted_O} and Remark \ref{rmk:chern_class}). While this result
is known (see \cite[Corollary 1.7]{furuta} and
\cite[Proposition 2.0.18]{MOY}), we provide a proof for completeness. 

\begin{lemma}\label{lemma:degree_Ok}
  Let $p,q$ be positive integers that are coprime and let $k$ be an
  integer. The degree of the complex line orbi-bundle
  $\mathcal{O}_{p,q}(k) \to \C P^1(p,q)$ equals $k/pq$. 
\end{lemma}

\begin{proof}
  Let $\mathcal{O}(k) \to \C P^1$ denote the complex line bundle of
  degree $k$. We claim that the map
  \begin{equation}
    \label{eq:18}
    \begin{split}
      \mathrm{Pr}: \mathcal{O}(k) & \to \mathcal{O}_{p,q}(k) \\
      [z_0,z_1,v] &\mapsto [z^p_0,z^q_1,v]
    \end{split}
  \end{equation}
  is an orbifold covering map that makes the following diagram commute
  \begin{equation}\label{eq:21}
    \begin{tikzcd}
      \mathcal{O}(k) \arrow[r,"\mathrm{Pr}"] \ar{d} &
      \mathcal{O}_{p,q}(k)\arrow{d} \\
      \C P^1 \arrow[r,"\mathrm{pr}"] & \C P^1(p,q),
    \end{tikzcd}
  \end{equation}
  where $\mathrm{pr} : \C P^1 \to \C P^1(p,q)$ is the orbifold
  covering map of Example \ref{exm:cp1_covering_cp1pq}. This follows
  by the same argument as in Example \ref{exm:cp1_covering_cp1pq} once
  we observe that $\mathrm{Pr}$ is induced by the smooth map
  \begin{equation*}
    \begin{split}
      \widehat{\mathrm{Pr}}: (\C^2 \smallsetminus \{0\}) \times \C & \to (\C^2 \smallsetminus \{0\}) \times \C \\
      (z_0,z_1,v) &\mapsto (z^p_0,z^q_1,v),
    \end{split}
  \end{equation*}
  which makes the following diagram commute
  \begin{equation*}
    \begin{tikzcd}
      (\C^2 \smallsetminus \{0\}) \times \C \arrow[r,"\widehat{\mathrm{Pr}}"] \ar{d} &
      (\C^2 \smallsetminus \{0\}) \times \C\arrow{d} \\
      \C^2 \smallsetminus \{0\} \arrow[r,"\widehat{\mathrm{pr}}"] & \C^2 \smallsetminus \{0\},
    \end{tikzcd}
  \end{equation*}
  where $\widehat{\mathrm{pr}} : \C^2 \smallsetminus \{0\} \to \C^2
  \smallsetminus \{0\}$ is given by \eqref{eq:20}. Moreover, this
  argument also implies that $\mathrm{pr}^*(\mathcal{O}_{p,q}(k))
  \simeq \mathcal{O}_k$.

  Let $S(\mathcal{O}_k) \to \C P^1$ and
  $S(\mathcal{O}_{p,q}(k)) \to \C P^1(p,q)$ denote the principal
  $S^1$-orbi-bundles associated to $\mathcal{O}(k) \to \C P^1$ and
  $\mathcal{O}_{p,q}(k) \to \C P^1(p,q)$ respectively. Then
  $\mathrm{Pr} : \mathcal{O}(k) \to \mathcal{O}_{p,q}(k)$ induces an
  orbifold covering map $S(\mathcal{O}_k) \to S(\mathcal{O}_{p,q}(k))$
  that we also denote by $\mathrm{Pr}$ by a slight abuse of
  notation. By construction, the analog of the commutative diagram
  \eqref{eq:21} holds replacing the top row with $\mathrm{Pr}:
  S(\mathcal{O}_k) \to S(\mathcal{O}_{p,q}(k))$, i.e., $\mathrm{Pr}$
  sends fibers to fibers. Moreover, since $\mathrm{pr}^*(\mathcal{O}_{p,q}(k))
  \simeq \mathcal{O}_k$, the preimage of a fiber of
  $S(\mathcal{O}_{p,q}(k)) \to \C P^1(p,q)$ consisting of regular
  points is the disjoint union of $pq$ fibers of $S(\mathcal{O}_k) \to
  \C P^1$. Since $\mathrm{pr} : \C P^1 \to \C P^1(p,q)$ has geometric
  degree $pq$ and since $S(\mathcal{O}_k) \to \C P^1$ has Euler class
  equal to $k$, by the covering formula of \cite[Section 3]{bs}, the Euler class of $S(\mathcal{O}_{p,q}(k)) \to \C
  P^1(p,q)$ equals $k/pq$. The result follows from Remark \ref{rmk:chern_class}.
\end{proof}

Combining Lemma \ref{def::EulerOrbibundle}, Remark
\ref{rmk:seifert-orbifolds} and Lemma \ref{lemma:degree_Ok}, we obtain
the following result immediately.

\begin{corollary}\label{cor:degree_seifert_cp1pq}
  Let $p,q$ be positive integers that are coprime and let $k$ be an
  integer. Let $S(\mathcal{O}_{p,q}(k)) \to \C P^1(p,q)$ denote the principal
  $S^1$-orbi-bundle associated to $\mathcal{O}_{p,q}(k) \to \C
  P^1(p,q)$ and let $(0;b_0, (p,l_p),(q,l_q))$ be its Seifert
  invariant. Then
  \begin{equation}
    \label{eq:24}
    \frac{k}{pq} = b_0 + \frac{l_p}{p} + \frac{l_q}{q}.
  \end{equation}
\end{corollary}

\subsection{Suborbifolds}\label{sec:suborbifolds}
In this section, we follow \cite{borz_brun,mestre_weilandt,weilandt}. 

\begin{definition}\label{defn:suborbifold_atlas}
  Let $|M|$ be a topological space, let $\mathcal{A} =
  \{(\widehat{U}_{\alpha},\Gamma_{\alpha},\varphi_{\alpha})\}_{\alpha \in A}$ be
  an orbifold atlas on $|M|$, and let $|N| \subseteq |M|$ be a subset
  endowed with the subspace topology. A {\bf suborbifold atlas} on
  $|N|$ is a subset $B \subseteq A$ together with a family of
  connected embedded submanifolds $\{\widehat{V}_{\beta} \subseteq
  \widehat{U}_{\beta}\}_{\beta \in B}$ such that
  \begin{enumerate}[label=(\arabic*),ref=(\arabic*),leftmargin=*]
  \item \label{item:14} for all $\beta \in B$,
    $\varphi_{\beta}(\widehat{V}_{\beta})$ is a non-empty open subset
    of $|N|$,
  \item \label{item:16} $\bigcup\limits_{\beta \in B}
    \varphi_{\beta}(\widehat{V}_{\beta}) = |N|$, and 
  \item \label{item:15} for all $\beta \in B$, there exists a subgroup
    $\Delta_{\beta}$ of $\Gamma_{\beta}$ so that
    \begin{itemize}[leftmargin=*]
    \item $\widehat{V}_{\beta}$ is $\Delta_{\beta}$-invariant, and
    \item if $\gamma \in \Gamma_\beta$ and $\widehat{x} \in
      \widehat{V}_{\beta}$ are such that $\gamma \cdot \widehat{x} \in
      \widehat{V}_{\beta}$, then there exists $\delta \in
      \Delta_{\beta}$ with $\gamma \cdot \widehat{x} = \delta \cdot
      \widehat{x}$. 
    \end{itemize}
  \end{enumerate}
  If, in addition, $\Delta_{\beta} = \Gamma_{\beta}$ for all $\beta
  \in B$, the suborbifold atlas is said to be {\bf full}. 
\end{definition}

As proved in \cite[Proposition 3.3]{weilandt}, if $M$ is an orbifold
and a subset $|N| \subseteq |M|$ has a suborbifold atlas, then $|N|$
inherits the structure of an orbifold, which we denote by
$N$. Moreover, the notion of full suborbifold atlas corresponds to
that of suborbifold atlas introduced in \cite[Definition
13.2.7]{Thurston}.

\begin{definition}\label{defn:suborbifold}
  Let $M$ be an orbifold. A {\bf (full) suborbifold} of $M$ is a subset $|N|$
  that admits a (full) suborbifold atlas.
\end{definition}

A full orbifold is also known as a {\em strong} suborbifold in the
literature (see,
e.g. \cite[Part (b) of Definition 2.6]{torusOrbifolds}).

\begin{remark}\label{rmk:full_suborbifolds_normal}
  If $N$ is a full suborbifold of $M$, the tangent orbi-bundle to $N$ can
  be seen as an orbi-subbundle of $TM$ and so the {\bf normal
    orbi-bundle} $\nu_N$ to $N$
  in $M$ may be defined as usual (see the discussion following
  \cite[Definition 2.6]{torusOrbifolds}).
\end{remark}

The proof of the existence of a tubular neighborhood for submanifolds
works {\em mutatis mutandis} for full suborbifolds. Hence, the
following result holds.

\begin{theorem}[Tubular neighborhood theorem for full suborbifolds]\label{thm:tub_neigh_orb}
  Let $M$ be an orbifold and let $N$ be a full suborbifold. Then there
  exist an open neighborhood $U \subseteq M$ of $N$, an open
  neighborhood $V \subseteq \nu_N$ of the zero section and a
  diffeomorphism $\varphi : U \to V$ that equals the zero section when
  restricted to $N$.
\end{theorem}

\begin{example}[Weighted projective spaces, continued]\label{exm:sub_weighted_proj_spaces}
  Fix an integer $n \geq 1$, positive
  integers $m_0,\dots, m_{n-1}$ satisfying
  $\gcd(m_0,\dots,m_{n-1})=1$, and a positive integer $m_n$. As in the
  case of manifolds, the
  weighted projective space $\C P^{n-1}(m_0,\ldots, m_{n-1})$ can be
  viewed as a (full) suborbifold of $\C P^n(m_0,\ldots,m_{n-1},m_n)$ as
  follows. We observe that the inclusion $\C^n \smallsetminus \{0\}
  \hookrightarrow \C^{n+1} \smallsetminus \{0\}$ given by

  $$(z_0,\ldots, z_{n-1}) \mapsto (z_0,\ldots, z_{n-1},0) $$
  is $\C^*$-equivariant under the actions on $\C^n
  \smallsetminus \{0\}$ and $\C^{n+1}
  \smallsetminus \{0\}$ defined in Example
  \ref{exm:weighted_proj_standard}. Hence, there is an induced map
  $\C P^{n-1}(m_0,\ldots, m_{n-1}) \to \C P^{n}(m_0,\ldots, m_{n-1},
  m_n)$ that is, in fact, a topological embedding; the image of this
  map is precisely
  \begin{equation}
    \label{eq:15}
     \{[z_0:\ldots: z_{n-1}:z_n] \in \C P^{n}(m_0,\ldots, m_{n-1},
     m_n) \mid z_n = 0\}.
  \end{equation}
  It can be checked, either directly, or using \cite[Theorem
  1]{borz_brun}, that the subset \eqref{eq:15} is a (full) suborbifold of $ \C P^{n}(m_0,\ldots, m_{n-1},
  m_n)$ that is diffeomorphic to $\C
  P^{n-1}(m_0,\ldots, m_{n-1})$. The normal orbi-bundle to $\C
  P^{n-1}(m_0,\ldots, m_{n-1})$ in $ \C P^{n}(m_0,\ldots, m_{n-1},
  m_n)$ can be identified with the complement of $[0:\ldots:0:1]$ in $ \C P^{n}(m_0,\ldots, m_{n-1},
  m_n)$ and is isomorphic to $\mathcal{O}_{m_0,\dots,m_{n-1}}(m_n)$. \hfill$\Diamond$   
\end{example}

\begin{remark}\label{ex::weightedprojSeifertInvariants}
  Let $p,q,k$ be pairwise coprime positive integers. We can use Example \ref{exm:sub_weighted_proj_spaces} to calculate the
  Seifert invariant of $S(\mathcal{O}_{p,q}(k)) \to \C P^1(p,q)$,
  since this is the principal $S^1$-orbi-bundle associated to the
  normal bundle to $\C P^1(p,q)$ in $\C P^2(p,q,k)$. We know that the
  Seifert invariant is of the form $(0;b_0, (p,l_p),(q,l_q))$.

  First we determine $l_p$ and $l_q$. Since $p,q$ are coprime, by
  B\'{e}zout's lemma, there exist unique integers $n_1,n_2$ with
  $|n_1|<q$ and $|n_2|<p$ such that
  \begin{equation}
    \label{eq:25}
    n_1p+n_2q=1.
  \end{equation}
  Moreover, since
  $p,q>0$, $n_1 n_2 < 0$. Without loss of generality we may assume
  that $n_1<0$. If we set $n_1^\prime:=-n_1$, $\beta:=n_2k$ and
  $\alpha:=n_1^\prime k$, then, by \eqref{eq:25},
  \begin{align}\label{ex_seifertinv}
    q\beta-p\alpha=k.
  \end{align}
  If we view $\C P^2(p,q,k)$ as $(\C^3 \smallsetminus \{0\})/\C^*$ as
  in Example \ref{exm:weighted_proj_standard}, then $l_p$ (respectively $l_q$)
  is obtained by understanding the action of $\Z_p$ (respectively
  $\Z_q$) near the point $(1,0,0)$ (respectively $(0,1,0)$). If
  $\lambda \in \Z_p$ and $\mu \in \Z_q$, by \eqref{ex_seifertinv}, 
  \begin{equation*}
    \lambda \cdot (z_0,z_1,z_2) = (z_0, \lambda^q z_1, \lambda^{q\beta}
    z_2) \text{ and } \mu \cdot (z_0,z_1,z_2) = (\mu^p  z_0, z_1,\mu^{-p\alpha}z_2)
  \end{equation*}
  for any $(z_0,z_1,z_2) \in \C^3 \smallsetminus \{0\}$. Let
  \begin{equation}
    \label{eq:32}
    \beta^\prime\equiv\beta\mod p \quad \text{and} \quad -\alpha^\prime\equiv-\alpha\mod q
  \end{equation}
  be such that $0<\beta^\prime<p$ and $0<q-\alpha^\prime<q$ respectively. 
  Since the two  structure groups act with exponent $\beta^\prime$ and
  $q-\alpha^\prime$ in the fiber direction of the normal bundle (given
  by the $z_2$-coordinate), $l_p = \beta'$ and $l_q =
  q-\alpha^{\prime}$.

  It remains to calculate $b_0$. To this end, by definition of
  $\alpha'$ and $\beta'$ and by
  \eqref{ex_seifertinv}, there exists an integer $N$ such that 
  \begin{equation}\label{eq:degree}
    \beta^\prime q-\alpha^\prime p=k-Npq.
  \end{equation} 
  By Corollary \ref{cor:degree_seifert_cp1pq},
  $$ \frac{k}{pq} = b_0
  +\frac{\beta^\prime}{p}+\frac{q-\alpha^\prime}{q}.$$
  Hence, by \eqref{eq:degree}, $b_0=N-1$, so that the Seifert
  invariant is
  $$(0;N-1, (p,\beta'),(q,q-\alpha^{\prime})), $$
  where
  $\alpha^{\prime}, \beta'$ are defined in \eqref{eq:32}.
\end{remark}

If a complex line bundle $L\rightarrow \mathbb{C}P^1/\mathbb{Z}_m$  is
trivial, then the Seifert invariant of the corresponding  Seifert
fibration satisfies the following a useful condition.
\begin{lemma}\label{lem::deg_zero}
  Let $L\rightarrow \mathbb{C}P^1/\mathbb{Z}_m$ be a complex line
  orbi-bundle whose associated circle orbi-bundle has Seifert invariant $(0;b,(m,l),(m,\hat{l}))$. If the degree of $L$ is zero then $l+\hat{l}=m$.
\end{lemma}

\begin{proof} If the degree of $L$ is zero, then, by Proposition~\ref{def::EulerOrbibundle}, we have
$$0 = \deg(L)=b+\frac{l+\hat{l}}{m},$$ 
and so $l+\hat{l}=\alpha\cdot m$ for some positive integer $\alpha$ since $1\leq l,\hat{l}<m$.  If  $\alpha\geq 2$, then we have 
$$l = \alpha\cdot m-\hat{l} \geq 2 m -\hat{l}> m $$ 
which is impossible. Hence, $\alpha=1$ and the result follows.
\end{proof}

\begin{example}[Projectivized plane bundles, continued]\label{exm:projectivized_suborbifolds}
  Let $\Sigma$ be a compact, connected orientable orbi-surface and let
  $\mathrm{pr} : L \to \Sigma$ be a complex line orbi-bundle. As in
  the case of manifolds, the
  projectivization $\mathbb{P}(L \oplus \C) \to \Sigma$ constructed in
  Example \ref{exm:projectivized} admits two `standard' $C^{\infty}$
  sections, namely $x \mapsto [0:1]$ and $x \mapsto [1:0]$, where the
  former lifts to a section of the rank-two complex orbi-bundle $L \oplus
  \C \to \Sigma$ and the latter does not. The images of these two
  sections are full suborbifolds of $\mathbb{P}(L \oplus \C)$ that are
  diffeomorphic to $\Sigma$. \hfill$\Diamond$   
\end{example}
\subsection{Group actions on orbifolds}\label{sec:group-acti-orbif}

\subsubsection{Definition and examples}\label{sec:defintion-examples} 
\begin{definition}
  Let $G$ be a Lie group and let $M$ be an orbifold. A \textbf{(smooth)
    action} of $G$ on $M$ is a continuous action of $G$ on $|M|$ such
  that the map
  \begin{equation*}
    \begin{split}
      \phi:G\times M &\rightarrow M \\
      (g,x) &\mapsto g \cdot x
    \end{split}
  \end{equation*}
  is a good $C^{\infty}$ map.
\end{definition}

\begin{remark}\label{rmk:infinitesimal_action}
  As in the case of manifolds, given a $G$-action on an orbifold $M$,
  any element $\xi$ of the Lie algebra $\mathfrak{g}$ of $G$ gives
  rise to a vector field $\xi^M $ on $M$ via the usual formula:
  $$ \xi^M(x) := \frac{d}{d t}\bigg\rvert_{t = 0} \phi(\exp(t\xi),x)
  \quad \text{for any } x \in M, $$
  where $\exp : \mathfrak{g} \to G$ denotes the exponential map.
\end{remark}

\begin{example}[Orbi-vector spaces, continued]\label{exm:linear_actions_orbi-vector_space}
  Let $\widehat{V}/\Gamma$ be an orbi-vector space. We define
  $\mathrm{GL}(\widehat{V}/\Gamma): = N(\Gamma)/\Gamma$, where $N(\Gamma) \leq
  \mathrm{GL}(\widehat{V})$ is the normalizer of $\Gamma$ in
  $\mathrm{GL}(\widehat{V})$. A {\bf representation} of a group $G$ on $\widehat{V}/\Gamma$ is a homomorphism $\rho : G \to
  \mathrm{GL}(\widehat{V}/\Gamma)$. If $G$ is a Lie group, a smooth action of
  $G$ on $\widehat{V}/\Gamma$ is {\bf linear} if there exists a representation
  $\rho$ of $G$ on $\widehat{V}/\Gamma$ such that $\phi(g,x) = \rho(g)(x)$ for
  all $g \in G$ and all $x \in \widehat{V}/\Gamma$. In analogy with \cite[Lemma
  3.1]{LermanTolman}, given a linear action of $G$ on $\widehat{V}/\Gamma$,
  there exist a Lie group $\widehat{G}$, a representation $\rho:
  \widehat{G} \to N(\Gamma)$ and a short exact sequence of Lie groups
  \begin{equation*}
    \begin{tikzcd}
      1\arrow{r} & \Gamma\arrow[r,hook,"\iota"]& \widehat{G}\arrow[r,"\pi"] & G\arrow{r} & 1,
    \end{tikzcd}
  \end{equation*}
  such that the following diagram commutes
  \begin{equation*}
    \begin{tikzcd}
      1\arrow{r} & \Gamma\arrow[r,hook,"\iota"] \arrow[d,equal]&
      \widehat{G}\arrow[r,"\pi"] \arrow[d,"\widehat{\rho}"]
      & G\arrow{r} \arrow[d,"\rho"] & 1 \\
      1\arrow{r} & \Gamma\arrow[r,hook,"\iota"]& N(\Gamma)\arrow{r}
      & \mathrm{GL}(V/\Gamma) \arrow{r} & 1.
    \end{tikzcd}
  \end{equation*}
  Moreover, $\rho$ is faithful if and only if $\widehat{\rho}$ is
  faithful. \hfill$\Diamond$   
\end{example}

Many of the smooth actions that we consider in this paper arise from
the following construction: Let $H$ be a Lie group that acts properly
and in a locally free fashion on the manifold $\widehat{M}$. Let $G$
be a Lie group that acts on $\widehat{M}$ so that the $G$- and
$H$-actions commute. The action of $g \in G$ on $\widehat{x} \in
\widehat{M}$ is denoted by $g\cdot x$. Then there is a continuous
$G$-action on $|\widehat{M}/H|$ given by
\begin{equation*}
  \begin{split}
    \phi: G \times |\widehat{M}/H| & \to |\widehat{M}/H| \\
    (g, [\widehat{x}]) & \mapsto [g \cdot \widehat{x}].
  \end{split}
\end{equation*}
By Lemma \ref{lemma:obvious_good_maps}, $\phi$ is a good $C^{\infty}$
map. We stress that the $G$-action on $\widehat{M}/H$ need not be
effective even if the $G$-action on $\widehat{M}$ is. The following
examples illustrate the above construction.
      
\begin{example}[Weighted projective spaces, continued]\label{exm:circle_action_weighted}
  Fix an integer $n \geq 1$, positive
  integers $m_0,\dots, m_{n}$ satisfying
  $\gcd(m_0,\dots,m_{n})=1$, and integers $a_0,\ldots, a_n$. The map
  \begin{equation*}
    \begin{split}
      \phi: S^1 \times \C P(m_0,\ldots, m_n) &\to \C P(m_0,\ldots,
      m_n) \\
      (\lambda, [z_0:\ldots:z_n]) &\mapsto [\lambda^{a_0}z_0:\ldots: \lambda^{a_n}z_n]
    \end{split}
  \end{equation*}
  is an $S^1$-action on $\C P(m_0,\ldots, m_n)$. \hfill$\Diamond$   
\end{example}

\begin{example}[Example \ref{tswEx_1}, continued]\label{exm:circle_action_Talvacchia}
  We let $M =  (\C P^1 \times \C P^1)/\Z_4$ be the orbifold
  constructed in Example \ref{tswEx_1}. The map
  \begin{equation*}
    \begin{split}
      \phi: S^1 \times M &\to M \\
      \left(\lambda,\left[[z_0:z_1],[w_0:w_1]\right]\right)&\mapsto
      \left[[\lambda^{1/2}z_0:z_1],[ \lambda^{1/2} w_0:w_1]\right]
    \end{split}
  \end{equation*}
  is an $S^1$-action on $M$. We observe that the expression involving
  square roots of elements of $S^1$ are well-defined because of the
  definition of the $\Z_4$-action on $\C P^1 \times \C P^1$. \hfill$\Diamond$   
\end{example}

\begin{example}[Example \ref{ex:quotient_cp1xcp1_finite}, continued]\label{exm:circle_action_cp1cp1}
  Given an integer $c > 1$, let $M_c =  (\C P^1 \times \C P^1)/\Z_c$ be the orbifold constructed in Example
  \ref{ex:quotient_cp1xcp1_finite} and let $m,n$ be positive integers. The map
  \begin{equation*}
    \begin{split}
      \phi: S^1 \times M_c &\to M_c \\
      (\lambda, [[z_0: z_1],[w_0: w_1]]_c) &\mapsto [[ z_0: \lambda^{ \frac{m}{c}} z_1], [ w_0: \lambda^{\frac{n}{c}} w_1]]_c
    \end{split}
  \end{equation*}
  is an $S^1$-action on $M_c$, where, as in Example
  \ref{exm:circle_action_Talvacchia}, there is no ambiguity in the
  expressions involving $c$-th roots because of the $\Z_c$ action on
  $\C P^1 \times \C P^1$. \hfill$\Diamond$   
\end{example}

\begin{example}[Projectivized plane bundles, continued]\label{exm:circle_action_proj}
  Let $\pi : M \to \Sigma$ be a Seifert fibration. The map
  \begin{equation*}
    \begin{split}
      \phi : S^1 \times ((M
      \times \C P^1)/S^1) & \to (M
      \times \C P^1)/S^1 \\
      (\lambda, [x,[z_0:z_1]]) &\mapsto [x,[\lambda z_0:z_1]]
    \end{split}
  \end{equation*}
  is an $S^1$-action on the orbifold $(M
  \times \C P^1)/S^1$ constructed in Example \ref{exm:projectivized}
  (see \eqref{eq:16} for the $S^1$-action on $M \times \C
  P^1$). Let $\mathrm{pr} : L \to \Sigma$ be the complex line
  orbi-bundle corresponding to $\pi : M \to \Sigma$. Using the
  identification between $(M
  \times \C P^1)/S^1$ and $\mathbb{P}(L \oplus \C)$ given in Example
  \ref{exm:projectivized}, we obtain an $S^1$-action on the latter. Note that  if $[x]\in \Sigma$ is a singular point with orbifold
 structure group $\Z_k$, then by \eqref{eq:16}
  $$
 [x,[\xi_k z_0:z_1]] = [x,[z_0:\xi_k^{-1} z_1]] =  [\xi_k x,[z_0: z_1]] = [x,[z_0: z_1]],
  $$
  and so each point in the fiber over $[x]$ has $\mathbb{Z}_k$ as stabilizer.
  \hfill$\Diamond$   
\end{example}

\subsubsection{Linearization at a fixed point}\label{sec:line-at-fixed}
Lemma \ref{lemma:good_homomorphism} has the
following immediate consequence that is important for our purposes.

\begin{corollary}\label{cor:action_preserves_structure_group}
  Let $G$ act on an orbifold $M$. If $x, y \in M$ lie in the same
  $G$-orbit, then $\Gamma_x \simeq \Gamma_y$. In particular, if $M$ has
  isolated singular points and $G$ is
  connected, then any singular point in $M$ is fixed by the $G$-action.
\end{corollary}

Let $G$ be a compact Lie group acting on an orbifold $M$. The local structure of a $G$-action in a neighborhood of a fixed point
$x \in M$ is relevant for the purposes of our paper. First, we observe
that there exists a l.u.c. $\oc$
centered at $x$ such that the open neighborhood $U$ of $x$ that is
uniformized by $\oc$ is $G$-invariant. We call such a l.u.c. {\bf
  $\boldsymbol{G}$-invariant}. Moreover, the $G$-action on $U$
can be described entirely by an action of a `larger' Lie group on
$\widehat{U}$. More precisely, the following result holds (for a proof, see \cite[Section 2]{LermanTolman} or
\cite[Proposition 2.8 and Corollary 2.9]{torusOrbifolds}; cf. Example \ref{exm:linear_actions_orbi-vector_space}). 

\begin{proposition}\label{prop:luc_centered_equivariant}
  Let $G$ be a compact Lie group acting on an orbifold $M$, let $x
  \in M$ be a fixed point and let $\oc$ be a $G$-invariant
  l.u.c. centered at $x$. There exist a Lie group $\widehat{G}$, a
  smooth action of $\widehat{G}$ on $\widehat{U}$ fixing $0$, and a
  short exact sequence of Lie groups
  \begin{equation}\label{exact_sequ}
    \begin{tikzcd}
      1\arrow{r} & \Gamma\arrow[r,hook,"\iota"]& \widehat{G}\arrow[r,"\pi"] & G\arrow{r} & 1
    \end{tikzcd}
  \end{equation}
  such that the following diagram commutes
  \begin{center}
    \begin{tikzcd}
      \widehat{G}\times \widehat{U} \arrow[r,"\widehat{\phi}"] \arrow[d, "{(\pi,\varphi)}"]
      & \widehat{U} \arrow[d, "\varphi"] \\
      G\times U\arrow[r, "\phi"]
      & U,
    \end{tikzcd}
  \end{center}
  where $\widehat{\phi} : \widehat{G} \times \widehat{U} \to
  \widehat{U}$ encodes the $\widehat{G}$-action. Moreover, $\Gamma$ commutes with every connected subgroup of $\widehat{G}$.
\end{proposition}

% \begin{remark}\label{rmk:quotient_orbifold}
%   By inspecting the proof of \cite[Proposition 2.8 and Corollary
%   2.9]{torusOrbifolds}, the following {\em global} version of Proposition
%   \ref{prop:luc_centered_equivariant} holds: Let $G$ be a compact Lie
%   group acting on a quotient orbifold $\widehat{M}/H$, where $H$ is a
%   compact Lie group. Then there exist a Lie group $\widehat{G}$, a
%   smooth action of $\widehat{G}$ on $\widehat{M}$, and a
%   short exact sequence of Lie groups
%   \begin{equation}
%     \begin{tikzcd}
%       1\arrow{r} & H\arrow[r,hook,"\iota"]& \widehat{G}\arrow[r,"\pi"] & G\arrow{r} & 1
%     \end{tikzcd}
%   \end{equation}
%   satisfying the conclusions of Proposition \ref{prop:luc_centered_equivariant}.
% \end{remark}

On the other hand, as in the case of manifolds, there is a
linear $G$-action on the tangent space $T_xM$ given by taking
derivatives of the $G$-action at $x$ (see Remark
\ref{rmk:tangent_orbi-bundle} and Example
\ref{exm:linear_actions_orbi-vector_space}). The next result explains
the relation between the $G$-action near $x$ and the linear $G$-action
on $T_xM$; it is the analog of the Bochner linearization theorem for
orbifolds (see \cite[Theorem 2.2.1]{dk}).

\begin{corollary}\label{cor:linearization_fixed_point}
  Let $G$ be a compact Lie group acting on an orbifold $M$. If $x$ is a
  fixed point, then there exist $G$-invariant open neighborhoods $U
  \subseteq M$ and $W \subseteq T_xM$ of $x$ and $[0]$ respectively,
  and a $G$-equivariant diffeomorphism $f : U \to W$ that sends $x$ to $[0]$.
\end{corollary}

\begin{proof}
  Let $\oc$ be a $G$-invariant l.u.c. centered at $x$ and let
  $\widehat{G}$ be the Lie group extending $G$ given by Proposition
  \ref{prop:luc_centered_equivariant}. Since $\widehat{G}$ is compact,
  by the Slice Theorem (see \cite[Theorem
  2.4.1]{dk}), there is no loss of generality in assuming that the
  $\widehat{G}$-action on $\widehat{U}$ is the restriction of a linear
  $\widehat{G}$-action. By Proposition
  \ref{prop:luc_centered_equivariant}, the map $\overline{\varphi} :
  \widehat{U}/\Gamma \to U$ is a $G$-equivariant
  diffeomorphism, where $\widehat{U}/\Gamma \subseteq \R^n/\Gamma$ is
  endowed with the restriction of a linear $\widehat{G}/\Gamma \simeq
  G$-action. Since the latter
  linear action corresponds to the linear $G$-action on $T_x M$ under
  the isomorphism $T_x M \simeq \R^n/\Gamma$, the result follows.
\end{proof}

We conclude this section by stating the following consequence of
Corollary \ref{cor:linearization_fixed_point} (see \cite[Corollary
2.10]{torusOrbifolds} for a proof).

\begin{corollary}\label{cor:fixed_pt_set_suborbi}
  Let $G$ be a compact Lie group that acts effectively on an orbifold
  $M$. Any connected component of the fixed point set of the action
  that has positive dimension\footnote{This restriction is
  necessary as we are working only with {\em effective}
  orbifolds.} is
  a full suborbifold of $M$.
\end{corollary}

\subsubsection{Quotients of orbifolds by finite
  groups}\label{sec:quot-orbif-finite}
In this subsection we show that the quotient of an orbifold by an
effective action of a finite group inherits the structure of an
orbifold (see Corollary \ref{cor:quotient_orbifold} below), and then we
illustrate this construction with several examples that are relevant
for our purposes.

Let $M$ be an orbifold of dimension $n$. The {\bf orthonormal frame} orbi-bundle
$\mathrm{Fr}(M) \to M$ can be constructed as in \cite[Section
1.3]{adem}: The total space can be obtained by gluing together the
orthonormal frame bundles of l.u.c.'s for $M$. (Alternatively, one can
define a Riemannian metric for an orbifold, prove that any orbifold
admits such a metric and construct the orthonormal frame
orbi-bundle as in the case of manifolds. We leave the details to the reader.) This principal
$\mathrm{O}(n)$-orbi-bundle has the following important property (see
\cite[Theorem 1.23]{adem} for a proof).

\begin{theorem}\label{thm:frame}
  Let $M$ be an orbifold of dimension $n$. The total space
  $\mathrm{Fr}(M)$ of the
  orthonormal frame orbi-bundle of $M$ is a smooth manifold endowed with an
  effective, locally free $\mathrm{O}(n)$-action such that the
  quotient orbifold $\mathrm{Fr}(M)/\mathrm{O}(n)$ is diffeomorphic to
  $M$. 
\end{theorem}

As a consequence of Theorem \ref{thm:frame}, every orbifold $M$ can be
thought of as a quotient orbifold $\widehat{M}/H$, where $H$ is
compact. This fact has the following useful consequence, stated below
without proof (see \cite[Remark 2.15]{kleiner_lott}).

\begin{corollary}\label{cor:quotient_orbifold}
  If $G$ is a finite group acting on an orbifold $M$, then the
  quotient $M/G$ can be endowed with the structure of an orbifold such
  that the quotient map $M \to M/G$ is an orbifold covering map. 
\end{corollary}

\begin{example}[Cyclic quotients of weighted projective lines]\label{exm:quotient_wpl}
  Let $p,q$ be positive integers that are pairwise coprime and let
  $a,b$ be integers. We consider the following smooth action of
  $\mathbb{Z}_m$  on $\mathbb{C}P^1(p,q)$:
\begin{equation}
  \label{eq:26}
  \begin{split}
    \phi:\mathbb{Z}_m\times\mathbb{C}P^1(p,q)&\rightarrow \mathbb{C}P^1(p,q)\\
    (\mu,[z_0:z_1])&\mapsto[\lambda^az_0:\lambda^bz_1].
  \end{split}
\end{equation}
By the weighted homogeneity of the coordinates in $\C P^1(p,q)$,
\begin{equation}
  \label{eq:27}
  \mu \cdot [z_0:z_1] =  [z_0: \mu^{\frac{bp-aq}{p}}z_1] =
  [\mu^{\frac{aq-bp}{q}}z_0:z_1]
\end{equation}
for any $\mu \in \Z_m$ and any $[z_0:z_1] \in \C P^1(p,q)$. We set
$k_1:=|aq-bp|$.  A necessary and sufficient condition for this action to be effective
and so that the set of points with non-trivial stabilizer is discrete
is that $\gcd(k_1,m) = 1$. In this case, the only points with
non-trivial stabilizers are $[1:0]$ and $[0:1]$. By Corollary
\ref{cor:quotient_orbifold}, the quotient $\C P^1(p,q)/\Z_m(a,b)$ inherits the
structure of an orbifold that has two singular points (corresponding
to $[1:0]$ and $[0:1]$), and with respect
to which the quotient map $\C P^1(p,q) \to \C P^1(p,q)/\Z_m(a,b)$ is
an orbifold covering map. This orbifold structure can also be obtained
as a quotient orbifold (see Example
\ref{exm:cyclic_quotients_symp}).

We claim that the orbifold structure group of the singular point
corresponding to $[1:0]$ is $\Z_{mp}$. To see this, we consider the
standard $\Z_{mp}$-action on $\C$ and observe that an open
neighborhood of $[1:0] \in \C P^1(p,q)$ is diffeomorphic to an open
neighborhood of $[0] \in \C/\Z_p$, where $\Z_p$ acts on $\C$ via the
inclusion $\Z_p \hookrightarrow \Z_{mp}$. Since $\gcd(k_1,m) = 1$, the residual $\Z_{mp}/\Z_p
\simeq \Z_m$-action on $\C/\Z_p$ can be identified with the
$\Z_m$-action on $\C/\Z_p$ given by \eqref{eq:27}. Analogously, the structure group of the singular point
corresponding to $[0:1]$ is $\Z_{mq}$. In particular, by Theorem
\ref{thm::classorbisurface}, if $a', b' \in \Z$ are also such that
$k'_1:=|a'q-b'p|$ is coprime with $m$, then $\C P^1(p,q)/\Z_m(a,b)$
and $\C P^1(p,q)/\Z_m(a',b')$ are diffeomorphic. Hence, we write
simply $\C P^1(p,q)/\Z_m$ for this quotient. 
\end{example}

\begin{example}[Complex
  line bundles over cyclic quotients of weighted projective lines]\label{ex::orbifold_quotient}
  In analogy with Example \ref{exm:quotient_wpl}, let $p,q,k$ be positive integers that are pairwise coprime and let
  $a,b,c$ be integers. We consider the following smooth action of $\mathbb{Z}_m$  on $\mathbb{C}P^2(p,q,k)$:
  \begin{equation}
    \label{eq:17}
    \begin{split}
      \phi:\mathbb{Z}_m\times\mathbb{C}P^2(p,q,k)&\rightarrow \mathbb{C}P^2(p,q,k)\\
      (\lambda,[z_0:z_1:z_2])&\mapsto[\lambda^az_0:\lambda^bz_1:\lambda^cz_2].
    \end{split}
  \end{equation}
  By the weighted homogeneity of the coordinates in
  $\mathbb{C}P^2(p,q,k)$, 
  \begin{equation}
    \label{eq:22}
    \begin{split}
      \lambda \cdot [z_0:z_1:z_2] & = [z_0: \lambda^{\frac{bp-aq}{p}}z_1:
      \lambda^{\frac{cp-ak}{p}}z_2] \\
      & = [\lambda^{\frac{aq-bp}{q}}z_0: z_1:
      \lambda^{\frac{cq-bk}{q}}z_2] \\
      & = [\lambda^{\frac{ak-cp}{k}}z_0: \lambda^{\frac{bk-cq}{k}} z_1:z_2],
    \end{split}
  \end{equation}
  for any $\lambda \in \Z_m$ and any
  $[z_0:z_1:z_2] \in \mathbb{C}P^2(p,q,k)$. 
  We set $k_1:=|aq-bp|, k_2:=|bk-cq|$ and $k_3:=|cp-ak|$. A necessary and sufficient condition for this action to be effective
  and so that the set of points with non-trivial stabilizer is discrete
  is
  \begin{equation*}
    \gcd(k_1,m) = \gcd(k_2,m) = \gcd(k_3,m) = 1.
  \end{equation*}
  In this case, the points that have non-trivial stabilizers are
  $[1:0:0]$, $[0:1:0]$ and $[0:0:1]$. By Corollary
  \ref{cor:quotient_orbifold}, the quotient $\mathbb{C}P^2(p,q,k)/\Z_m(a,b,c)$
  inherits the structure of an orbifold that has three singular points,
  corresponding to $[1:0:0]$, $[0:1:0]$ and $[0:0:1]$. This
  orbifold structure can also be realized as a quotient orbifold (see
  Example \ref{exm:cyclic_quotients_symp} below). By \eqref{eq:22},
  the orbifold structure groups of the singular points
  corresponding to $[1:0:0]$, $[0:1:0]$ and $[0:0:1]$ are 
  $\Z_{mp}$, $\Z_{mq}$ and $\Z_{mk}$.
  
  In analogy with Example \ref{exm:sub_weighted_proj_spaces}, the subset
  of $\mathbb{C}P^2(p,q,k)/\Z_m(a,b,c)$ given by $\{[z_0:z_1:0] \in \C P^2(p,q,k)\}/\Z_m $
  is a (full) suborbifold that can be identified with the orbifold 
  $\C P^1(p,q)/\Z_m = \C P^1(p,q)/\Z_m(a,b)$ of Example \ref{exm:quotient_wpl}. The
  normal orbi-bundle to $\C P^1(p,q)/\Z_m$ in $\mathbb{C}P^2(p,q,k)/\Z_m(a,b,c)$ is
  precisely $\mathcal{O}_{p,q}(k)/\Z_m(a,b,c) \to \C P^1(p,q)/\Z_m$, where the
  $\Z_m$-action on $\mathcal{O}_{p,q}(k)$ is given by a formula entirely
  analogous to \eqref{eq:17}. \hfill$\Diamond$ 
\end{example}

For our purposes, we need to compute the degree of the complex line
orbi-bundle $\mathcal{O}_{p,q}(k)/\Z_m(a,b,c) \to \C P^1(p,q)/\Z_m$
of Example \ref{ex::orbifold_quotient} and to relate it to the Seifert invariant of the corresponding circle
bundle. The first of these results can be proved with the techniques of Lemma
\ref{lemma:degree_Ok} using the orbifold covering map $\C P^1(p,q) \to
\C P^1(p,q)/\Z_m$. Hence, we omit its proof.

\begin{lemma}\label{lem::linebundle}
  Let $p,q,k$ be positive integers that are pairwise coprime, let $m$
  be a positive integer, and let
  $a,b,c$ be integers such that
  $$ \gcd(|aq-bp|,m) = \gcd(|bk-cq|,m) = \gcd(|cp-ak|,m) = 1.$$
  The degree of the complex line orbi-bundle
  $$\mathcal{O}_{p,q}(k)/\Z_m(a,b,c) \to \C P^1(p,q)/\Z_m$$
  is $k/pqm$. 
\end{lemma}

Combining Lemma \ref{def::EulerOrbibundle}, Remark
\ref{rmk:seifert-orbifolds} and Lemma \ref{lem::linebundle}, the
following result holds.

\begin{corollary}\label{cor:degree_quotient}
  Let $p,q,k$ be positive integers that are pairwise coprime, let $m$
  be a positive integer, and let
  $a,b,c$ be integers such that
  $$ \gcd(|aq-bp|,m) = \gcd(|bk-cq|,m) = \gcd(|cp-ak|,m) = 1.$$
  Let $S(\mathcal{O}_{p,q}(k)/\Z_m(a,b,c)) \to \C
  P^1(p,q)/\Z_m$ denote the principal $S^1$-orbi-bundle
  associated to $\mathcal{O}_{p,q}(k)/\Z_m(a,b,c) \to \C
  P^1(p,q)/\Z_m$ and let 
  $$(0;b_0, (mp, l_{mp}),
  (mq,l_{mq}))$$
  be the Seifert invariant of $S(\mathcal{O}_{p,q}(k)/\Z_m(a,b,c)) \to \C
  P^1(p,q)/\Z_m$. Then 
  \begin{equation}
    \label{eq:28}
    \frac{k}{pqm} = b_0 + \frac{l_{mp}}{mp} + \frac{l_{mq}}{mq}.
  \end{equation}
\end{corollary}

\subsection{Symplectic orbifolds and Hamiltonian
  actions} \label{sec:sympl-orbif-hamilt}

\subsubsection{Definition and Examples}\label{sec:definition-examples-1}
The definition of a symplectic form carries on from manifolds to
orbifolds. % To this end, we observe that, given an orbifold $M$ and a 2-form $\omega \in
% \Omega^2(M)$, there is a morphism of vector
% orbi-bundles\footnote{Strictly speaking, we have not defined this
%   notion. However, we trust that, with the background material in
%   Section \ref{sec:definition-examples}, it is clear how to define
%   it.} $\omega^{\flat} : TM \to T^*M$ given by contraction. We say
% that $\omega$ is {\bf non-degenerate} if $\omega^{\flat}$ is an
% isomorphism. 

\begin{definition}\label{defn:symplectic}
  Let $M$ be an orbifold. A {\bf symplectic form} on $M$ is a closed,
  non-degenerate 2-form $\omega \in \Omega^2(M)$. A \textbf{symplectic
    orbifold} is a pair $(M,\omega)$, where $\omega$ is a symplectic
  form on $M$. A \textbf{symplectomorphism} between two symplectic
  orbifolds $(M_1,\omega_1),(M_2,\omega_2)$ is a diffeomorphism
  $\varphi:M\rightarrow M'$ that satisfies
  $\varphi^*\omega'=\omega$. We denote a symplectomorphism by $\varphi
  : (M_1,\omega_1) \to (M_2,\omega_2)$.
\end{definition}

As for manifolds, the existence of a non-degenerate
2-form on a connected orbifold $M$ implies that the dimension of $M$
is even. The simplest examples of symplectic orbifolds are
linear. 

\begin{example}[Symplectic orbi-vector spaces]\label{exm:symp_ovs}
  Let $(\widehat{V},\widehat{\omega})$ be a real symplectic vector space
  and let $\mathrm{Sp}(\widehat{V},\widehat{\omega}) < \mathrm{GL}(V)$
  be the subgroup of linear maps that are symplectomorphisms. If $\Gamma < \mathrm{Sp}(\widehat{V},\widehat{\omega})$ is a
  finite subgroup, then the orbi-vector space $\widehat{V}/\Gamma$
  inherits a symplectic form $\omega$. We call $(\widehat{V}/\Gamma,
  \omega)$ a {\bf symplectic} orbi-vector space.
\end{example}

\begin{remark}\label{rmk:symp_ovs}
  Let $(\widehat{V},\widehat{\omega})$ be as in Example
  \ref{exm:symp_ovs} and let $\widehat{J}$ be a $\Gamma$-invariant (almost) complex structure
  on $\widehat{V}$ that is compatible with $\widehat{\omega}$. (Such a
  structure exists since $\Gamma$ is compact.) This allows us to consider $\widehat{V}$ as a Hermitian vector
  space. Since the unitary group
  $\mathrm{U}(\widehat{V})$ is the maximal compact
  subgroup of $\mathrm{Sp}(\widehat{V},\widehat{\omega})$ (see
  \cite[Proposition 2.22]{mcduff_salamon}), $\Gamma <
  \mathrm{U}(\widehat{V})$. In particular, the real orbi-vector
  space $\widehat{V}/\Gamma$ can be endowed with the structure of a
  {\bf Hermitian} orbi-vector spaces.
\end{remark}

There are two basic constructions of symplectic orbifolds starting
from symplectic manifolds, which should be thought of as the analogs
of global quotients and quotients in the symplectic category (see Section
\ref{sec:basic-definitions}). Let $(\widehat{M},\widehat{\omega})$ be
a symplectic manifold. First, we suppose that $\Gamma$ is a finite group that acts on
$\widehat{M}$ effectively and by symplectomorphisms. The global
quotient $M:= \widehat{M}/\Gamma$ inherits a symplectic form
$\omega$; this is clearly the analog of a global quotient. More generally, we suppose that $G$ is a Lie group that acts properly and
in a locally free fashion on $\widehat{M}$ so that the action is {\bf
  Hamiltonian}, i.e., there exists a smooth $G$-equivariant map $\Phi : M \to
\mathfrak{g}^*$ to the dual of the Lie algebra of $G$, called {\em
  moment map}, such that
$$ \omega(\xi^M, \cdot) = d \langle \xi, \Phi \rangle \quad \text{ for
  all } \xi \in \mathfrak{g}, $$
where $G$ acts on $\mathfrak{g}^*$ with the coadjoint action, $\xi^M
\in \mathfrak{X}(M)$ is the vector field determined by $\xi$ as in
Remark \ref{rmk:infinitesimal_action}, and $\langle \cdot,\cdot \rangle$ is the natural pairing
$\mathfrak{g} \times \mathfrak{g}^* \to \R$. If $\alpha \in \mathfrak{g}^*$ is a
regular value that is fixed by the coadjoint action of $G$, then the
{\bf reduced space} $M_{\alpha}:=\Phi^{-1}(\alpha)/G$ is an orbifold
that inherits a symplectic form $\omega_\alpha$ (see \cite{mar_wein}
and \cite[Lemma 3.9]{LermanTolman}). This is the analog of a quotient
in symplectic geometry. The following examples are various instances of the above two constructions.

\begin{example}[Weighted projective spaces, continued]\label{exm:wps_as_symp}
  Fix an integer $n \geq 1$ and positive
  integers $m_0,\dots, m_{n}$ satisfying $\gcd(m_0,\dots,m_n)=1$. We
  endow $\C^{n+1}$ with the standard symplectic form $\omega_0$. The
  $S^1$-action on $\C^{n+1}$ given by \eqref{eq:8} is
  Hamiltonian. Hence, $\C P^n(m_0,\ldots, m_n)$ inherits a symplectic
  form $\omega$, which we call the {\bf standard
    symplectic form} on $\C P^n(m_0,\ldots, m_n)$. We observe that, if
  $m_i = 1$ for all $i =0, \ldots, n$, then $\C P^n(m_0,\ldots, m_n)
  \simeq \C P^n$ and $\omega$ is the Fubini-Study form.
\end{example}

\begin{example}[Example \ref{tswEx_1}, continued]\label{exm:tsw_symp}
  We let $M =  (\C P^1 \times \C P^1)/\Z_4$ be the orbifold
  constructed in Example \ref{tswEx_1}. We endow $\C P^1 \times \C
  P^1$ with the symplectic form $\widehat{\omega}$ given by taking the
  product of the Fubini-Study form on each factor. Since the $\Z_4$-action
  of \eqref{eq:29} is symplectic, $M$ inherits a symplectic form $\omega$.
\end{example}

\begin{example}[Example \ref{ex:quotient_cp1xcp1_finite}, continued]\label{exm:cp1cp1zc_symp}
  Given an integer $c > 1$, let $M_c =  (\C P^1 \times \C P^1)/\Z_c$ be the orbifold constructed in Example
  \ref{ex:quotient_cp1xcp1_finite}. We endow $\C P^1 \times \C
  P^1$ with the same symplectic form $\widehat{\omega}$ as in Example
  \ref{exm:tsw_symp}. Since the $\Z_c$-action of \eqref{eq:30} is
  symplectic, $M_c$ inherits a symplectic form $\omega_c$. 
\end{example}

\begin{example}[Cyclic quotients of weighted projective spaces, continued]\label{exm:cyclic_quotients_symp}
  The orbifolds $\C P^1(p,q)/\Z_m = \C P^1(p,q)/\Z_m(a,b)$ and $\C
  P^2(p,q,k)/\Z_m(a,b,c)$ of Examples \ref{exm:quotient_wpl} and
  \ref{ex::orbifold_quotient} admit symplectic forms. To see this, we
  present them as quotient orbifolds. In what follows, we deal only
  with $\C P^1(p,q)/\Z_m$, as the argument for $\C
  P^2(p,q,k)/\Z_m(a,b,c)$ is entirely analogous. Moreover, we fix the
  notation and conditions of Example \ref{exm:quotient_wpl}. 

  We start by considering $\C^2$ endowed with the standard symplectic
  form $\omega_0$. Let $\Z_m \times S^1$ act on $\C^2$ by 
  \begin{equation}
    \label{eq:31}
    \begin{split}
      \mathbb{Z}_m\times S^1 \times \C^2&\rightarrow \C^2\\
      ((\lambda,\mu),(z_0,z_1))&\mapsto(\lambda^a\mu^p z_0,\lambda^b \mu^qz_1),
    \end{split}
  \end{equation}
  (cf. \eqref{eq:8} and \eqref{eq:26}). This action is locally free
  and Hamiltonian with moment map $H : \C^2 \to \R$ given
  by
  $$ H(z_0,z_1) = \frac{1}{2}\left(p|z_0|^2 +q|z_1|^2\right).$$
  Since $pq \neq 0$, $pq$ is a regular value of $H$. Hence, the reduced space
  $M_{pq}(a,b)$ is an orbifold that inherits a symplectic form
  $\omega_{\mathrm{red}}$. The techniques used in Example
  \ref{exm:weighted_proj_standard} can be used to show that $M_{pq}(a,b)$
  is diffeomorphic to $\C P^1(p,q)/\Z_m$. Throughout this paper,
  we use this diffeomorphism to identify $M_{pq}(a,b)$ with $\C
  P^1(p,q)/\Z_m$, trusting that this does not cause confusion.
\end{example}

Next we show that certain projectivized plane
bundles over orbi-surfaces admit a symplectic form (see Example
\ref{exm:projectivized}).

\begin{example}[Projectivized plane bundles, continued]\label{exm:proj_symp}
  Let $\pi : M \to \Sigma$ be a Seifert fibration with Euler class
  different from zero and such that $M$ is a manifold. Let
  $\mathrm{pr} : L \to \Sigma$ be the corresponding complex line
  orbi-bundle. We aim to construct a family of symplectic forms on the
  total space of the projectivization $\mathbb{P}(L\oplus \C) \to
  \Sigma$. To this end, we fix the identification between $ \mathbb{P}(L\oplus \C)$ and the
  quotient $(M \times \C P^1)/S^1$ as in Example
  \ref{exm:projectivized}.

  Since the Euler class of $\pi : M \to \Sigma$ does not vanish, by
  \cite[Theorem 1.4]{kegel_lange}, there exists an $S^1$-invariant 1-form $\lambda \in
  \Omega^1(M)$ such that $\lambda(X^M) =1$, where $X^M \in
  \mathfrak{X}(M)$ is the infinitesimal generator of the
  $S^1$-action. (This can be thought of as a connection 1-form on the
  Seifert fibration.) Moreover, in analogy with the case of principal
  $S^1$-bundles over manifolds, since the Euler class of $\pi : M \to
  \Sigma$ does not vanish, there exists a {\em symplectic form}
  $\omega_\Sigma \in \Omega^2(\Sigma)$ such that $d\lambda = \pi^*
  \omega_{\Sigma}$. (The 1-form $\lambda$ is a so-called {\em Besse
    contact form}, see \cite{kegel_lange} for more details.)

  We fix $s > 0$. We consider the Fubini-Study symplectic form $\omega_{\C P^1}$ on
  $\C P^1$ and we let $h : \C P^1 \to \R$ be the smooth function given
  by
  $$[z_0:z_1] \mapsto \frac{|z_0|^2}{|z_0|^2+|z_1|^2} + \frac{1}{s}.$$
  This is a moment map for the $S^1$-action on $(\C P^1, \omega_{\C P^1}$ given by
  $\lambda \cdot [z_0:z_1] = [z_0:\lambda z_1]$.
  We denote the projections from $M \times \C P^1$
  to the first and second component by $\mathrm{pr}_1$ and
  $\mathrm{pr}_2$ respectively and consider the
  following 2-form on $M \times \C P^1$:
  $$ s\left(\mathrm{pr}^*_2 \omega_{\C P^1} - d\left(\mathrm{pr}^*_2h\,
    \mathrm{pr}^*_1 \lambda \right)\right).$$
  By construction, this form is closed, has maximal rank equal to
  four and its kernel is precisely the span of the infinitesimal
  generator of the $S^1$-action of Example
  \ref{exm:projectivized}. Hence, we obtain a symplectic form $\omega_s
  \in \Omega^2((M \times \C P^1)/S^1)$ as desired.
\end{example}

We conclude this section with the following observation, which,
intuitively speaking, states that we may work with {\em full}
submanifolds of symplectic suborbifolds as in the case of manifolds.

\begin{remark}\label{rmk:full_suborbifold_symplectic}
  By Remark \ref{rmk:full_suborbifolds_normal}, the tangent bundle of
  a full suborbifold of $M$ can be seen as an orbi-subbundle of $TM$. Hence, the notion of {\bf
    Lagrangian/isotropic/symplectic} full suborbifolds of a symplectic
  orbifold $(M,\omega)$ can be defined as in the case of manifolds. 
\end{remark}

To conclude this section, we state the following important
property of symplectic orbifolds, which is the analog of a standard
result for symplectic manifolds (see \cite[Proposition 8]{mu_ro} for a
proof). 

\begin{proposition}\label{prop:acs}
  Let $(M,\omega)$ be a symplectic orbifold. Then there exists a
  compatible almost complex structure $J$ on $M$, i.e.,
  $\omega(\cdot,J\cdot)$ is a Riemannian metric on $M$. 
\end{proposition}

\begin{remark}\label{rmk:contractible}
  Since the proof of \cite[Proposition 8]{mu_ro} is
  entirely analogous to the case of symplectic manifolds, the space of
  compatible almost complex structures on a given symplectic orbifold
  is contractible.
\end{remark}

Proposition \ref{prop:acs} has the following immediate consequence.

\begin{corollary}\label{cor:complex_bundle}
  Let $(M,\omega)$ be a symplectic orbifold and let $N \subset M$ be a
  full symplectic suborbifold. Then $TM$ admits the structure of a
  complex vector orbi-bundle such that the restriction $TM|_N$ has
  $TN$ as a complex vector orbi-subbundle. In particular, the normal
  bundle to $N$ has the structure of a complex vector orbi-bundle. 
\end{corollary}

We observe that, by Remark \ref{rmk:contractible}, the structures of
complex vector orbi-bundles of Corollary \ref{cor:complex_bundle} do
not depend on the choice of compatible almost complex structure up to isomorphism.

\subsubsection{Darboux's theorem and its consequences in dimension four}\label{sec:darbouxs-theorem-its}

A version of Darboux's theorem holds for symplectic orbifolds (see,
for instance, \cite[Proposition 11]{mu_ro}), i.e., Example
\ref{exm:symp_ovs} is locally the only one up to symplectomorphism. 

\begin{theorem}[Darboux's theorem for orbifolds]\label{thm:darboux_orbifolds}
  Let $(M,\omega)$ be a symplectic orbifold. If $x \in M$ is a
  point with orbifold structure group $\Gamma$, then there exists a
  l.u.c. $\oc$ centered at $x$ such that $\widehat{U}$ is an open
  subset of $\R^{2n} \simeq \C^n$, $\Gamma < \mathrm{U}(n)$ and $\omega|_U$ is
  as in Example \ref{exm:symp_ovs}. 
\end{theorem}

\begin{remark}\label{rmk:darboux_as_linearisation}
  With the notation of Theorem \ref{thm:darboux_orbifolds}, we observe
  that the orbi-vector space $T_x M$ can be identified with
  $\R^{2n}/\Gamma$ and that it inherits a linear symplectic form
  $\omega_x$. Hence, it is a symplectic orbi-vector space (see Example \ref{exm:symp_ovs}). 
\end{remark}

Since the fixed point set of any linear unitary action has real
codimension at least two, the following result is an immediate
consequence of Theorem \ref{thm:darboux_orbifolds}.

\begin{corollary}\label{cor:symp_codim_2}
  Let $(M,\omega)$ be a symplectic orbifold. Then the singular set of
  $M$ has codimension at least two.
\end{corollary}
\color{black}

Let $(M,\omega)$ be a symplectic orbifold. By Theorem
\ref{thm:darboux_orbifolds}, the local invariants of a neighborhood of a point $x \in M$ with
structure group $\Gamma$ up to symplectomorphisms are precisely the invariants of a Hermitian
orbi-vector space $\C^n/\Gamma$ up to unitary isomorphisms. For this
reason, and since we are particularly interested in the case of
four-dimensional symplectic orbifolds with isolated singular points and abelian orbifold structure groups,
in what follows we study Hermitian
orbi-vector spaces %$\C^2/\Z_m$ for some integer $m \geq 2$ with the
$\C^2/\Gamma$ where $\Gamma$ is abelian and $[(0,0)]$
is the only singular point.

Fix a faithful representation $\Gamma \hookrightarrow
\mathrm{U}(2)$. Since $\Gamma$
is finite and abelian, we may assume that the image of $\Gamma$ in $\mathrm{U}(2)$ is contained
in the maximal torus consisting of diagonal matrices $(S^1)^2 <
\mathrm{U}(2)$. Moreover, since the rank of that torus is two,
$\Gamma$ has at most two generators.

\begin{lemma}\label{lemma:abelian+isolated_cyclic}
  Suppose that $\Gamma$ is a finite abelian group and let $\Gamma \hookrightarrow
  \mathrm{U}(2)$ be a faithful representation. If the quotient
  $\C^2/\Gamma$ orbifold has isolated singular points, then $\Gamma$
  is cyclic.
\end{lemma}
\begin{proof}
  Suppose that there exist
  integers $m_1,m_2 \geq 2$ with $\gcd(m_1,m_2) > 1$ such that $\Gamma
  \simeq \Z_{m_1} \times \Z_{m_2}$. We may assume that there exist
  $k_1,k_2,j_1,j_2 \in \Z$ such that the $\Gamma$-action on
  $\C^2$ induced by the representation is given by
  \begin{equation*}
    \begin{split}
      \Gamma \times \C^2 &\to \C^2 \\
      ((e^{\frac{2\pi i l_1}{m_1}}, e^{\frac{2\pi i
          l_2}{m_2}}),(z_1,z_2)) &\mapsto (e^{\frac{2\pi i l_1k_1}{m_1}}e^{\frac{2\pi i
          l_2k_2}{m_2}}z_1, e^{\frac{2\pi i l_1j_1}{m_1}}e^{\frac{2\pi i
          l_2j_2}{m_2}}z_2).
    \end{split}
  \end{equation*}
  A simple calculation shows that the above action is effective if and
  only if
  \begin{equation}
    \label{eq:62}
    \gcd(|k_2j_1 - k_1j_2|,m_1) = 1 = \gcd(|k_2j_1 - k_1j_2|,m_2).
  \end{equation}
  We set $\gcd(m_1,m_2) = c$ and write $m_i = m'_ic$ for $i=1,2$. We
  set $l_1 = m'_1k_2$ and $l_2 = -m'_2k_1$. Then the subgroup of $\Gamma$ generated by $(e^{\frac{2\pi i l_1}{m_1}}, e^{\frac{2\pi i
      l_2}{m_2}})$ fixes all points in the subspace $z_2 =
  0$. Moreover, either $l_1$ is not divisible by $m_1$ or
  $l_2$ is not divisible by $m_2$. Indeed, if otherwise, then $c > 1$ divides both $k_1$
  and $k_2$, contradicting \eqref{eq:62}. Hence, all points in the
  subspace $z_2=0$ have non-trivial stabilizer.
\end{proof}

By Lemma \ref{lemma:abelian+isolated_cyclic}, we may assume that
$\Gamma$ is cyclic, i.e., there exists an integer $m \geq 2$ such that
$\Gamma \simeq \Z_m$. As in the proof of Lemma
\ref{lemma:abelian+isolated_cyclic},
given any representation $\Z_m \to \mathrm{U}(2)$, there exist integers $a,b \in \Z$ such that
the induced action on $\C^2$ is given by 
\begin{equation}
  \label{eq:5}
  \begin{split}
    \Z_m \times \C^2 &\to \C^2 \\
    (e^{\frac{2\pi ik}{m}},(z_1,z_2)) &\mapsto (e^{\frac{2\pi i
        ak}{m}}z_1, e^{\frac{2\pi ibk}{m}}z_2). 
  \end{split}
\end{equation}
A simple calculation yields the following result, which gives a
necessary and sufficient condition for the action of \eqref{eq:5} to be such that the origin is the only point
with non-trivial stabilizer.

\begin{lemma}\label{lemma:cyclic_rep_isolated_sing}
  Consider the $\Z_m$-action on $\C^2$ given by \eqref{eq:5}. The
  quotient $\C^2/\Z_m$ is an orbifold such that $[(0,0)]$ is the only
  singular point if and only if
  \begin{equation}
    \label{eq:2}
    \gcd(a,m) = 1 =\gcd(b,m).
  \end{equation}
\end{lemma}

\begin{remark}\label{rmk:wlog_1}
  If \eqref{eq:2} holds, by composing with a suitable isomorphism
  $\Z_m \to \Z_m$, without loss of generality we may assume that
  either $a$ or $b$ in \eqref{eq:5} equals 1.
\end{remark}

Remark \ref{rmk:wlog_1} motivates introducing the
following terminology that is also used in complex algebraic geometry
(see, e.g., \cite{reid}). 

\begin{definition}\label{def:model_singular_point}
  Let $m \geq 2$ be an integer and let $l \in \Z$ be such that $1\le
  l<m$ and $\gcd(l,m)=1$. Let $\mathbb{C}^2/\mathbb{Z}_m$ be the
  Hermitian orbi-vector space that is the quotient of the following
  $\mathbb{Z}_m$-action on $\C^2$: 
  \begin{align} \label{eq:typesing}
    \mathbb{Z}_m \times \mathbb{C}^2 & \to \mathbb{C}^2  \nonumber \\
    (\xi_m , (z_1,z_2))  & \mapsto (\xi_m z_1,\xi_m^l z_2).
  \end{align}
  We call $\C^2/\Z_m$ the \textbf{model for a cyclic
    isolated singular point of type $\frac{1}{m}(1,l)$}. The integers
  $(1,l)$ are the {\bf weights} of the $\mathbb{Z}_m$-action on
  $\C^2$. 
\end{definition}

\begin{remark}\label{rmk:ambiguity}
  In Definition \ref{def:model_singular_point}, we choose to make the
  weight of the $\Z_m$-action of the first coordinate to be equal to one. The other choice
  would yield an integer  $1\le l'<m$ such that $\gcd(l',m)=1$ and $l
  \cdot l'= 1 \mod m$.
\end{remark}

The following result states that the ambiguity of Remark
\ref{rmk:ambiguity} disappears when considering the unitary
isomorphism type of the corresponding Hermitian orbi-vector space. Its
proof is a straightforward calculation and is left to the reader.

\begin{lemma}\label{lemma:model_iso}
  There exists a unitary isomorphism between models for a cyclic
  singular point of types $\frac{1}{m}(1,l)$ and $\frac{1}{m}(1,l')$ if and only if either $l = l'$ or
  $l \cdot l' = 1 \mod m$.
\end{lemma}

% \begin{example}\label{ex:eq1}
%   We consider the  $\mathbb{Z}_7$-action on $\mathbb{C}^2$ given by
%   \begin{equation*}
%     \begin{split}
%       \mathbb{Z}_7 \times \mathbb{C}^2 & \to \mathbb{C}^2  \\
%       (e^{\frac{2\pi i k}{7}} , (z_1,z_2))  & \mapsto (e^{\frac{6\pi i k}{7}} z_1, e^{\frac{10\pi i k}{7}} z_2).
%     \end{split}
%   \end{equation*}
%   Then, taking the quotient
%   $\mathbb{C}^2/\mathbb{Z}_7$, we have that
%   \begin{align*}
%     (z_1,z_2)\sim (\xi_7^3 z_1,\xi_7^5 z_2) & \sim (\xi_7^6 z_1,\xi_7^3 z_2)  \sim (\xi_7^2 z_1,\xi_7 z_2) \sim(\xi_7^5 z_1,\xi_7^6 z_2) \sim \\  & \sim (\xi_7  z_1,\xi_7^4 z_2) \sim (\xi_7^4 z_1,\xi_7^2 z_2),
%   \end{align*}
%   for some generate $\xi_7$ of $\mathbb{Z}_7$.
%   Consequently,  the singular point $[0,0]_7$ in $\mathbb{C}^2/\mathbb{Z}_7$ is a singularity of type $\frac{1}{7}(1,4)$. If we permute $z_1$ and $z_2$, we obtain type $\frac{1}{7}(1,2)$. (Note that $4 \cdot 2 = 8 = 1 \mod 7$.)
% \end{example}

As an immediate consequence of Theorem \ref{thm:darboux_orbifolds},
Lemma \ref{lemma:cyclic_rep_isolated_sing} and Remark \ref{rmk:wlog_1}, we
obtain the following result.

\begin{corollary}\label{cor:sing_point_dim_4}
  Let $(M,\omega)$ be a four dimensional symplectic orbifold. Let $x
  \in M$ be an isolated singular point with structure group $\Z_m$ for
  some integer $m \geq 2$. Then there exist an integer $1 \leq l < m$
  with $\gcd(l,m)=1$, an open neighborhood $U \subseteq M$ of $x$, an
  open neighborhood $V \subseteq \C^2/\Z_m$ of $[0,0]$, and a symplectomorphism
  $\varphi : (U,\omega) \to (V,\omega_0)$ that sends $x$ to $[0,0]$,
  where $\C^2/\Z_m$ is the global quotient constructed from $l$ as in
  Definition~\ref{def:model_singular_point} and $\omega_0$ is the
  symplectic form on $\C^2/\Z_m$ that descends from the standard
  symplectic form on $\C^2$.  
\end{corollary}

Motivated by Definition \ref{def:model_singular_point} and Corollary
\ref{cor:sing_point_dim_4}, we introduce the following terminology.

\begin{definition}\label{def::quotientSingularity}
  Let $x$ a singular point in a symplectic orbifold of
  dimension four. Let $\Z_m$ be the orbifold structure group of $x$
  and let $1 \leq l < m$
  with $\gcd(l,m)=1$ be the integer of Corollary
  \ref{cor:sing_point_dim_4}. We say that is a \textbf{cyclic
    isolated singular point of type $\frac{1}{m}(1,l)$}. 
\end{definition}

\begin{example}
  Consider the weighted projective plane $\C P^2(3,5,7)$. The type of
  the singular
  points $[1:0:0]$, $[0:1:0]$ and $[0:0:1]$ is $\frac{1}{3}(1,2)$,
  $\frac{1}{5}(1,4)$ and  $\frac{1}{7}(1,4)$ respectively.  \hfill$\Diamond$
\end{example}

\begin{example}
The type of the three singular points in the quotient $\C
P^2(3,5,7)/\mathbb{Z}_{11}(1,2,10)$ are $\frac{1}{33}(1,23)$,
$\frac{1}{55}(1,19)$ and  $\frac{1}{77}(1,25)$.  \hfill$\Diamond$
\end{example}

\begin{example}\label{tswEx_types} (Example~\ref{tswEx_1}, continued) 
Let $M:=(\mathbb{C}P^1\times\mathbb{C}P^1)/\mathbb{Z}_4$ be the
orbifold constructed in Example~\ref{tswEx_1}. It has three isolated
singular points, $[[0:1],[0:1]], [[1:0],[1:0]]$ and $[[0:1],[1:0]]$
with orbifold structure groups given by $\mathbb{Z}_4$,
$\mathbb{Z}_4$ and $\mathbb{Z}_2$ respectively. By diagonalizing the action of $\mathbb{Z}_4$ on $\C P^1\times \C P^1$ around  $\left([0:1],[0:1]\right)\quad \text{and} \quad
\left([1:0],[1:0]\right)$,
there exist coordinates around these points for which the $\mathbb{Z}_4$-action is given by
\begin{align*}
 \mathbb{Z}_4 \times \mathbb{C}^2 & \to \mathbb{C}^2  \nonumber \\
(\xi_4 , (z,w))  & \mapsto (\xi_4 z,\xi_4^3 w).
\end{align*}
%$$
%\xi_4 \cdot (z,w) = (\xi_4 z, \xi_4^3w),
%$$
%where $\xi_4$  is a primitive $4$-th root of unity.
Similarly, there exist coordinates around $\left([0:1],[1:0]\right)\in \C P^1\times \C P^1$ for which the $\mathbb{Z}_2$-action is given by
\begin{align*}
  \mathbb{Z}_2 \times \mathbb{C}^2 & \to \mathbb{C}^2  \nonumber \\
(\xi_2 , (z,w))  & \mapsto (\xi_2 z,\xi_2 w).
\end{align*}
%$$
%\xi_2 \cdot (z,w) = (\xi_2 z, \xi_2 w),
%$$
%where $\xi_2=-1$. 
(Note that $\left([1:0],[0:1]\right)$ and $\left([0:1],[1:0]\right)$ are in the same $\mathbb{Z}_4$-orbit).
We conclude that the type of $[[0:1],[0:1]], [[1:0],[1:0]]$ is
$\frac{1}{4}(1,3)$, while that of $[[0:1],[1:0]]$ is $\frac{1}{2}(1,1)$.
 \hfill$\Diamond$

\end{example} 

\subsubsection{Compact, connected, symplectic orbi-surfaces}\label{sec:comp-conn-sympl}

In dimension two, symplectic geometry is simply the geometry of
oriented (orbi-)surfaces; moreover, by Corollary
\ref{cor:symp_codim_2}, the singular set is necessarily isolated. Hence, a connected, oriented
orbi-surface $\Sigma$ with isolated singular points comes equipped with a
symplectic form $\omega$. If $x \in \Sigma$ and $\oc$ is a l.u.c. centered at $x$ as in Remark
\ref{rmk:centred}, then $\Gamma$ is a subgroup of $\mathrm{SO}(2)
\simeq \mathrm{U}(1)$
and $\omega|_U$ is the quotient of the standard symplectic form
$\omega_0$ on $\R^2 \simeq \C$ by the action of $\Gamma$.

In analogy with the case of manifolds, compact, connected,
symplectic orbi-surfaces can be
classified up to symplectomorphism (see \cite[Section 5]{dui_pel} and Theorem
\ref{thm::classorbisurface}).

\begin{theorem}\label{thm:classification_symplectic_orbi-surfaces}
  For $i=1,2$, let $(\Sigma_i,\omega_i)$ be a compact, connected,
  symplectic orbi-surface. The
  symplectic orbifolds $(\Sigma_1,\omega_1)$ and $(\Sigma_2,\omega_2)$
  are symplectomorphic if and only if $\Sigma_1$ is diffeomorphic to
  $\Sigma_2$ and $\int_{\Sigma_1} \omega_1 = \int_{\Sigma_2} \omega_2$.  
\end{theorem}

\subsubsection{Hamiltonian actions on symplectic orbifolds}\label{sec:hamilt-acti-sympl}
In this paper we are particularly interested in the following type of
actions on symplectic orbifolds.

\begin{definition}\label{defn:hamiltonian}
  Let $G$ be a Lie group with Lie algebra $\mathfrak{g}$ and let
  $(M,\omega)$ be a symplectic orbifold.
  \begin{itemize}[leftmargin=*]
  \item An action of $G$ on $M$ is {\bf Hamiltonian} if there exists a
    good $C^{\infty}$ map $H : M \to \mathfrak{g}^*$, called {\bf
      moment map}, that is $G$-equivariant with respect to the
    coadjoint action of $G$ on $\mathfrak{g}^*$ and such that
    $$ \omega(\xi^M ,\cdot) = d \langle \xi,\Phi \rangle (\cdot) \quad
    \text{for all } \xi \in \mathfrak{g}, $$
    \noindent
    where $\xi^M \in \mathfrak{X}(M)$ is the vector field
    determined by $\xi$ as in Remark \ref{rmk:infinitesimal_action}, and $\langle \cdot, \cdot \rangle$ is the
    natural pairing $\mathfrak{g} \times \mathfrak{g}^* \to \R$.
  \item A {\bf Hamiltonian $G$-space} is a triple $(M,\omega,\Phi)$,
    where $\Phi : M \to \mathfrak{g}^*$ is the moment map for a
    Hamiltonian $G$-action on $M$.
  \item Two Hamiltonian $G$-spaces $(M_1,\omega_1,\Phi_1)$ and
    $(M_2,\omega_2,\Phi_2)$ are {\bf isomorphic} if there exists a
    $G$-equivariant symplectomorphism $\varphi : (M_1,\omega_1) \to
    (M_2,\omega_2)$ such that $\Phi_2 \circ \varphi = \Phi_1$. 
\end{itemize}

\end{definition}

\begin{remark}\label{rmk:hamiltonian}
  As for symplectic manifolds, a Hamiltonian action on a symplectic
  orbifold is by symplectomorphisms. Moreover, we observe that, if $G$
  is compact and connected, then $G$-equivariance in the above
  definition of an isomorphism of Hamiltonian $G$-spaces is automatic
  as long as the symplectomorphism intertwines the moment maps. 
\end{remark}

\subsubsection*{Linear symplectic actions on symplectic
  orbi-vector spaces}
Linear symplectic actions play an important role in
understanding the behavior of Hamiltonian group actions near fixed
points (see Theorem \ref{thm:darboux} and Section
\ref{ChapterLocalForms}). For this reason, we recall here some of
their properties, following \cite[Lemmas 3.1 and 3.2 ]{LermanTolman}. Let $(\widehat{V}/\Gamma,\omega)$ be a symplectic orbi-vector space
obtained as a global quotient of a real symplectic vector space
$(\widehat{V},\widehat{\omega})$ as in Example
\ref{exm:symp_ovs}. By a slight abuse of notation, we set $\widehat{N}$ to be the normalizer of
$\Gamma$ in $\mathrm{Sp}(\widehat{V},\widehat{\omega})$. This is a
subgroup of the normalizer $N(\Gamma)$ of $\Gamma$ in
$\mathrm{GL}(\widehat{V})$ of Example
\ref{exm:linear_actions_orbi-vector_space}. As in \cite[Section
3]{LermanTolman}, we set
$\mathrm{Sp}(\widehat{V}/\Gamma,\omega):=\widehat{N}/\Gamma$. 

To start, we recall the following result without proof (see \cite[Lemma 3.2]{LermanTolman} for
a proof).

\begin{lemma}\label{lemma:sp-acts-ham}
  The natural
  $\mathrm{Sp}(\widehat{V}/\Gamma,\omega)$-action on
  $(\widehat{V}/\Gamma,\omega)$ is Hamiltonian and a moment map for
  this action is given by $\Phi : \widehat{V}/\Gamma \to
  \mathfrak{sp}(\widehat{V}/\Gamma,\omega)^*$ defined by
  $$ \langle \xi, \Phi \rangle([\widehat{v}]) = \frac{1}{2} \omega(\xi [\widehat{v}],
  [\widehat{v}]) \quad \text{for any } [\widehat{v}] \in
  \widehat{V}/\Gamma \text{ and any } \xi \in
  \mathfrak{sp}(\widehat{V}/\Gamma,\omega),$$
  where $\xi [\widehat{v}]$ is the value of the vector field $\xi$ at
  $[\widehat{v}]$.

  Moreover, let $\pi : \widehat{N} \to \mathrm{Sp}(\widehat{V}/\Gamma,\omega)$ and $\pi^* :
  \mathfrak{sp}(\widehat{V}/\Gamma,\omega)^* \to
  \widehat{\mathfrak{n}}^*$
  denote the quotient map and the induced isomorphism between the duals of the
  Lie algebras respectively, and let $\widehat{\Phi} : \widehat{V} \to
  (\mathfrak{sp}(\widehat{V},\widehat{\omega}))^*$ be the homogeneous
  moment map for the standard
  $\mathrm{Sp}(\widehat{V},\widehat{\omega})$-action on
  $(\widehat{V},\widehat{\omega})$. Then the following diagram commutes
  \begin{equation}
    \label{eq:37}
    \begin{tikzcd}
      \widehat{V} \arrow[r, "q"]
      \arrow[d,"\widehat{\Phi}_{\widehat{N}}" ]&
      \widehat{V}/\Gamma \arrow[d,"\Phi"]\\
      \widehat{\mathfrak{n}}^*& \mathfrak{sp}(\widehat{V}/\Gamma,\omega)^* \arrow[l,"\pi^*"],
    \end{tikzcd}
  \end{equation}
  where $q : \widehat{V} \to
  \widehat{V}/\Gamma$ is the quotient map.
\end{lemma}

Let $G$ be a Lie group and let $\rho : G \to
\mathrm{Sp}(\widehat{V}/\Gamma,\omega)$ be a Lie group homomorphism,
that we call a {\bf symplectic representation} of $G$. In analogy
with Example \ref{exm:linear_actions_orbi-vector_space} (see also
\cite[Lemma 3.1]{LermanTolman}), there exist a Lie group $\widehat{G}$, a representation $\rho:
\widehat{G} \to \widehat{N}$ and a short exact sequence of Lie groups
\begin{equation}\label{eq:exact_sequ_lin}
  \begin{tikzcd}
    1\arrow{r} & \Gamma\arrow[r,hook]& \widehat{G}\arrow[r,"\pi"] & G\arrow{r} & 1,
  \end{tikzcd}
\end{equation}
such that the following diagram commutes
\begin{equation*}
  \begin{tikzcd}
    1\arrow{r} & \Gamma\arrow[r,hook] \arrow[d,equal]&
    \widehat{G}\arrow[r,"\pi"] \arrow[d,"\widehat{\rho}"]
    & G\arrow{r} \arrow[d,"\rho"] & 1 \\
    1\arrow{r} & \Gamma\arrow[r,hook]& \widehat{N}\arrow[r,"\pi"]
    & \mathrm{Sp}(\widehat{V}/\Gamma,\omega) \arrow{r} & 1.
  \end{tikzcd}
\end{equation*}
The representation $\rho$ is faithful if and only if $\widehat{\rho}$ is
faithful (see \cite[Lemma 3.1]{LermanTolman}). By a standard argument
using functoriality, the following result is an immediate consequence
of Lemma \ref{lemma:sp-acts-ham}.

\begin{corollary}\label{cor:linear-symp_action}
  Let $\rho : G \to
  \mathrm{Sp}(\widehat{V}/\Gamma,\omega)$ be a symplectic
  representation. The induced linear $G$-action on $(\widehat{V}/\Gamma,\omega)$ is Hamiltonian and a moment map $\Phi_G :
  \widehat{V}/\Gamma \to \mathfrak{g}^*$ for
  this action is given by $\rho^* \circ
  \Phi$. Moreover, let $\Phi_{\widehat{G}} : \widehat{V}
  \to \widehat{\mathfrak{g}}^*$ be the moment map for the Hamiltonian
  $\widehat{G}$-action on $(\widehat{V},\widehat{\omega})$. Then the
  following diagram commutes

   \begin{equation*}
    \begin{tikzcd}
      \widehat{V} \arrow[r, "q"]
      \arrow[d,"\widehat{\Phi}_{\widehat{G}}" ]&
      \widehat{V}/\Gamma \arrow[d,"\Phi_G"]\\
      \widehat{\mathfrak{g}}^*& \mathfrak{g}^* \arrow[l,"\pi^*"].
    \end{tikzcd}
  \end{equation*}
\end{corollary}

\begin{remark}\label{rmk:G-compact}
  If $G$ is
  compact, then so is $\widehat{G}$. As in Remark \ref{rmk:symp_ovs},
  we may view $\widehat{V}$ (respectively $\widehat{V}/\Gamma$) as a
  Hermitian (respectively orbi-)vector space. Let $N_{\mathrm{U}(\widehat{V})}(\Gamma)$ denote the normalizer of
  $\Gamma$ in $\mathrm{U}(\widehat{V})$ and set $\mathrm{U}(\widehat{V}/\Gamma) :=
  N_{\mathrm{U}(\widehat{V})}(\Gamma)/\Gamma$. Then $\widehat{\rho}$
  and $\rho$ factor through $N_{\mathrm{U}(\widehat{V})}(\Gamma)$ and
  $\mathrm{U}(\widehat{V}/\Gamma)$ respectively. Hence, in this case,
  the study of symplectic representations of $G$ on $(\widehat{V}/\Gamma,\omega)$ reduces to that of
  unitary representations on $\widehat{V}/\Gamma$.

  Moreover, the subset of $\widehat{V}/\Gamma$ of fixed
  points under the $G$-action is, in fact, a $G$-invariant orbi-vector
  complex subspace: To
  see this, we identify the above subset with the quotient by $G$
  of the complex subspace of $\widehat{V}$ that consists of fixed
  point under the $\widehat{G}$-action. In particular,  the subset of $\widehat{V}/\Gamma$ that consists of fixed
  points under the $G$-action is a full symplectic suborbifold (see Definition
  \ref{defn:suborbifold} and Remark \ref{rmk:full_suborbifold_symplectic}). 
\end{remark}

\subsubsection*{Examples of Hamiltonian circle actions}
In this paper, we are mostly interested in Hamiltonian actions of
compact tori, with particular emphasis on $S^1$. In
this case, the moment map is invariant under the group action.
Moreover, we use the notation $H$ for the moment map of a circle action
and we fix once and for all an identification between the dual of the
Lie algebra of $S^1$ and $\R$.

\begin{example}[Weighted projective spaces, continued]\label{exm:weighted_ham}
  Fix an integer $n \geq 1$ and positive
  integers $m_0,\dots, m_{n}$ satisfying $\gcd(m_0,\dots,m_n)=1$. We
  endow the weighted projective space $\C P^n(m_0,\ldots, m_n)$ with
  the standard symplectic form $\omega$ of Example
  \ref{exm:wps_as_symp}. Let $a_0,\ldots, a_n$ be integers. The
  $S^1$-action on $\C P^n(m_0,\ldots, m_n)$ of Example
  \ref{exm:circle_action_weighted} is Hamiltonian and a moment map is
  given by
  \begin{equation}
    \label{eq:33}
    \begin{split}
      H : \C P^n(m_0,\ldots, m_n) & \to \R \\
      [z_0:\ldots:z_n] &\mapsto \frac{1}{2} \sum\limits_{j=0}^n a_j
      \left(\prod_{s \neq j} m_s \right) \frac{|z_j|^2}{\|z\|^2}, 
    \end{split}
  \end{equation}
  where $\|z\|^2:= \sum_{j=0}^n |z_j|^2$. \hfill$\Diamond$ 
\end{example}

\begin{example}[Example \ref{tswEx_1}, continued]\label{exm:tsw_ham}
  Let $M =  (\C P^1 \times \C P^1)/\Z_4$ be the orbifold
  constructed in Example \ref{tswEx_1} and let $\omega$ be the
  symplectic form on $M$ constructed in Example
  \ref{exm:tsw_symp}. The $S^1$-action on $M$ of Example
  \ref{exm:circle_action_Talvacchia} is Hamiltonian and a moment map
  is given by
  \begin{equation}
    \label{eq:34}
    \begin{split}
      H : M & \to \R \\
      \left([[z_0:z_1],[w_0:w_1]]\right) &\mapsto
      \frac{1}{4}\left(\frac{|z_0|^2}{|z_0|^2+|z_1|^2}+\frac{|w_0|^2}{|w_0|^2+|w_1|^2} \right).
    \end{split}
  \end{equation}
  \hfill$\Diamond$ 
\end{example}

\begin{example}[Example \ref{ex:quotient_cp1xcp1_finite}, continued]\label{exm:finite_quotient_cp1cp1_ham}
  Let $c > 1$ be an integer and let $M_c = (\C P^1 \times \C P^1)/\Z_c$
  be the orbifold constructed in Example
  \ref{ex:quotient_cp1xcp1_finite} with $l =1$. Let $\omega_c$ be the
  symplectic form on $M_c$ constructed in Example
  \ref{exm:cp1cp1zc_symp}. We consider $m=1, n= c-1$
  and the corresponding $S^1$-action on $M_c$ of Example
  \ref{exm:circle_action_cp1cp1}. This action is Hamiltonian and a moment map is
  given by
  \begin{equation}
    \label{eq:35}
    \begin{split}
      H : M_c & \to \R \\
      \left([[z_0:z_1],[w_0:w_1]]\right) &\mapsto
      \frac{1}{2c}\left(\frac{|z_1|^2}{|z_0|^2+|z_1|^2}+(c-1)\frac{|w_1|^2}{|w_0|^2+|w_1|^2} \right).
    \end{split}
  \end{equation}
  \hfill$\Diamond$ 
\end{example}

\begin{example}[Projectivized plane bundles, continued]\label{exm:proj_plane_ham}
  Let $\pi : M \to \Sigma$ be a Seifert fibration with Euler class
  different from zero and such that $M$ is a manifold. Given $s
  >0$, let $\omega_s$ be the symplectic form on $(M \times \C
  P^1)/S^1$ constructed in Example \ref{exm:proj_symp}. The $S^1$-
  action on $(M \times \C
  P^1)/S^1$ of Example \ref{exm:circle_action_proj} is Hamiltonian and
  a moment map is given by 
  \begin{equation}
    \label{eq:36}
    \begin{split}
      H : (M \times \C
      P^1)/S^1 & \to \R \\
      [x,[z_0:z_1]] &\mapsto \frac{s|z_0|^2}{2(|z_0|^2 + |z_1|^2)}. 
    \end{split}
  \end{equation}
  \hfill$\Diamond$ 
\end{example}

\subsubsection*{The equivariant Darboux theorem and its generalizations}
Next we present the linearization theorem for
fixed points of Hamiltonian group actions, first proved in \cite[Lemma
3.5]{LermanTolman}. In some sense, this result combines Corollary
\ref{cor:linearization_fixed_point} and Theorem \ref{thm:darboux_orbifolds}. To this end, we
recall that, as in the case of manifolds, if a Lie group $G$ acts on $(M,\omega)$
in a Hamiltonian fashion and $x \in M$ is a fixed point of the action,
then the linear $G$-action on $T_xM$ preserves the linear symplectic
form $\omega_x$ (see Remark \ref{rmk:darboux_as_linearisation}). In
particular, we obtain a symplectic representation $\rho : G \to
\mathrm{Sp}(T_xM,\omega_x)$ that is known as the {\bf
  symplectic slice representation} at $x$. The linear $G$-action on $T_xM$ is
Hamiltonian with moment map $\Phi_{\mathrm{lin}} : T_x M \to
\mathfrak{g}^*$; moreover, if $G$ is compact, the representation
$\rho$ can be taken to be unitary (see Corollary \ref{cor:linear-symp_action} and Remark
\ref{rmk:G-compact}). 

\begin{theorem}[Equivariant Darboux theorem for orbifolds]\label{thm:darboux}
   Let $G$ be a compact Lie group and let $(M,\omega, \Phi)$ be a
   Hamiltonian $G$-space. If $x$ is a
   fixed point, then there exist $G$-invariant open neighborhoods $U
  \subseteq M$ and $W \subseteq T_xM$ of $x$ and $[0]$ respectively,
  and a $G$-equivariant symplectomorphism $f : (U,\omega) \to (W,\omega_x)$ that sends $x$
  to $[0]$ and such that $\Phi = \Phi(x) + \Phi_{\mathrm{lin}} \circ f$.
\end{theorem}

By Remark \ref{rmk:G-compact}, Theorem \ref{thm:darboux} has the
following immediate consequence (cf. Corollary
\ref{cor:fixed_pt_set_suborbi}).

\begin{corollary}\label{cor:fixed_point_set_symp_full_suborbifold}
  Let $G$ be a compact Lie group and let $(M,\omega, \Phi)$ be a
  Hamiltonian $G$-space. Any connected component of the fixed point
  set is a full symplectic suborbifold. 
\end{corollary}

To conclude this section, we state a useful generalization of Theorem
\ref{thm:darboux} to arbitrary full symplectic suborbifolds (see
Theorem \ref{thm:esnt} below). To this end, let $(M,\omega)$ be a
symplectic orbifold and let $N$ be a full symplectic suborbifold. Then
the normal orbi-bundle to $N$ is a symplectic orbi-bundle (cf. Remark \ref{rmk:full_suborbifolds_normal}). Since the standard Cartan calculus holds for orbifolds (see Remark
\ref{rmk:flow}), the proof of the equivariant symplectic neighborhood
theorem for manifolds using the Moser trick works for fully embedded
symplectic suborbifolds. Hence, the following result holds.

\begin{theorem}[Equivariant symplectic neighborhood theorem for orbifolds]\label{thm:esnt}
  Let $G$ be a compact Lie group and let $(M_i,\omega_i, \Phi_i)$ be a
  Hamiltonian $G$-space for $i=1,2$. Let $N$ be an orbifold endowed
  with a $G$-action and let
  $I_i : N \to M_i$ be a full $G$-equivariant embedding for
  $i=1,2$. If the normal orbi-bundles of $N$ in $M_1$ and $M_2$
  are $G$-equivariantly isomorphic as symplectic orbi-bundles, then for $i=1,2$ there exists a $G$-invariant
  open neighborhood $U_i \subseteq M_i$ of $I_i(N_i)$ and an
  isomorphism $\varphi : (U_1, \omega|_{U_1}, H_1|_{U_1}) \to (U_2,
  \omega|_{U_2}, H_2|_{U_2})$ that extends the diffeomorphism
  identifying $I_1(N) \simeq I_2(N)$.    
\end{theorem}

\begin{remark}\label{rmk:iso_normal}
  If $N$ is a fully embedded symplectic suborbifold of a symplectic
  orbifold $(M,\omega)$, then the normal orbi-bundle $\nu_N$ to $N$ is both a 
  complex and symplectic orbi-bundle, and the two structures are
  compatible in the standard sense (see Corollary
  \ref{cor:complex_bundle}). In fact, just as in the case of
  manifolds, the isomorphism class of $\nu_N$ as a symplectic
  orbi-bundle is the same as that of $\nu_N$ as a complex orbi-bundle
  (see \cite[Corollary A.6]{karshon} for the case of manifolds).
\end{remark}

\subsubsection*{Properties of Hamiltonian torus actions on compact
  symplectic orbifolds}
To conclude this section, we recall a foundational result in the study
of Hamiltonian actions of tori on compact symplectic orbifolds, which
is one of the main results of
\cite{LermanTolman}.

\begin{theorem}\label{thm:connected_fibers}
  Let $T$ be a torus and let $(M,\omega, \Phi)$ be a Hamiltonian
  $T$-space. If $M$ is compact, then $\Phi^{-1}(\alpha)$ is connected
  for every $\alpha \in \mathfrak{t}^*$. 
\end{theorem}

We remark that, as in the case of manifolds, the proof of Theorem
\ref{thm:connected_fibers} uses Morse theory, which, in the case of
orbifolds, was started in \cite{LermanTolman} and further continued in
\cite{hepworth}. Finally, the following result is an immediate consequence of
Theorem \ref{thm:connected_fibers} that we use throughout this paper.

\begin{corollary}\label{lem::uniqueMaxMin}
  Let $(M,\omega,H)$ be Hamiltonian $S^1$-space. If $M$
  is compact, the function $H$ has a unique local maximum and minimum. % In particular,  the space  can  have at most two fixed symplectic orbi-surfaces,  {\color{red} where $H$ attains a global extremum}.
\end{corollary}

\section{Local normal forms}\label{ChapterLocalForms}

\subsection{Orbi-weights of Hamiltonian $S^1$-actions}\label{sec:unit-repr-s1}
Since we are primarily interested in Hamiltonian $S^1$-actions on
symplectic orbifolds, motivated by Corollary
\ref{cor:linear-symp_action}, Remark \ref{rmk:G-compact} and Theorem \ref{thm:darboux}, we take a closer look at
unitary representations of $S^1$ on the Hermitian vector space
$\C^n/\Gamma$, with the view of proving Lemma \ref{lem::orbiweights}
and introducing orbi-weights at a fixed point (see Definition \ref{defn:orbi-weights}). We begin by proving the following result, which
provides a description of (abelian) finite extensions of tori and, as
such, may be of independent interest.

\begin{theorem}\label{thm::groupext}
  Let $\Gamma$ be a finite group and let $T$ be a torus. Consider the following short exact sequence of Lie groups
  
  \begin{equation} \label{eq:exactseq}
    \begin{tikzcd} 
      1\arrow{r} &\Gamma\arrow[r,"i"]& \widehat{G}\arrow[r,"j"]& T\arrow{r} & 1.
    \end{tikzcd}
  \end{equation}
  Let $\widehat{G}_0\trianglelefteq \widehat{G}$ denote the  identity component of $ \widehat{G}$  and set $\Gamma_0:=i^{-1}(\widehat{G}_0)\trianglelefteq\Gamma.$ Then 
  \begin{itemize}[leftmargin=*]
  \item $\widehat{G}_0$ is a torus of the same dimension as $T$ and $j|_{\widehat{G}_0}:\widehat{G}_0\rightarrow T$ is a covering map,
  \item $\pi_0(\widehat{G})\simeq \Gamma/\Gamma_0$, and
  \item $\widehat{G}\simeq (\Gamma\times \widehat{G}_0)/\Gamma_0$, where $\Gamma_0$ is identified with a subgroup of $\Gamma\times \widehat{G}_0$ via
    \begin{align*}
      \Gamma_0&\hookrightarrow \Gamma\times \widehat{G}_0\\
      \gamma&\mapsto (\gamma,i(\gamma^{-1})).
    \end{align*}
  \end{itemize}
  Moreover, if $\Gamma$ is abelian, then $\widehat{G}\simeq(\Gamma/\Gamma_0)\times\widehat{G}_0$.
\end{theorem}

\begin{proof}
  Since $j : \widehat{G} \to T$, is a covering map, it induces an isomorphism of Lie algebras $
  j_*:\mathrm{Lie}(\widehat{G})\rightarrow \mathrm{Lie}(T)$. Hence,
  $\widehat{G}_0$ is abelian and of the same dimension as $T$. Since $\widehat{G}$ is compact, $\widehat{G}_0$
  is a torus. Moreover, since $\exp_T \circ j_* = j \circ
  \exp_{\widehat{G}}$, since
  $\exp_{\widehat{G}}(\mathrm{Lie}(\widehat{G})) =
  \widehat{G}_0$ and since $\exp_T(\mathrm{Lie}(T)) =T$, the
  restriction of $j$ to $\widehat{G}_0$ is onto, so that
  \begin{equation}
    \label{eq:3}
    \begin{tikzcd} 
      1\arrow{r} &\Gamma_0\arrow[r,"i_0"]& \widehat{G}_0\arrow[r,"j_0"]& T\arrow{r} & 1
    \end{tikzcd}
  \end{equation}
  \noindent
  is a short exact sequence of Lie groups, where $i_0$ (respectively $j_0$) denotes
  the restrictions of $i$ (respectively $j$) to $\Gamma_0$ (respectively
  $\widehat{G}_0$). Hence, $j_0 : \widehat{G}_0 \to T$ is a covering
  map.

  Since
  \begin{center}
    \begin{tikzcd}
      &1\arrow[d]& 1\arrow[d]&1\arrow[d]& \\
      1\arrow{r} &\Gamma_0 \arrow[d,  "i_0" ]\arrow[r,hook]& \Gamma \arrow[d, "i"]\arrow[r, two heads]&
      \Gamma/\Gamma_0 \arrow[d,]\arrow{r} & 1\\
      1\arrow{r} & \widehat{G}_0\arrow[r,hook]\arrow[d,
      two heads, "j_0"]& \widehat{G}\arrow[d,two heads,"j"]\arrow[r, two heads]& \pi_0(\widehat{G})\arrow[d]\arrow[r] & 1\\
      1\arrow[r]	& T\arrow[r]\arrow[d]&T\arrow[r]\arrow[d]&1 & \\
      & 1& 1& & 
    \end{tikzcd}
  \end{center}
  is a diagram of exact sequences of Lie groups, we have $\pi_0(\widehat{G})
  \simeq \Gamma/\Gamma_0$. 

  Next we show that every element of $\Gamma$ commutes with every
  element of $\widehat{G}_0$. To show this, let $g_0\in
  \widehat{G}_0$ and $h\in \Gamma$. Since $T$ is abelian, $j
  \left(g_0 \, h \, g_0^{-1} \right) = e \in T$. Hence, $ g_0
  \, h \, g_0^{-1} \in \Gamma$. Since $\widehat{G}_0$ is connected
  and $\Gamma$ is discrete, for a fixed $h \in \Gamma$, the map $
  \widehat{G}_0 \to \Gamma$ that sends $g_0$ to $g_0 h g_0^{-1}$
  is constant. Hence, $g_0 \, h \, g_0^{-1} = h$ for every $g_0\in
  \widehat{G}_0$ and every $h \in \Gamma$, as desired.

  Since $\Gamma$ commutes with $\widehat{G}_0$,  the map 
  \begin{align} \label{eq:group}
    \mu:\Gamma \times \widehat{G}_0&\rightarrow \widehat{G}\\ \nonumber
    (g_1,g_2)&\mapsto i(g_1) g_2
  \end{align}
  is a smooth homomorphism. Moreover, it is surjective. Indeed,  let $g\in \widehat{G}$. Then $g\widehat{G}_0
  \subset \widehat{G}$ is the connected component of $\widehat{G}$
  containing $g$. Since $\Gamma$ surjects onto $\pi_0(\widehat{G})$,
  there exists an element $h\in\Gamma$ such that $i(h) \in
  g\widehat{G}_0$. Hence, $i(h)=g \, g_0$ for some $g_0 \in
  \widehat{G}_0$, as desired. Since 
  \begin{align*}
    \ker\mu&
             =\{(g_1,g_2)\in\Gamma\times\widehat{G}_0 \mid g_2=i(g_1^{-1})\}\simeq i^{-1}(\widehat{G}_0) = \Gamma_0,
  \end{align*}
 we have $\widehat{G}\simeq(\Gamma\times\widehat{G}_0)/\Gamma_0$.

  If $\Gamma$ is abelian, then
  $\widehat{G}\simeq(\Gamma\times\widehat{G}_0)/\Gamma_0$ is also
  abelian. Furthermore, since $\widehat{G}_0$ is a torus, it is a
  divisible group. Hence, we can extend the identity map $\widehat{G}_0\rightarrow \widehat{G}_0$ to a homomorphism $\widehat{G}\rightarrow
  \widehat{G}_0$ (see \cite[Lemma 4.2]{lang}). The restriction of
  this homomorphism  to
  $\Gamma$ is a homomorphism
  $I:\Gamma\rightarrow\widehat{G}_0$ that extends
  $\Gamma_0\rightarrow\widehat{G}_0$. We consider
  the  isomorphism 
  \begin{align*}
    \phi:\Gamma\times\widehat{G}_0&\rightarrow \Gamma\times\widehat{G}_0\\
    (g_1,g_2)&\mapsto(g_1,I(g_1)g_2).
  \end{align*}
  % \begin{align*}
  %   \phi((g_1,g_2)(g_1',g_2'))&=\phi(g_1g_1',g_2g_2')
  %   =(g_1g_1',I(g_1g_1')g_2g_2') \\ 
  %   & =(g_1g_1',I(g_1)I(g_1')g_2g_2')
  %   =\phi(g_1,g_2)\phi(g_1',g_2'),
  % \end{align*}
  By a direct calculation, we see that
  $\phi(\ker\mu)=\Gamma_0\times\{1\}$, where $\mu:\Gamma \times
  \widehat{G}_0 \to \widehat{G}$ is as in \eqref{eq:group},
  and that  $\phi$ descends to an isomorphism
  $(\Gamma\times\widehat{G}_0)/\Gamma_0\rightarrow
  \Gamma/\Gamma_0\times\widehat{G}_0$, as desired.
\end{proof}

In what follows, we fix the notation of the subsection on linear
symplectic actions in Section \ref{sec:hamilt-acti-sympl}. Moreover, we fix the standard symplectic form
$\omega_0$ on $\C^n$, so that $\C^n/\Gamma$ inherits a linear
symplectic form. By an abuse of notation, we denote the resulting
symplectic vector orbi-space also by $\C^n/\Gamma$.

\begin{lemma}\label{lem::orbiweights}  Let $\Gamma$ be a subgroup of $\mathrm{U}(n)$ and let $\rho : S^1 \to
  \mathrm{U}(\mathbb{C}^n/\Gamma)$ be a representation.
  Let $\widehat{G}$ be the extension of $S^1$ by $\Gamma$ of
  \eqref{eq:exact_sequ_lin}, let $\widehat{\rho} : \widehat{G} \to
  N_{\mathrm{U}(n)}(\Gamma)$ be the representation given by Remark
  \ref{rmk:G-compact}, and let $\widehat{G}_0$ denote the connected
  component of the identity of $\widehat{G}$. 
  There exist a positive integer $p$ that divides $|\Gamma|$,
  $a_1,\ldots, a_n\in \mathbb{Z}$, and an isomorphism
  $\widehat{G}_0 \simeq S^1$such that
  \begin{itemize}[leftmargin=*]
  \item the covering map $\widehat{G}_0 \simeq S^1 \to S^1$ is given
    by $\widehat{\lambda} \to \widehat{\lambda}^p$, and
  \item for all $\widehat{\lambda}\in \widehat{G}_0 \simeq S^1$ and for
    all $(z_1,\ldots, z_n) \in \mathbb{C}^n$, 
    \begin{equation*}
    \widehat{\rho}(\widehat{\lambda})(z_1,\ldots, z_n) = (\widehat{\lambda}^{a_1}z_1,\ldots, \widehat{\lambda}^{a_n}z_n).
  \end{equation*}
  \end{itemize}
  Finally, the induced $S^1$-action on $\mathbb{C}^n/\Gamma$ is
  Hamiltonian and a moment map $H : \mathbb{C}^n/\Gamma \to \R$ is
  given by  
  \begin{equation}
    \label{eq:6}
    H([z_1,\ldots,z_n]) =
    \frac{1}{2}\left(\frac{a_1}{p}|z_1|^2+\cdots + \frac{a_n}{p}|z_n|^2 \right).
  \end{equation}
\end{lemma}

% %Set $\Gamma_0 := \Gamma \cap
%   \widehat{G}_0$. Then
%   $$\widehat{\rho}(\gamma_0) \in \Gamma < \mathrm{U}(n) \quad \text{ for
%     any } \gamma_0 \in \Gamma_0. $$

As an immediate consequence of Lemma \ref{lem::orbiweights}, the
$S^1$-action on $\C^n/\Gamma$ is given by
\begin{equation}
  \label{eq:1}
  \lambda \cdot [z_1,\ldots, z_n] = [\lambda^{\frac{a_1}{p}}z_1,\ldots,
  \lambda^{\frac{a_n}{p}}z_n] \text{ for any } \lambda \in S^1,
\end{equation}
(cf. Examples
\ref{exm:circle_action_Talvacchia} and \ref{exm:circle_action_cp1cp1}). 

\begin{proof}[Proof of Lemma \ref{lem::orbiweights}]
  Since the following diagram commutes
  \begin{equation}\label{commdiag}
    \begin{tikzcd}
      1\arrow{r} &\Gamma\arrow[d,equal]\arrow[r,hook]& \widehat{G}\arrow[d,"\widehat{\rho}"]\arrow[r,"j"]& S^1\arrow[d,"\rho"]\arrow{r} & 1\\
      1\arrow{r} & \Gamma\arrow[r,hook]& N_{\mathrm{U}(n)}(\Gamma)\arrow[r,"\pi"]& \mathrm{U}(\mathbb{C}^n/\Gamma)\arrow{r} & 1
    \end{tikzcd}
  \end{equation}
  (see the discussion preceding Corollary
  \ref{cor:linear-symp_action}, and Remark \ref{rmk:G-compact}),
  by Theorem~\ref{thm::groupext}, the restriction of $j$ to $\widehat{G}_0$ is a covering map
  $\widehat{G}_0 \to S^1$. Hence, there exist an
  isomorphism $\widehat{G}_0 \simeq S^1$ and a positive integer $p$
  such that the subgroup $\Gamma_0 < \Gamma$ of Theorem
  \ref{thm::groupext} is isomorphic to $\Z_p$ and the covering map $\widehat{G}_0 \simeq S^1 \to S^1$ sends
  $\widehat{\lambda}$ to $\widehat{\lambda}^p$. In particular, $p$
  divides $|\Gamma|$. Moreover, since
  $\widehat{\rho}$ is unitary, there exist $a_1,\ldots,
  a_n \in \Z$ such that
  \begin{equation}
    \label{eq:4}
    \widehat{\rho}_0(\widehat{\lambda}) = \mathrm{diag}(\widehat{\lambda}^{a_1},\ldots,
    \widehat{\lambda}^{a_n}) \quad \text{for any } \widehat{\lambda} \in \widehat{G}_0. 
  \end{equation}
  Finally, the $\widehat{G}$-action on $(\C^n,\omega_0)$ is Hamiltonian and a
  moment map $\widehat{H} \colon \C^n \to \R$ is given by 
  \begin{equation}
    \label{eq:7}
    \widehat{H}(z_1,\ldots,z_n) = \frac{1}{2}\left(a_1|z_1|^2+\ldots +
      a_n|z_n|^2 \right).
  \end{equation}
  Since the derivative $\R \to \R$ of the covering map $S^1 \to S^1$
  is given by multiplication by $p$ and since $\widehat{H}$ is
  $\Gamma$-invariant, the last statement follows by Corollary
  \ref{cor:linear-symp_action}. 
\end{proof}

\begin{remark}\label{rmk:non-abelian_local_extremum}
  With the notation of Lemma \ref{lem::orbiweights}, set $n = 2$,
  suppose that the representation $\rho$ is faithful and that $\Gamma$
  is not abelian. We claim that, in this case, $a_1 = a_2 = \pm 1$. To
  see this, we observe that, by \cite[Lemma 3.1]{LermanTolman}, $\rho$ is injective if and only if
  $\widehat{\rho}$ is injective. Hence, $\gcd(a_1,a_2) =1$. By
  definition, $\widehat{\rho}(\widehat{G}_0)$ is contained in
  $N_{\mathrm{U}(2)}(\Gamma)$. Thus, if $\left(
    \begin{smallmatrix}
      \alpha & \beta \\
      \gamma & \delta
    \end{smallmatrix}
  \right) \in \Gamma$, then
  \begin{equation}
    \label{eq:65}
    \begin{pmatrix}
      \widehat{\lambda}^{a_1} & 0 \\
      0 & \widehat{\lambda}^{a_2}
    \end{pmatrix}
    \begin{pmatrix}
      \alpha & \beta \\
      \gamma & \delta
    \end{pmatrix}
    \begin{pmatrix}
      \widehat{\lambda}^{-a_1} & 0 \\
      0 & \widehat{\lambda}^{-a_2}
    \end{pmatrix} = \begin{pmatrix}
      \alpha & \widehat{\lambda}^{a_1-a_2}\beta \\
      \widehat{\lambda}^{a_2-a_1}\gamma & \delta
    \end{pmatrix} \in \Gamma
  \end{equation}
  for all $\widehat{\lambda} \in \widehat{G}_0 \simeq S^1$. Since
  $\Gamma$ is finite and not abelian, it follows that $a_1 = a_2$,
  whence the result.
\end{remark}

The rational numbers that appear in Lemma \ref{lem::orbiweights} are the analog of
weights in the realm of orbifolds. To see this, let $(M,\omega,H)$ be
a Hamiltonian $S^1$-space and let $x \in M$ be a fixed point. By Theorems
\ref{thm:darboux_orbifolds} and \ref{thm:darboux}, the action near $x$ is modeled on an $S^1$-action
on $\C^n/\Gamma$ induced by a unitary representation, where $\Gamma$
is the orbifold structure group of $x$ that acts on $\C^n$ in a
unitary fashion. This is precisely the situation considered in Lemma
\ref{lem::orbiweights}. Hence, we may introduce the following terminology, following
\cite{LermanTolman,olocalization}.

\begin{definition}\label{defn:orbi-weights}
  Let $(M,\omega,H)$ be
  a Hamiltonian $S^1$-space and let $x \in M$ be a fixed point with
  orbifold structure group $\Gamma$. Let $p,a_1,\ldots, a_n \in \Z$ be the integers given
  by Lemma \ref{lem::orbiweights}. We say that the rational numbers $\frac{a_1}{p},\ldots,
  \frac{a_n}{p} \in \mathbb{Q}$ are the {\bf
    orbi-weights} at $x$. 
\end{definition}

\begin{example}
Let $\Gamma=\mathbb{Z}_m$ act on $\mathbb{C}^2$ by
\begin{align*}
 \mathbb{Z}_m \times \mathbb{C}^2 & \to \mathbb{C}^2  \nonumber \\
(\mu, (z_1,z_2))  & \mapsto (z_1,\mu z_2).
\end{align*}
%$$
%\xi_m\cdot(z_1,z_2)=(z_1,\xi_m z_2),
%$$
If the action of $S^1$  on $\mathbb{C}^2/\mathbb{Z}_m$ is given by
$$
\lambda \cdot [z_1,z_2]_m=[\lambda z_1, z_2]_m, \quad \lambda \in S^1,
$$
then the extension $\widehat{G}$ of $S^1$ by $\mathbb{Z}_m$ in
\eqref{commdiag} is $\widehat{G} = S^1\times \mathbb{Z}_m$ 
and $\Gamma_0:=\widehat{G}_0 \cap \mathbb{Z}_m = \{id\}$. In this
case, the orbi-weights at $[0]$ are $1,0$.

If, on the other hand, the action of $S^1$ is given by
$$
\lambda \cdot [z_1,z_2]_m=[ z_1, \lambda z_2]_m, \quad \lambda \in S^1,
$$
then $\widehat{G} = S^1$ and $\Gamma_0= \mathbb{Z}_m  $. In this
case, the orbi-weights at $[0]$ are $0, \frac{1}{m}$. \hfill$\Diamond$
\end{example}	

  \begin{example}\label{ex:projective}(Weighted projective planes)
	Consider the weighted projective plane
        $M=\mathbb{C}P^2(p,s,q)$ with $p,q,s$ pairwise relatively
        prime, equipped with (a multiple of)  the standard  symplectic
        form of Example~\ref{ex_weightedprj}  and the Hamiltonian
        circle action of Examples~\ref{exm:circle_action_weighted}
        and \ref{exm:weighted_ham} given by
  \begin{equation}\label{eq:exactionwps}
    \begin{split}
      \phi: S^1 \times \C P^2(p,s,q) &\to \C P^2(p,s,q) \\
      (\lambda, [z_0:z_1:z_n]) &\mapsto[\lambda^a z_0:\lambda^b z_1:\lambda^cz_2],
    \end{split}
  \end{equation}
where $a,b,c\in \mathbb{Z}$ are such that $k_0:= aq -cp $, $k_1:=bq-cs$ and $k_2:=as-bp$ are positive and pairwise relatively prime. 

This action has three isolated fixed points,  $F_1 = [1 : 0 : 0]$,
$F_2 = [0 : 1 : 0]$ and $F_3 = [0 : 0 : 1]$, with orbifold structure
group $\Z_p$, $\Z_s$ and $\Z_q$ respectively. We consider the global
quotient $\C^2/\Z_p$ given by the following $\Z_p$-action:
\begin{equation*}
  \begin{split}
    \mathbb{Z}_p \times \mathbb{C}^2 & \to \mathbb{C}^2  \\
    (\mu, (z_1,z_2))  & \mapsto (\mu^sz_1,\mu^q z_2).
  \end{split}
\end{equation*}
A l.u.c. centered at $F_1$ is given by the map that sends $[z,w]_p\in
\mathbb{C}^2/\Z_p$ to $[1:z:w]\in M$.
By the weighted homogeneity of the coordinates in  $\mathbb{C}P^2(p,s,q)$, the $S^1$-action in these coordinates  is given by
\begin{align*}
  \lambda\cdot[z_0:z_1:z_2]&=[\lambda^a z_0:\lambda^b z_1:\lambda^cz_2] =[z_0:\lambda^{-\frac{as-bp}{p}}z_1:\lambda^{-\frac{aq-cp}{p}}z_2].
\end{align*} 	
Since the above action is linear, the orbi-weights at $F_1$ are $ -\frac{k_0}{p},-\frac{k_2}{p}$. Similarly, 	
\begin{align*}	
\lambda\cdot[z_0:z_1:z_2]&=[\lambda^{\frac{as-bp}{s}}z_0:z_1:\lambda^{-\frac{bq-cs}{s}}z_2]\\
&=[\lambda^{\frac{aq-cp}{q}}z_0:\lambda^{\frac{bq-cs}{q}}z_1:z_2].
\end{align*} 
Hence, the orbi-weights at $F_2$ and $F_3$ are
$\frac{k_2}{s},-\frac{k_1}{s}$ and $\frac{k_0}{q}, \frac{k_1}{q}$ respectively.

If $k_1=0$ and $k_0,k_2$ are relatively prime, the fixed point set of the
action in \eqref{eq:exactionwps} is the union of $F_1$ and the points of the orbi-sphere
$$
\Sigma=\{[0:z_1:z_2]\in M\} \simeq \C P^1(s,q). 
$$
The orbi-weights at $F_1$ are $ -\frac{q}{p},-\frac{s}{p}$ and the
orbi-weights at each point in $\Sigma$ are $0,1$.  Note that, since
$k_1=0$, we have $bq=cs$. Hence, since $p,q,s$ are pairwise relatively prime
and $k_0,k_2$ are relatively prime, we must have $k_2=s$ and
$k_0=q$. \hfill$\Diamond$ 
\end{example}

\subsection{Local classification near fixed points in the four dimensional case}\label{sec:local-class-near}
The aim of this section is to prove the local classification of Hamiltonian
$S^1$-spaces on four dimensional symplectic orbifolds near a cyclic isolated
singular point of type $\frac{1}{m}(1,l)$ -- see Theorem
\ref{cor::localformpt} below. We start by proving the following
result, which gives necessary and sufficient conditions for given
rational numbers to arise as orbi-weights at a fixed point that is
singular for the orbifold under the assumption that the $S^1$-action
be effective. 

\begin{proposition}\label{prop_local}
  Let $m \geq 2$ be an integer, let $p$ be a positive integer that
  divides $m$, and let $a_1,a_2 \in \Z$. Let
  $\mathbb{C}^2/\mathbb{Z}_m$ be a model for a cyclic isolated
  singular point of type $\frac{1}{m}(1,l)$. There
  exists a faithful representation $\rho : S^1 \to
  \mathrm{U}(\mathbb{C}^2/\mathbb{Z}_m)$ such that $p, a_1,a_2$ are as
  in Lemma \ref{lem::orbiweights} if and only if
  \begin{equation}
    \label{eq:50}
    \gcd(a_1,a_2)=1 \quad \text{and} \quad \gcd(|a_1l-a_2|,m)=p.
  \end{equation}
  % one of
  % the following conditions holds:
  % \begin{enumerate}[leftmargin=*]
  % \item \label{item:17} $a_1,a_2\ne0$, $\gcd(a_1,a_2)=1$, and
  %   $\gcd(|a_1l-a_2|,m)=p$, or 
  %   % \item \label{item:19} $\lvert a_1\rvert =1$, $a_2=0$ and $p=1$, or
  % \item \label{item:18} $\{|a_1|, |a_2|\} =\{0,1\}$ and $p=1$.
  % \end{enumerate}
  Moreover, $[0,0]$ is an isolated fixed point for the induced $S^1$-action on
  $\C^2/\Z_m$ exactly if $a_1a_2 \neq 0$. %condition \eqref{item:17}
  % holds.
  Finally, the $S^1$-action extends to a
  toric action on  $\mathbb{C}^2/\mathbb{Z}_m$. %in any of the above cases,
\end{proposition}
    
\begin{proof}
  Suppose first that there is a faithful representation $\rho : S^1 \to
  \mathrm{U}(\mathbb{C}^2/\mathbb{Z}_m)$ such that the orbi-weights at
  $[0]$ of the induced $S^1$-action are $a_1/p,
  a_2/p$. Let $\widehat{G}$ be the extension of $S^1$ by $\Z_m$ of
  \eqref{eq:exact_sequ_lin}, let $\widehat{\rho} : \widehat{G} \to
  N_{\mathrm{U}(2)}(\Z_m)$ be the representation given by Remark
  \ref{rmk:G-compact}, and let $\widehat{G}_0$ denote the connected
  component of the identity of $\widehat{G}$. By Lemma
  \ref{lem::orbiweights}, there exists a positive integer $p$ and an isomorphism $\widehat{G}_0
  \simeq S^1$ such that $\Gamma_0 = \Z_m \cap
  \widehat{G}_0 \simeq \Z_p$. Since the following diagram commutes
  \begin{equation}
    \begin{tikzcd}
      1\arrow{r} &\Z_m\arrow[d,equal]\arrow[r,hook]& \widehat{G}\arrow[d,"\widehat{\rho}"]\arrow[r,"j"]& S^1\arrow[d,"\rho"]\arrow{r} & 1\\
      1\arrow{r} & \Z_m\arrow[r,hook]& N_{\mathrm{U}(2)}(\Z_m)\arrow[r,"\pi"]& \mathrm{U}(\mathbb{C}^2/\Z_m)\arrow{r} & 1,
    \end{tikzcd}
  \end{equation}
 for any $\gamma_0 \in \Z_p$   we have  $\widehat{\rho}(\gamma_0) \in
  \Z_m < \mathrm{U}(2)$. Hence, since $[0,0]$ is an isolated singular
  point of type $\frac{1}{m}(1,l)$, for any $\gamma_0 \in \Z_p$, there
  exists $\mu \in \Z_m$ such that
  \begin{equation}
    \label{eq:38}
     \widehat{\rho}(\gamma_0)(z_1,z_2) = (\mu z_1,\mu^l z_2) \text{
   for all } (z_1,z_2) \in \C^2.
  \end{equation}
  On the other hand, by Lemma \ref{lem::orbiweights},
  \begin{equation}
    \label{eq:39}
    \widehat{\rho}(\widehat{\lambda})(z_1,z_2) =
    (\widehat{\lambda}^{a_1}z_2,\widehat{\lambda}^{a_2}z_2) \text{ for
      all } \widehat{\lambda} \in S^1 \simeq \widehat{G}_0.
  \end{equation}
  By \eqref{eq:38} and \eqref{eq:39}, and since $\Gamma_0 \simeq
  \Z_p$, we have that
  \begin{equation} \label{eq:condition1}
    a_1l - a_2 = 0 \mod p.
  \end{equation}

By \cite[Lemma 3.1]{LermanTolman}, $\rho$ is injective if and only if
$\widehat{\rho}$ is injective. By \eqref{eq:39}, if
$\widehat{\rho}$ is injective, then $\gcd(a_1,a_2)=1$. Moreover,
by \eqref{eq:1}, $\rho(\lambda) =
\mathrm{diag}(\lambda^{\frac{a_1}{p}},
\lambda^{\frac{a_2}{p}})$. Hence, setting $\lambda = e^{2\pi i
  \theta}$ for some $\theta \in [0,1)$, we have that $\lambda \in \ker
\rho$ if and only if 
$$ e^{\frac{2\pi i\theta a_1}{p}}  =\mu \text{ and }\quad
e^{\frac{2\pi i\theta a_2}{p}}=\mu^l \text{ for some }\mu \in
\mathbb{Z}_m. $$
Equivalently,
\begin{equation}
  \theta(a_1l-a_2) \in \mathbb{Z},\quad  \theta(a_1l-a_2)  = 0 \mod p,  \quad\text{and}\quad e^{\frac{2\pi i\theta a_1}{p}}\in \mathbb{Z}_m. \label{inj_cond}
\end{equation}

\begin{claim}
  If $\gcd(a_1,a_2)=1$ and $a_1a_2\ne 0$, then $\rho$ is injective if and only if $\gcd\left(\frac{\lvert a_1l-a_2\rvert}{p},\frac{m}{p}\right)=1$.
\end{claim}

\begin{proof}[Proof of Claim]
  For each implication, we split the proof in two cases, namely $p=1$ and $p >1$. First we suppose that
  $\gcd\left(\frac{|a_1l-a_2|}{p},\frac{m}{p}\right)=1$. \\

  \noindent
  {\bf Case $p=1$:} Since $\gcd(0,m)=m\neq 1$, then $\lvert a_1l - a_2\rvert \neq 0$. If $\lvert a_1l-a_2 \rvert=1$, then the
  result follows by \eqref{inj_cond}. Hence, we may assume that
  $|a_1l-a_2|>1$. By \eqref{inj_cond}, there exists an integer $0\leq q< \lvert
  a_1l-a_2\rvert $ such that
  $$\theta=
  \frac{q}{\lvert a_1l-a_2\rvert} \quad  \text{ and } \quad e^{\frac{2\pi i q a_1}{\lvert a_1l-a_2\rvert}}\in \Z_m,$$ 
  so that $e^{\frac{2\pi iq a_1m}{\lvert a_1l-a_2\rvert}}=1$. Since $\gcd(\lvert a_1l-a_2\rvert,m)=1$ and 
  $$\gcd(\lvert a_1l-a_2\rvert, a_1)=\gcd(a_2,a_1)=1,$$ 
  this can only be satisfied if $q=0$ and so $\rho$ is injective. \\
  
  \noindent
  {\bf Case $p >1$:} By \eqref{eq:condition1}, there
  exists an integer $\alpha \geq 0$ such that $\lvert a_1l-a_2\rvert
  =\alpha\, p$. Therefore 
  \begin{equation}\label{eq:gcd}
    \gcd\left(\frac{\lvert a_1l-a_2\rvert }{p},\frac{m}{p}\right)=\gcd(\alpha,k)=1,
  \end{equation}
  where $k$ is such that $m=k \, p$. Suppose first that $\alpha=0$. Then
  $$1=\gcd(\alpha,k)=\gcd(0,k)=k,$$ 
  and so $p=m$. Hence, by  \eqref{inj_cond}, $\theta a_1\in \mathbb{Z}$. Moreover, 
  $$1=\gcd(\lvert a_1l-a_2\rvert,a_1)=\gcd(0,a_1)=a_1.$$
  Hence, $\theta\in\mathbb{Z}$ and $\rho$ is injective. Next
  we assume that $\alpha=1$. Then, by \eqref{inj_cond}, $\theta\in
  \mathbb{Z}$ and $\rho$ is injective. Finally, if  $ \alpha >1$, then, by the second equation in \eqref{inj_cond}, there exists an integer
  $0\leq q < p$ such that  
  $$\theta=\frac{q}{\alpha} \quad \text{ and} \quad e^{\frac{2\pi i q
      a_1k}{\alpha}}=1.$$
  Since 	
  $$1=\gcd(a_1,a_2)= \gcd(\lvert a_1l-a_2\rvert ,a_1)=\gcd( \alpha\,  p,a_1),$$  
  we have that $\gcd(a_1,\alpha)=1$. Moreover, by assumption,
  $\gcd(k,\alpha)=1$ and so $q=0 \mod \alpha$. Hence, since $q<\alpha$, it follows that $q=\theta=0$ and so
  $\rho$ is injective. \\
  
  Conversely, suppose that  $\gcd\left(\frac{\lvert
    a_1l-a_2\rvert}{p},\frac{m}{p}\right)=c\ne 1$. \\

  \noindent
  {\bf Case $p=1$:} If $\lvert a_1l-a_2\rvert=0$, then
  \eqref{inj_cond} is always satisfied  by $\theta=\frac{1}{m}$
  and so $\rho$ is not injective. Otherwise, since $\lvert
  a_1l-a_2\rvert \neq 1$, then $\lvert a_1l-a_2\rvert >1$. In this case, $\theta=1/c$ satisfies \eqref{inj_cond}  and so  
  $\rho$ is not injective. \\
  
  \noindent
  {\bf Case $p > 1$:} Note that there exists an integer  $\alpha \neq
  1$ such that $\lvert a_1l-a_2\rvert =\alpha\,  p$. If $\alpha =0$,  then \eqref{inj_cond}  is satisfied  by 
  $\theta=\frac{1}{k}$ and so  $\rho$ is not injective.  If
  $\alpha>1$, then, since $c=\gcd(\alpha,k)$, we have that $\theta=1/c$ satisfies \eqref{inj_cond} and so 
  $\rho$ is not injective.
\end{proof}

If $a_1a_2\neq 0$, then, by the above Claim and since $\gcd\left(\lvert
  a_1l-a_2\rvert,p\right)=p$ by \eqref{eq:condition1}, we have that $\gcd\left(\lvert
  a_1l-a_2\rvert,m\right)=p$, so that condition~\eqref{eq:50} holds. Else, assume that $a_2=0$ without loss of generality. Since
$\gcd(a_1,0)=1$, then $\lvert a_1\rvert =1$. This readily implies that
$p =1$ for, if not, 
$$p=\gcd(\lvert a_1l\rvert,p)=\gcd(l,p),$$
which contradicts $1=\gcd(l,m)$. Hence, condition \eqref{eq:50} holds.
	
It is a simple calculation to check that \eqref{eq:50} is sufficient to define a faithful representation $\rho : S^1 \to
\mathrm{U}(\mathbb{C}^2/\mathbb{Z}_m)$ such that the orbi-weights at
$[0,0]$ of the induced $S^1$-action are $a_1/m,
a_2/m$. We leave this to the reader and remark only that
\eqref{eq:condition1} ensures that $\rho$ is well-defined, and that the
above Claim implies that it is injective in the case $a_1a_2 \neq 0$.

Suppose that $[0,0] \in \C^2/\Z_m$ is not an isolated fixed point of
the $S^1$-action. Since the action is given by
$$ \lambda \cdot [z_1,z_2] = [\lambda^{\frac{a_1}{p}}z_1,
\lambda^{\frac{a_1}{p}}z_2],$$
and since $\rho$ is faithful, if $[z_1,z_2] \neq [0,0]$ is a fixed
point, then exactly one of $z_1$ or $z_2$ equals 0. Assume without
loss of generality that $z_2=0$. Then, for all $\lambda \in S^1$,
there exists $\mu \in \Z_m$ such that $\lambda^{\frac{a_1}{p}} =
\mu$. This can only occur if $a_1 = 0$. Hence, since $[0,0] \in
\C^2/\Z_m$ is a fixed point of the $S^1$-action, it is isolated
exactly if $a_1a_2 \neq 0$.

Finally, we want to prove that the circle action on
$\mathbb{C}^2/\mathbb{Z}_m$ extends to a torus action. By Theorem
\ref{thm::groupext}, $\widehat{G}$ is abelian. Hence, since $\widehat{\rho}$
is a unitary representation, there exists a basis of $\C^2$ such that $\widehat{\rho}(\widehat{G})$ consists
solely of diagonal elements. Since the maximal torus $(S^1)^2 \subset
\mathrm{U}(2)$ equals the subset of diagonal elements of
$\mathrm{U}(2)$ and since $\Z_m < (S^1)^2$, the result follows. 
\end{proof}

\begin{remark}\label{rmk:type}
  Let $\mathbb{C}^2/\mathbb{Z}_m$ be a model for a cyclic isolated
  singular point of type $\frac{1}{m}(1,l)$, let $\rho : S^1 \to
  \mathrm{U}(\C^2/\Z_m)$ be a faithful representation, and let $p,
  a_1,a_2$ be the integers associated to $\rho$ by Lemma \ref{lem::orbiweights}. Let $1 \leq l'
  < m$ be the unique integer such that $l \cdot l' = 1 \mod m$. Then the model for a cyclic isolated
  singular point of type $\frac{1}{m}(1,l')$ is isomorphic to
  $\mathbb{C}^2/\mathbb{Z}_m$ (see Lemma \ref{lemma:model_iso}). We
  fix such an isomorphism and let
  $\rho'$ denote the corresponding faithful unitary
  representation. Let $p', a_1',a_2'$ be the integers associated to
  $\rho'$ by Lemma \ref{lem::orbiweights}. Then $p = p'$ and the
  multisets $\{a_1,a_2\}$ and $\{a_1',a_2'\}$ are equal. Moreover,
  since $l \cdot l' = 1 \mod m$, equation \eqref{eq:50} is satisfied by $p, a_1,a_2$ if and only if they
  are satisfied by $p', a_1',a_2'$.

  In particular, if $a_1 \neq a_2$, this allows us to {\em choose} a
  preferred representative for the type of a cyclic isolated singular
  point: Namely, we choose  $\frac{1}{m}(1,l)$ so that $a_1 >
  a_2$, i.e., the larger of the two orbi-weights is in the direction of
  the coordinate corresponding to $1$. If $a_1 = a_2$, then by Proposition \ref{prop_local} $a_1
  = a_2 = \pm 1$. Moreover, if $l^2 = 1 \mod m$, then $l = l'$, so that there is
  no choice to be made. Finally, if $l^2 \neq 1 \mod m$, then there is
  no preferred choice for the expression $\frac{1}{m}(1,l)$. This
  comes into play in Section \ref{sec:label-mult-hamilt}.
\end{remark}

The main result of this section, Theorem
\ref{cor::localformpt}, is an immediate consequence of Theorem \ref{thm:darboux}, Lemma \ref{lem::orbiweights} and
Proposition \ref{prop_local}.

\begin{theorem}\label{cor::localformpt}
  Let $(M,\omega,H)$ be a four dimensional Hamiltonian $S^1$-space and
  let $x \in M$ be a cyclic quotient isolated singular point of type
  $\frac{1}{m}(1,l)$ that is also a fixed point for the $S^1$-action.  Let $\C^2/\Z_m$ be the corresponding model and
  let $\omega_0$ be the
  symplectic form on $\C^2/\Z_m$ that descends from the standard
  symplectic form on $\C^2$. Let $p, a_1,a_2 \in \Z$ be the integers
  associated to the representation $\rho : S^1 \to \mathrm{U}(T_x M)
  \simeq \mathrm{U}(\C^2/\Z_m)$ defining the linear $S^1$-action on
  $T_xM$ by Lemma \ref{lem::orbiweights}.

  There exist an open $S^1$-invariant neighborhood $U
  \subseteq M$ of $x$, an
  open neighborhood $V \subseteq \C^2/\Z_m$ of $[0,0]$, and an
  isomorphism of Hamiltonian $S^1$-spaces between $(U,\omega, H)$ and
  $(V,\omega_0, H_{\mathrm{lin}})$, where the map 
  $H_{\mathrm{lin}} : \C^2/\Z_m \to \R$ given by 
 \begin{equation*}
    H_{\mathrm{lin}}([z_1,z_2]):=H(x)+\frac{1}{2}\left(\frac{a_1}{p}\lvert z_1\rvert^2+\frac{a_2}{p}\lvert z_2\rvert^2\right)
  \end{equation*}  
  is the moment map for the linear $S^1$-action.

  Moreover, $x$ is an isolated fixed point of the $S^1$-action if and
  only if
  $$ a_1a_2 \neq 0 \, , \, \gcd(a_1,a_2)=1 \text{ and }
  \gcd(|a_1l-a_2|,m)=p. $$
  Finally, $x$ is not an isolated fixed point if and only if $p=1$ and
  $\{|a_1|,|a_2|\} = \{0,1\}$.
  % \vspace{2cm}
  % be an isolated singular point that is a 
  % singularity of type $\frac{1}{m}(1,l)$ with 
  % $$\gcd(l,m)=1\quad \text{and} \quad 1\leq l<m.$$
  % \begin{enumerate}
  % \item[i)]  {\color{red} If $x$ is an isolated fixed point  of the $S^1$-action}
  %   then there exist  unique non zero integers $a_1, a_2, p $ satisfying 
  %   $$\gcd(a_1,a_2)=1 \quad \text{and} \quad \gcd(|a_1l-a_2|,m)=p,$$ 
  %   such that the linearized action of $S^1$ on $(T_xM,\omega_x)$ is given by
  %   \begin{equation}\label{equ:lemRep}
  %     \lambda \cdot[z_1,z_2]_m=[\lambda^{\frac{a_1}{p}} z_1,\lambda^{\frac{a_2}{p}} z_2]_m.
  %   \end{equation}
  % \item[ii)]  {\color{red} If $x$ is a non-isolated fixed point then the linearized action of $S^1$ on $(T_xM,\omega_0)$ is given by
  %     \begin{equation}\label{equ:lemRepii}
  %       \lambda \cdot[z_1,z_2]_m=[\lambda^{a_1} z_1,\lambda^{a_2} z_2]_m.
  %     \end{equation}
  %     with either $\lvert a_1\rvert = 1$ and $a_2 =0$ or $\lvert a_2\rvert = 1$ and $a_1 =0$.}
  % \end{enumerate}  
  % {\color{red}     In both cases the moment map  for the linearized action on $(T_xM,\omega_0)$  is} 
\end{theorem}

As a consequence of Theorem \ref{cor::localformpt}, its analog for
manifolds (see \cite[Corollary A.7]{karshon}), and Corollary
\ref{cor:fixed_point_set_symp_full_suborbifold}), we obtain the
following result, which describes properties of the fixed point set of
Hamiltonian $S^1$-actions on four dimensional symplectic orbifolds
with discrete singular sets all of whose orbifold structure groups are
cyclic.

\begin{corollary}\label{cor:fixed_point_set}
  Let $(M,\omega,H)$ be a four dimensional Hamiltonian $S^1$-space
  with discrete singular set and such that all orbifold structure
  groups are cyclic. Any connected component of the fixed point set of
  $M$ is a full symplectic suborbifold of dimension at most
  two. Moreover, if $\Sigma$ is a two dimensional component, all
  points in $\Sigma$ are local extrema of $H$. 
\end{corollary}

\begin{remark}\label{rmk:fixed_fully_embedded}
  Under the hypotheses of Corollary \ref{cor:fixed_point_set}, we can
  prove a slightly stronger result, namely that any two dimensional
  component $\Sigma$ of the fixed point set is {\em fully embedded} in the
  sense of \cite[Definition 14]{mestre_weilandt}, i.e., the inclusion
  map $\Sigma \hookrightarrow M$ is a good $C^{\infty}$ map. This uses
  the fact that the singular set of $M$ is discrete\footnote{The example of a
  full suborbifold that is not fully embedded of \cite[Example
  5.4]{weilandt} can be endowed with a suitable $S^1$-action to
  illustrate what may go wrong if the singular set of $M$ is not discrete.}.
\end{remark}

Let $\mathbb{C}^2/\mathbb{Z}_m$ be a model for a cyclic isolated
singular point of type $\frac{1}{m}(1,l)$. Our next result states that
the integers $p, a_1, a_2$ associated to a faithful unitary representation
of $S^1$ on $\C^2/\Z_m$ determine the isomorphism class of the
representation.

\begin{proposition}\label{prop:triple_determines}
  Let $\mathbb{C}^2/\mathbb{Z}_m$ be a model for a cyclic isolated
  singular point of type $\frac{1}{m}(1,l)$. Let $\rho, \rho' : S^1
  \to \mathrm{U}(\C^2/\Z_m)$ be faithful representations. Let $p,
  a_1,a_2$ (respectively $p',
  a'_1,a'_2$) be the integers associated to $\rho $ (respectively
  $\rho'$) by Lemma \ref{lem::orbiweights}. There exists a unitary
  isomorphism $A \in \mathrm{U}(\C^2/\Z_m)$ such that
  \begin{equation}
    \label{eq:40}
    A \circ \rho(\lambda) = \rho'(\lambda) \circ A \text{ for all }
    \lambda \in S^1
  \end{equation}
  if and only if $p = p'$ and either the multisets $\{a_1,a_2\}$ and
  $\{a'_1,a'_2\}$ are equal and $l^2= 1 \mod m$, or $a_1 = a_1'$ and $a_2 =
  a_2'$ and $l^2 \neq 1 \mod m$.
\end{proposition}

In order to prove Proposition \ref{prop:triple_determines}, we need
the following result, stated below without proof as it is a simple
calculation.

\begin{lemma}\label{lemma:normalizer}
  Let $\mathbb{C}^2/\mathbb{Z}_m$ be a model for a cyclic isolated
  singular point of type $\frac{1}{m}(1,l)$. Then the normalizer
  $N_{\mathrm{U}(2)}(\Z_m)$ is given by
  \begin{enumerate}[label=(\arabic*), ref=(\arabic*),leftmargin=*]
  \item \label{item:20} $\mathrm{U}(2)$ if $l =1$,
  \item \label{item:21} the subgroup generated by the diagonal
    elements and the matrix $
    \left(\begin{smallmatrix}
        0 & 1 \\
        1 & 0
      \end{smallmatrix} \right)$ if $l^2 = 1 \mod m$ and $l \neq 1 \mod m$, and
  \item \label{item:22} the subgroup of diagonal matrices in all other cases.
  \end{enumerate}
\end{lemma}

\begin{proof}[Proof of Proposition \ref{prop:triple_determines}]
  First, we assume that there exists $A \in \mathrm{U}(\C^2/\Z_m)$
  such that \eqref{eq:40} holds. Moreover, we start by assuming that
  $l^2 \neq 1 \mod m$. By part \ref{item:22} of Lemma \ref{lemma:normalizer},
  $N_{\mathrm{U}(2)}(\Z_m)$ is the subgroup of diagonal
  matrices. Hence, there exist $\psi, \theta \in \R$ such that
  \begin{equation}
    \label{eq:41}
    [e^{2\pi i \psi} \lambda^{\frac{a_1}{p}}z_1, e^{2\pi i \theta}
    \lambda^{\frac{a_2}{p}}z_2] = [e^{2\pi i \psi} \lambda^{\frac{a'_1}{p'}}z_1, e^{2\pi i \theta}
    \lambda^{\frac{a'_2}{p'}}z_2]
  \end{equation}
  for all $ [z_1,z_2] \in \C^2/\Z_m$ and for all $\lambda \in S^1$. Let $[z_1,z_2]$ be such that $z_1z_2 \neq 0$.  Then by \eqref{eq:41} and the fact that
  $\mathrm{U}(\C^2/\Z_m) = N_{\mathrm{U}(2)}(\Z_m)/\Z_m$ we have that  \begin{equation*}
    \lambda^{\frac{m a_1}{p}} =  \lambda^{\frac{m a'_1}{p'}} \quad
    \text{and} \quad \lambda^{\frac{m a_2}{p}} = \lambda^{\frac{m a'_2}{p'}},
  \end{equation*}
  for all $\lambda \in S^1$. Hence, since $m \neq 0$, we have
  \begin{equation}
    \label{eq:43}
    \frac{a_1}{p} = \frac{a'_1}{p'} \quad \text{and} \quad \frac{a_2}{p} = \frac{a'_2}{p'}.
  \end{equation}
Note that, $a_i = 0$ if and only if $a'_i =0$. Moreover, in
  this case, we have by Proposition \ref{prop_local}, that   $a_j = a'_j = \pm 1$ for $j \neq i$ (note that  $p$ and $p'$ are
  positive) and so $p = p'$. If none of
  $a_1,a_2,a_1', a_2'$ are zero, then, by \eqref{eq:43},  we have $a_1 a_2^\prime = a_1^\prime a_2$ and then, since 
  $$\gcd(a_1,a_2) =\gcd(a_1',a_2') = 1,$$
  (and  $p$ and $p^\prime$ are positive), we have that 
  $a_1 = a_1'$, $a_2 = a_2'$ and  $p=p^\prime$.

  Next suppose that $l^2 = 1 \mod m$ and $l \neq 1$. In this case, by part
  \ref{item:21} of Lemma
  \ref{lemma:normalizer}, we have that $N_{\mathrm{U}(2)}(\Z_m)$ is the subgroup
  generated by diagonal
  matrices and $
    \left(\begin{smallmatrix}
        0 & 1 \\
        1 & 0
      \end{smallmatrix} \right)$. Hence, there exist $\psi, \theta
  \in \R$ such that either \eqref{eq:41} or
  $$  [e^{2\pi i \psi} \lambda^{\frac{a_2}{p}}z_2, e^{2\pi i \theta}
  \lambda^{\frac{a_1}{p}}z_1] = [e^{2\pi i \psi} \lambda^{\frac{a'_1}{p'}}z_2, e^{2\pi i \theta}
  \lambda^{\frac{a'_2}{p'}}z_1] $$
  holds for all $ [z_1,z_2] \in \C^2/\Z_m$ and for all $\lambda \in
  S^1$. The former case is dealt with above; the latter is entirely
  analogous and yields that $a_1 = a_2'$ and $a_2 = a_1'$.

  Finally, assume that $l = 1$. By part
  \ref{item:20} of Lemma \ref{lemma:normalizer}, the group
  $N_{\mathrm{U}(2)}(\Z_m)$ equals $\mathrm{U}(2)$. Hence, there exist
  $\alpha, \beta \in \C$ with $|\alpha|^2 + |\beta|^2 =1$ and $\theta
  \in \R$ such that
  \begin{equation}
    \label{eq:42}
    \begin{split}
      & [\alpha \lambda^{\frac{a_1}{p}}z_1 + \beta
      \lambda^{\frac{a_2}{p}}z_2, e^{2 \pi i \theta} (\bar{\alpha}
      \lambda^{\frac{a_2}{p}}z_2 - \bar{\beta}
      \lambda^{\frac{a_1}{p}}z_1)] \\
      = & [\lambda^{\frac{a_1'}{p'}}(\alpha z_1 + \beta z_2),
      \lambda^{\frac{a_2'}{p'}}e^{2 \pi i \theta} (\bar{\alpha} z_2 -  \bar{\beta} z_1)]
    \end{split}
  \end{equation}
  for all $ [z_1,z_2] \in \C^2/\Z_m$ and for all $\lambda \in
  S^1$. By considering the cases $z_1= 0$ and $z_2 \neq 0$, and
  $z_1\neq 0$ and $z_2 = 0$ and using the fact that $|\alpha|^2 +
  |\beta|^2 =1$, the argument is the same as that of the case $l^2 = 1
  \mod m$ and $l \neq 1$. 
\end{proof}

To conclude this section we prove a result that describes the points with non-trivial
stabilizer of an effective unitary linear $S^1$-action on the model $\C^2/\Z_m$.

\begin{lemma}\label{lem::isotropy}
  Let $\mathbb{C}^2/\mathbb{Z}_m$ be a model for a cyclic isolated
  singular point of type $\frac{1}{m}(1,l)$. Let $\rho : S^1 \to
  \mathrm{U}(\mathbb{C}^2/\mathbb{Z}_m)$ be a faithful representation
  and let $p, a_1,a_2$ be the integers associated to $\rho$ as in
  Lemma \ref{lem::orbiweights}. Any point $[z_1,z_2] \in
  \C^2/\Z_m$ such that $z_1z_2 \neq 0$ has
  trivial stabilizer for the induced $S^1$-action. Moreover, the stabilizer of a point $[z_1,0]$
  (respectively $[0,z_2]$) with $z_1 \neq 0$ (respectively $z_2 \neq
  0$) is $\Z_{|ka_1|}$ (respectively $\Z_{|ka_2|}$), where $k = m/p$
  and we set $\Z_0 \equiv S^1$. 
% Consider a neighborhood of a singular point of type $\frac{1}{m}(1,l)$ which is an isolated fixed point of a Hamiltonian $S^1$-space, locally modeled by $\mathbb{C}^2/\mathbb{Z}_m$ as in Theorem~\ref{cor::localformpt},
% with orbi-weights $a_1/p,a_2/p$. Set $k:=m/p$.  Then, in these coordinates,
%  the points of the form $[z_1,0]$ have stabilizer
% $\mathbb{Z}_{ka_1}$, while the points of the form $[0,z_2]$ have
% stabilizer $\mathbb{Z}_{ka_2}$. All the other points in this neighborhood  have trivial stabilizer.
\end{lemma}
 
\begin{proof}
  Suppose first that $[z_1,z_2] \in \C^2/\Z_m$ is such that $z_1z_2
  \neq 0$ and that $\lambda \in S^1$ fixes it. By \eqref{eq:typesing}
  and \eqref{eq:1}, there exists $\mu \in \Z_m$ such that 
  \begin{equation*}
    \lambda^{\frac{a_1}{p}} = \mu \text{ and } \lambda^{\frac{a_2}{p}}
    = \mu^{\ell}.
  \end{equation*}
  In particular, such a $\lambda$ fixes every point $[z_1,z_2] \in
  \C^2/\Z_m$ with $z_1z_2
  \neq 0$. By continuity, $\lambda$ fixes every point in
  $\C^2/\Z_m$. Since $\rho$ is faithful, $\lambda =1$.

  Next we determine the stabilizer of a point of the form $[z_1,0]$
  with $z_1 \neq 0$. If $a_1 = 0$, then $[z_1,0]$ is stabilized by
  $S^1$. So we may assume that $a_1 \neq 0$. If $\lambda \in S^1$
  fixes $[z_1,0]$, then  $\lambda^{\frac{a_1}{p}}\in \mathbb{Z}_m$ and so
  $\lambda^{\frac{ m a_1}{p}} =1$ and  $\lambda^{k a_1} =1$. Thus
  the stabilizer of $[z_1,0]$ is contained in $\Z_{|ka_1|}$. On the
  other hand, it is immediate to see that $e^{
    \frac{2 \pi i }{ka_1}}$ fixes $[z_1,0]$. Therefore, the stabilizer
  of $[z_1,0]$ equals $\Z_{|ka_1|}$. Similarly, the points of the form
  $[0,z_2]$ with $z_2 \neq 0$ have stabilizer $\mathbb{Z}_{|k a_2|}$.
\end{proof}

\begin{remark}\label{rmk:fixed_orbisurfaces_iso_spheres}
  By Lemma \ref{lem::isotropy}, a linear Hamiltonian $S^1$-action on a model for a cyclic isolated
  singular point of type $\frac{1}{m}(1,l)$ may enjoy properties that
  cannot occur for linear linear Hamiltonian $S^1$-actions on four
  dimensional symplectic vector spaces. We fix the notation in Lemma
  \ref{lem::isotropy} and suppose that $k > 1$. If both $a_1, a_2 \neq
  0$, then the greatest common divisor of the two (finite) stabilizers
  $\Z_{|ka_1|}$ and $\Z_{|ka_2|}$ equals $k$, i.e., they are not coprime. On the other
  hand, if one of $a_1$ or $a_2$ is equal to zero, then there is a one dimensional orbi-vector subspace of $\C^2/\Z_m$ that
  is pointwise fixed by the action and a complementary one that
  consists of points with finite non-trivial stabilizer. 
\end{remark}
\color{black}

% \begin{remark}
%   An alternative way of finding the orbi-weights and the conditions for
%   effectiveness is by starting with any effective action of the
%   identity component $\widehat{G}_0\simeq S^1$ of the extended group
%   on the uniformizing chart. Such an action may not induce an
%   effective action on the quotient and so one has to
%   take the quotient by all elements that act trivially on the
%   orbi-space.\hfill$\diamond$ 
% \end{remark}

\section{The labeled multigraph of a Hamiltonian $S^1$-space}\label{sec:labelled-multigraph}
By an abuse of notation, henceforth whenever we say that
$(M,\omega,H)$ is a {\bf Hamiltonian $S^1$-space}, we assume that
\begin{itemize}[leftmargin=*]
\item $M$ is compact,
\item the singular set of $M$ is discrete, and
\item every singular point of $M$ has a cyclic orbifold structure group.
\end{itemize}

\subsection{Isotropy orbi-spheres}\label{sec:isotr-orbi-spher}
Let $(M,\omega,H)$ be a Hamiltonian $S^1$-space. For any subgroup $K$
of $S^1$, we consider two subsets of $M$, namely $\mathrm{Stab}(M;K)
\subseteq M$
and $M^K \subseteq M$, that are defined as follows
\begin{equation}
  \label{eq:61}
  \begin{split}
    \mathrm{Stab}(M;K) &:=\{ x \in M \mid \mathrm{Stab}(x) = K\}, \\
    M^K & := \{ x \in M \mid h \cdot x = x \text{ for all } h \in K
    \},
  \end{split}
\end{equation}
where $\mathrm{Stab}(x)$ denotes the stabilizer of $x$. We endow both
subsets with the subspace topology. We observe that the closure of
$\mathrm{Stab}(M;K)$ is $M^K$ and that these two sets coincide if $H = S^1$. By Corollary
\ref{cor:fixed_point_set}, a connected component of $M^{S^1}$ is
either an isolated fixed point $x$ or a full symplectic submanifold $\Sigma$ of dimension
two. In the former case, a neighborhood of $x$ is equivariantly
determined by Theorem \ref{cor::localformpt}. In the latter case, since $M$ is compact, $\Sigma$ is too. Hence,
by Theorem \ref{thm:classification_symplectic_orbi-surfaces}, it is
classified by its diffeomorphism type (see Theorem
\ref{thm::classorbisurface}), and its symplectic area.

% be the subset that is stabilized by
% $H$, i.e.
% $$ M^H := \{ x \in M \mid h \cdot x = x \text{ for all } h \in H
% \}. $$
% We endow $M^H$ with the subspace topology. 

\begin{definition}\label{defn:fixed_orbi-surface}
  Let $(M,\omega,H)$ be a Hamiltonian $S^1$-space. A connected
  component of $M^{S^1}$ of dimension two is called a {\bf fixed orbi-surface}.
\end{definition}

\begin{remark}\label{rmk:type_fixed}
  Let $\Sigma$ be a fixed orbi-surface of a Hamiltonian
  $S^1$-space. By Corollary \ref{cor:fixed_point_set}, the
  orbi-surface $\Sigma$ is a local extremum of $H$. Hence, by
  Corollary \ref{lem::uniqueMaxMin}, it is a global extremum of
  $H$. Moreover, by
  Theorem \ref{cor::localformpt}, if $x \in \Sigma$ is a singular
  point, then the integer $p$ of Lemma \ref{lem::orbiweights} equals
  one and the
  orbi-weights $a_1,a_2$ at $x$ satisfy $\{|a_1|,|a_2|\} =
  \{0,1\}$. In particular, we can choose a preferred representative $\frac{1}{m}(1,l)$
  for the type of $x$ (see Remark \ref{rmk:type}).
\end{remark}

Next we turn to understanding a connected component of
$\mathrm{Stab}(M;K)$ for some fixed
proper subgroup $K \simeq \Z_k$ of $S^1$, where $k \geq 2$. To this end, let $x \in M$
be a point with stabilizer equal to $\Z_k$. By Corollary
\ref{cor:action_preserves_structure_group}, $x$ is regular. Therefore,
the `standard' local model for the $S^1$-action near $x$ holds (see
\cite[Lemma A.14]{karshon} and, more generally, \cite{gs,marle}). To
recall this model, we let $\rho : \Z_k \to \mathrm{U}(1)$ be the
symplectic slice representation at $x$ and consider
$$Y:=S^1\times_{\mathbb{Z}_k}\R \times \C = (S^1\times\R \times \C)/\Z_k,$$
where the $Z_k$-action on $S^1\times\R \times \C$ is given by
$$ \mu \cdot (e^{2 \pi i\theta},h,z) = (\mu \cdot e^{2 \pi i\theta},
h, \rho(\mu^{-1})z), $$
for any $\mu \in \Z_k$ and any $(e^{2 \pi i\theta},h,z) \in
S^1\times\R \times \C$. There exists a symplectic form $\omega_Y$ on
$Y$ such that the $S^1$-action given by
$$ \lambda \cdot [e^{2 \pi i\theta},h,z] = [\lambda \cdot e^{2 \pi
  i\theta},h,z] $$
is Hamiltonian with homogeneous moment map $\Phi_Y : Y \to \R$ given
by
$$ \Phi_Y([e^{2 \pi i\theta},h,z]) = z. $$
The following
result is a special case of the Marle-Guillemin-Sternberg local normal
form (see \cite{gs,marle}).

\begin{lemma}~\label{lemma:nbhd}
  Let $(M,\omega,H)$ be a Hamiltonian $S^1$-space and let $x \in M$
  be a point with stabilizer equal to $\Z_k$. There exist an
  $S^1$-invariant $U \subseteq M$ open neighborhood of $x$ and an open
  neighborhood $V \subseteq Y$ of $[1,0,0]$ such that
  $(U,\omega|_U,\Phi|_U)$ and $(V, \omega_Y|_V, \Phi_Y|_V)$ are
  isomorphic in the sense of Definition \ref{defn:hamiltonian}.
\end{lemma}

% \begin{remark}
% Note that  $\left (Lie(S^1)/Lie(S^1_x)\right)^*\simeq \mathbb{R}$ and that the symplectic slice for the action is the symplectic vector space 
% $$
% V= T_x \mathcal{O}^\omega / T_x \mathcal{O} \simeq \mathbb{C}
% $$
% (a fiber of the symplectic normal bundle of $\mathcal{O}$ in $M$).
% \end{remark}
%If  $x\in M$ is  a regular point with stabilizer $S^1_x\simeq\mathbb{Z}_k$ for some $k\geq 2$, then by the Symplectic Slice Theorem, there exists a neighborhood of the orbit $S^1\cdot x$ that is equivariantly symplectomorphic to 
%$$S^1\times_{\mathbb{Z}_k}(I\times D^2),$$ where $I\subset\mathbb{R}$ is an interval about $0$, and $D^2\subset\mathbb{C}^2$ is a disc about $0$  (see Theorem~\ref{sympslicce}). Here we consider $\mathbb{R}\simeq(Lie(S^1)/Lie(S^1_x))$,
%$$\mathbb{C}^2\simeq \widehat{T}_x(S^1\cdot x)^\omega/\left(\widehat{T}_x(S^1\cdot x)\cap \widehat{T}_x(S^1\cdot x)^\omega\right).$$ and an action of $\mathbb{Z}_k$ in  $S^1\times (I\times D^2)$ given by
%$$\nu\cdot  (e^{i\theta},h,z)=(e^{i\theta}\xi_k^{-1},h,\xi_k^lz)$$ 
%for a generator $\xi_k$ of  $\mathbb{Z}_k$, for some integer $l$ such that $1\leq l< k$. Moreover,  the $S^1$-action is given, in these coordinates, by
%$$\lambda\cdot[e^{i\theta},h,z]_k=[\lambda e^{i\theta},h,z]_k.$$
%and the symplectic form by 
%$$\omega=dh\wedge d\theta+\frac{i}{2}dz\wedge d\bar{z},$$ 
%which yields a moment map $\Phi=h$. 

Theorem \ref{cor::localformpt}, its analog for manifolds (see \cite[Corollary A.7]{karshon}) and Lemma \ref{lemma:nbhd} are the
fundamental tools to prove the following result.

\begin{proposition}\label{prop:iso_sphere}
  Let $(M,\omega,H)$ be a Hamiltonian $S^1$-space and let $k \geq 2$
  be an integer. Let $\Sigma \subseteq \mathrm{Stab}(M;\Z_k)$ be a non-empty
  connected component. The closure $\overline{\Sigma}$ of $\Sigma$ is a
  compact, full symplectic submanifold of $(M,\omega)$ of dimension
  2. Moreover, if the stabilizer of $y \in \overline{\Sigma}$ contains $\Z_k$ as a
  proper subgroup, then $y \in M^{S^1}$. Finally, $H|_{\Sigma} : \overline{\Sigma} \to \R$ is the moment map of the
  effective Hamiltonian $S^1/\Z_k$-action on $\overline{\Sigma}$ induced by the
  $S^1$-action on $M$.
\end{proposition}

\begin{proof}
  Since $M$ is compact, so is $\overline{\Sigma}$. Let $x \in \Sigma$. By the Principal Orbit Theorem applied to $\Sigma$ (see \cite[Theorem 2.8.5]{dk}), there
  exists an open, dense subset of $\overline{\Sigma}$ consisting of points with
  stabilizer equal to $\Z_k$. Hence, the quotient $S^1/\Z_k$ acts
  effectively on $\overline{\Sigma}$. Moreover, by Theorem
  \ref{cor::localformpt}, its analog for manifolds (see
  \cite[Corollary A.7]{karshon}) and Lemma \ref{lemma:nbhd}, $\overline{\Sigma}$
  is a full symplectic submanifold of $(M,\omega)$ of dimension
  2, and if the stabilizer of $y \in \overline{\Sigma}$ contains $\Z_k$ as a
  proper subgroup, then $y \in M^{S^1}$. In fact, $\overline{\Sigma}$ is a fully embedded symplectic
  submanifold of $(M,\omega)$, i.e., the inclusion $\overline{\Sigma}
  \hookrightarrow M$ is a good $C^{\infty}$ map. To see this, we may argue as in Remark
  \ref{rmk:fixed_fully_embedded}, using the fact that the singular
  points of $M$ are isolated. Hence, the restriction of $H$ to
  $\overline{\Sigma}$ is a good $C^{\infty}$ map. 
\end{proof}

In other words, if $\Sigma \subseteq \mathrm{Stab}(M;\Z_k)$ is a
non-empty connected component then $(\overline{\Sigma}, \omega|_{\overline{\Sigma}},
H|_{\overline{\Sigma}})$ is a compact symplectic toric orbifold of dimension
two.

\begin{corollary}\label{cor:Z_k-orbi-sphere}
  Let $(M,\omega,H)$ be a Hamiltonian $S^1$-space and let $k \geq 2$
  be an integer. Let $\Sigma \subseteq \mathrm{Stab}(M;\Z_k)$ be a non-empty
  connected component and let $\overline{\Sigma} \subseteq M^{\Z_k}$
  denote its closure. Then
  \begin{itemize}[leftmargin=*]
  \item there are precisely two points $x_{\min}, x_{\max} \in \overline{\Sigma}$
    that are fixed points for the $(S^1/\Z_k)$-action on $\overline{\Sigma}$ and,
    hence, lie in $M^{S^1}$,
  \item the following equation holds
    \begin{equation}
      \label{eq:44}
      H(x_{\max}) - H(x_{\min}) = \frac{k}{2 \pi} \int\limits_{\overline{\Sigma}}
      \omega|_{\overline{\Sigma}}, \text{ and}
    \end{equation}
  \item letting $\Z_{n_{\min}}$ and $\Z_{n_{\max}}$ denote the orbifold
    structure groups of $x_{\min}$ and $x_{\max}$ respectively and
    setting $\gcd(n_{\min},n_{\max}) = m$, then $\overline{\Sigma}$ is
    diffeomorphic to $\C P^1(n'_{\min}, n'_{\max})/\Z_m$, where
    $n_{\min}=m n'_{\min}$ and $n_{\max} = m n'_{\max}$. 
\end{itemize}
\end{corollary}

\begin{proof}
  By the classification of Lerman and Tolman \cite[Theorem
  1.5]{LermanTolman}, since $\dim \overline{\Sigma} =2$ and since we are
  considering an effective action of $(S^1/\Z_k)$, setting $H(\overline{\Sigma})
  = [h_{\min},h_{\max}]$, we have that
  \begin{itemize}[leftmargin=*]
  \item there are precisely two fixed points $x_{\min}, x_{\max} \in
    \overline{\Sigma}$ for the $(S^1/\Z_k)$-action on $\overline{\Sigma}$, where
    $H(x_{\min}) = h_{\min}$ and $H(x_{\max}) = h_{\max}$,
  \item equation \eqref{eq:44} holds, 
  \item the set of singular points of $\overline{\Sigma}$ is contained in
    $\{x_{\min}, x_{\max} \}$, and
  \item the labels attached to $h_{\min}$ and $h_{\max}$ are precisely
    the orders of the orbifold structure groups of $x_{\min}$ and
    $x_{\max}$ respectively. 
  \end{itemize}
  The first and second statements follow immediately. To prove the last, we
  observe that, by
  \cite[Example 1.1 and Theorem
  1.5]{LermanTolman}, the genus of the underlying topological surface
  $|\Sigma|$ is zero. By Example \ref{exm:quotient_wpl}, $\C
  P^1(n'_{\min}, n'_{\max})/\Z_m$ as in the statement has genus zero
  and at most two
  singular points that have orbifold structure groups given by
  $\Z_{n_{\min}}$ and $\Z_{n_{\max}}$. Hence, the result follows immediately by
  Theorem \ref{thm::classorbisurface}.
\end{proof}

Motivated by Corollary \ref{cor:Z_k-orbi-sphere} and \cite{karshon}, we introduce the following terminology.

\begin{definition}\label{defn:Z_k-sphere}
  Let $(M,\omega,H)$ be a Hamiltonian $S^1$ and let $ k \geq 2$ be an
  integer. The closure $\overline{\Sigma} \subseteq M^{\Z_k}$ of a non-empty connected component of $\Sigma \subseteq \mathrm{Stab}(M;\Z_k)$ 
  is called a
  \textbf{$\mathbb{Z}_k$-isotropy orbi-sphere}. We refer to the two
  fixed points in $\overline{\Sigma} \cap M^{S^1}$ as the {\bf north} and {\bf
    south} poles of $\overline{\Sigma}$, depending on the value of the moment map. Finally, whenever $\Z_k$ is not relevant, we drop it from the terminology. 
\end{definition}

\begin{remark}
  By Lemma \ref{lem::isotropy} and Corollary
  \ref{cor:Z_k-orbi-sphere}, a fixed point with
  orbifold structure group given by $\Z_m$ has an orbi-weight of $-
  \frac{k}{m}$ (respectively $\frac{k}{m}$) if and only if it is the
  north (respectively south) pole of a $\Z_k$-isotropy orbi-sphere.
\end{remark}

\begin{remark}
  Unlike the case of smooth manifolds (see \cite{karshon}), a
  $\Z_k$-isotropy orbi-sphere need not be a connected component of
  $M^{\Z_k}$. The reason is that a connected component of $M^{\Z_k}$
  in a linear action on an orbi-vector space need not be a
  suborbifold (cf. Remark \ref{rmk:fixed_orbisurfaces_iso_spheres}). 
\end{remark}

\begin{example}\label{ex:projective25}(Weighted projective planes, continued)
We fix notation as in Example~\ref{ex:projective}. Let
$M=\mathbb{C}P^2(p,s,q)$ be a weighted projective plane endowed with
the standard symplectic form. We consider the Hamiltonian $S^1$-action
of \eqref{eq:exactionwps}. Suppose first that $k_0,k_1,k_2\neq 0$. All
the points in the orbi-sphere  $\{[0:z_1:z_2]\in M\}$ through the
fixed points $F_2$ and $F_3$ are fixed by the subgroup
$\mathbb{Z}_{k_1}$ of $S^1$. Hence, this is a $\Z_{k_1}$-isotropy
orbi-sphere. Similarly, the orbi-sphere $\{[z_0:z_1:0]\in M\}$
(respectively $\{[z_0:0:z_2]\in M\}$) is a $\Z_{k_2}$-isotropy
orbi-sphere (respectively $\Z_{k_3}$). To consider the remaining cases, without loss of generality we may assume that $k_1=0$. Then the orbi-sphere
$\{[z_0:z_1:0]\in M\}$ is a $\Z_s$-isotropy orbi-sphere and
$\{[z_0:0:z_2]\in M\}$ is a $\Z_q$-isotropy orbi-sphere. \hfill$\Diamond$ 
\end{example}

\begin{example}\label{tswIS} (Example~\ref{tswEx_1}, continued) 
  The Hamiltonian $S^1$-space constructed in Examples
  \ref{tswEx_1}, \ref{exm:circle_action_Talvacchia}, \ref{exm:tsw_symp}
  and \ref{exm:tsw_ham} has exactly two $\mathbb{Z}_2$-isotropy
  orbi-spheres that share the same north and south poles, which are the
  two fixed points of type $\frac{1}{4}(1,3)$  -- see Example
  \ref{tswEx_types}. In fact, the union of these two isotropy
  orbi-spheres and the three fixed points of Example \ref{tswEx_types} are the only points in $M$ with a
  non-trivial stabilizer.  
  \hfill$\Diamond$
\end{example}

To conclude this section, we state the following result that is an immediate  consequence of Theorem \ref{cor::localformpt} and Lemma \ref{lem::isotropy}. 
\begin{corollary}\label{cor:north-south}
	Let $(M,\omega,H)$ be a Hamiltonian $S^1$-space and let $x \in M^{S^1}$. Then
	\begin{itemize}[leftmargin=*]
		\item $x$ belongs to at most two isotropy orbi-spheres;
		\item if $H(x)$ is not extremal and $x$ lies on two isotropy orbi-spheres, then $x$ is the south pole of one and the north pole of the other; and
		\item if $x$ lies on a fixed orbi-surface and on an isotropy orbi-sphere, then it is a singular point of $M$ and lies on exactly one isotropy orbi-sphere.
	\end{itemize}
\end{corollary}

\subsection{The labeled multigraph}\label{sec:label-mult-hamilt}
We associate a {\bf labeled multigraph} to a
Hamiltonian $S^1$-space $(M,\omega,H)$ as follows (see
Table~\ref{table:multigraph}). \\

%\vspace{.2 cm}
\noindent{\bf Vertex set:}

\begin{itemize}[leftmargin=*]
\item Each connected component of $M^{S^1}$ corresponds to a vertex. A
  vertex corresponding to a fixed orbi-surface is called {\bf
    fat}\footnote{Correspondingly, we use a
  thicker symbol when drawing a fat vertex.}. Every vertex is labeled with the moment map value
  of the corresponding connected component of $M^{S^1}$.
\item If a vertex $v$ corresponds to a singular point $x \in M$, we
  label $v$ with the type of the singular point (see Remark
  \ref{rmk:type}). More explicitly, if the
  orbi-weights at $x$ are distinct, then we label $v$ with
  $\frac{1}{m}(1,l)$, where $1$ is in the direction of the larger
  orbi-weight. If the orbi-weights are equal and $l^2 = 1 \mod m$,
  then again we label $v$ with
  $\frac{1}{m}(1,l)$. Finally, in the remaining case, we label $v$
  with {\em both} $\frac{1}{m}(1,l)$ and $\frac{1}{m}(1,l')$.
\item If a vertex $v$ corresponds to a fixed orbi-surface $\Sigma$, we
  label $v$ with the following data:
  \begin{itemize}
  \item the genus $g$ of $|\Sigma|$,
  \item the type of each singular point in $\Sigma$ (see
    Remark \ref{rmk:type_fixed}), and
  \item the normalized symplectic area $a(\Sigma):= \frac{1}{2\pi} \int_\Sigma \omega|_{\Sigma}$ of $\Sigma$.
  \end{itemize}
\end{itemize}

\noindent{\bf Edge set:} 
\begin{itemize}[leftmargin=*]
\item For any $k \geq 2$, each $\mathbb{Z}_k$-isotropy orbi-sphere corresponds to an edge
  that joins the vertices corresponding to the north and south poles
  (see Definition \ref{defn:Z_k-sphere}). We label the edge with the
  integer $k$.
\end{itemize}

\begin{table}[h!]\label{table:multigraph}
  \begin{tabular}{ l|c| c  }
    & \textbf{Symbol} &\textbf{Labels} \\
    \hline
    &&\\
    \textbf{Fixed point} &
                           \begin{tikzpicture}
                             \coordinate (P1) at (-1,3);
                             \fill (P1) circle (4pt);
                           \end{tikzpicture}
                      & type of singular point,\\
    && moment map value \\
    &&\\
    \textbf{Fixed surface} & 
                             \begin{tikzpicture}
                               \fill (0,0) circle (8pt);
                             \end{tikzpicture}
                      &  type of each singular point, \\&&  genus,
                                                           area,
                                                            moment map value \\
    &&\\
    \textbf{Isotropy $\Z_k$-orbi-sphere}& 
                                   \begin{tikzpicture}[shorten >=1pt,node distance=3cm,auto]%on grid,
                                     \tikzstyle{state}=[shape=circle,thick,draw,minimum size=0.5cm]
                                     \draw[very thick] (0,0) -- (0,1);
                                   \end{tikzpicture}
                      & $k$ \\
  \end{tabular}
  \vspace{0.2cm}
  \caption{The labeled multigraph at a glance.}
\end{table}

% The multigraph and its integer and rational labels depend only on the orbifold and on the circle action, while the real labels are determined by the cohomology class of the symplectic form, up to a simultaneous shift of all the values by the same amount. The proof of this result is similar to the manifold case (see \cite[Lemma 2.5]{karshon}).
 
% \begin{lemma}\label{lemma:cc} Let $(M,\omega,H)$ be a Hamiltonian $S^1$-space. Then the cohomology class of $\omega$ determines the values of the moment map at the fixed points up to a simultaneous shift and determines the areas of the fixed orbi-surfaces. 
% \end{lemma}
 
% \begin{proof}
% Let $x$ and $y$ be fixed points such that $H(x)>H(y)$ and let $\gamma:[0,1] \to M$ be a path from $x$ to $y$. Let $f:[0,2\pi]\times [0,1] \to M$ be the map
% $$
% f(s,t)= e^{i s} \cdot \gamma(t).
% $$
% Then
% \begin{equation}\label{eq:int}
% \int_{[0,2\pi]\times [0,1] } f^* \omega = 2\pi \int_0^1 \gamma^* (\iota(X_M) \omega) = 2\pi (H(y)-H(x)).
% \end{equation}
% Since $f$ defines a cycle in homology, the integral in \eqref{eq:int} only depends on the cohomology class of $\omega$. Hence, the first part of the result follows. The second part is a trivial consequence of the fact that  fixed orbi-surfaces are closed suborbifolds of $M$.
% \end{proof}

For simplicity, we often omit the moment map labels in figures depicting a labeled multigraph. Moreover, whenever it is
not relevant for the specific example or proof, we label a vertex that corresponds to a singular point
with only one expression of the form $\frac{1}{m}(1,l)$ even if the
orbi-weights are equal and $l^2 \neq 1 \mod m$. We also stress that
the relative position of non-extremal vertices does not necessarily
reflect the corresponding moment map values. However, the relative
position of the north/south pole of a given isotropy orbi-sphere always
corresponds to the top/bottom vertex of the corresponding edge. Finally, by a slight
abuse of notation, when showing the
labeled multigraph of a general family of Hamiltonian $S^1$-spaces (as
in Example \ref{ex:projective2} below), we include edges that may
correspond to `free' $S^1$-invariant orbi-spheres (in Figure \ref{Figure:mgwps} below, any
one of $k_0,k_1,k_2$ may be equal to one).

\begin{example}\label{ex:projective2}(Weighted projective planes, continued)
  We determine the labeled multigraph of the Hamiltonian $S^1$-space
  $(M=\mathbb{C}P^2(p,s,q),\omega,H)$ of
  Example~\ref{ex:projective} under the condition that  the integers $k_0,k_1,k_2$ are all positive (see Figure
  \ref{Figure:mgwps} below). To this end, we fix the notation in
  Example~\ref{ex:projective}. We remark that
  \begin{itemize}[leftmargin=*]
  \item the real numbers $\alpha$ and $\beta$ in the moment map labels are
    positive, and
  \item the integers $\ell_p$, $\ell_q$ and $\ell_s$ in the types of
    singular points can be calculated explicitly as in Remark \ref{ex::weightedprojSeifertInvariants}.
  \end{itemize}
  Moreover, the following hold
  \begin{equation} \label{eq:weights}
    H(F_1)>H(F_2)>H(F_3) \quad \text{and} \quad s k_0 = q k_2+ p k_1.
  \end{equation}
  The moment map labels can be easily determined using the
  localization formula for equivariant cohomology
  (see Theorem~\ref{prop:localization} below). To this end, using the equivariant
  symplectic form $\omega^{S^1}=\omega - H y$, we have that
  \begin{equation}
    \label{eq:45}
    0=\int\limits_M\omega^{S^1} =-\frac{1}{y}\left( \frac{ q\, H_{\text{min}}}{k_0k_1}+   \frac{p\, H_{\text{max}}}{k_0k_2} - \frac{s\, H(x_s)}{k_1k_2} \right),
  \end{equation}
  where $H_{\min} = H(F_3)$ and $H_{\max} = H(F_1)$. Set $\alpha :=
  H(F_2)$ and let $\beta$ be the normalized symplectic area of the
  $\mathbb{Z}_{k_0}$-orbi-sphere. By \eqref{eq:weights}, \eqref{eq:45}
  and Corollary \ref{cor:Z_k-orbi-sphere}, we obtain the moment map labels in
  Figure~\ref{Figure:mgwps}. 

  \renewcommand{\thefigure}{\thesection.\arabic{figure}}
  \begin{figure}[h!]
    \centering
    \begin{tikzpicture}	
      \coordinate (P1) at (-1,3);
      \fill (P1) circle (4pt);
      \coordinate (P2) at (-1,0);
      \fill (P2) circle (4pt);		
      \coordinate (P3) at (0,1.5);
      \fill (P3) circle (4pt);		
      \draw[very thick](P1) -- (P2);
      \draw[very thick](P2) -- (P3);
      \draw[very thick](P1) -- (P3);
      \coordinate[label=left: ${k_0:=aq-cp}$] (AB) at ($(P1)!0.5!(P2)$);
      \coordinate[label=right: ${\,k_2:= as-bp}$] (BC) at ($(P1)!0.4!(P3)$);	
      \coordinate[label=right: ${\,k_1:=bq-cs}$] (BC) at ($(P2)!0.4!(P3)$);	
      \coordinate[label=above:${\frac{1}{p}(1,\ell_p), \color{red} \alpha + \frac{q k_2 }{s} \beta }$] (BC) at ($(P1)$);
      \coordinate[label=below:${\frac{1}{q}(1,\ell_q), \color{red} \alpha - \frac{p k_1}{s }  \beta}$] (BC) at ($(P2)$);
      \coordinate[label=right:$\,{\frac{1}{s}(1,\ell_s), \color{red} \alpha} $] (BC) at ($(P3)$);    
    \end{tikzpicture}
    \caption{Labeled multigraph of Example~\ref{ex:projective}  if there is no fixed surface.}\label{Figure:mgwps}
  \end{figure}
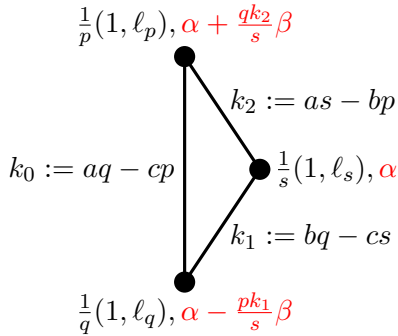
  \hfill$\Diamond$
\end{example}

\begin{example}\label{tswEx} (Example~\ref{tswEx_1}, continued) 
  By bringing together Examples \ref{tswEx_1},
  \ref{exm:circle_action_Talvacchia}, \ref{exm:tsw_symp},
  \ref{exm:tsw_ham} and \ref{tswIS}, we obtain that the Hamiltonian $S^1$-space studied in
  \cite{counterex} has the labeled multigraph in Figure
  \ref{fig::typeBspace}.
  \renewcommand{\thefigure}{\thesection.\arabic{figure}}
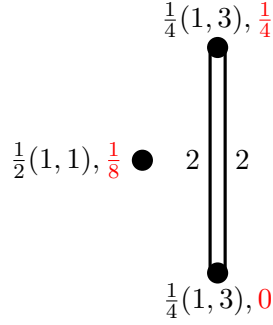
\begin{figure}[h!]
	\centering
	\begin{tikzpicture}	
		\coordinate (Q1) at (7,3);
		\fill (Q1) circle (4pt);
		\coordinate (Q2) at (7,0);
		\fill (Q2) circle (4pt);	
		
		\coordinate (Q0) at (6,1.5);
		\fill (Q0) circle (4pt);	
	
		\coordinate[label=left:$2$] (BC) at ($(6.9,1.5)$);	
		\coordinate[label=right:$2$] (BC) at ($(7.1,1.5)$);
		
		\coordinate[label=below:${\frac{1}{4}(1,3), \color{red}{0}}$] (BC) at ($(Q2)$);
		\coordinate[label=above:${\frac{1}{4}(1,3), \color{red}{\frac{1}{4}}}$] (BC) at ($(Q1)$);
		\coordinate[label=left:${\frac{1}{2}(1,1),\color{red} \frac{1}{8}}\,\,$] (BC) at ($(Q0)$);
		
		\draw[very thick](7.1,3) -- (7.1,0);
		\draw[very thick](6.9,3) -- (6.9,0);
	\end{tikzpicture}
\caption{Labeled multigraph of the Hamiltonian $S^1$-space studied in \cite{counterex}.}
\label{fig::typeBspace}
\end{figure}
In particular, this example illustrates that, unlike what happens in
the case of manifolds (see \cite{karshon}), the class of combinatorial
invariants that we associate to a Hamiltonian $S^1$-space {\em must
  necessarily} include multigraphs. Moreover, two edges that are incident to
the same two vertices can have equal labels (cf. Remark \ref{rmk:fixed_orbisurfaces_iso_spheres})!
\hfill$\Diamond$
\end{example}

As expected, the labeled multigraph of a Hamiltonian $S^1$-space is an
invariant of the isomorphism class. To state this result formally, by
a slight abuse of terminology, we
say that two labeled multigraphs are {\bf equal} if there exists a
bijection between them that preserves all the labels.

\begin{lemma}\label{lemma:trivial_direction}
  Let $(M_i,\omega_i, H_i)$ be a Hamiltonian $S^1$-space for
  $i=1,2$. If $(M_1,\omega_1, H_1)$ and $(M_2,\omega_2, H_2)$ are
  isomorphic, their labeled multigraphs are equal.
\end{lemma}

\begin{proof}
  Let $\varphi : (M_1,\omega_1) \to (M_2,\omega_2)$ be an
  isomorphism. By definition, it is an $S^1$-equivariant
  diffeomorphism. Hence, $\varphi$ establishes a bijection between
  connected components of $\mathrm{Stab}(M_1;K)$ and $\mathrm{Stab}(M_1;K)$ for any subgroup $K$ of
  $S^1$. Moreover, this bijection preserves the property of being a pole of an isotropy
  orbi-sphere. Therefore, there is a bijection
  $\overline{\varphi}$ between the
  multigraphs of $(M_1,\omega_1, H_1)$ and $(M_2,\omega_2,
  H_2)$. Since $H_2 \circ \varphi = H_1$, the moment map labels are
  preserved by $\overline{\varphi}$. Since $\varphi$ is an $S^1$-equivariant
  diffeomorphism, $\overline{\varphi}$ preserves the
  labels of edges. Moreover, by Theorem \ref{cor::localformpt} and
  Proposition \ref{prop:triple_determines}, $\overline{\varphi}$
  preserves the labels of vertices corresponding to singular points
  that are isolated (see also Remark \ref{rmk:type}). Finally, since
  $\varphi$ is a symplectomorphism, it restricts to a
  symplectomorphism between fixed orbi-surfaces. By Theorems
  \ref{thm::classorbisurface} and
  \ref{thm:classification_symplectic_orbi-surfaces},
  $\overline{\varphi}$ preserves the labels of fat vertices. 
\end{proof}

One of the main objectives of the rest of this paper is to establish
the converse to Lemma \ref{lemma:trivial_direction} (see Theorem
\ref{thm:uniqueness}). % In what follows we prove that the data of the
% labeled multigraph allows to determine the Hamiltonian $S^1$-space
% {\em locally} near isolated fixed points. 

\subsection{The labeled multigraph determines the fixed point set}\label{sec:label-mult-determ}
In this section, we show that the labeled multigraph of a Hamiltonian $S^1$-space can be used to `read off' local invariants of each connected component of the fixed point set of the $S^1$-action. First we consider the case of a vertex corresponding to an isolated fixed point.

\begin{lemma}\label{lemma:gamma_thin_vertex}
  Let $(M,\omega,H)$ be a Hamiltonian $S^1$-space and let $v$ be a vertex
  of the labeled multigraph of $(M,\omega,H)$ that corresponds to an isolated fixed point $x \in
  M$. Then the labeled multigraph of $(M,\omega,H)$ determines an open neighborhood of $x$ up to
  isomorphism.  
\end{lemma}

\begin{proof}
	By Corollary \ref{cor:sing_point_dim_4}, Theorem
        \ref{cor::localformpt} and Proposition
        \ref{prop:triple_determines}, it suffices to show that the labeled multigraph of $(M,\omega,H)$ determines the type of $x$ as a singular point of $M$, the moment map value $H(x)$, and the integers $p, a_1,a_2$ that are associated to $x$ as in Lemma \ref{lem::orbiweights}. 
	
	By definition of the labeled multigraph of $(M,\omega,H)$, the
        label of $v$ (of the form $\frac{1}{m}(1,l)$ or
        $\left\{\frac{1}{m}(1,l),\frac{1}{m}(1,l')\right\}$) determines
        the type of $x$. Moreover, the moment map label of $v$ is
        precisely $H(x)$. By Corollary \ref{cor:north-south}, the
        vertex $v$ is incident to at most two edges. We only consider
        the case in which $v$ is incident to exactly two edges,
        leaving the other cases to the reader since they are entirely
        analogous. We suppose that the edges are labeled by the
        positive integers $k_1,k_2$. By Lemma \ref{lem::isotropy}, we
        have that 
	$$p =  \gcd(k_1,k_2) \, , \quad \lvert a_1\rvert = \frac{k_1}{\gcd(k_1,k_2)}\quad \text{ and } \quad \lvert a_2\rvert = \frac{k_2}{\gcd(k_1,k_2)}. $$
	
	It remains to determine the signs of $a_1$ and $a_2$. To this
        end, since $v$ corresponds to an isolated fixed point $x \in
        M$, by Corollary \ref{lem::uniqueMaxMin} and Theorem
        \ref{cor::localformpt}, then either $a_1a_2 < 0$ or $a_1a_2 >
        0$. The former occurs if and only if the moment map label of
        $v$ is not extremal among all moment map labels. In this case,
        by Corollary \ref{cor:north-south}, the fixed point $x$ is the
        south pole of one isotropy orbi-sphere and the north pole of
        another. Hence, in this case, the labeled multigraph of
        $(M,\omega,H)$ can be used to determine the signs of $a_1$ and
        $a_2$. The case $a_1a_2 > 0 $ occurs precisely if the moment
        map label of $v$ is either the maximum or minimum among
        all moment map labels. By Corollary \ref{lem::uniqueMaxMin}
        and Theorem \ref{cor::localformpt}, the former (respectively
        latter) occurs precisely if $a_1$ and $a_2$ are both negative
        (respectively positive). Hence, the signs of $a_1$ and $a_2$
        can be determined from the labeled multigraph of 
        $(M,\omega,H)$ also in this case.
\end{proof}

% \color{blue}
% \begin{remark} From Lemma~\ref{lemma:gamma_thin_vertex} we conclude
%   that a fixed point with orbifold structure group of order $m$ has an
%   orbi-weight $-\frac{k}{m}$ with $k\geq 2$ if and only if it is the
%   north pole of a $\Z_k$-isotropy orbi-sphere (and corresponds to the
%   top vertex of an edge labelled by $k$). Similarly, it has an
%   orbi-weight $\frac{k}{m}$  if and only if it is the south pole of a
%   $\Z_k$-isotropy orbi-sphere (and  corresponds to the bottom vertex
%   of an edge labelled by $k$). 
% \end{remark}
% \color{red}

In order to consider the remaining case of a fat vertex, we discuss
first a localization formula coming from equivariant cohomology in our context.

\subsubsection*{Intermezzo: the ABBV localization formula}\label{sec:loc}
%We recall here some results on equivariant cohomology that can be found, for instance, in \cite{equivariant} and \cite{atiyahBott}.
The material described below is used throughout the paper
and readers are referred to \cite{olocalization} for more details.

Let $(M,\omega, H)$ be a Hamiltonian $S^1$-space. We fix a generator
of the Lie algebra of $S^1$. Hence, we obtain an infinitesimal
generator $\xi^M \in \mathfrak{X}(M)$ for the action, as well as a generator
$y$ of the dual of the Lie algebra of $S^1$.
Consider the space $\Omega^*(M)^{S^1}$ of $S^1$-invariant differential forms on $M$ and let 
$$
\Omega^*(M)^{S^1}[y]= \Omega^*(M)^{S^1} \otimes \mathbb{R}[y]
$$
be the ring of polynomials in $y$ with coefficients in
$\Omega^*(M)^{S^1}$, where $y$ has degree two. The \textit{Cartan
  differential} $d_{\xi^M}$ on $\Omega(M)^{S^1}[y]$ is defined on
homogeneous elements by
$$d_{\xi^M}(\alpha \otimes y^i):=d\alpha\otimes y^i-
(\iota_{\xi^M}\alpha)\otimes y^{i+1}.
$$
The $S^1$-{\em equivariant de Rham complex of} $M$ is 
$\left(\Omega^*(M)^{S^1}[y],d_X\right)$, and its cohomology $H^*_{S^1}(M)$ is the $S^1$-{\em equivariant de Rham cohomology of} $M$. For example, the equivariant symplectic form 
\begin{equation}\label{equivsympform}
\omega^{S^1}:=\omega - H y
\end{equation}
is an $S^1$-equivariant cohomology class of degree two.

Let $\Sigma$ be a fixed orbi-surface and let
$\nu_{\Sigma}$ denote its normal orbi-bundle. By fixing an
$S^1$-invariant almost complex structure that is compatible with
$\omega$, we may see $\nu_{\Sigma}$ as a complex line bundle. Since
$\Sigma$ is fixed, we obtain a fiberwise linear complex $S^1$-action
on $\nu_{\Sigma}$. Moreover, since the singular points of $M$ are
isolated by assumption, the regular locus of $\Sigma$ is
connected. Hence, just as in the case of smooth manifolds, the
isotropy weight $\lambda$ of the fiberwise linear action does not depend on the
regular point and must be equal to $\pm 1$ by effectiveness. Finally, we define the {\bf
  equivariant Euler class} of $\nu_F$ by
\begin{align}\label{eq:46} 
	e^{S^1}(\nu_{\Sigma})=e(\nu_{\Sigma})+y\,\lambda_{\Sigma}.
\end{align}
We remark that, in this case, the equivariant Euler class is equal to
the equivariant Chern class (cf. Remark \ref{rmk:chern_class}).

\begin{example}
  Suppose that $\Sigma$ has $k$ singular points and that the degree of $\nu_{\Sigma}$ is $\beta_\Sigma$, so that
  $c_1(\nu_{\Sigma}) = \beta_{\Sigma}u$ for a generator $u\in
  H^2(\Sigma;\mathbb{Z})$. Let $\iota_{\Sigma} : \Sigma
  \hookrightarrow M$ denote the inclusion. Then 
  \begin{equation}\label{eq:*}
    \iota_{\Sigma}^*\, c_1^{S^1}(TM) =c_1^{S^1}(T\Sigma)+\beta_{\Sigma}\,
    u + \lambda \, y  =  \left( 2-2g-k+ \beta_\Sigma
      +\sum_{i=1}^k\frac{1}{m_i} \right)u + \lambda y, 
  \end{equation} 
  where $c_1^{S^1}(TM)$ is the equivariant first Chern class of $TM$,
  the integer $m_i$ is the order of the orbifold structure group of the singular
  point $x_i$ on $\Sigma$ and $g$ is the genus of $\Sigma$. In
  \eqref{eq:*} we use the fact that, since $\Sigma\subset M^{S^1}$, the circle acts with weight $0$ on $T\Sigma$ and also that $c_1^{S^1}(T\Sigma)=c_1(T\Sigma) = \chi(\Sigma) u$, where 
  \begin{equation}\label{eq:**}
    \chi(\Sigma)  = 2-2g+\sum_{i=1}^k\left(\frac{1}{m_i}-1\right)
  \end{equation}
  is the orbifold Euler characteristic of $\Sigma$ (see \cite{furuta}). \hfill$\Diamond$
\end{example}

The following result is a special case of the ABBV localization
formula for orbifolds (see \cite[Theorem 2.1]{olocalization}).

\begin{theorem}
  \label{prop:localization} Let $(M,\omega, H)$ be a Hamiltonian
  $S^1$-space and let $\alpha$ be an
  $S^1$-equivariant closed form. Then
  \begin{equation}\label{eq:localization}\int_M\alpha=\sum\limits_\Sigma\int_\Sigma\frac{\iota_{\Sigma}^*\alpha}{e^{S^1}(\nu_\Sigma)}
    + \sum\limits_x \frac{\alpha(x)}{m_x\lambda_{1,x}\lambda_{2,x}},\end{equation}
  where the first sum is over the fixed orbi-surfaces, the second
  is over the isolated fixed points, and, for each isolated fixed
  point $x$, the integer $m_x$ is the order of its
  orbifold structure group and $\lambda_{1,x},\lambda_{2,x}$
  are its orbi-weights.
\end{theorem}

\subsubsection*{Back to the case of a fat vertex}
Let $\Sigma$
denote the corresponding fixed orbi-surface. By Corollary
\ref{cor:fixed_point_set}, $\Sigma$ is a full symplectic suborbifold
of $(M,\omega)$. Hence, by Corollary \ref{cor:complex_bundle}, the
normal bundle to $\Sigma$ can be endowed with the structure of a
complex vector orbi-bundle. Since $\dim \Sigma = 2$, the normal bundle
to $\Sigma$ can be viewed as a complex line orbi-bundle. 

\begin{proposition}\label{prop:graph_seifert}
  Let $(M,\omega,H)$ be a Hamiltonian $S^1$-space and let $v $ be a
  fat vertex of the labeled multigraph of
  $(M,\omega,H)$ that corresponds to a fixed orbi-surface $\Sigma$. Then the labeled multigraph of
  $(M,\omega,H)$ determines the Seifert invariant of the
  principal $S^1$-orbi-bundle associated to the normal orbi-bundle to
  $\Sigma$. 
\end{proposition}

\begin{proof}
  Let $n$ be the number of singular points that lie on $\Sigma$. Then
  the Seifert invariant that we are interested in is of the form
  $$ (g;\beta_0, (\alpha_1,\beta_1),\dots,(\alpha_n,\beta_n)), $$
  \noindent
  where $g$ is the genus of $|\Sigma|$, $(\alpha_i,\beta_i)$ is the
  Seifert invariant of the singular point $x_i \in \Sigma$, and
  $\beta_0 \in \Z$ is uniquely determined by equation
  \eqref{eq:13}.

  Since there is a label for each singular point on $\Sigma$, then $n$
  can be recovered from the labeled multigraph of
  $(M,\omega,H)$. Moreover, the genus label of $v$ is precisely
  $g$. Next we show that the Seifert invariant of $x_i$ is determined
  by the labeled multigraph of
  $(M,\omega,H)$ for each $i=1,\ldots, n$. We fix $i=1,\ldots,
  n$ and let $\frac{1}{m}(1,l)$ be the type label that corresponds to
  $x_i$. Let $1\leq
  l^\prime<m$ be such that $l\cdot l^\prime=1 \mod m$. Then the
  Seifert invariant of $x_i$ is either $(m,l^\prime)$, if $H(\Sigma)$ is
  the minimum of $H$, or $(m,l)$, if $H(\Sigma)$ is
  the maximum of $H$. (By Remark \ref{rmk:type_fixed}, these are the
  only cases to consider.)

  It remains to show that $\beta_0$ is determined by the labeled multigraph of
  $(M,\omega,H)$. By Proposition \ref{def::EulerOrbibundle}, and
  Remarks \ref{rmk:chern_class} and \ref{rmk:seifert-orbifolds}, it
  suffices to show that the degree of the normal orbi-bundle to $\Sigma$ is determined by the labeled multigraph of
  $(M,\omega,H)$. This can be checked directly using the localization formula for orbifolds (see
  Theorem~\ref{prop:localization}), as follows. First we assume that there are two extremal fixed orbi-surfaces
  $\Sigma_\pm$, labeled so that $H(\Sigma_-)$ (respectively
  $H(\Sigma_+)$) is the minimum $H_{\min}$ (respectively maximum
  $H_{\max}$) of $H$. By applying Theorem \ref{prop:localization} to the equivariant 
  symplectic form $\omega^{S^1}:= \omega - H y$, we obtain
  \begin{equation}
    \label{eq:47}
    \begin{split}
      0 & =  \int_{\Sigma_{-}}\frac{\omega-H_{\min}y}{\beta_{-}u_-+y}+\int_{\Sigma_{+}}\frac{\omega-H_{\max}y}{\beta_{+}u_+-y}-\sum_{i=1}^n\frac{1}{m_{i}}\frac{H(x_i)y}{\lambda_{i,1}\lambda_{i,2}\,\, y^2}\\
           & = \int_{\Sigma_{-}}{\scriptstyle \sum\limits_{j=0}^\infty (-1)^j \left(\frac{\beta_{-}u_-}{y}\right)^j}
               \left(\frac{\omega}{y}-H_{\min}\right)+\int_{\Sigma_{+}}{\scriptstyle\sum\limits_{j=0}^\infty
               \left(\frac{\beta_{+}u_+}{y}\right)^j}\left(H_{\max}-\frac{\omega}{y}\right)
             \\
             & \mbox{} -\sum_{i=1}^n\frac{1}{m_{i}}\frac{H(x_i)}{\lambda_{i,1}\lambda_{i,2}\,
               y} \\
          & =  \frac{1}{y}\left(area(\Sigma_{-})+\beta_{-} H_{\min}+\beta_{+} H_{\max}-area(\Sigma_{+})  -\sum_{i=1}^n\frac{1}{m_{i}}\frac{H(x_i)}{\lambda_{i,1}\lambda_{i,2}}\right),
    \end{split}
  \end{equation}
  % {\tiny \begin{align*}
  %          0 & =\int_M\omega^{S^1} =\int_{\Sigma_{-}}\frac{\omega-H_{\min}y}{e^{S^1}(\nu_-)}+\int_{\Sigma_{+}}\frac{\omega-H_{\max}y}{e^{S^1}(\nu_+)}-\sum\limits_{i=1}^n\frac{1}{m_{i}}\frac{H(x_i)y}{e^{S^1}(\nu_{x_i})}\\
  %          = & \int_{\Sigma_{-}}\frac{\omega-H_{\min}y}{b_{-}u_-+y}+\int_{\Sigma_{+}}\frac{\omega-H_{\max}y}{b_{+}u_+-y}-\sum_{i=1}^n\frac{1}{m_{i}}\frac{H(x_i)y}{\lambda_{i,1}\lambda_{i,2}\,\, y^2}\\
  %          = &\int_{\Sigma_{-}}{\scriptstyle \sum\limits_{j=0}^\infty (-1)^j \left(\frac{b_{-}u_-}{y}\right)^j}
  %              \left(\frac{\omega}{y}-H_{\min}\right)+\int_{\Sigma_{+}}{\scriptstyle\sum\limits_{j=0}^\infty
  %              \left(\frac{b_{+}u_+}{y}\right)^j}\left(H_{\max}-\frac{\omega}{y}\right)
  %              -\sum_{i=1}^n\frac{1}{m_{i}}\frac{H(x_i)}{\lambda_{i,1}\lambda_{i,2}\,
  %              y} \\
  %          = & \frac{1}{y}\left(area(\Sigma_{-})+b_{-} H_{\min}+b_{+} H_{\max}-area(\Sigma_{+})  -\sum_{i=1}^n\frac{1}{m_{i}}\frac{H(x_i)}{\lambda_{i,1}\lambda_{i,2}}\right),
  %        \end{align*}}  
  where
  \begin{itemize}[leftmargin=*]
  \item $u_{\pm} \in H^2(\Sigma_{\pm};\Z)$ is a generator,
  \item $\nu_{\pm}$ is the normal orbi-bundle to $\Sigma_{\pm}$
    with degree given by $\beta_{\pm}$ and equivariant Euler class equal to $e^{S^1}(\nu_{\pm})$, 
  \item we assume that there are $n$ isolated fixed points
    $x_1,\ldots, x_n$, and 
  \item the order of the orbifold structure group of $x_i$ is
    $m_i$ and the orbi-weights are
    $\lambda_{i,1},\lambda_{i,2}$. 
  \end{itemize}
  In order to see that \eqref{eq:47} holds, we remark that, since $\nu_{\pm}$ is a complex line orbi-bundle, the
  equivariant Euler class equals the equivariant first Chern
  class $c_1^{S^1}(\nu_{\pm})$ (see Remark
  \ref{rmk:chern_class}). Moreover, since every the non-zero
  orbi-weight of any point in $\Sigma_{\pm}$ is $\mp 1$, we have that
  $c_1^{S^1} (\nu_{\Sigma_\pm}) = \beta_{\pm} u_\pm \mp y$. On the other
  hand, integrating the equivariant form $1$ and using the
  localization formula again, we have that

  \begin{equation}
    \label{eq:48}
    0 =\int_M1 = \frac{1}{y^2}\left(-\beta_{-}-\beta_{+}+\sum\limits_{i=1}^n\frac{1}{m_{i}}\frac{1}{\lambda_{i,1}\lambda_{i,2}}\right).
  \end{equation}
  % Moreover, integrating  the equivariant form $1$, gives
%  {\tiny \begin{align*}
%  		0 & =\int_M1=\int_{\Sigma_{-}}\frac{1}{b_{\Sigma_{-}}u_-+y}+\int_{\Sigma_{+}}\frac{1}{b_{\Sigma_{+}}u_+-y}+\sum_{i=1}^n\frac{1}{d_{i}}\frac{1}{e(\nu_{x_i})}\\
% % 		&=\frac{1}{y}\int_{\Sigma_{-}}\frac{1}{1-\left(\frac{-b_{\Sigma_{-}}u_-}{y}\right)}-\frac{1}{y}\int_{\Sigma_{+}}\frac{1}{1-\left(\frac{b_{\Sigma_{+}}u_+}{y}\right)}+\sum_{i=1}^n\frac{1}{d_{x_i}}\frac{1}{\lambda_{i_1}\lambda_{i_2}\,y^2}\\
%  %		&=\frac{1}{y}\int_{\Sigma_{-}}\sum_{j=0}^\infty (-1)^j \left(\frac{b_{\Sigma_{-}}u_-}{y}\right)^j - \frac{1}{y}\int_{\Sigma_{+}}\sum_{j=0}^\infty \left(\frac{b_{\Sigma_{+}}u_+}{y}\right)+\sum_{i=1}^n\frac{1}{d_{i}}\frac{1}{\lambda_{i,1}\lambda_{i,2}\,y^2}\\
%  		&=\frac{1}{y^2}\left(-b_{\Sigma_{-}}-b_{\Sigma_{+}}+\sum_{i=1}^n\frac{1}{d_{i}}\frac{1}{\lambda_{i,1}\lambda_{i,2}}\right).
%  \end{align*}}
  Hence, by \eqref{eq:47} and \eqref{eq:48}, the following holds:
  \begin{equation}\label{eq:linear_system}
    \begin{pmatrix}
      H_{\min} & H_{\max} \\
      1 & 1 
    \end{pmatrix}
    \begin{pmatrix}
      \beta_- \\
      \beta_+
    \end{pmatrix} =
    \begin{pmatrix}
      area(\Sigma_{+}) - area(\Sigma_{-}) +
      \sum\limits_{i=1}^n\frac{1}{m_{i}}\frac{H(x_i)}{\lambda_{i,1}\lambda_{i,2}}
      \\
      \sum\limits_{i=1}^n\frac{1}{m_{i}}\frac{1}{\lambda_{i,1}\lambda_{i,2}}
    \end{pmatrix}.
  \end{equation}
  Since $H_{\min}<H_{\max}$, \eqref{eq:linear_system} admits a
  unique solution $(\beta_-,\beta_+)$. Hence, since all the other
  quantities in \eqref{eq:linear_system} can be reconstructed from the
  labeled multigraph, the result follows.
  
  Suppose next that $(M,\omega,H)$  has a unique fixed orbi-surface
  $\Sigma$ and denote by $\beta$ the degree of its normal orbi-bundle. By integrating $1$ and using the localization formula  of
  Theorem~\ref{prop:localization}, we obtain that
% {\tiny \begin{align*}
% 		0 & =\int_M1=\int_{\Sigma_{+}}\frac{1}{b_{\Sigma_{+}}u_+-y}+\sum_{i=0}^n\frac{1}{d_{i}}\frac{1}{e(\nu_{x_i})} \\ &
% 		= - \frac{1}{y}\int_{\Sigma_{+}}\sum_{j=0}^\infty \left(\frac{b_{\Sigma_{+}}u_+}{y}\right)+\sum_{i=0}^n\frac{1}{d_{i}}\frac{1}{\lambda_{i,1}\lambda_{i,2}\,y^2}
% 		=\frac{1}{y^2}\left(-b_{\Sigma_{+}}+\sum_{i=0}^n\frac{1}{d_{i}}\frac{1}{\lambda_{i,1}\lambda_{i,2}}\right)
% \end{align*}}
% and so 	
  \begin{equation}\label{eq:bsigma-}
    \beta = \sum_{i=1}^n\frac{1}{m_{i}}\frac{1}{\lambda_{i,1}\lambda_{i,2}}.
  \end{equation}
  Since the right hand side of \eqref{eq:bsigma-} can be reconstructed
  from the labeled multigraph, the result follows.
\end{proof}

\begin{remark}\label{rmk:last_case}
  The last case in the proof of Proposition
  \ref{prop:graph_seifert} can be shown alternatively as
  follows. Without loss of generality, by adding a suitable constant we may assume that $H(\Sigma)
  \neq 0$. Then, if we integrate $\omega^{S^1}$, we obtain that
  \begin{equation}\label{eq:comp_cond2}
    \beta = \beta_{\pm} = \frac{1}{H(\Sigma_\pm)} \left(\pm \,area(\Sigma_\pm)+\sum\limits_{i=1}^n\frac{1}{m_{i}}\frac{H(x_i)}{\lambda_{i,1}\lambda_{i,2}}\right),
  \end{equation}
  depending on whether $\Sigma=\Sigma_\pm$ is a maximal or a minimal
  orbi-surface. Note that \eqref{eq:bsigma-} together with
  \eqref{eq:comp_cond2} give us a necessary compatibility condition
  for the labels of the multigraphs of a Hamiltonian $S^1$-space that
  has a unique fixed orbi-surface.
\end{remark}

\begin{example}[Weighted projective spaces, continued]\label{exm:wps2}
   We determine the labeled multigraph of the Hamiltonian $S^1$-space
  $(M=\mathbb{C}P^2(p,s,q),\omega,H)$ of
  Example~\ref{ex:projective} under the condition that $k_1=0$ (see
  Figure~\ref{fig:mgwpsk1_0} and cf. Example \ref{ex:projective2}). Here $A$ is a positive
  real number that represents the area label. Moreover, as in Example \ref{ex:projective2}, the positive integers
  $\ell_p$, $\ell_q$ and $\ell_s$ are obtained as in Remark
  \ref{ex::weightedprojSeifertInvariants}. Finally, the relation
  between the moment map labels of the minimum and maximum can be
  obtained by combining Lemma \ref{lemma:degree_Ok} and \eqref{eq:comp_cond2}.
\renewcommand{\thefigure}{\thesection.\arabic{figure}}
  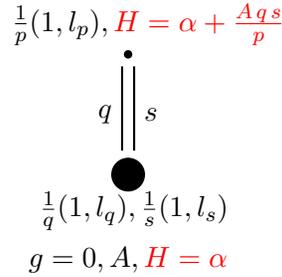
\begin{figure}[h!]
    \begin{center}
      % \resizebox{12cm}{!}{%
      \begin{tikzpicture}[scale=.8]			
        \coordinate (point) at  (-6.5,3);
        \fill (point) circle (2pt);
        \fill  (-6.5,1) circle (8pt);
        
        \draw[thick](-6.6,2.8) -- (-6.6,1.4);
        \draw[thick](-6.4,2.8) -- (-6.4,1.4);
        \coordinate[label=right:{$s$}] (c1) at (-6.4,2);
        \coordinate[label=left:{$q$}] (c1) at (-6.6,2);
        \coordinate[label=right:{$\frac{1}{p}(1,l_{p}), \color{red} H=\alpha + \frac{A \, q \,s}{p}$}] (c2) at (-8.6,3.5);
        \coordinate[label=above:{$\frac{1}{q}(1,l_q),\frac{1}{s}(1,l_s)$}] (AAA)  at (-6.4,-.1);
        \coordinate[label=above:{$g=0,A , \color{red} H=\alpha$}] (ABB)  at (-6.5,-.8);
        % \draw[->, thick] (-4,2.2) to (-3.2,2.2);	
      \end{tikzpicture}	\hspace{1cm}
      % }	
    \end{center}
    \caption{Labeled multigraph of Example~\ref{ex:projective} with
      one fixed surface.}
    \label{fig:mgwpsk1_0}
  \end{figure}
  This example shows that, if $M$ is an orbifold, the labeled
  multigraph of $(M,\omega,H)$ may have edges that are incident to a
  fat vertex, something that cannot occur for manifolds! \hfill $\Diamond$
	%The moment map labels can be easily determined from \eqref{eq:comp_cond2}.
	% using the Localization Formula for equivariant cohomology (cf. Theorem~\ref{prop:localization}). Indeed, using the equivariant 
	%symplectic form $\omega^{S^1}=\omega - H y$ (cf. Appendix~\ref{sec:loc}), we have
	%$$
	%0 = \int_M \omega^\sharp =   \frac{p \,\iota_F^*\omega^\sharp}{q \,s y^2} +  \int_\Sigma  \frac{\iota_\Sigma^* \omega^\sharp}{b_\Sigma \, \mathit{u} + y } =
	% -  \frac{p H_{max}y}{q \, s\, y^2}  +  \int_\Sigma \frac{\omega - H_{min} \, y}{b_\Sigma \, \mathit{u} +y},
	%$$
	%where we denote by $H_{min}$ and $H_{max}$ respectively the minimum and  the maximum   values of $H$. We conclude that
	%$$
	% \frac{p H_{max}}{q\, s y}  = \int_\Sigma \left(\frac{\omega}{y} - H_{max}\right)\, \sum_{j=0}^\infty (-1)^j \, \left(\frac{b_\Sigma \mathit{u}}{y}\right)^j
	%$$
	%and so we have $A:=\mathrm{area}(\Sigma)=b_\Sigma \left(H_{max}- H_{min} \right)$. 
\end{example}

\begin{example}[Projectivized plane bundles, continued]\label{exm:projectivized_suborbifolds2}
  Let $\Sigma$ be a compact, connected orientable orbi-surface and let
  $\mathrm{pr} : L \to \Sigma$ be a complex line orbi-bundle. Let
  $$(g;\beta_0, (m_1,l_1),\ldots, (m_n,l_n))$$
  be the Seifert invariant of the principal $S^1$-orbi-bundle
  associated to $L$.  Consider the
  projectivization $\mathbb{P}(L \oplus \C) \to \Sigma$ constructed in
  Example \ref{exm:projectivized} and fix $s >0$. We calculate the labeled
  multigraph of the Hamiltonian $S^1$-space $(\mathbb{P}(L \oplus \C),
  \omega_s, H)$, where $\omega_s$ is the symplectic form of
  Example~\ref{exm:proj_symp}, and $H$ is defined in
  Example~\ref{exm:proj_plane_ham} (see Figure \ref{fig:proj}
  below). We observe that the fixed point set consists of two copies
  $\Sigma_+$ and $\Sigma_-$
  of $\Sigma$, and that the normal bundles to $\Sigma_+$ and
  $\Sigma_-$ have opposite degrees $\beta_+=-\beta_-$ determined by Proposition~\ref{def::EulerOrbibundle}. Moreover, the type labels in  Figure \ref{fig:proj} for the
  singular points of $\Sigma_{\pm}$ can be calculated as in the proof
  of Proposition \ref{prop:graph_seifert}, knowing that, since  the normal orbi-bundles of $\Sigma_\pm$ have opposite orientations, if $(m_i,l_i)$ is a Seifert invariant of $\Sigma_-$, then the corresponding Seifert invariant of $\Sigma_+$ is $(m_i,\hat{l}_i)$ with $\hat{l}_i:=m_i-l_i$ (see \cite[Theorem 6, p. 184]{seifert}). As usual, in  Figure \ref{fig:proj},  we denote by $l_i^\prime$ the integers $1\leq l_i^\prime <m_i$ such that $l_i \cdot l_i^\prime=1 \mod m_i$. The existence of isotropy
  orbi-spheres and the labels of the corresponding edges can be
  deduced from Lemma \ref{lem::isotropy} and Remark
  \ref{rmk:type_fixed}. Finally, the relation between the
  area labels can be obtained by integrating the equivariant
  symplectic form $\omega_s^{S^1}$ and using the localization formula
  of Theorem \ref{prop:localization}.  
%   Let $\Sigma$ be a compact, connected orientable orbi-surface and let $\mathrm{pr} : L \to \Sigma$ be a complex line orbi-bundle of degree $b_\Sigma$. Consider the
%   projectivization $\mathbb{P}(L \oplus \C) \to \Sigma$ constructed in
%   Example \ref{exm:projectivized}  and let $\Sigma_-$, $\Sigma_+$ be the two copies of $\Sigma$ inside  $\mathbb{P}(L \oplus \C)$ (cf. Example~\ref{exm:projectivized_suborbifolds}), 
%   where the moment map $H$ (defined in Example~\ref{exm:proj_plane_ham}
%   for a multiple of the symplectic form of Example~\ref{exm:proj_symp})
%   takes its minimal and maximal values. Let
% $\{b_-,g,(p_1,\hat{l}_1^\prime),\ldots, (p_n,\hat{l}_n^\prime)\}$  be
% the Seifert invariant of the circle orbi-bundle associated to the
% normal orbi-bundle $\nu_{\Sigma_-}$ of $\Sigma_-$ in $\mathbb{P}(L
% \oplus \C)$ (where  $1\leq \hat{l}_i^\prime < p_i$) 
% % \text{and} \quad l_i^\prime \cdot \hat{l}_i^\prime=1 \mod p_i,$$ 
% and let $\{b_+,g,(p_1,l_1),\ldots, (p_n,l_n)\}$  be  the Seifert
% invariant associated to $\Sigma_+$. Note that, since the Seifert
% fibration associated to $\nu_{\Sigma_+}$ is obtained from the one
% associated to $\nu_{\Sigma_-}$ by changing orientation, we have  $l_i
% +\hat{l}_i^\prime=p_i$ and $b_-+b_+=-n$ (and
% $\deg(\nu_{\Sigma_+})=-\deg(\nu_{\Sigma_+})$) \cite[Theorem 6,
% p.388]{seifert_threlfall}. 

% The multigraph associated to this $S^1$-space is depicted in
% Figure~\ref{fig:proj}, where the area labels are obtained from the
% moment map labels using \eqref{eq:linear_system}. 
\renewcommand{\thefigure}{\thesection.\arabic{figure}}
  \begin{figure}[h!]
  \centering
  \begin{tikzpicture}[scale=.8]	
    \coordinate (point) at  (-0.5,3);
    \fill  (5.5,3) circle (8pt);
    \fill  (5.5,1) circle (8pt);
    \coordinate[label=above:{$\frac{1}{m_1}(1,l_1^\prime),\dots,\frac{1}{m_n}(1,l_n^\prime)$}] (VFD) at (5.7,0);
    \coordinate[label=above:{$g,A,\color{red} H=\alpha$}] (VFE) at (5.5,-0.5);			
    \coordinate[label=below:{$\frac{1}{m_1}(1,\hat{l}_1),\dots,\frac{1}{m_n}(1,\hat{l}_n)$}] (VCC) at (5.6,4.7);
    \coordinate[label=below:{$g, A - \beta_- s, \color{red} H=\alpha + s$}] (VCC) at (5.5,4);			
    \draw[thick](5,2.6) -- (5,1.4);
    \draw[thick](6,2.6) -- (6,1.4);
    \node at ($(5,2)!.5!(6,2)$) {\ldots};
    \coordinate[label=right:{$m_n$}] (c1) at (6,2);
    \coordinate[label=left:{$m_1$}] (c1) at (5,2);		
  \end{tikzpicture}\caption{Multigraph of the Hamiltonian $S^1$-space $(\mathbb{P}(L \oplus \C),
  \omega_s, H)$.} \label{fig:proj}
\end{figure}
\end{example}

\section{Equivariant operations}\label{ChapterEquivariantOperations}
\subsection{Symplectic weighted blow-up}\label{sec:sympl-weight-blow}
In this section we describe how to perform equivariant symplectic
weighted blow-ups at singular points, a construction that we use to
desingularize Hamiltonian $S^1$-spaces equivariantly. 
Recall that, topologically, a symplectic blow-up at a regular point
corresponds to cutting out a symplectic ball around the point and
collapsing the boundary of the resulting space along the fibers of the
Hopf fibration. Smoothly this can be achieved by symplectic
reduction using an auxiliary circle action (see \cite{gs_birational}).
This was generalized to {\em symplectic cutting} by Lerman \cite{cut}.
Symplectic {\em weighted} blow-ups were introduced by 
Meinrenken and Sjamaar in \cite{MS}, and by Godinho in
\cite{lgodinho}: Topologically, the idea is to remove a symplectic orbi-ball centered
around a singular point\footnote{Strictly speaking, one may carry out
  a weighted blow-up at a regular point, but, since we do not consider
this case in this paper, we restrict to singular points.} and to collapse the orbits of some
(weighted) circle action on the boundary of the resulting space. In
what follows, we describe this operation. 

Let $(M^{2n},\omega)$ be a $2n$-dimensional symplectic orbifold and
let $x\in M$ be a singular point that is isolated with cyclic orbifold
structure group $\Gamma$.
By Darboux's Theorem for orbifolds (see Theorem \ref{thm:darboux_orbifolds}), there exists 
a l.u.c.  $(\widehat{U},\Gamma,\varphi)$  centered at $x$ such that
$\widehat{U}$ is an open subset  of $\R^{2n}\simeq \C^n$, where
$\Gamma < \mathrm{U}(n)$ acts in a linear fashion, $\omega_{\lvert
  U}$ is as in Example~\ref{exm:symp_ovs}, and $[0]$ is the only
singular point in $U$. Since $\Gamma <
\mathrm{U}(n)$ is abelian, we may assume without loss of generality that $\Gamma$ is contained in
the maximal torus of $\mathrm{U}(n)$ consisting of diagonal matrices. We fix a $\Gamma$-invariant
ball $\widehat{B} \subseteq \widehat{U} \in \C^n$ around $0$ and
$\widehat{m}:=(\widehat{m}_1,\dots,\widehat{m}_n)\in\mathbb{Z}^n_+$ so
that $\gcd(\widehat{m}_1,\dots,\widehat{m}_n) = 1$. Moreover, we
denote a circle acting on $\C^n$ by $\widehat{S}^1$.
The $\widehat{S}^1$-action on $\widehat{B}$ given by
\begin{equation}
  \label{eq:49}
  \begin{split}
    \widehat{S}^1\times\widehat{B}&\rightarrow\widehat{B} \\
    (\lambda,z_1,\ldots, z_2)&\mapsto(\lambda^{\widehat{m}_1} z_1,\dots,\lambda^{\widehat{m}_n}z_n),
  \end{split}
\end{equation}
is effective, fixes only $0$ and commutes with the
$\Gamma$-action. Since $\widehat{m} \in \Z^n_+$, the moment map
$\widehat{H} : \widehat{B} \to \R$ for the above
$\widehat{S}^1$-action given by
$$ \widehat{H}(z_1,\ldots, z_n) = \frac{1}{2}\sum\limits_{i=1}^n
\widehat{m}_i |z_i|^2 $$
is proper. The $\widehat{S}^1$-action of \eqref{eq:49} descends to an effective Hamiltonian action
of the quotient $\widehat{S}^1/(\widehat{S}^1 \cap \Gamma) \simeq S^1$ on
the orbifold $B=\widehat{B}/\Gamma \subseteq U$ that fixes only $[0]$. Let $H : B \to \R$ be the moment
map for this action such that $H([0]) = 0$ and let
$m:=(m_1,\dots,m_n) \in \mathbb{Q}^n_+$ be the orbi-weights for this action at
$[0]$. (By Lemma \ref{lem::orbiweights}, for each $i=1,\ldots, n$, the orbi-weight $m_i$ equals
$\widehat{m}_i/p$, where $p = |\widehat{S}^1 \cap \Gamma|$.) By
Corollary \ref{cor:linear-symp_action},
$$ H([z_1,\ldots,z_n]_{\Gamma}) = \frac{1}{2} \sum\limits_{i=1}^n m_i |z_i|^2. $$
Let $\varepsilon > 0$ be sufficiently small. Since $p\varepsilon$ is a
regular value of $\widehat{H}$, $\varepsilon$ is a regular value of
$H$. Moreover, since $[0]_{\Gamma} \in B$ is the only singular point, $H^{-1}(\varepsilon)$ is a smooth
submanifold of $B$. Hence, the quotient $H^{-1}(\varepsilon)/S^1$ is an orbifold
that inherits a symplectic form from $\omega$ by symplectic reduction.
Moreover, it satisfies the following properties:
\begin{enumerate}[label=(\alph*),ref=(\alph*),leftmargin=*]
\item \label{item:23} the singular points are isolated and have cyclic orbifold
  structure groups, and
\item \label{item:24} it is diffeomorphic to $\C P^{n-1}(\widehat{m}_1,\ldots,
  \widehat{m}_n)/\Gamma'$ , where
  \begin{equation}
    \label{eq:51}
    \Gamma':= \Gamma/(\Gamma\cap \widehat{S}^1).
  \end{equation}
\end{enumerate}
Since a singular point in $H^{-1}(\varepsilon)/S^1$ corresponds to an
orbit with finite stabilizer, property \ref{item:23} is a consequence of the local normal
form for Hamiltonian $S^1$-actions (see \cite{gs,marle}). On the other
hand, property \ref{item:24} follows from the fact that $\widehat{H}$ is
proper and the definition of $\C P^{n-1}(\widehat{m}_1,\ldots,
\widehat{m}_n)$ -- see Example \ref{ex_weightedprj}. 

For any $\eta >0$ sufficiently small, we set
$B_{\eta}(x):=H^{-1}([0,\eta))$. This is an $S^1$-invariant open symplectic
ellipsoid centered at $x$ and contained in $U$. Given $\varepsilon,
\delta > 0$ sufficiently small, we consider the following $S^1$-action
on the product
$Z_{\varepsilon,\delta}:=B_{\varepsilon+\delta}(x)\times\mathbb{C}$
\begin{align*}
	S^1\times Z_{\varepsilon,\delta}& \,\, \rightarrow \,\, Z_{\varepsilon,\delta}\\
	\left(\lambda,[z_1,\dots,z_n]_\Gamma,w\right)&\,\, \mapsto \,\, \left([\lambda^{m_1}z_1,\dots,\lambda^{m_n}z_n]_\Gamma,\lambda^{-1}w\right).
\end{align*}
We endow $Z_{\varepsilon,\delta}$ with the symplectic form given by
the sum of the pullback of $\omega$ and of the pullback of the
standard symplectic form on $\C$. Then the above circle action is
Hamiltonian and a moment map is given by 
$$\widetilde{H}\left([z_1,\dots,z_n]_\Gamma,w\right)=\frac{1}{2}\left(m_1 \lvert z_1\rvert ^2+\cdots+m_n \lvert z_n\rvert ^2-\lvert w\rvert^2\right).$$
Since $\varepsilon$ is a regular value of $H$, it is also a regular
value of $\widetilde{H}$. Moreover, since $([0]_{\Gamma},0) \in
Z_{\varepsilon,\delta}$ is the only singular point, $\widetilde{H}^{-1}(\varepsilon)$ is, in fact, a smooth
submanifold of $Z_{\varepsilon,\delta}$. We observe that the restriction of
the inclusion $B_{\varepsilon+\delta}(x) \hookrightarrow
Z_{\varepsilon,\delta}$ to $H^{-1}(\varepsilon)$ is a smooth embedding
$H^{-1}(\varepsilon) \hookrightarrow
\widetilde{H}^{-1}(\varepsilon)$. Hence, it can be checked that the symplectic reduction
$(\widetilde{H}^{-1}(\varepsilon)/S^1, \omega_{\mathrm{red}})$ is a
symplectic orbifold that satisfies the following properties:
\begin{enumerate}[label=(\Alph*),ref=(\Alph*),leftmargin=*]
\item \label{item:25} the smooth embedding $H^{-1}(\varepsilon) \hookrightarrow
  \widetilde{H}^{-1}(\varepsilon)$ descends to a full embedding
  $H^{-1}(\varepsilon)/S^1 \hookrightarrow
  \widetilde{H}^{-1}(\varepsilon)/S^1$  -- see Remark
  \ref{rmk:fixed_fully_embedded},
\item \label{item:26} the only singular points in
  $\widetilde{H}^{-1}(\varepsilon)/S^1$ lie in the image of the above
  embedding, and
\item \label{item:27} the complement of $H^{-1}(\varepsilon)/S^1$ in
  $\widetilde{H}^{-1}(\varepsilon)/S^1$ is symplectomorphic to the
  open subset $H^{-1}((\varepsilon, \varepsilon + \delta)) =
  B_{\varepsilon + \delta}(x) \smallsetminus \overline{B_{\varepsilon}(x)}$ of
  $B_{\varepsilon + \delta}(x)$. 
\end{enumerate}

% \todo{23/9/5: I commented out the `partition' of the reduced space into the
%   exceptional divisor and the rest, and the comment about the blow-up
%   being a disk orbi-bundle over the exceptional divisor. Do we need
%   these comments? }

% and so the level set 
% \begin{align*}
% 	&  \widetilde{H}^{-1}(\varepsilon)  =\left\{\left(\left[z_1,\dots,z_n\right]_\Gamma,w\right)\in Z_{\varepsilon,\delta}\mid \, m_1 \lvert z_1\rvert ^2+\cdots+m_n\lvert z_n\rvert ^2=\lvert w\rvert ^2+ 2\varepsilon\right\}\\ 
% 	& =\left\{\left(\left[z_1,\dots,z_n\right]_\Gamma,0\right)\in Z_{\varepsilon,\delta} \mid \, m_1\lvert z_1\rvert ^2+\cdots+m_n \lvert z_n\rvert ^2= 2 \varepsilon\right\} \quad \sqcup\\ 
% 	&\hspace{.7cm} \left\{\left(\left[z_1,\dots,z_n\right]_\Gamma,w\right)\in Z_{\varepsilon,\delta}\mid  \,\, m_1 \lvert z_1\rvert ^2+\cdots+m_n\lvert z_n\rvert ^2- 2\varepsilon=\lvert w \rvert ^2  \right. \\  & \hspace{3.5cm}\left. \text{and} \quad m_1\lvert z_1\rvert^2+\cdots+m_n \lvert z_n \rvert^2>2 \varepsilon\right\}\\ 
% 	 & \simeq  \,\,\, S^{2n-1}/\Gamma \, \sqcup	\,	\left\{\left[z_1,\dots,z_n\right]_\Gamma \in B_{\varepsilon+\delta}(x) \mid  m_1\lvert z_1\rvert^2+\cdots+m_n\lvert z_n\rvert^2>2\varepsilon\right\}\times S^1
% \end{align*} 
% is a smooth submanifold of $Z_{\varepsilon,\delta}$. 

% the reduced space $\widetilde{H}^{-1}(\varepsilon)/S^1$ is a disk
% orbi-bundle over $\Sigma_0 \simeq \mathbb{C}P^{n-1}(m^\prime)/D$.

We denote the image of the embedding $H^{-1}(\varepsilon)/S^1 \hookrightarrow
\widetilde{H}^{-1}(\varepsilon)/S^1$ by $\Sigma_0$. Explicitly, it is
given by 
\begin{equation} \label{eq:ed}
\Sigma_0  := \left\{\left([z_1,\dots,z_n]_\Gamma,0\right)\in  \widetilde{H}^{-1}(\varepsilon) \right \}/\,\,S^1. 
\end{equation}

By
property \ref{item:27}, the symplectic orbifolds $(M
\smallsetminus \overline{B_{\varepsilon}(x)},\omega)$ and
$(\widetilde{H}^{-1}(\varepsilon)/S^1, \omega_{\mathrm{red}})$ can be glued along
$B_{\varepsilon + \delta}(x) \smallsetminus
\overline{B_{\varepsilon}(x)}$ to obtain the symplectic orbifold
\begin{equation}
  \label{eq:wbu}
  \left(\widetilde{M}_{x,\varepsilon}:=\left( M\smallsetminus
  \overline{B_{\varepsilon}(x)}\right)
  \coprod_{B_{\varepsilon+\delta}(x) \smallsetminus
  \overline{B_{\varepsilon}(x)}} \widetilde{H}^{-1}(\varepsilon)/S^1, \widetilde{\omega}_{\varepsilon}\right).
\end{equation}

% \todo{23/9/7: why don't we make the symplectic orbifold depend on
%   $\varepsilon$? Why don't we make the symplectic form depend on $x$?}

\begin{definition}\label{def:bup}
  Let $(M,\omega)$ be a symplectic orbifold and let $x \in M$ be a
  singular point with cyclic orbifold structure group $\Gamma$. The
  \textbf{(symplectic) weighted blow-up of} $(M,\omega)$ at $x$ of
  {\bf size} $\varepsilon$ and {\bf weights} $(m_1,\ldots, m_n) \in \mathbb{Q}^n_+$ is the symplectic orbifold 
  $(\widetilde{M}_x, \widetilde{\omega})$ in \eqref{eq:wbu}. The
  \textbf{exceptional divisor} of the weighted blow-up is the (fully
  embedded) suborbifold
  $$\Sigma_0 \simeq \mathbb{C}P^{n-1}(\widehat{m}_1,\ldots,
  \widehat{m}_n)/\Gamma'$$
  of \eqref{eq:ed}, where $\Gamma'$ is as in \eqref{eq:51}.
\end{definition}

\begin{remark}\label{rmk:size_blow_up}
  By looking at the construction of a weighted blow-up at a singular
  point, it is clear that there are simple constraints on the possible
  size, e.g., the symplectic size of a Darboux chart. To simplify the
  exposition, we omit stating explicitly that $\varepsilon > 0$
  satisfies these constraints, assuming implicitly that it does. We
  trust that this does not cause confusion.
\end{remark}

Since we are primarily interested in performing weighted blow-ups on symplectic orbifolds with
isolated singular points all of whose orbifold structure groups are
cyclic, the following result plays an important role.

\begin{lemma}\label{lemma:blow_up_preserves}
  Let $(M,\omega)$ be a $2n$-dimensional symplectic orbifold with isolated singular
  points such that all orbifold structure groups are cyclic. Let $x
  \in M$ be a singular point. For all $\varepsilon >0$ and all
  $(m_1,\ldots, m_n) \in \mathbb{Q}^n_+$, the
  weighted blow-up of $(M,\omega)$ at $x$ of size $\varepsilon$ and
  weights $(m_1,\ldots, m_n)$ has isolated singular points and is such
  that all orbifold structure groups are cyclic.
\end{lemma}

\begin{proof}
  This result follows immediately by the conditions on the singular
  points of $M$, the above construction of the weighted blow-up, and
  properties \ref{item:23}, \ref{item:25} and \ref{item:26} above.
\end{proof}

Let $G$ be a compact Lie group. As expected, a weighted blow-up can be
performed on a Hamiltonian $G$-space so that the result is still a
Hamiltonian $G$-space by making $G$-equivariant choices and using a
local circle action that commutes with the $G$-action (see
\cite[Section 4.1]{MS})\footnote{Strictly speaking, in \cite[Section
  4.1]{MS} it is not checked that the moment map is good in the sense
  of Definition \ref{defn:good_map}. However, this follows at once by
  observing that a weighted blow-up is obtained by symplectic cutting
  and that the map used to perform the cut is good by construction.}. Moreover, for a fixed singular point $x$, whenever the size is sufficiently
small, the result only depends on the weights and the size up to isomorphisms of Hamiltonian $G$-spaces (see \cite[Proposition
4.2]{MS}). In the case of a $G\simeq S^1$-action, we denote an
$S^1$-equivariant weighted blow-up at an isolated singular point $x$
with cyclic structure group $\Gamma$ of size $\varepsilon$ and weights
$(m_1,\ldots,m_n)$ by
$$
\text{WBU}(\lvert \Gamma \rvert, k_1,\ldots, k_n, m_1,\ldots,m_n,\varepsilon),
$$
where $k_1 \geq \ldots \geq k_n$  are the orbi-weights of the
$S^1$-action at $x$ (see Definition \ref{defn:orbi-weights}).

\subsubsection{Weighted blow-ups of Hamiltonian $S^1$-spaces}
In this section, we consider two families of $S^1$-equivariant weighted blow-ups on
Hamiltonian $S^1$-spaces that play an important role in some of our
key results (see Proposition~\ref{prop::desing}, and Theorems \ref{fig::minimalfixedsurface} and \ref{thm:existisolated}). Before
discussing these results further, we observe that weighted
blow-ups preserve Hamiltonian $S^1$-spaces (see the beginning of
Section \ref{sec:labelled-multigraph}). This is an immediate
consequence of Lemma \ref{lemma:blow_up_preserves} and of
\cite[Section 4]{MS}.

\begin{corollary}\label{cor:weighted}
  Let $(M,\omega,H)$ be a
  Hamiltonian $S^1$-space and let $x\in M$ be a singular point. Any
  $S^1$-equivariant weighted blow-up of $(M,\omega, H)$ at $x$ is a
  Hamiltonian $S^1$-space.
\end{corollary}

Our aim is
to describe the effect of the aforementioned blow-ups on the labeled multigraphs
(see Propositions~\ref{prop::desing} and \ref{prop::desing2} below). Our
choice to split these results is dictated by how we use each family in
our classification results: Proposition \ref{prop::desing} is used in
the desingularization process of Theorem \ref{thm:desingularization}, while
(the reverse construction to that of) Proposition \ref{prop::desing2} is needed to obtain the minimal
models of Theorems \ref{fig::minimalfixedsurface} and
\ref{thm:existisolated}. Moreover, there are some small geometric
differences between the two families (see Remarks
\ref{rem::weightedblowupConnectsum} and \ref{rmk:other_blow-up}). Finally, we believe that this makes it
easier for the reader to extract the necessary information. 

\begin{proposition}\label{prop::desing} Let $(M,\omega,H)$ be a
  Hamiltonian $S^1$-space and let $x\in M$ be a singular point of type
  $\frac{1}{m}(1,l)$. Let $1 \leq l' < m$ be the integer such that $l
  \cdot l' = 1 \mod m$ and let $\gamma \in \Z$ be such that $ll'
  + m \gamma = 1$. Let $\frac{a_1}{p},\frac{a_2}{p}$ be the
  orbi-weights of the $S^1$-action at $x$, where $\gcd(a_1,a_2)=1$ and
  $a_1\geq a_2$. 
  % are such that $\gcd (a_1,a_2)=1$ and
  % $\gcd(\lvert a_1l - a_2\rvert, m)=p$
  Let $b_1,b_2>0$ be integers such that
  \begin{equation}\label{eq:ebucond}
    \gcd(b_1,b_2)=1\quad \text{and}\quad b_1l - b_2 = b\, m, \quad \text{for some $b \in\mathbb{Z}$}.
  \end{equation}
  Let $s_1$ be the integer that satisfies $1\leq s_1 < b_1$ and
  $s_1= b \mod b_1$ if  $a_1m \geq a_2b_1-a_1b_2$, or
  $s_1 b = 1 \! \!\!\mod b_1$
  otherwise. Let $s_2$ be the integer that satisfies $1\leq s_2 < b_2$ and $s_2 = - (b l' +
  b_1\gamma)\! \!\!\mod b_2  $ if  $a_2 m\geq a_1b_2 - a_2b_1$, or
  $(b l' + b_1\gamma)s_2 =  - 1\! \!\!\mod b_2$
  otherwise. Finally, we set $k:=m/p$.
%	Moreover,  for integers $x,y$ such that $l x + m y=1$,  $s_2\in \mathbb{Z}$ such that  $1\leq s_2 < b_2$ is defined by
%		$$s_2 = - (\beta x + b_1y)\! \!\!\mod b_2  $$
%	if  $a_2 m\geq a_1b_2 - a_2b_1$, and	by
%		$$(\beta x + b_1y)s_2 =  - 1\! \!\!\mod b_2$$
%		otherwise. 

  An $S^1$-equivariant weighted blow-up of $(M,\omega)$ at $x$
  $$
  \text{WBU}\left(m,\frac{a_1}{p},\frac{a_2}{p},\frac{b_1}{m},\frac{b_2}{m},\varepsilon\right)
  $$
  results in the following `local' changes in the labeled multigraph of
  $(M,\omega,H)$ near the vertex that corresponds to the connected
  component of
  the fixed point set to which $x$ belongs:\\ \\
  \noindent\begin{minipage}{.55\linewidth}
    % \item[(I)] 
    $(I)$ Non-extremal fixed point: \\  \\
    \resizebox{6cm}{!}{%
      \begin{tikzpicture}			
        \coordinate[label=right:] (BB) at (-1.5,1.3);
        \coordinate[label=right:${\,\,\frac{1}{b_1}(1,s_1), \color{red} \alpha + \frac{k \lvert a_1\rvert}{ b_1}\varepsilon}$] (AA) at (-2,0.7);
        \coordinate[label=right:${\,\, \frac{1}{b_2}(1,s_2), \color{red} \alpha - \frac{k\lvert a_2\rvert}{ b_2}\varepsilon }$] (CC) at (-2,-0.5);
					\coordinate[label=right:] (DD) at (-1.5,-1.1);

					\coordinate[label=left:$\lvert a_1\rvert k$](X) at ($(AA)!1!(BB)$);
					\coordinate[label=left:$\lvert a_2\rvert k$](X) at ($(CC)!1!(DD)$);
					\coordinate[label=left:$\frac{\lvert a_1b_2-a_2b_1\rvert}{p}$](X) at ($(AA)!0.5!(CC)$);
					
					\draw[very thick] (BB)--(AA)--(CC)--(DD);
					\fill (AA) circle (4pt);
					\fill (CC) circle (4pt);
					
					\hspace{-.5cm}\draw[->, very thick] (-4,0) to (-3.5,0);
					
					\coordinate[label=left:${\frac{1}{m}(1,l),\, \color{red} \alpha\,\,}$] (A) at (-4.5,0);
					\coordinate[label=right:] (B) at (-4,1);
					\coordinate[label=right:] (C) at (-4,-1);
					\coordinate[label=left:$\lvert a_1\rvert k$] (D)at ($(A)!0.7!(B)$);
					\coordinate[label=left:$\lvert a_2\rvert k$] (D)at ($(A)!0.7!(C)$);
					\draw[very thick] (B) -- (A)--(C);
					\fill (A) circle (4pt);
				\end{tikzpicture} }
			\end{minipage}%
\noindent\begin{minipage}{.55\linewidth}
%\item[(II)] 
  $(II)$ Isolated extremum  {\small ($a_1b_2\neq a_2 b_1$):} \\
  \\ %\\
				\resizebox{6cm}{!}{%
				\begin{tikzpicture}				
					\coordinate[label=right:] (BB) at (-1,1);
					\coordinate[label=above:${\hspace{3cm}\frac{1}{b_i}(1,s_i),   \color{red}  \alpha + \frac{k \, |a_i|}{b_i}\varepsilon}$] (AA) at (-0.7,0.5);
					\coordinate[label=below:${\,\,\,\frac{1}{b_j}(1,s_j), \color{red}  \alpha + \frac{k\, |a_j|}{b_j}\varepsilon  }$] (CC) at (0,-0.5);
					\coordinate[label=right:] (DD) at (0.7,0.5);

					\coordinate[label=left:$\lvert a_i\rvert k$](X) at ($(AA)!1!(BB)$);
					\coordinate[label=right:$\lvert a_j\rvert k$](X) at ($(CC)!0.7!(DD)$);
					\coordinate[label=left:$\frac{\lvert a_1b_2-a_2b_1\rvert}{p}$](X) at ($(AA)!0.7!(CC)$);
					
					\draw[very thick] (BB)--(AA)--(CC)--(DD);
					\fill (AA) circle (4pt);
					\fill (CC) circle (4pt);
					
					\draw[->, very thick] (-2.5,-.4) to (-2,-.4);
					
					\coordinate[label=below:${\frac{1}{m}(1,l), \color{red} \alpha}$] (A) at (-4,-0.5);
					\coordinate[label=right:] (B) at (-4.7,0.5);
					\coordinate[label=right:] (C) at (-3.3,0.5);
					\coordinate[label=left:$\lvert a_i\rvert k$] (D)at ($(A)!0.7!(B)$);
					\coordinate[label=right:$\lvert a_j\rvert k$] (D)at ($(A)!0.7!(C)$);
					\draw[very thick] (B) -- (A)--(C);
					\fill (A) circle (4pt);
				\end{tikzpicture}}
			\end{minipage}	\\	\\
				
\noindent \begin{minipage}{.55\linewidth}
%\item[(III)] 
$(III)$ Isolated extremum \\ \hspace{-.5cm} {\small ($b_i= |a_i|, p=m$):} \\ \\
\resizebox{5cm}{!}{%
			\begin{tikzpicture}				
					\coordinate[label=left:] (AA) at (-1.7,0.5);
					\coordinate[label=below:] (CC) at (-1,-0.5);
					\coordinate[label=below:${\hspace{2cm}\frac{1}{\lvert a_1\rvert }(1,s_1),\frac{1}{\lvert a_2\rvert}(1,s_2),}$] (XX) at (-1.8,-0.9);
					\coordinate[label=below:${\hspace{2cm}A=\frac{m}{a_1a_2}\varepsilon,g=0, \color{red} \alpha + \varepsilon}$] (XX) at (-1.9,-1.5);
					\coordinate[label=right:] (DD) at (-0.3,0.5);

					\coordinate[label=right:$\lvert a_2 \rvert$](X) at ($(CC)!0.7!(DD)$);
					\coordinate[label=left:$\lvert a_1\rvert$](X) at ($(AA)!0.3!(CC)$);
					
					\draw[ very thick] (AA)--(CC)--(DD);
					\fill (CC) circle (8pt);
					
					\draw[->, very thick] (-2.8,-.3) to (-2.3,-0.3);
					
					\coordinate[label=below: ${\vspace{.4cm}\frac{1}{m}(1,l), \color{red} \alpha}$] (A) at (-4,-0.5);
					\coordinate[label=right:] (B) at (-4.7,0.5);
					\coordinate[label=right:] (C) at (-3.3,0.5);
					\coordinate[label=left:$\lvert a_1 \rvert$] (D)at ($(A)!0.7!(B)$);
					\coordinate[label=right:$\lvert a_2 \rvert$] (D)at ($(A)!0.7!(C)$);
					\draw[very thick] (B) -- (A)--(C);
					\fill (A) circle (4pt);
				\end{tikzpicture} }
			\end{minipage}%
\noindent\begin{minipage}{.55\linewidth}
				%\item[(IV)] 
$(IV)$ Point on fixed orbi-surface \\ {\small ($p=1$):}\\ \\
\resizebox{5cm}{!}{%
				\begin{tikzpicture}	\hspace{.8cm}		
						
					\coordinate[label=right:${\frac{1}{b_i}(1,s_i), \color{red} \alpha + \frac{m}{b_i}\varepsilon}$] (AA) at (-1.5,0.5);
					\coordinate[label=below:] (CC) at (-1.5,-0.5);
					\coordinate[label=below:${\hspace{1cm}\frac{1}{b_j}(1,s_j),A- \frac{1}{b_j}\varepsilon,g, \color{red} \alpha}$] (XX) at (-1.5,-.8);				
					\coordinate[label=right:] (DD) at (-1.5,1.5);
					\coordinate[label=right:$m$](X) at ($(AA)!0.7!(DD)$);
					\coordinate[label=right:$b_j$](X) at ($(AA)!0.5!(CC)$);
					\draw[very thick] (CC)--(DD);
					\fill (AA) circle (4pt);
					\fill (CC) circle (8pt);
					
					\draw[->, very thick] (-3,0) to (-2.5,0);

					\coordinate[label=below:] (A) at (-4,-0.5);
					\coordinate[label=below:${\frac{1}{m}(1,l),A,g,\color{red} \alpha}$] (X) at (-4.5,-0.8);			
					\coordinate[label=right:] (B) at (-4,1);
					\coordinate[label=left:$m$] (D) at ($(A)!0.7!(B)$);
					\draw[very thick] (A) -- (B);		
					\fill (A) circle (8pt);		
				\end{tikzpicture}}
			\end{minipage}	\\ \\
		%\end{enumerate}
together with cases $(II)$ -- $(IV)$ turned upside down\footnote{In
  all cases upside down, the
  signs of the terms involving $\varepsilon$ in the moment map labels are changed.}. Here
if $b_2|a_1|>b_1|a_2|$ and $a_1\neq 0$ then $i=1$ and $j=2$, and $i=2$ and $j=1$ otherwise.
\end{proposition}

\begin{remark}\label{rmk:desing_well_defined}
  If $a_1 = a_2 = \pm 1$ and $l^2 \neq 1 \mod m$, then $x$ has two
  labels for its type (see Section \ref{sec:label-mult-hamilt}): these
  are of the form $\frac{1}{m}(1,l)$ and $\frac{1}{m}(1,l')$. In
  this case, the integers
  $b_1,b_2,s_1,s_2$ and $b_1', b_2', s_1', s_2'$ of Proposition
  \ref{prop::desing} are related as follows:
  $$ b_1' = b_2 \, ,\, b_2' = b_1 \, ,\, s_1' = s_2 \, ,\, s_2' =
  s_1.$$
  In particular, there is no ambiguity in case $(II)$ above (this is
  the only one that may occur). To simplify the exposition, we do not
  mention this situation explicitly in the proof of Proposition \ref{prop::desing}.
\end{remark}

\begin{remark}\label{rmk:exc_divisor}
  In cases $(I)$, $(II)$ and $(IV)$, the exceptional divisor
  is an isotropy orbi-sphere and, thus, corresponds to the `extra'
  edge in the labeled multigraph. On the other hand, in case $(III)$,
  the exceptional divisor is a connected component of the fixed point
  set. Moreover,
  if $x$ is a singular point of type $\frac{1}{m}(1,l)$, the exceptional divisors of the weighted blow-ups 
  $$
  \text{WBU}\left(m,\frac{a_1}{p},\frac{a_2}{p},\frac{l'}{m},\frac{1}{m},\varepsilon\right) \quad \text{or} \quad \text{WBU}\left(m,\frac{a_1}{p},\frac{a_2}{p},\frac{1}{m},\frac{l}{m},\varepsilon\right),
  $$ 
  with $1\leq l' <m$ and $l \cdot  l' = 1 \mod m$, 
  have a unique singular point and the order of its orbifold structure group is smaller than $m$.	
\end{remark}

\begin{proof}[Proof of Proposition \ref{prop::desing}]
  First, without loss of generality, we may assume that, if $x$ is an
  extremal fixed point of $H$, then $H(x)$ is minimal. (The case of a
  maximum is entirely analogous.) Second, by construction, away from an open neighborhood of $x$ and of the
  exceptional divisor, the Hamiltonian $S^1$-space $(M,\omega,H)$ and
  its $S^1$-equivariant blow-up at $x$ given by
  $\text{WBU}\left(m,\frac{a_1}{p},\frac{a_2}{p},\frac{b_1}{m},\frac{b_2}{m},\varepsilon\right)$
  are isomorphic. Hence, by definition of the labeled multigraph of a 
  Hamiltonian $S^1$-space, it suffices to consider the effects of
  the above blow-up in an open neighborhood of $x$. Clearly, since we
  remove an open neighborhood of $x$, we are removing the
  corresponding fixed point (or a neighborhood thereof from the fixed
  minimal orbi-surface). Hence, the corresponding vertex (or
  label of the type of the singular point on a fat vertex) disappears from the labeled multigraph. Moreover, by
  Theorem \ref{cor::localformpt}, we may assume that such an open
  neighborhood is a suitable open neighborhood of $[0,0]_m$ in a local
  model $\C^2/\Z_m$ for a cyclic isolated singular point of type
  $\frac{1}{m}(1,l)$, that $x = [0,0]_m$ and that the moment map value
  of $x$ is $\alpha = 0$. Finally, to simplify the exposition, we assume
  that the above open neighborhood of $[0,0]_m$ equals $\C^2/\Z_m$. 

  We denote the circle that we use to perform the
  blow-up by $\widetilde{S}^1$. In order to perform the desired
  weighted blow-up at $[0,0]_m$, we consider the following
  $\widetilde{S}^1$-action:
  $$ \widetilde{\lambda} \cdot [z_1,z_2]_m =
  [\widetilde{\lambda}^{\frac{b_1}{m}}z_1,\widetilde{\lambda}^{\frac{b_2}{m}}z_2]_m.$$
  %on  a sufficiently small neighborhood $B_{\varepsilon+\delta}(x)$ containded in $U$  given by	
	% \begin{align*}
	% 	\widetilde{S}^1\times B_{\varepsilon+\delta}(x)&\rightarrow B_{\varepsilon+\delta}(x)\\
	% 	(\lambda,[z_1,z_2]_m)&\mapsto[\lambda^{\frac{b_1}{m}}z_1,\lambda^{\frac{b_2}{m}}z_2]_m,
	% \end{align*}
  where $b_1,b_2$ satisfy \eqref{eq:ebucond}. %and $b_1,b_2>0$.
  %	$$b_2 -b_1l = 0 \mod m \quad \text{and} \quad \gcd(b_1, b_2)=1.$$
  By Theorem~\ref{cor::localformpt}, this circle action is
  well-defined and effective. Moreover, it has a unique fixed
  point  $x=[0,0]_m$ and commutes with the original $S^1$-action
  that is given by
  \begin{equation}
    \label{eq:58}
    \lambda \cdot [z_1,z_2]_m =
     [\lambda^{\frac{a_1}{p}}z_1,\lambda^{\frac{a_2}{p}}z_2]_m. 
  \end{equation}
  The extended $\widetilde{S}^1$-action on $\C^2/\Z_m \times \C$ that we use to perform
  symplectic cutting is given by
  \begin{equation}
    \label{equ::blowupAction}
    \widetilde{\lambda} \cdot ([z_1,z_2]_m,w) = ([\widetilde{\lambda}^{\frac{b_1}{m}}z_1,\widetilde{\lambda}^{\frac{b_2}{m}}z_2]_m,\widetilde{\lambda}^{-1}w).
  \end{equation}
  It is Hamiltonian with moment map
  \begin{equation}\label{eq:mmwbu}
    \widetilde{H}([z_1,z_2]_m,w)=\frac{1}{2}\left(\frac{b_1}{m}\lvert z_1\rvert ^2+\frac{b_2}{m}\lvert z_2\rvert ^2-\lvert w\rvert^2\right).
  \end{equation}
  The exceptional divisor is given by 
  \begin{equation}
    \label{eq:53}
    \Sigma_0  := \left\{\left([z_1,z_2]_m,0\right)\in
      \widetilde{H}^{-1}(\varepsilon) \right \}/\,\, \widetilde{S}^1.
  \end{equation}
  By \eqref{equ::blowupAction} and Lemma~\ref{lem::isotropy},
  $\Sigma_0$ has two points with non-trivial orbifold structure group
  and these are $\mathbb{Z}_{b_1}$ and $\mathbb{Z}_{b_2}$,
  corresponding to
  \begin{equation}
    \label{eq:57}
    x_1 := \left(\left[\left(\frac{2 \varepsilon
            m}{b_1}\right)^{1/2},0\right]_m,0\right) \quad \text{and}
    \quad x_2:= \left(\left[0,\left(\frac{2 \varepsilon
            m}{b_2}\right)^{1/2}\right]_m,0\right)
  \end{equation}
  respectively. Moreover,
  since $\gcd(b_1l-b_2,m)= m$, the orbifold structure group
  $\mathbb{Z}_m$ of $[0,0]_m$ acts as a subgroup of the identity
  component of the extension  of $\widetilde{S}^1$ by
  $\mathbb{Z}_m$. By property \ref{item:24} in Section
  \ref{sec:sympl-weight-blow}, it follows that $\Sigma_0$ is
  diffeomorphic to $\C P^1(b_1,b_2)$.

  Next we calculate the order of the isotropy group for the
  $S^1$-action of the points on the exceptional divisor $\Sigma_0$. 
  Let $r\in \mathbb{Z}$ be such that $a_2- a_1l = r  p$. The original $S^1$-action descends to a neighborhood of the exceptional divisor in the weighted blow-up as
  \begin{equation}
    \label{eq:52}
    \begin{split}
      S^1\times \widetilde{H}^{-1}(\varepsilon)/\widetilde{S}^1&\rightarrow \widetilde{H}^{-1}(\varepsilon)/\widetilde{S}^1\\
      (\lambda,[z_1,z_2,w])&\mapsto[\lambda^{\frac{a_1}{p}}z_1,\lambda^{\frac{a_2}{p}}z_2,w],
    \end{split}
  \end{equation}
  and so
  \begin{align}\label{eq:actionbup}
    \lambda \cdot [z_1,z_2,w]&=[z_1,\lambda^{\frac{(a_2b_1-a_1b_2)}{p b_1}}z_2,\lambda^\frac{a_1k}{b_1}w]=[\lambda^{\frac{(a_1b_2-a_2b_1)}{p b_2}}z_1, z_2,\lambda^{\frac{a_2k}{b_2}}w].
  \end{align}
  Since 
  $$a_2=a_1l+ r  p, \quad b_2=b_1l-b m \quad \text{and} \quad m=kp,$$ 
  it follows that  
  $$\frac{\lvert a_1b_2-a_2b_1\rvert}{p}\in \Z_{\geq 0}.$$ 
  This is the desired order. Moreover, by \eqref{eq:53}, \eqref{eq:52}
  and \eqref{eq:actionbup}, the singular points on $\Sigma_0$ are
  fixed points of the $S^1$-action and the orbi-weights of the point
  corresponding to $x_1$ are
  \begin{equation}
    \label{eq:55}
    \frac{a_2b_1 - a_1b_2}{pb_1} \, , \, 
    \frac{a_1m}{pb_1}, 
  \end{equation}
  while those of the point
  corresponding to $x_2$ are
  \begin{equation}
    \label{eq:56}
    \frac{a_1b_2 - a_2b_1}{pb_2} \, , \, 
    \frac{a_2m}{pb_2}.
  \end{equation}
  If $ a_1b_2-a_2b_1\neq 0$,
  then the exceptional divisor is an isotropy orbi-sphere with
  singular points that have orbifold structure groups of order $b_1$
  and $b_2$ and these are the `poles' of the orbi-sphere. This
  occurs in cases $(I),(II)$ and $(IV)$, in which there is a
  `new' edge in the graph, which corresponds to the exceptional
  divisor (cf. Remark \ref{rmk:exc_divisor}). Moreover, if one of
  $a_1$ or $a_2$ vanishes, then precisely one of the two singular
  points has one zero orbi-weight, so that it lies on the minimal fixed
  orbi-surface (see case $(IV)$). If $ a_1b_2-a_2b_1= 0$,
  then the exceptional divisor is a component of the fixed point set. Finally, since
  $$
  \gcd{(a_1,a_2)}=\gcd{(b_1,b_2)}=1 \text{ and } b_1,b_2 >0,
  $$
  we have that
  $$b_1=\lvert a_1\rvert,  \quad b_2=\lvert a_2\rvert,  \quad
  a_1a_2>0 \quad \text{and} \quad p = m,$$ 
  i.e., we are in case $(III)$. \\

  % If we further  assume that $ a_1b_2-a_2b_1\neq 0$ then, through weighted blow-up, we  locally ``substitute'' the singular point  of order $m$ with an $S^1$-invariant orbi-sphere (the exceptional divisor) having singularities of order $b_1$ and $b_2$ at the poles.  In this case the labeled multigraph locally  changes as in $(I),(II)$ and $(IV)$ or these turned upside down.

Let us determine the type of the two singular points on
$\Sigma_0$. Since these are fixed points of the $S^1$-action, by
definition of the labeled multigraph, this determines the
corresponding labels of the multigraph. To this end, first note that $\gcd(\lvert b \rvert,b_1)=1$, since
$$
1 = \gcd(b_1,b_2)=\gcd(b_1,b_1l - b m)=\gcd (b_1,\lvert b \rvert m)
$$
and $\gcd(b_1,\lvert b \rvert )$  divides $\gcd(b_1 , \lvert b
\rvert  m)$. By \eqref{equ::blowupAction}, the action of
$\mathbb{Z}_{b_1}$ on $\C^2/\Z_m \times \C$ %on
% $Z_{\varepsilon,\delta}$
as a subgroup of $\widetilde{S}^1$ can be written as 
\begin{equation}
  \label{eq:54}
%  \label{eq:labels2}
\xi_{b_1}\cdot \left([z_1,z_2]_m,w\right)= ([z_1,\xi_{b_1}^{-b}z_2]_m,\xi_{b_1}^{-1}w).
\end{equation}
%the singular point of order $b_1$ in the exceptional divisor is an
%isolated singular point of the blow-up.
%Hence, if $\mathbb{Z}_{b_1}$ acts with weight $1$ on the fiber over the point, then it acts with weight $\beta$ on the direction tangent to $\Sigma_0$. 
Hence, if $\frac{1}{b_1}(1,s_1)$ is the type of the singular point
corresponding to $x_1$, then by \eqref{eq:55}, either
$s_1=b \mod b_1$ if $a_1m\geq a_2b_1-a_1 b_2$ or $s_1 b = 1
\mod b_1$. For the singular point that corresponds to $([0,1]_m,0)$,
first we show that that $\gcd(b_1\gamma+b l',b_2)=1$, where
$l',\gamma$ are as in the statement. To this end, we set $d:=
\gcd(b_1\gamma+b l',b_2)$ and write $b_2=b_2^\prime \, d$ and
$b_1\gamma+b l'=c\, d$ for some integers $b_2^\prime$ and
$c$. Then 
\begin{align*}
  c\, d=b_1\left(\frac{1-ll'}{m}\right)-l'\left(\frac{b_2-b_1l}{m}\right)=\frac{b_1-b_2l'}{m},
\end{align*}
so that $d$ divides $b_1-b_2l'$. Since $b_2=b_2^\prime d$, it follows
that $d$ divides $b_1$. Hence, since $\gcd(b_1,b_2)= 1$, then $d=1$.
In analogy with \eqref{eq:53}, the action of  $\mathbb{Z}_{b_2}$ on
$\C^2/\Z_m \times \C$ as a subgroup of $\widetilde{S}^1$ can be written as
\begin{equation}\label{eq:labels2}
e^{\frac{2\pi i}{b_2}}\cdot\left([z_1,z_2]_m,w\right) =
([e^{\frac{2\pi i(b l' +b_1\gamma) }{b_2}}z_1,z_2]_m,e^{-\frac{2\pi i}{b_2}}w).
%		&=\left([e^{\frac{2\pi ib_1}{mb_2}}z_1,e^{\frac{2\pi ib_2}{mb_2}}z_2]_m,e^{-\frac{2\pi i}{b_2}}w\right) \\ &
%		=\left([e^{\frac{2\pi i(b_1lx+b_1my)}{mb_2}}z_1,e^{\frac{2\pi i(lx+my)}{m}}z_2]_m,e^{-\frac{2\pi i}{b_2}}w\right)\nonumber\\ 
%		&=\left([e^{\frac{2\pi i\left((b_2+\beta m)x+b_1my\right)}{mb_2}}z_1,e^{\frac{2\pi ilx}{m}}z_2]_m,e^{-\frac{2\pi i}{b_2}}w\right) \\&
%		=\left([e^{\frac{2\pi ix}{m}+\frac{2\pi i (b_1y+\beta x)}{b_2}}z_1,e^{\frac{2\pi ilx}{m}}z_2]_m,e^{-\frac{2\pi i}{b_2}}w\right)\nonumber\\
%		&=([e^{\frac{2\pi i(b_1y+\beta x) }{b_2}}z_1,z_2]_m,e^{-\frac{2\pi i}{b_2}}w).	
%	\end{align}
\end{equation}
%the singular point of order $b_2$ in the exceptional divisor is an
%isolated singular point of the blow-up.
%Moreover,
%there exist positive integers $\tilde{s},\tilde{t}$ such that $1=b_2\tilde{t}-(b_1y+\beta x)\tilde{s}$. 
%We easily see from  \eqref{eq:labels2} that if $\mathbb{Z}_{b_2}$ acts with weight $1$ on the fiber over the point, then it acts with weight $-(b_1y+\beta x)$ on the direction tangent to $\Sigma_0$. 
Hence, if $\frac{1}{b_2}(1,s_2)$ is the type of the singular point
corresponding to $x_2$, then by \eqref{eq:56}, either
$s_2=-(b l' + b_1\gamma) \mod b_2$ if $a_2m\geq a_1b_2-a_2 b_2$,
or $-(b l'+ b_1\gamma)  s_2 = 1 \mod b_2$. \\

The moment map value of the singular points on $\Sigma_0$ can be
obtained easily by combining \eqref{eq:mmwbu} and the homogeneous
moment map for the $S^1$-action on $\C^2/\Z_m$ of
\eqref{eq:58}. These agree with those given in the statement. Hence,
this completes the proof of cases $(I)$ and $(II)$. \\

To complete the proof of cases $(III)$ and $(IV)$, by definition of
the labeled multigraph, we need to
determine the symplectic invariants of the (possibly new) fixed
minimal orbi-surface (see Theorem
\ref{thm:classification_symplectic_orbi-surfaces}). In Case $(III)$,
the fixed orbi-surface is `created' by the weighted blow-up and, as
shown above, it has genus zero and precisely two singular points whose
type we have already determined. In Case $(IV)$, there is a fixed
minimal orbi-surface both before and after the weighted blow-up.
Topologically, what this operation does is to remove a disk from the
fixed orbi-surface and to replace it with another disk, i.e., the
genus is unaffected. Moreover, the new label of the (possibly)
singular point on the fixed orbi-surface can be calculated as
above.

Hence, it remains to calculate the area of the fixed
orbi-surface after the blow-up in both cases. We do this using the
Localization Formula for equivariant cohomology (see
Theorem~\ref{prop:localization}): More explicitly, we compute the
integrals of the  equivariant form $1$ and of the equivariant
symplectic form before and after the blow-up. First, we deal with case
$(III)$: 
\begin{align*}
		0& =\int_{\widetilde{M}_{x,\varepsilon}} 1 =\int_{M} 1 -  \frac{m}{ a_1 a_2\, y^2} +  \int_{\Sigma_{0}} \frac{1}{\beta_{0}u_0 + y}
		\\ & =  -  \frac{m}{ a_1 a_2\, y^2} + \frac{1}{y} \int_{\Sigma_{0}} \sum_{j=0}^\infty \left(-\frac{\beta_{0}u_0}{y}\right)^j =\frac{1}{y^2}\left(-  \frac{m}{ a_1 a_2}  - \beta_{0} \right), 	
\end{align*}
where $u_0\in H^2(\Sigma_0;\mathbb{Z})$ is a generator and $\beta_0$
is the degree of the normal orbi-bundle of $\Sigma_0$ in
$\widetilde{M}_{x,\varepsilon}$. Hence, $\beta_{0}=  -  \frac{m}{ a_1 a_2}$. Moreover,
\begin{equation}
  \label{eq:59}
  \begin{split}
    0& =\int_{\widetilde{M}_{x,\varepsilon}}\tilde{\omega}^{S^1} =\int_{M}\omega^{S^1} + \frac{m^2 \, H(x) }{m\, a_1 a_2\, y} +  \int_{\Sigma_{0}} \frac{\iota_{\Sigma_0}\tilde{\omega}^{S^1}}{\beta_{0}u_{0} + y}
    \\ &= \int_{\Sigma_{0}} \frac{\iota_{\Sigma_0}\tilde{\omega}^{S^1}}{\beta_{0}u_{0} + y}
    % = \frac{\alpha\, m}{a_1 a_2 y} +  \int_{\Sigma_{0}}\left(\sum_{j=0}^\infty \left(-\frac{b_{\Sigma_{0}}u_0}{y}\right)^j\right)\left(\frac{\tilde{\omega}}{y}-H\right)\\
    =\frac{1}{y}\left(area(\Sigma_{0})+ \beta_{0} \varepsilon\right),
    % \\ & = \frac{1}{y}\left(\frac{\alpha\, m}{ a_1 a_2} +area(\Sigma_{0}) -  \frac{m}{ a_1 a_2}(\alpha + \varepsilon) \right) =  \frac{1}{y}\left( area(\Sigma_{0})-  \frac{m}{a_1 a_2} \varepsilon \right),
  \end{split}
\end{equation}
where we use that $\alpha = H(x) = 0$ by assumption (cf. also
\eqref{eq:47}). Hence, the area of the exceptional divisor
$\Sigma_{0}$ is $\frac{m}{a_1 a_2} \varepsilon$. This completes the
proof of case $(III)$.

Finally, in case $(IV)$, we denote by $\Sigma$ and
$\widetilde{\Sigma}$ the fixed orbi-surfaces before and after the
weighted blow-up. Correspondingly, we use the notation $u$ and
$\widetilde{u}$ for the generators in cohomology, and $\beta$ and
$\widetilde{\beta}$ for the degrees of the normal bundles. Then, 
\begin{align*}
		0& =\int_{\widetilde{M}_{x,\varepsilon}} 1 =\int_{M} 1
                   -  \int_{\Sigma} \frac{1}{\beta u + y} +  \int_{\widetilde{\Sigma}} \frac{1}{\widetilde{\beta}\widetilde{u} + y} - \frac{b_i}{y^2m \,b_j}
%	\\ & =   0 -   \frac{1}{y} \int_{\Sigma} \sum_{j=0}^\infty \left(-\frac{b_{\Sigma}u_\Sigma}{y}\right)^j +  \frac{1}{y} \int_{\tilde{\tilde{\Sigma}}} \sum_{j=0}^\infty \left(-\frac{b_{\tilde{\Sigma}}u_{\tilde{\Sigma}}}{y}\right)^j  - \frac{b_i}{y^2m \,b_j}
	 =\frac{1}{y^2}\left(\beta - \widetilde{\beta}  - \frac{b_i}{m \, b_j} \right).
\end{align*}
Hence, $\widetilde{\beta} - \beta = - \frac{b_i}{m \,
  b_j}$. Analogously, since the moment map value of the new isolated
fixed point is $\frac{m \varepsilon}{b_i}$, we have that

\begin{align*}
		0&
                   =\int_{\widetilde{M}_{x,\varepsilon}}\widetilde{\omega}^{S^1}
                   =\int_{M}\omega^{S^1} -  \int_{\Sigma}
                   \frac{\iota_{\Sigma}\omega^{S^1}}{\beta u + y} +
                   \int_{\widetilde{\Sigma}}
                   \frac{\iota_{\widetilde{\Sigma}}\tilde{\omega}^{S^1}}{\widetilde{\beta}
                   \widetilde{u} + y} 
		+ \frac{1}{b_j\, y}\varepsilon\\ 
		%& =-  \int_{\Sigma} \frac{\omega - H^M y}{b_{\Sigma}u_\Sigma + y} +  \int_{\tilde{\Sigma}}\frac{\tilde{\omega} - H^{\tilde{M}_x} y}{b_{\tilde{\Sigma}}u_{\tilde{\Sigma}}+y}  + \frac{b_i \, H^{\tilde{M}_x} (z)}{m\,b_j\, y}  \\ 
%		& =  - \int_{\Sigma} \left(\sum_{j=0}^\infty \left(-\frac{b_{\Sigma}u_\Sigma}{y}\right)^j\right)\left(\frac{\omega}{y}-H^M\right) + \int_{\tilde{\Sigma}} \left(\sum_{j=0}^\infty \left(-\frac{b_{\tilde{\Sigma}}u_{\tilde{\Sigma}}}{y}\right)^j\right)\left(\frac{\tilde{\omega}}{y}-H^{\tilde{M}}\right)  +\frac{b_i \, H^{\tilde{M}_x} (z)}{m\,b_j\, y}  \\
		%&=\frac{1}{y}\left(area(\tilde{\Sigma}) - area(\Sigma) + b_{\tilde{\Sigma}} H_{\text{min}}^{\tilde{M}_x}  - b_{\Sigma} H_{\text{min}}^{M}  + \frac{b_i \, H^{\tilde{M}_x} (z)}{m\,b_j\, y} \right)  \\ & = \frac{1}{y}\left( area(\tilde{\Sigma}) - area(\Sigma)  + \alpha (b_{\tilde{\Sigma}}   - b_{\Sigma} )+ \frac{b_i }{m\,b_j} (\alpha +\frac{m}{b_i}\varepsilon) \right) 
	&=  \frac{1}{y}\left( area(\tilde{\Sigma}) - area(\Sigma)  + \frac{1}{b_j} \varepsilon \right).
\end{align*}
Hence, $area(\widetilde{\Sigma} )= area(\Sigma)  - \frac{1}{b_j}
\varepsilon$, as desired.
%Taking the primitive $b_2$-th root of unity
%	$$\xi_{b_2}:=e^{\frac{2\pi i(b_1y+\beta x)}{b_2}},$$ 
%we have from \eqref{eq:labels2} that  a neighborhood of the singularity of order $b_2$ can be modeled by $\mathbb{C}^2/\mathbb{Z}_{b_2}$, where 	
%\begin{align}
%	(z_1,w)\sim (\xi_{b_2}z_1, \xi_{b_2}^{\tilde{s}}w),
%	\end{align}
%	and so, by \eqref{eq:actionbup}, the vertex corresponding to the singularity of order $b_2$ in  $\widehat{H}^{-1}(\varepsilon)/\widehat{S}^1$ is labeled by $\frac{1}{b_2}(1,s_2)$, where $$1\leq s_2 <b_2\quad \text{and} \quad s_2 = \tilde{s} \mod b_2$$ 
%	(i.e. $(\beta x+b_1y) s_2 = -1 \mod b_2$), if $a_2m < a_1b_2-a_2b_1$. Otherwise $s_2= - (\beta x+b_1y) \mod b_2$. 
\end{proof}

\begin{remark}\label{rem::weightedblowupConnectsum}
  As proved in \cite[Lemma 5.1]{lgodinho} and in  \cite[page
  3]{desing}, topologically each weighted blow-up described in
  Proposition~\ref{prop::desing} is the connect sum of $M$ with a
  suitable weighted projective space with reverse orientation. By
  construction, in each case the orders of the orbifold structure
  groups of the singular points on the resulting exceptional
  divisors are coprime.
\end{remark}

The proof of the following result, which deals with the second family
of weighted blow-ups, is entirely analogous to that of Proposition
\ref{prop::desing}. For brevity, we omit it.

\begin{proposition}\label{prop::desing2}
Let $(M,\omega,H)$ be a Hamiltonian $S^1$-space and let $x\in M$ be a singular point of type $\frac{1}{m}(1,l)$. Let $1 \leq l' < m$ be the integer such that $l
  \cdot l' = 1 \mod m$ and let $\gamma \in \Z$ be such that $ll'
  + m \gamma = 1$. Let $\frac{a_1}{p},\frac{a_2}{p}$ be the
  orbi-weights of the $S^1$-action at $x$, where $\gcd (a_1,a_2)=1$, $a_1\geq a_2$ and
  $\gcd(\lvert a_1l - a_2\rvert, m)=p<m$. We set $k:=m/p$. Let
  $b_1,b_2>0$ and $b$ be integers such that
  \begin{equation}\label{eq:Prop}
    \gcd(b_1,b_2)=1, \quad b_1l-b_2=b p \quad \text{and} \quad
    \gcd(b, k)=1.
  \end{equation}
  Let $s_1$ be the integer that satisfies $1\leq s_1 < b_1 k$ and
  $s_1=b \! \!\!\mod b_1 k$ if  $a_1p\geq a_2b_1-a_1b_2$, or $s_1
  b= 1 \! \!\!\mod b_1 k$ otherwise. Let $s_2$ be the integer that
  satisfies $1\leq s_2 < b_2 k$ and $s_2 = - (b l' + b_1 k \gamma)\! \!\!\mod b_2  k$
  if  $a_2 p\geq a_1b_2 - a_2b_1$, or $(b l'+b_1 k \gamma)s_2=
  -1\! \!\!\mod b_2 k$ otherwise.

  An $S^1$-equivariant weighted blow-up of $(M,\omega)$ at $x$
  $$
  \text{WBU}\left(m,\frac{a_1}{p},\frac{a_2}{p},\frac{b_1}{p},\frac{b_2}{p},\varepsilon\right)
  $$
  results in the following `local' changes in the labeled multigraph of
  $(M,\omega,H)$ near the vertex that corresponds to the connected
  component of
  the fixed point set to which $x$ belongs: \\ \\
  % \vspace{.2cm}		
  % \begin{enumerate}			
  \noindent\begin{minipage}{.55\linewidth}
    $(I)$ Non-extremal fixed point: \\  \\ 
    \resizebox{6cm}{!}{%
				\begin{tikzpicture}		
					\coordinate[label=right:] (BB) at (-1.5,1.3);
					\coordinate[label=right:${\,\,\frac{1}{b_1 k}(1,s_1), \color{red} \alpha + \frac{\lvert a_1\rvert}{b_1}\varepsilon}$] (AA) at (-2,0.7);
					\coordinate[label=right:${\,\, \frac{1}{b_2 k}(1,s_2), \color{red} \alpha - \frac{\lvert a_2\rvert}{b_2}\varepsilon }$] (CC) at (-2,-0.5);
					\coordinate[label=right:] (DD) at (-1.5,-1.1);
									
					\coordinate[label=left:$\lvert a_1\rvert k$](X) at ($(AA)!1!(BB)$);
					\coordinate[label=left:$\lvert a_2\rvert k$](X) at ($(CC)!1!(DD)$);
					\coordinate[label=left:$\frac{\lvert a_1b_2-a_2b_1\rvert k}{p}$](X) at ($(AA)!0.5!(CC)$);
					
					\draw[very thick] (BB)--(AA)--(CC)--(DD);
					\fill (AA) circle (4pt);
					\fill (CC) circle (4pt);
					
					\hspace{-.5cm}\draw[->, very thick] (-4,0) to (-3.5,0);
					
					\coordinate[label=left:${\frac{1}{m}(1,l),\, \color{red} \alpha\,\,}$] (A) at (-4.5,0);
					\coordinate[label=right:] (B) at (-4,1);
					\coordinate[label=right:] (C) at (-4,-1);
					\coordinate[label=left:$\lvert a_1\rvert k$] (D)at ($(A)!0.7!(B)$);
					\coordinate[label=left:$\lvert a_2\rvert k$] (D)at ($(A)!0.7!(C)$);
					\draw[very thick] (B) -- (A)--(C);
					\fill (A) circle (4pt);
				\end{tikzpicture}}
			\end{minipage}%
\noindent\begin{minipage}{.55\linewidth}
$(II)$ Isolated extremum {\small ($a_1b_2\neq a_2 b_1$):} \\ \\				
\resizebox{6cm}{!}{%
				\begin{tikzpicture}				
					\coordinate[label=right:] (BB) at (-1,1);
					\coordinate[label=above:${\hspace{3cm}\frac{1}{b_i k}(1,s_i),   \color{red}  \alpha \!+\! \frac{\lvert a_i\rvert}{b_i}\varepsilon}$] (AA) at (-0.7,0.5);
					\coordinate[label=below:${\,\,\,\frac{1}{b_j k}(1,s_j), \color{red}  \alpha + \frac{\lvert a_j \rvert}{b_j}\varepsilon  }$] (CC) at (0,-0.5);
					\coordinate[label=right:] (DD) at (0.7,0.5);

					\coordinate[label=left:$\lvert a_i\rvert k$](X) at ($(AA)!1!(BB)$);
					\coordinate[label=right:$\lvert a_j\rvert k$](X) at ($(CC)!0.7!(DD)$);
					\coordinate[label=left:$\frac{\lvert a_1b_2-a_2b_1\rvert k}{p}$](X) at ($(AA)!0.7!(CC)$);
					
					\draw[very thick] (BB)--(AA)--(CC)--(DD);
					\fill (AA) circle (4pt);
					\fill (CC) circle (4pt);
					
					\draw[->, very thick] (-2.5,0.2) to (-2,0.2);
					
					\coordinate[label=below:${\frac{1}{m}(1,l), \color{red} \alpha}$] (A) at (-4,-0.5);
					\coordinate[label=right:] (B) at (-4.7,0.5);
					\coordinate[label=right:] (C) at (-3.3,0.5);
					\coordinate[label=left:$\lvert a_i\rvert k$] (D)at ($(A)!0.7!(B)$);
					\coordinate[label=right:$\lvert a_j\rvert k$] (D)at ($(A)!0.7!(C)$);
					\draw[very thick] (B) -- (A)--(C);
					\fill (A) circle (4pt);
				\end{tikzpicture}}
			\end{minipage}	\\	\\
			
\noindent\begin{minipage}{.55\linewidth}
$(III)$ Isolated extremum  {\small ($b_i=\lvert a_i\rvert$):} \\ \\
\resizebox{6cm}{!}{%
				\begin{tikzpicture}				
					\coordinate[label=left:] (AA) at (-1.7,0.5);
					\coordinate[label=below:] (CC) at (-1,-0.5);
					\coordinate[label=below:${\hspace{.5cm}\frac{1}{\lvert a_1 \rvert k}(1,s_1),\frac{1}{\lvert a_2\rvert k}(1,s_2)}$] (XX) at (-1,-0.9);
					\coordinate[label=below:${\hspace{.5cm}A=\frac{p^2}{a_1a_2 m}\varepsilon,g=0, \color{red} \alpha + \varepsilon}$] (XX) at (-1,-1.5);
					\coordinate[label=right:] (DD) at (-0.3,0.5);

					\coordinate[label=right:$\lvert a_2 \rvert k$](X) at ($(CC)!0.7!(DD)$);
					\coordinate[label=left:$\lvert a_1\rvert k$](X) at ($(AA)!0.3!(CC)$);
					
					\draw[ very thick] (AA)--(CC)--(DD);
					\fill (CC) circle (8pt);
					
					\draw[->, very thick] (-2.8,-0.2) to (-2.3,-0.2);
					
					\coordinate[label=below: ${\frac{1}{m}(1,l), \color{red} \alpha}$] (A) at (-4,-0.5);
					\coordinate[label=right:] (B) at (-4.7,0.5);
					\coordinate[label=right:] (C) at (-3.3,0.5);
					\coordinate[label=left:$\lvert a_1\rvert k$] (D)at ($(A)!0.7!(B)$);
					\coordinate[label=right:$\lvert a_2 \rvert k$] (D)at ($(A)!0.7!(C)$);
					\draw[very thick] (B) -- (A)--(C);
					\fill (A) circle (4pt);
				\end{tikzpicture}}
			\end{minipage}%
			\noindent\begin{minipage}{.55\linewidth}
$(IV)$ Point on fixed orbi-surface \\ {\small ($p=1$):}\\ \\
	\resizebox{6cm}{!}{%
				\begin{tikzpicture}	\hspace{-.2cm}		
						
					\coordinate[label=right:${\frac{1}{b_i m}(1,s_i), \color{red} \alpha + \frac{a_i}{b_i}\varepsilon}$] (AA) at (-1.5,0.5);
					\coordinate[label=below:] (CC) at (-1.5,-0.5);
					\coordinate[label=below:${\hspace{1cm}\frac{1}{b_j m}(1,s_j),A- \frac{1}{mb_j} \varepsilon,g, \color{red} \alpha }$] (XX) at (-1,-.8);				
					\coordinate[label=right:] (DD) at (-1.5,1.5);
					\coordinate[label=right:$m$](X) at ($(AA)!0.7!(DD)$);
					\coordinate[label=right:$b_j m$](X) at ($(AA)!0.5!(CC)$);
					\draw[very thick] (CC)--(DD);
					\fill (AA) circle (4pt);
					\fill (CC) circle (8pt);
					
					\draw[->, very thick] (-3,0) to (-2.5,0);

					\coordinate[label=below:] (A) at (-4,-0.5);
					\coordinate[label=below:${\frac{1}{m}(1,l),A,g,\color{red} \alpha}$] (X) at (-4,-0.8);			
					\coordinate[label=right:] (B) at (-4,1);
					\coordinate[label=left:$m$] (D) at ($(A)!0.7!(B)$);
					\draw[very thick] (A) -- (B);		
					\fill (A) circle (8pt);		
				\end{tikzpicture}}
			\end{minipage} 	\\ \\
together with cases $(II)$-$(IV)$ turned upside down\footnote{In
  all cases upside down, the
  signs of the terms involving $\varepsilon$ in the moment map labels are changed.}. Here
$i=1$ and $j=2$ if $b_2|a_1|>b_1|a_2|$ and $a_1\neq 0$, and $i=2$ and
$j=1$ otherwise.
\end{proposition}

\begin{remark}\label{rmk:other_blow-up}
  Topologically, each weighted blow-up described in
  Proposition~\ref{prop::desing} is the connect sum of $M$ with a
  suitable quotient of a weighted projective space by a finite  cyclic
  group with reverse orientation (cf. Remark
  \eqref{rem::weightedblowupConnectsum}). We remark that $p < m$ in all cases
  considered in Proposition \eqref{prop::desing2}. 
    % The weighted blow-ups described in Proposition~\ref{prop::desing} correspond to  connect sums of $M$ with  a suitable weighted projective space with reverse orientation. The 
  % In some situations, however,  weighted blow-ups may correspond to  connect sums of $M$ with  a suitable quotient of a weighted projective space by a finite  cyclic group, with reverse orientation. These are described in Proposition~\ref{prop::desing2} and occur when the orbi-weights of the circle action that is used to cut and those of the original $S^1$-action on $M$ have the same denominator $p$ {\bf strictly smaller} than the order $m$ of the orbifold structure group of the blown-up point. The orders of the singular points on the exceptional divisors obtained in these cases always  have a common denominator.
  % % In all cases we can see  weighted symplectic blow-ups as being obtained by removing the orbi-ball $B_{\varepsilon}(x)$ around $x$ and collapsing the orbits of the weighted $S^1$-action on the boundary of the resulting orbifold.\\
  % To desingularize our Hamiltonian $S^1$-spaces we only need to use the  weighted blow-ups described  in Proposition~\ref{prop::desing}. However, to obtain the minimal models in Theorems~\ref{fig::minimalfixedsurface} and \ref{thm:existisolated}, we may need to blow-down $S^1$-invariant orbi-spheres with singular points whose orders have   a common divisor. In these cases, we have to use the reverse constructions of the second type of weighted blow-ups, which are described in Proposition~\ref{prop::desing2}.
  % \hfill$\diamond
\end{remark}

By Propositions
\ref{prop::desing} and \eqref{prop::desing2} there are clear
restrictions on the size of an $S^1$-equivariant weighted blow-up,
which are entirely analogous to those that occur in the smooth case
(see \cite[Definition 7.1]{karshon}): Namely, the moment map labels
need to stay strictly monotone along `chains' of edges and the area
labels need to stay positive. We call these
blow-ups {\bf admissible}. \\

We conclude this section by stating a result that we need in Section
\ref{ChapterClassification}. It keeps track of the change in the
degree of a fixed orbi-surface after performing a blow-up as in case
$(IV)$ in either of the above propositions and is an immediate consequence of
\eqref{eq:48} and \eqref{eq:bsigma-}. Moreover, we emphasize that the
notation below is different to that of Propositions \ref{prop::desing}
and \ref{prop::desing2}.

\begin{corollary}\label{lem::blowupchange0}
Let $(M,\omega,H)$ be  a Hamiltonian $S^1$-space with a fixed
orbi-surface $\Sigma$ %with normal orbi-bundle $\nu_\Sigma$
and let $x\in\Sigma$ be a singular point with orbifold structure group
of order $m$.  Let $(\widetilde{M},\widetilde{\omega},\widetilde{H})$
be the Hamiltonian $S^1$-space obtained from $(M,\omega,H)$ by
performing a weighted blow-up at $x$ as in case $(IV)$ of either
Proposition~\ref{prop::desing} or \ref{prop::desing2}. Let the orders
of the orbifold structure groups of the two singular points on the
exceptional divisor be $p$ and $q$, where the former corresponds to
the singular point lying on the fixed orbi-surface
$\widetilde{\Sigma}$ in
$(\widetilde{M},\widetilde{\omega},\widetilde{H})$. Let $\beta,
\widetilde{\beta}$ denote the degrees of the normal bundles of
$\Sigma, \widetilde{\Sigma}$ in $(M,\omega,H)$ and
$(\widetilde{M},\widetilde{\omega},\widetilde{H})$ respectively. Then
\begin{equation}
  \label{eq:60}
  \widetilde{\beta} - \beta = -\frac{q}{m p}.
\end{equation}

\end{corollary}
\subsubsection{Desingularization of Hamiltonian $S^1$-spaces}
We may proceed to the proof of our first main result using the
description of the effects of the families of $S^1$-equivariant
weighted blow-ups in Propostion~\ref{prop::desing}.

\begin{theorem}[Desingularization]\label{thm:desingularization}
  Let $(M,\omega,H)$ be a Hamiltonian $S^1$-space. After applying at
  most finitely many $S^1$-equivariant weighted blow-ups to
  $(M,\omega,H)$, the resulting Hamiltonian $S^1$-space has no
  singular point.
  % \vspace{2cm}
  % Given a  Hamiltonian $S^1$-space $(M,\omega,H)$, there exists an
  % $S^1$-equivariant desingularization, which transforms the space into
  % a compact 4-dimensional Hamiltonian $S^1$-manifold. 
\end{theorem}

\begin{proof}
  If $(M,\omega,H)$ has no singular points, there is nothing to
  prove. Otherwise, since $M$ has isolated singular points and is
  compact, the singular set of $M$ is finite and non-empty. Let $x \in
  M$ be a singular point that has orbifold structure group with
  minimal order among all singular points. Let $\frac{1}{m}(1,l)$ and
  $a_1/p, a_2/p$ be the type and the orbi-weights of $x$
  respectively. Let $1\leq m'<m$ be the unique integer satisfying $m' l =
  1 \mod m$.

  By Corollary \ref{cor:weighted}, for some $\varepsilon >
  0$ sufficiently small, the
  $S^1$-equivariant weighted blow-up at $x$ 
  $$\text{WBU}\left(m,\frac{a_1}{p},\frac{a_2}{p},\frac{m'}{m},\frac{1}{m},\varepsilon
  \right)$$
  yields a new Hamiltonian $S^1$-space that we denote by $(M',
  \omega', H')$. This is one of the blow-ups considered in Proposition
  \ref{prop::desing}. More precisely,
  \begin{itemize}[leftmargin=*]
  \item if $x$ is not extremal, then the above blow-up is
    of type $(I)$;
  \item if $x$ is extremal, isolated and
    $a_1 \neq a_2 m'$ holds, then the above blow-up is
    of type $(II)$;
  \item if $x$ is extremal, isolated and $a_1 = a_2 m'$, then the above blow-up is
    of type $(III)$ -- cf. \eqref{eq:55} and \eqref{eq:56};
  \item if $x$ is extremal and not isolated, then the above blow-up is
    of type $(IV)$. 
  \end{itemize}
  By Proposition \ref{prop::desing} and Remark \ref{rmk:exc_divisor}, the cardinality of the singular set of $(M',
  \omega', H')$ is at most that of $(M,\omega,H)$. Moreover, if the
  two cardinalities are equal, then the minimal order of the
  orbifold structure group occurring in the singular set of $(M',
  \omega', H')$ is smaller that $m$. Since both the cardinality
  of the singular set of $(M,\omega,H)$ and $m$ are finite, the result follows by iterating
  the above procedure finitely many times.
\end{proof}

\subsubsection*{Intermezzo: Relation to Hirzebruch-Jung continued fractions and minimal resolutions of cyclic singularities in complex algebraic geometry}
%\subsection{Hirzebruch-Jung continued fractions}\label{sec::HJ}
%\mbox{}
%\todo{23/7/18: I changed Definition \ref{def::quotientSingularity} a
%  bit and so the references below should be checked accordingly. Maybe
%  Definition \ref{def:model_singular_point} is more appropriate in
%  some places.}
%
%The theory of cyclic quotient singularities on complex surfaces $\mathbb{C}^2/\mathbb{Z}_m$ has  been extensively studied in algebraic geometry (see, for example,  \cite{reid,cox,popescu}). These singularities are locally classified by two integers $m,l$ as in Definition $\ref{def::quotientSingularity}$.
%

A closer look at the algorithm for desingularizing Hamiltonian $S^1$-spaces outlined in the proof of Theorem \ref{thm:desingularization} reveals a connection with the well-known theory of resolutions of cyclic singularities of complex dimension two using the so-called Hirzebruch-Jung continued fractions (see \cite{reid}). First, we introduce these continued fractions. 

Let $1 \leq m_1 < m_0$ be integers such that $\gcd(m_0,m_1) = 1$. By the
Euclidean algorithm, there exist unique integers $c_1 \geq 2$ and $1
\leq m_2 < m_1$ such that $m_0 = c_1m_1 - m_2$. In particular,
\begin{align}\label{equ_hj_recursion}
	0<c_1-\frac{m_0}{m_1}<1;
\end{align}
moreover, since $\gcd(m_0,m_1) = 1$, then $\gcd(m_1,m_2) = 1$. Hence, we
may iterate this process finitely many times to obtain finite
sequences of positive integers $(m_i)^{n}_{i=0}$ and $(c_i)_{i=1}^n$
such that
\begin{itemize}[leftmargin=*]
\item $1 \leq m_{i+1} < m_i$ and $\gcd(m_i,m_{i+1}) = 1$ for all
  $i=0,\ldots, n-1$,
\item $c_i \geq 2$ for all $i=1,\ldots, n$, and
\item $m_i = c_i m_i - m_{i+1}$ for all $i=0,\ldots, n-2$ and $m_{n-1}
  = c_n m_n = c_n$. 
\end{itemize}
By construction, for all $i=0,\ldots, n-1$,
\begin{equation}\label{eq:HJCF}
  \frac{m_i}{m_{i+1}}=c_{i+1}-\frac{1}{c_{i+2}-\frac{1}{\ddots
      -\frac{1}{c_n}}}=:[c_{i+1}, \ldots, c_n].
\end{equation}

\begin{definition}\label{def:HJCF}
  Given integers $1 \leq m_1 < m_0$ with
  $\gcd(m_0,m_1) = 1$, the continued fraction of \eqref{eq:HJCF} with
  $i=0$ is the \textbf{Hirzebruch-Jung continued fraction} of $m_0/m_1$.
\end{definition}

Suppose that we run the algorithm of the proof of Theorem \ref{thm:desingularization}: Let $x_0$ be a singular point of a Hamiltonian $S^1$-space $(M_0,\omega_0,H_0)$ and let $\frac{1}{m_0}(1,l_0)$ be the type of $x_0$. Let $\frac{a_1}{p} \geq \frac{a_2}{p}$ be the orbi-weights of $x_0$. We denote by $m_1$ the unique integer such that $1 \leq m_1 < m_0$ and $m_1 l_0 = 1 \mod m_0$. We perform an  $S^1$-equivariant weighted blow-up at $x$ 
  $$\text{WBU}\left(m_0,\frac{a_1}{p},\frac{a_2}{p},\frac{m_1}{m_0},\frac{1}{m_0},\varepsilon
  \right)$$
for some $\varepsilon > 0$ sufficiently small, thereby obtaining a new Hamiltonian $S^1$-space $(M_1,\omega_1,H_1)$. We denote by $x_1$ the (possibly) singular point in $(M_1,\omega_1,H_1)$ that the above blow-up generates and denote its type by $\frac{1}{m_1}(1,l_1)$. Since $a_1 m_0 \geq a_2 m_1 - a_1$, by Proposition \ref{prop::desing}, the integer $l_1$ is equal to $b$ modulo $m_1$, where $b$ is the unique integer such that $m_1 l_0 - 1 = bm_0$. 

Let $[c_1,\ldots, c_n]$ denote the Hirzebruch-Jung continued fraction of $m_0/m_1$. Since $l_1 = b \mod m_1$, by definition of $b$,
$$(c m_1  - m_0)l_1 = 1 \mod m_1$$
for any integer $c$. In particular, the above equality holds for $c_1$ and so by \eqref{equ_hj_recursion} the integer $m_2 = c_1 m_1 - m_0$ is the unique one such that $1 \leq m_2 < m_1$ and $m_2 l_1 = 1 \mod m_1$. By iterating this argument, we see that the sequence of integers $(m_i)_{i=0}^n$ given by the order of the cyclic orbifold structure groups of the singular points constructed by running the algorithm of the proof of Theorem \ref{thm:desingularization} is precisely that of the Hirzebruch-Jung continued fraction of $m_0/m_1$!

\begin{example}\label{tswEx2} 
We present here two  equivariant desingularizations of the Hamiltonian $S^1$-space $(M,\omega,H)$ of Example \ref{tswEx}.

%Following Theorem \ref{thm:desingularization}, we apply a sequence of weighted blow-ups as described, in order to obtain a smooth $S^1$-manifold.
We consider first the singular point of of type  $\frac{1}{4}(1,3)$ where the moment map attains its minimum value. Since $l=3$ and $3 \cdot 3 =  1 \mod 4$, we consider $m_1=3$  and 
the Hirzebruch-Jung continued fraction $\frac{4}{3}=[2,2,2]$. We obtain the sequence 
	\begin{align*}
		m=4, \quad   m_1 & =3, \quad m_2= c_1 m_1- m = 2 \cdot 3 - 4 = 2 \quad \text{and} \\ m_3& = c_2\, m_2 - m_1 = 2 \cdot 2  - 3 = 1,
	\end{align*}
and perform three weighted blow-ups of Type $(II)$ as in Proposition~\ref{prop::desing} (see Figure~\ref{fig::desing1}):
	\begin{itemize}[leftmargin=*]
		\item  a $\text{WBU}(4,\frac{1}{2},\frac{1}{2},\frac{3}{4},\frac{1}{4},\varepsilon_1)$ that creates a new singular point of type $\frac{1}{3}(1,2 )$  that is a minimum of the resulting moment map;
		\item a $\text{WBU}(3,\frac{2}{3},\frac{1}{3},\frac{2}{3},\frac{1}{3},\varepsilon_2)$ that creates a new singular point of type $\frac{1}{2}(1,1)$ that is again a minimum;
		\item a $\text{WBU}(2,\frac{1}{2},\frac{1}{2},\frac{1}{2},\frac{1}{2},\varepsilon_3)$ that creates a smooth fixed surface of self-intersection $-2$ as the exceptional divisor, where the resulting moment map attains its minimum value.
	\end{itemize}
	We proceed similarly for the maximal point of $H$ and then we desingularize the singular point of type $\frac{1}{2}(1,1)$. Since $1 \cdot1   = 1 \mod 2$ and  $\frac{2}{1}=[2]$, this can be achieved with one weighted blow-up $\text{WBU}(2,\frac{1}{2},-\frac{1}{2},\frac{1}{2},\frac{1}{2},\varepsilon)$ of Type $(I)$ in Proposition~\ref{prop::desing}.	
\renewcommand{\thefigure}{\thesection.\arabic{figure}}
	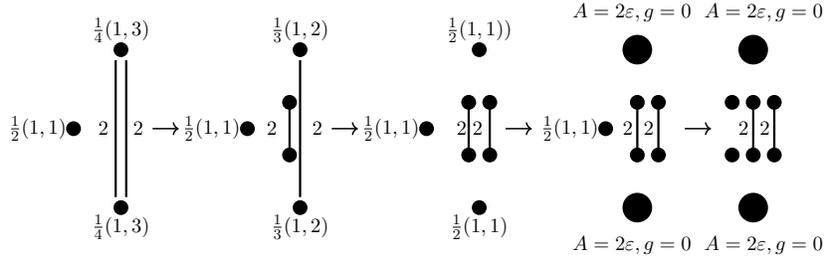
\begin{figure}[h!]
		\centering		
		\resizebox{11cm}{!}{%
		\begin{tikzpicture}	
			\coordinate (Q1) at (0,3);
			\fill (Q1) circle (4pt);
			\coordinate (Q2) at (0,0);
			\fill (Q2) circle (4pt);	
			\coordinate (Q0) at (-.9,1.5);
			\fill (Q0) circle (4pt);	
			
			\coordinate[label=left:$2$] (BC) at ($(-0.1,1.5)$);	
			\coordinate[label=right:$2$] (BC) at ($(0.1,1.5)$);
			
			\coordinate[label=below:${\frac{1}{4}(1,3)}$] (BC) at ($(Q2)$);
			\coordinate[label=above:${\frac{1}{4}(1,3)}$] (BC) at ($(Q1)$);
			\coordinate[label=left:${\frac{1}{2}(1,1)}$] (BC) at ($(Q0)$);
			
			\draw[very thick](0.1,2.8) -- (0.1,0.2);
			\draw[very thick](-0.1,2.8) -- (-0.1,0.2);
			%%%%%%%%%%%%%%%%%%%%%%%%%%%%%%%%%%%%%%
			\draw[->, thick] (0.6,1.5) to (1.1,1.5);
			
			\coordinate (Q1) at (3.4,3);
			\fill (Q1) circle (4pt);
			\coordinate (Q2) at (3.4,0);
			\fill (Q2) circle (4pt);
			\coordinate (Q3) at (3.2,2);
			\fill (Q3) circle (4pt);
			\coordinate (Q4) at (3.2,1);
			\fill (Q4) circle (4pt);	
			\coordinate (Q0) at (2.4,1.5);
			\fill (Q0) circle (4pt);	
			
			\coordinate[label=left:$2$] (BC) at ($(3.1,1.5)$);	
			\coordinate[label=right:$2$] (BC) at ($(3.5,1.5)$);

			\coordinate[label=below:${\frac{1}{3}(1,2)}$] (XX) at ($(Q2)$);
			\coordinate[label=above:${\frac{1}{3}(1,2)}$] (X) at ($(Q1)$);
			\coordinate[label=left:${\frac{1}{2}(1,1)}$] (BC) at ($(Q0)$);
			
			\draw[very thick](3.4,2.8) -- (3.4,0.2);
			\draw[very thick](3.2,2) -- (3.2,1);
			%%%%%%%%%%%%%%%%%%%%%%%%%%%%%%%%%%%%%%%%
			\draw[->, thick] (4.0,1.5) to (4.5,1.5);
			
			\coordinate (Q1) at (6.8,3);
			\fill (Q1) circle (4pt);
			\coordinate (Q2) at (6.8,0);
			\fill (Q2) circle (4pt);	
			\coordinate (Q0) at (5.8,1.5);
			\fill (Q0) circle (4pt);
			
			\coordinate (B1) at (7,1);
			\fill (B1) circle (4pt);
			\coordinate (B2) at (6.6,1);
			\fill (B2) circle (4pt);
			\coordinate (B3) at (6.6,2);
			\fill (B3) circle (4pt);
			\coordinate (B4) at (7,2);
			\fill (B4) circle (4pt);	
			
			\coordinate[label=left:$2$] (BC) at ($(6.7,1.5)$);	
			\coordinate[label=left:$2$] (BC) at ($(7,1.5)$);

			\coordinate[label=below:${\frac{1}{2}(1,1)}$] (XX) at ($(Q2)$);
			\coordinate[label=above:${\frac{1}{2}(1,1))}$] (X) at ($(Q1)$);
			\coordinate[label=left:${\frac{1}{2}(1,1)}$] (BC) at ($(Q0)$);
			
			\draw[very thick](7,1) -- (7,2);
			\draw[very thick](6.6,1) -- (6.6,2);

			%%%%%%%%%%%%%%%%%%%%%%%%%%%%%%%%%%%%%%%%
			\draw[->, thick] (7.3,1.5) to (7.8,1.5);
			
			\coordinate (Q1) at (9.8,3);
			\fill (Q1) circle (8pt);
			\coordinate (Q2) at (9.8,-0);
			\fill (Q2) circle (8pt);
			
			\coordinate (BB0) at (10.2,1);
			\fill (BB0) circle (4pt);
			\coordinate (BB1) at (10.2,2);
			\fill (BB1) circle (4pt);
			
			\coordinate (B1) at (9.8,1);
			\fill (B1) circle (4pt);
			\coordinate (B2) at (9.2,1.5);
			\fill (B2) circle (4pt);
			\coordinate (B4) at (9.8,2);
			\fill (B4) circle (4pt);	
			
			\coordinate[label=left:$2$] (BC) at ($(10.25,1.5)$);
			\coordinate[label=left:$2$] (BC) at ($(9.85,1.5)$);
			\coordinate[label=left:${\frac{1}{2}(1,1)}$] (BC) at ($(B2)$);
			
			\coordinate[label=below:${A=2\varepsilon,g=0}$] (XX) at ($(9.7,-0.4)$);
			\coordinate[label=above:${A=2\varepsilon,g=0}$] (X) at ($(9.7,3.4)$);

			\draw[very thick](10.2,1) -- (10.2,2);
			\draw[very thick](9.8,1) -- (9.8,2);
			%%%%%%%%%%%%%%%%%%%%%%%%%%%%%%%%%%%%%%%%
			\draw[->, thick] (10.7,1.5) to (11.2,1.5);
			
			\coordinate (Q1) at (12,3);
			\fill (Q1) circle (8pt);
			%index 2 points
			\coordinate (Q2) at (12,-0);
			\fill (Q2) circle (8pt);
			
			\coordinate (BB0) at (12.4,1);
			\fill (BB0) circle (4pt);
			\coordinate (BB1) at (12.4,2);
			\fill (BB1) circle (4pt);
			
			\coordinate (B1) at (12,1);
			\fill (B1) circle (4pt);
			\coordinate (B2) at (11.6,1);
			\fill (B2) circle (4pt);
			\coordinate (B3) at (11.6,2);
			\fill (B3) circle (4pt);
			\coordinate (B4) at (12,2);
			\fill (B4) circle (4pt);

			\coordinate[label=left:$2$] (BC) at ($(12.45,1.5)$);
			\coordinate[label=left:$2$] (BC) at ($(12.05,1.5)$);

			\coordinate[label=below:${A=2 \varepsilon,g=0}$] (XX) at ($(12.2,-0.4)$);
			\coordinate[label=above:${A= 2 \varepsilon,g=0}$] (X) at ($(12.2,3.4)$);

			\draw[very thick](12.4,1) -- (12.4,2);
			\draw[very thick](12,1) -- (12,2);	
			
		\end{tikzpicture}}
		\caption{An equivariant desingularization of  $(M,\omega,H)$ obtained using only the weighted blow-ups of Proposition~\ref{prop::desing}.}
		\label{fig::desing1}
	\end{figure}
	
	Alternatively, we can perform first two weighted blow-ups $\text{WBU}(4,\pm\frac{1}{2},\pm\frac{1}{2},\frac{1}{2},\frac{1}{2},\varepsilon)$ of Type $(III)$ in Proposition~\ref{prop::desing2} at the extremal fixed points. The corresponding exceptional divisors are two (extremal) fixed orbi-spheres, each one with two singular points of type $\frac{1}{2}(1,1)$. Then  we desingularize these points using two weighted blow-ups $\text{WBU}(2,1,0,\frac{1}{2},\frac{1}{2},\delta)$ and $\text{WBU}(2,0,-1,\frac{1}{2},\frac{1}{2},\delta)$ of Type $(IV)$ in Proposition~\ref{prop::desing}. Finally, we desingularize the remaining fixed point of type $\frac{1}{2}(1,1)$ using a weighted blow-up $\text{WBU}(2,\frac{1}{2},-\frac{1}{2},\frac{1}{2},\frac{1}{2},\varepsilon)$ of Type $(I)$ in Proposition~\ref{prop::desing} (see  Figure \ref{fig::desing2}).
\renewcommand{\thefigure}{\thesection.\arabic{figure}}
	\begin{figure}[h!]
		\centering
				\resizebox{10.5cm}{!}{%
		\begin{tikzpicture}
			\coordinate (Q1) at (0,3);
			\fill (Q1) circle (4pt);
			\coordinate (Q2) at (0,0);
			\fill (Q2) circle (4pt);	
			\coordinate (Q0) at (-1,1.5);
			\fill (Q0) circle (4pt);	
			
			\coordinate[label=left:$2$] (BC) at ($(-0.1,1.5)$);	
			\coordinate[label=right:$2$] (BC) at ($(0.1,1.5)$);
			
			\coordinate[label=below:${\frac{1}{4}(1,3)}$] (BC) at ($(Q2)$);
			\coordinate[label=above:${\frac{1}{4}(1,3)}$] (BC) at ($(Q1)$);
			\coordinate[label=left:${\frac{1}{2}(1,1)}$] (BC) at ($(Q0)$);
			
			\draw[very thick](0.1,2.8) -- (0.1,0.2);
			\draw[very thick](-0.1,2.8) -- (-0.1,0.2);
			%%%%%%%%%%%%%%%%%%%%%%%%%%%%%%%%%%%%%%
			\draw[->, thick] (1,1.5) to (1.7,1.5);
			
			\coordinate (Q1) at (4,3);
			\fill (Q1) circle (8pt);
			\coordinate (Q2) at (4,0);
			\fill (Q2) circle (8pt);	
			\coordinate (Q0) at (3,1.5);
			\fill (Q0) circle (4pt);	
			
			\coordinate[label=left:$2$] (BC) at ($(3.9,1.5)$);	
			\coordinate[label=right:$2$] (BC) at ($(4.1,1.5)$);

			\coordinate[label=below:${\frac{1}{2}(1,1),\frac{1}{2}(1,1),}$] (XX) at ($(4,-0.3)$);
			\coordinate[label=below:${A=\varepsilon,g=0}$] (XXX) at ($(4,-0.9)$);
			\coordinate[label=above:${\frac{1}{2}(1,1),\frac{1}{2}(1,1)}$] (X) at ($(4,3.2)$);
			\coordinate[label=above:${A=\varepsilon,g=0}$] (XI) at ($(4,3.8)$);
			\coordinate[label=left:${\frac{1}{2}(1,1)}$] (BC) at ($(Q0)$);
			
			\draw[very thick](4.2,2.6) -- (4.2,0.4);
			\draw[very thick](3.8,2.6) -- (3.8,0.4);
			%%%%%%%%%%%%%%%%%%%%%%%%%%%%%%%%%%%%%%%%
			\draw[->, thick] (5,1.5) to (5.7,1.5);
			
			\coordinate (Q1) at (8,3);
			\fill (Q1) circle (8pt);
			\coordinate (Q2) at (8,-0);
			\fill (Q2) circle (8pt);	
			\coordinate (Q0) at (7,1.5);
			\fill (Q0) circle (4pt);
			
			\coordinate (B1) at (8.2,1);
			\fill (B1) circle (4pt);
			\coordinate (B2) at (7.8,1);
			\fill (B2) circle (4pt);
			\coordinate (B3) at (7.8,2);
			\fill (B3) circle (4pt);
			\coordinate (B4) at (8.2,2);
			\fill (B4) circle (4pt);	
			
			\coordinate[label=left:$2$] (BC) at ($(7.9,1.5)$);	
			\coordinate[label=left:$2$] (BC) at ($(8.2,1.5)$);

			\coordinate[label=below:${A=\varepsilon-\delta,g=0}$] (XX) at ($(8,-0.4)$);
			\coordinate[label=above:${A=\varepsilon-\delta,g=0}$] (X) at ($(8,3.4)$);
			\coordinate[label=left:${\frac{1}{2}(1,1)}$] (BC) at ($(Q0)$);
			
			\draw[very thick](8.2,1) -- (8.2,2);
			\draw[very thick](7.8,1) -- (7.8,2);

			%%%%%%%%%%%%%%%%%%%%%%%%%%%%%%%%%%%%%%%%
			\draw[->, thick] (9.5,1.5) to (10.2,1.5);
			
			\coordinate (Q1) at (12,3);
			\fill (Q1) circle (8pt);
			\coordinate (Q2) at (12,-0);
			\fill (Q2) circle (8pt);
			
			\coordinate (BB0) at (12.4,1);
			\fill (BB0) circle (4pt);
			\coordinate (BB1) at (12.4,2);
			\fill (BB1) circle (4pt);
			
			\coordinate (B1) at (12,1);
			\fill (B1) circle (4pt);
			\coordinate (B2) at (11.6,1);
			\fill (B2) circle (4pt);
			\coordinate (B3) at (11.6,2);
			\fill (B3) circle (4pt);
			\coordinate (B4) at (12,2);
			\fill (B4) circle (4pt);

			\coordinate[label=left:$2$] (BC) at ($(12.45,1.5)$);
			\coordinate[label=left:$2$] (BC) at ($(12.05,1.5)$);

			\coordinate[label=below:${A=\varepsilon-\delta,g=0}$] (XX) at ($(12,-0.4)$);
			\coordinate[label=above:${A=\varepsilon-\delta,g=0}$] (X) at ($(12,3.4)$);

			\draw[very thick](12.4,1) -- (12.4,2);
			\draw[very thick](12,1) -- (12,2);
			
		\end{tikzpicture}}
		\caption{An alternative desingularization of $(M,\omega,H)$}\label{fig::desing2}
	\end{figure}
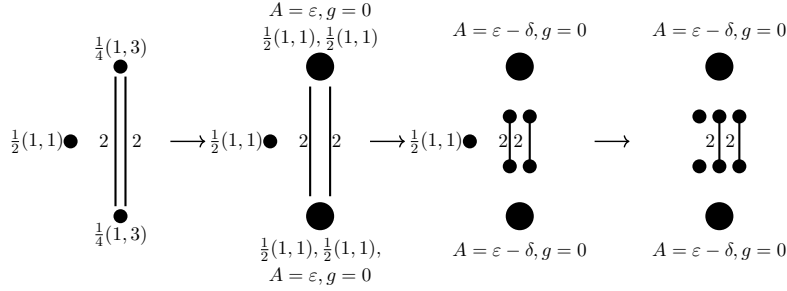\hfill$\Diamond$ 
	\end{example}

\subsection{Symplectic weighted blow-down}\label{sec:wbd}
In this section we consider the reverse procedure of a symplectic
weighted blow-up, called a symplectic weighted blow-down. While we use
the former to desingularize Hamiltonian $S^1$-spaces (see
Theorem \ref{thm:desingularization}), the latter are used to obtain
the so-called ``minimal models'' (see Section \ref{sec:existence}). \\

Let $(M^{2n},\omega)$ be a symplectic orbifold that has
isolated singular points. Suppose that $\Sigma \simeq \mathbb{C}P^{n-1}(p_1,\ldots,p_n)/\mathbb{Z}_m$
is a fully embedded symplectic suborbifold of $(M^{2n},\omega)$. By Theorems \ref{thm:tub_neigh_orb} and
\ref{thm:esnt}, there exists a neighborhood $N(\Sigma)$ of $\Sigma$ in
$M$ that is symplectomorphic to some neighborhood of the zero section
in the normal orbi-bundle $\nu_\Sigma$, which is a complex line orbi-bundle
over $\Sigma$. Assume that $\deg(\nu_\Sigma)<0$. Then $N(\Sigma)$ is
equivariantly symplectomorphic to a neighborhood  of the exceptional
divisor in some weighted blow-up of $\mathbb{C}^n/\Gamma$ at
$[0]_\Gamma$, for some cyclic group $\Gamma$  that contains
$\mathbb{Z}_m$ (see \cite[Section 5.2]{lgodinho}). To perform the weighted blow-down we  remove a
neighborhood of $\Sigma$ and glue back  the image of an orbi-ball in
$\mathbb{C}^n/\Gamma$ of the appropriate size by an extension of this
symplectomorphism.

\begin{remark}\label{rmk:symp_weighted_blow_down}
  Just like in the case of manifolds one can show that
  \begin{itemize}[leftmargin=*]
  \item the above construction is independent of the symplectomorphism
    between $N(\Sigma)$ and the neighborhood of the exceptional
    divisor in the weighted blow-up of $\mathbb{C}^n/\Gamma$;
  \item if $G$ is a compact Lie group that acts on $(M,\omega)$ in a
    Hamiltonian fashion and $\Sigma$ is $G$-invariant, then the above
    construction can be carried out $G$-equivariantly. In particular,
    the blown-down space inherits a Hamiltonian $G$-action; and
  \item if there exists a complex structure on $M$ that is compatible
    with $\omega$ (i.e., $(M,\omega)$ is Kähler), then the blown-down
    space inherits a complex structure that is compatible
    with $\omega$.
  \end{itemize}
\end{remark}

 \subsubsection{Equivariant weighted blow-downs on Hamiltonian $S^1$-spaces} 
Henceforth we restrict to Hamiltonian $S^1$-spaces. By the above
discussion, we may blow down a (fully embedded) invariant orbi-sphere if
and only if its normal orbi-bundle has negative degree. By Theorem
\ref{thm:esnt}, the following
result explains how one can obtain this degree from the
orbi-weights of the $S^1$-action.

\begin{lemma}\label{lem::euler}
  Let $\Sigma:=\mathbb{C}P^1(p,q)/\mathbb{Z}_n$ be the orbi-sphere
  corresponding to the 
  $\mathbb{Z}_n$-action on  $\mathbb{C}P^1(p,q)$ given by
  \begin{align*}
    \xi_n\cdot[z_0:z_1]_{p,q}=[\xi_n^a z_0:\xi_n^b z_1]_{p,q},
  \end{align*} 
  with $a,b\in \mathbb{Z}$ such that $aq-bp=1$. Let $L \to \Sigma$ be
  a complex line orbi-bundle.

  Consider the $S^1$-action on $\Sigma$ that rotates it $c$ times
  while fixing the singular points of order $nq$ and $np$, that we
  call the south and north poles respectively. Suppose that the action
  lifts to $L$ and let $m_{-}$ (respectively $m_+$) be the
  orbi-weight of the action restricted to the fiber over the south
  (respectively north) pole. Then 
  $$
  m_+-m_- = - c \deg(L).
  $$
  Moreover, $L$ is isomorphic to 
  \begin{equation}
    \label{eq:64}
    \mathcal{O}_{p,q}\left(-\frac{npq\left(m_+-m_-\right)}{c}\right)/\mathbb{Z}_n(a,b,b\,l_++a\,l_-),
  \end{equation}
  where $1\leq l_-< nq$ and $1\leq l_+< np$ are such that $(nq,l_-)$
  and $(np,l_+)$ are the Seifert invariants of the
  singular points of the principal $S^1$-orbi-bundle associated to $L$.
  \end{lemma}

\begin{proof}
  We split the orbi-sphere $\mathbb{C}P^1(p,q)/\mathbb{Z}_n$ into its
  southern  and northern hemispheres,  given by orbi-discs
  $D^-/\mathbb{Z}_{nq}$ and $D^+/\mathbb{Z}_{np}$, where $D^-$ and
  $D^+$ are disks  in $\mathbb{C}$. Since these are simply connected,
  we obtain trivializations $(D^-\times\mathbb{C})/\mathbb{Z}_{nq}$
  and $(D^+\times\mathbb{C})/\mathbb{Z}_{np}$ of
  $L$ for some $\mathbb{Z}_{nq}$- and  $\mathbb{Z}_{np}$-actions given
  by 
\begin{equation} \label{Znp_action}	
 \begin{array}{rcl}
{\scriptstyle	\mathbb{Z}_{nq}\times D^-\times\mathbb{C} }& \rightarrow & {\scriptstyle D^-\times\mathbb{C}}\\
	{\scriptstyle (\xi_{nq},z,w)} & \mapsto & {\scriptstyle(\xi_{nq}\, z,\xi^{l_-}_{nq}\, w)}
\end{array}	
\quad \text{and} \quad
\begin{array}{rcl}
	\scriptstyle{		\mathbb{Z}_{np} \times D^+\times\mathbb{C} }& \rightarrow & \scriptstyle{ D^+\times\mathbb{C} }\\
		\scriptstyle{	(\xi_{np},z,w)} & \mapsto &\scriptstyle{ (\xi_{np} \, z ,\xi_{np}^{l_+} \, w) }
		\end{array}
	\end{equation}
	for unique integers $l_-,l_+$ with $1\le l_-<nq$ and $1\leq l_+<np$, such that $\gcd(nq,l_-)=\gcd(np,l_+)=1$.
	%,  where   $\xi_{nq}$ is a primitive $nq$-th root of unity and $\xi_{np}$ is a primitive $np$-th root of unity. 
	The circle action on these quotients is given by
\begin{equation*}
\begin{array}{rcl}		
		\scriptstyle{S^1\times (D^-\times\mathbb{C})/\mathbb{Z}_{nq} }&\scriptstyle{\stackrel{\alpha_1}{\rightarrow}} &\scriptstyle{  (D^-\times\mathbb{C})/\mathbb{Z}_{nq}}\\
			\scriptstyle{(\lambda,[z,w]_{nq})}&\scriptstyle{\mapsto} & \scriptstyle{[\lambda^{\frac{c}{nq}}z,\lambda^{m_-}w]_{nq}}
		\end{array}
\,\, \text{and} \,\,
\begin{array}{rcl}
		\scriptstyle{S^1\times  (D^+\times\mathbb{C})/\mathbb{Z}_{np}}&\scriptstyle{\stackrel{\alpha_2}{\rightarrow}} & \scriptstyle{(D^+\times\mathbb{C})/\mathbb{Z}_{np}}\\
			\scriptstyle{(\lambda,[\tilde{z},w]_{np})}&\scriptstyle{\mapsto} &\scriptstyle{ [\lambda^{-\frac{c}{np}}\tilde{z},\lambda^{m_+}w]_{np}}.
		\end{array}
	\end{equation*}
Hence, by \eqref{eq:condition1}, the orbi-weights satisfy 
$$
-l_+c\equiv m_+np \mod np\quad\text{and}\quad l_-c\equiv m_-nq \mod nq,
$$	
where we recall that $m_+np$ (respectively $m_-nq$) is an integer that
is not necessarily
divisible by $np$ (respectively $nq$).

Let $g: (\partial D^-\times\mathbb{C})/\mathbb{Z}_{nq} \to (\partial
D^+\times\mathbb{C})/\mathbb{Z}_{np}$ be the clutching map for $L$. We
write it as 
\begin{equation}
  \label{equ::gluingmap}
  [z,w]_{nq}\mapsto	[z^{-\frac{q}{p}},z^{-\frac{\alpha}{np}}w]_{np}
\end{equation}
for some $\alpha\in \Z$. For $g$ to be  well-defined, we need that
\begin{align*}
g([z,w]_{nq})=g([e^\frac{2\pi i}{qn}z,e^\frac{2\pi i l_-}{qn}w]_{nq}).
\end{align*}
Hence,
\begin{align*}
[z^{-\frac{q}{p}},z^{-\frac{\alpha}{pn}}w]_{np} & =[e^\frac{-2\pi i}{pn}z^{-\frac{q}{p}},e^{-\frac{2\pi i \alpha}{n^2qp}+\frac{2\pi i  l_-}{qn}}z^{-\frac{\alpha}{pn}}w]_{np} \\ 
&=  [e^\frac{-2\pi i}{pn}z^{-\frac{q}{p}},e^{\frac{2\pi i(l_-np-\alpha)}{n^2qp}}z^{-\frac{\alpha}{pn}}w]_{np} \\ &=  [z^{-\frac{q}{p}},e^{\frac{2\pi i(l_+nq+l_-np-\alpha)}{n^2qp}}z^{-\frac{\alpha}{pn}}w]_{np},
\end{align*}
implying that
\begin{equation*}
\alpha  = l_-pn+l_+qn\mod n^2qp.\nonumber
\end{equation*}
Set $k:=\alpha/n\in \mathbb{Z}$; then $g$ is of the form
\begin{equation}\label{eq:mapg}
g([z,w]_{nq})=[z^{-\frac{q}{p}},z^{-\frac{k}{p}}w]_{np},
\end{equation}
for some $k$ satisfying
\begin{equation}\label{equ::welldefinedk}
  k=l_+q+l_-p \mod npq.
\end{equation}

Let $b\in  \Z$ be such that $k=l_+q+l_-p  + b \,npq$. Then $(0;b,(np,l_+),(nq,l_-))$ is the Seifert invariant of
the circle  orbi-bundle  associated to $L$ and, by Lemma
\ref{lem::linebundle} and Corollary
\ref{cor:degree_quotient}, 
$$\deg(L)=b+\frac{l_+}{np}+\frac{l_-}{nq}=\frac{k}{npq}. $$
Since the $S^1$-action on $\mathbb{C}P^1(p,q)/\mathbb{Z}_n$ lifts to $L$, the gluing map $g$ needs to be $S^1$-equivariant. Hence we need	
\begin{align*}
  &g\left(\alpha_1\left(\lambda,[z,w]_{nq}\right)\right) =\alpha_2\left(\lambda,g\left([z,w]_{np}\right)\right)   \Leftrightarrow \\   & \Leftrightarrow
                                                                                                                                         g\left([\lambda^{\frac{c}{nq}}z,\lambda^{m_-}w]_{nq}\right) =\alpha_2\left(\lambda,[z^{-\frac{q}{p}},z^{-\frac{k}{p}}w]_{np}\right)\\
  & \Leftrightarrow [\lambda^{-\frac{c}{np}}z^{-\frac{q}{p}},		
    \lambda^{-\frac{ck}{nqp}}z^{-\frac{k}{p}}\lambda^{m_-}w]_{np}  =
    [\lambda^{-\frac{c}{np}}z^{-\frac{q}{p}},\lambda^{m_+}z^{-\frac{k}{p}}w]_{np},
\end{align*}
for all $\lambda \in S^1$,
which gives $m_+=m_--\frac{ck}{npq}$. We conclude that 
$$k=-\frac{(m_+-m_-)npq}{c},$$
so that the first result holds. \\

Next we show that $L$ is isomorphic to the orbi-bundle of \eqref{eq:64}. Since the
$\mathbb{Z}_n$-action on $\mathcal{O}_{p,q}(k)$ is given by
\begin{align*} 
  \xi_n\cdot[z_0:z_1:w]_{p,q,k}&=[\xi_n^az_0:\xi_n^bz_1:\xi_n^{(bl_++al_-)}w]_{p,q,k} \\&
  =[z_0:\xi_n^{\frac{bp-aq}{p}}z_1:\xi_n^{\frac{(bl_++al_-)p-ak}{p}}w]_{p,q,k}\\
                               &=[z_0:\xi_{np}z_1:\xi_{np}^{(bl_++al_-)p-a(ql_++pl_-)}w]_{p,q,k} \\ 
                               % =[z_0:\xi_{np}z_1:\xi_{np}^{l_+(bp-aq)}w]_{p,q,k}\\
                               &=[z_0:\xi_{np}z_1:\xi_{np}^{l_+}w]_{p,q,k},
\end{align*}
% where $\xi_{n}$ is a primitive $n$-th root of unity and $\xi_{np}$ is a primitive $np$-th root of unity, where we used \eqref{equ::welldefinedk}. T
$\mathcal{O}_{p,q}(k)/\mathbb{Z}_n(a,b,b\,l_++a\,l_-)$ has a singular
point of type $\frac{1}{np}(1,l_+)$, with $\Z_n$ acting with weight
$l_+$ on the fiber over this singular point. Analogously, it has a
singular point of type $\frac{1}{nq}(1,l_-)$,  with $\Z_n$ acting with
weight $l_-$ on the fiber over this singular point. We conclude that
the Seifert invariants corresponding to the two singular points are
$(nq,l_-)$ and $(np,l_+)$ respectively.
By Lemma~\ref{lem::linebundle},
$$\deg \left( \mathcal{O}_{p,q}(k)/\mathbb{Z}_n(a,b,b\,l_++a\,l_-) \right)= \frac{k}{n\,p\,q}.$$ 
Hence, the Seifert invariants of the principal $S^1$-orbi-bundles
associated to $\mathcal{O}_{p,q}(k)/\mathbb{Z}_n(a,b,b\,l_++a\,l_-)$
and $L$ are equal. Therefore, the two complex line orbi-bundles are isomorphic. 
\end{proof}

Lemma \ref{lem::euler} immediately implies the following result, which
generalizes \cite[Corollary 5.5]{karshon}.

\begin{corollary}\label{cor::index2gradientsph}
  Let $(M,\omega,H)$ be a Hamiltonian $S^1$-space. Any
  fully embedded $S^1$-invariant orbi-sphere whose north and south
  poles are non-extremal fixed points has negative
  self-intersection\footnote{i.e., the degree of the normal orbi-bundle.}. 
%  Every $S^1$-invariant orbi-sphere whose north and south poles are non-extremal fixed points, has negative self-intersection and can therefore be blown down.
\end{corollary}

An equivariant weighted blow-down may give rise to a singular
point that has a cyclic orbifold structure group. By Theorem
\ref{thm:esnt}, the next result determines the order of its orbifold
structure group. 

\begin{lemma}\label{lem::weightedblowdown}
   Let $\Sigma:=\mathbb{C}P^1(p,q)/\mathbb{Z}_n$ be the orbi-sphere
  corresponding to the 
  $\mathbb{Z}_n$-action on  $\mathbb{C}P^1(p,q)$ given by
  \begin{align*}
    \xi_n\cdot[z_0:z_1]_{p,q}=[\xi_n^a z_0:\xi_n^b z_1]_{p,q},
  \end{align*} 
  with $a,b\in \mathbb{Z}$ such that $aq-bp=1$. Let $k <0$, $1\leq l_-
  < nq$ and $1\leq l_+ < np$ be integers such that $\gcd(l_-,nq)=\gcd(l_+,nq)=1$.
  Suppose that $S^1$ acts on $\Sigma$ and that the action lifts to the
  complex line orbi-bundle 
  $$\mathcal{O}_{p,q}(k)/\mathbb{Z}_n(a,b,b \,l_++a\,l_-) \to \Sigma.$$ 
  Then an $S^1$-equivariant symplectic weighted blow-down along the
  zero section gives rise to a singular point with orbifold structure
  group given by $\Z_{|nk|}$.
\end{lemma}

\begin{proof}
  By Example \ref{ex::orbifold_quotient}, $\mathcal{O}_{p,q}(-k)/\mathbb{Z}_n(a,b,b \,l_++a\,l_-)$ is isomorphic to the normal orbi-bundle of $\mathbb{C}P(p,q)/\mathbb{Z}_n$ inside 
  $$\mathbb{C}P^2(p,q,-k)/\mathbb{Z}_n(a,b,b \,l_++a\,l_-).$$ 
  Therefore, a neighborhood $N(\Sigma)$ of $\Sigma$ is equivariantly
  symplectomorphic  to a neighborhood  of the exceptional divisor in
  some weighted blow-up of $\mathbb{C}^n/\mathbb{Z}_{nk}$ at
  $[0]_{\mathbb{Z}_{nk}}$. Since the complex line orbi-bundle in the
  statement is the opposite of
  $\mathcal{O}_{p,q}(-k)/\mathbb{Z}_n(a,b,b \,l_++a\,l_-)$, the result follows.
  % We can therefore remove a neighborhood of
  %$\Sigma$ and glue back  the image of an orbi-ball in
  %$\mathbb{C}^n/\mathbb{Z}_{nk}$ of the appropriate size.  
\end{proof}

\section{Classification of Hamiltonian $S^1$-spaces}\label{ChapterClassification}
\subsection{Uniqueness}\label{uniqueness}
By Lemma \ref{lemma:trivial_direction}, the labeled multigraph of a Hamiltonian
$S^1$-space is an invariant of the isomorphism class. The aim of this
section is to show that it is the {\em only} invariant.

% By Lemma~\ref{lemma:trivial_direction}  we have that if $(M,\omega,\Phi)$ and $(M^\prime,\omega^\prime,\Phi^\prime)$ are two Hamiltonian $S^1$-spaces and there exists an equivariant symplectomorphism $\varphi:M \to M^\prime$ that respects the moment maps then the corresponding labelled multigraphs are equal. The following result shows the converse.
\begin{theorem}\label{thm:uniqueness}
  If two Hamiltonian $S^1$-spaces have equal labeled multigraphs then they are equivariantly symplectomorphic.
\end{theorem}
\begin{proof}
Suppose that $(M_1,\omega_1,\Phi_1)$ and
$(M_2,\omega_2,\Phi_2)$ have equal labeled
multigraphs. By Theorem \ref{thm:desingularization}, both $(M_1,\omega_1,\Phi_1)$ and
$(M_2,\omega_2,\Phi_2)$ can be desingularized by applying finitely
many $S^1$-equivariant weighted blow-ups. In fact, the proof of
Theorem \ref{thm:desingularization} provides an explicit algorithm.
Moreover, we observe that, except for the size of each blow-up, all
the data to run the algorithm are encoded in the labeled
multigraph. Since the labeled multigraphs of $(M_1,\omega_1,\Phi_1)$ and
$(M_2,\omega_2,\Phi_2)$ are assumed to be equal, we may perform the
desingularization using the {\em same} algorithm, i.e., we make sure
that the size of each blow-up is the same for both spaces at each
step. In particular, there exists a non-negative integer $k$ such that
the desingularization procedure for both spaces is completed after
exactly $k$ blow-ups.

We argue by induction on $k$. If $k=0$, then neither space has
singular points and the result follows by \cite[Theorem
4.1]{karshon}. Suppose that the result holds for any two Hamiltonian
$S^1$-spaces with equal labeled multigraphs whose desingularization
procedure requires $k$ blow-ups. By Proposition \ref{prop::desing}, after applying the first
$S^1$-equivariant weighted blow-up to both $(M_1,\omega_1,\Phi_1)$ and
$(M_2,\omega_2,\Phi_2)$, we obtain two Hamiltonian $S^1$-spaces
$(M'_1,\omega'_1,\Phi'_1)$ and $(M'_2,\omega'_2,\Phi'_2)$ with equal labeled multigraphs whose desingularization
procedure requires $k$ blow-ups. Hence, by the inductive hypothesis,
we may assume that they are, in fact, equal Hamiltonian
$S^1$-spaces. Since $(M_1,\omega_1,\Phi_1)$ and
$(M_2,\omega_2,\Phi_2)$ are (isomorphic to spaces) obtained from $(M'_1,\omega'_1,\Phi'_1)$
and $(M'_2,\omega'_2,\Phi'_2)$ by applying an $S^1$-equivariant
weighted blow-down to the same $S^1$-invariant orbi-sphere, the
result follows.
\end{proof}

\subsection{Existence}\label{sec:existence}
By Theorem \ref{thm:uniqueness}, the labeled multigraph  of a
Hamiltonian $S^1$-space is a complete invariant. In this section, we
turn to the question of determining which spaces actually occur. Our
aim is to show that every Hamiltonian $S^1$-space can be obtained from
a minimal space by a sequence of $S^1$-equivariant symplectic weighted
blow-ups. By ``minimal'' we mean that no equivariant symplectic
blow-down can be performed. This is the analog of \cite[Theorem 6.3]{karshon} for orbifolds. 

Our strategy is as follows. First, we obtain the ``candidates'' for
the labeled multigraphs of the minimal spaces in Theorems
\ref{fig::minimalfixedsurface} and \ref{thm:existisolated}. The
technique is to perform equivariant weighted symplectic blow-downs to
kill $S^1$-invariant orbi-spheres with normal orbi-bundles  of
negative degree, reversing the constructions described in
Propositions~\ref{prop::desing} and \ref{prop::desing2}. Second, we
show that each candidate actually arises as the labeled multigraph of
some Hamiltonian $S^1$-space in Sections~\ref{MSF} and \ref{MSI}.

\subsubsection{Chains of $S^1$-invariant orbi-spheres}\label{sec:gradientsphere}
Let $(M,\omega,H)$ be a Hamiltonian $S^1$-space. In analogy with
\cite[Section 3.1]{karshon}, there exists a  {\bf compatible metric}
on $M$,  i.e., an
$S^1$-invariant Riemannian metric $\langle \cdot , \cdot \rangle$ such
that the endomorphism $J:TM  \to TM$ defined by 
$$
\langle u, v \rangle = \widetilde{\omega} (u,Jv)
$$
is an almost complex structure. By construction, the gradient and Hamiltonian vector fields of
$H$ commute. Moreover, either both vanish or are linearly independent
at a point. Since $M$ is compact, we obtain
an action of $\mathbb{C}^* \simeq \mathbb{R} \times S^1$ on
$M$ extending the given Hamiltonian $S^1$-action. Let $\mathcal{O}$ be
a $\mathbb{C}^*$-orbit that contains more than a point. Since $M$ has
isolated singular points, $\mathcal{O}$ consists solely of regular
points. Hence, \cite[Lemma 3.3]{karshon} holds, so that $\mathcal{O}$
is symplectomorphic to a cylinder endowed with a suitable
$S^1$-action. By Theorem \ref{cor::localformpt}, the
closure of $\mathcal{O}$ is
topologically a 2-sphere that is called a {\bf gradient
  orbi-sphere}: the two points that lie in the closure are fixed
points that we call north and south poles of the orbi-sphere,
depending on the value of $H$. We remark that a gradient orbi-sphere need not be a
suborbifold, just as in the case of manifolds (see \cite[Lemma
4.9]{aharaHattori} and Lemma \ref{lemma:free_smooth} below).

Let $p \in \mathcal{O}$ be a point and let $K \subseteq S^1$ denote
its stabilizer under the $S^1$-action. Since $\mathcal{O}$ contains more than a point, $K
\neq S^1$. Moreover, every point in $\mathcal{O}$ has stabilizer for
the $S^1$-action equal to $K$. Hence, we say that a gradient
orbi-sphere is {\bf free} if it is the closure of a
$\mathbb{C}^*$-orbit all of whose points have trivial stabilizer. In
analogy with \cite[Lemma 3.5]{karshon}, the gradient orbi-spheres that
are not free are precisely the isotropy orbi-spheres of
$(M,\omega,H)$. In particular, by Proposition \ref{prop:iso_sphere} only free gradient orbi-spheres may
fail to be suborbifolds. 

If we vary the compatible metric on $(M,\omega,H)$, the gradient
orbi-spheres that are not free do not change, while those that are free may
do. Since the arguments in the proof of \cite[Lemma 3.6]{karshon}
apply to Hamiltonian $S^1$-spaces (because they have isolated singular
points), for a generic choice of compatible metric there are no free
gradient orbi-spheres both of whose poles are interior fixed points.

Fix a generic compatible metric on $(M,\omega,H)$. By the arguments in
\cite[Section 3]{aharaHattori}, if no pole of a free gradient
orbi-sphere is an isolated extremum, then the orbi-sphere is a
fully embedded suborbifold. Otherwise, the following result holds
(cf. \cite[Lemma 4.9]{aharaHattori}).

\begin{lemma}\label{lemma:free_smooth}
Let $(M,\omega,H)$ be a Hamiltonian $S^1$-space with a generic
compatible metric. Let $Z$ be a free gradient orbi-sphere one of whose
poles $F$ is an isolated extremum that has type
$\frac{1}{m}(1,l)$. Let $\frac{a_1}{m},\frac{a_2}{m}$ be the
orbi-weights at $F$, with  $a_1l-a_2= 0 \mod m$. Then $Z$ is fully
embedded at $F$ if and only if at least one of $a_1$ or $a_{2}$ equals
$1$ in absolute value.
\end{lemma}
\begin{proof}
Consider coordinates $[z_1,z_2]_m$ around $F$ as in
Theorem~\ref{cor::localformpt} and let $x$ be a point in $Z$ near $F$
with coordinates $[w_1,w_2]_m$. We may assume that
$F$ is a maximal point. If $w_i=0$ for $i=1$ or $2$ then $Z$ is fully
embedded at $F$ and $a_j=-1$ for $j\neq i$, because $Z$ is assumed to
be a gradient sphere. Otherwise, $Z$ is given near $F$ by coordinates $[z^{a_1}w_1,z^{a_2}w_2]_m$ with $z\in\C^*$. Since $a_1a_2\neq 0$, near $F$ the orbi-sphere $Z$ is defined by the equation 
\begin{equation}\label{eq:locus}
w_2^{a_1} z_1^{a_2} - w_1^{a_2} z_2^{a_1}=0
\end{equation}
and so it is fully embedded at $F$ if and only if at least one of $a_1,a_2$ is $-1$. (Note that this is exactly the condition for the locus of \eqref{eq:locus} to be smooth in $\C^2$.)
\end{proof}

To conclude this section, we introduce the
following notion for Hamiltonian $S^1$-spaces, that is analogous to
\cite[Definition 5.1]{karshon}. 

\begin{definition}\label{def:chain}
Let $(M,\omega,H)$ be a Hamiltonian $S^1$-space and let $g$ be a
compatible metric on $(M,\omega,H)$. A {\bf chain of $S^1$-invariant
  gradient orbi-spheres} is a
sequence of gradient orbi-spheres $C_1, \ldots, C_k$ such that the south
pole of $C_1$ is a minimum of $H$, the north pole of $C_i$
is the south pole of $C_{i+1}$ for $i=1, \ldots k-1$, and the north
pole of $C_k$ is a maximum of $H$. We say that such a chain is {\bf
  non-trivial} if it contains more than one orbi-sphere or if it
contains an orbi-sphere with non-trivial stabilizer. 
\end{definition}

\subsubsection{Blowing down to a minimal space with fixed orbi-surfaces}

\begin{lemma}\label{lemma:existsurface1}		
  Any Hamiltonian $S^1$-space $(M,\omega,H)$ that has at least one fixed orbi-surface
  can be obtained from a Hamiltonian  $S^1$-space that has at least
  one fixed orbi-surface and no non-extremal fixed points by applying finitely many $S^1$-equivariant
  weighted blow-ups.
\end{lemma}

\begin{proof}
  By Corollary~\ref{cor::index2gradientsph}, every isotropy
  orbi-sphere whose
  north and south poles are non-extremal has negative
  self-intersection and can, therefore, be blown down. Therefore we
  may assume that $(M,\omega,H)$ contains no such isotropy
  orbi-spheres. Let $x$ be a non-extremal fixed point of $(M,\omega,H)$. We claim that there is a fully embedded
  $S^1$-invariant orbi-sphere one of whose poles is $x$ with the
  property that the other pole lies on a fixed orbi-surface in
  $(M,\omega,H)$. To see this, we fix a generic compatible metric on
  $(M,\omega,H)$, so that $x$ is a pole of a gradient orbi-sphere whose
  other pole lies on a fixed orbi-surface. Hence this gradient
  orbi-sphere is fully
  embedded. Moreover, by Lemma \ref{lem::euler}, this
  gradient orbi-sphere has negative self-intersection and can be
  therefore blown down. Iterating this procedure finitely many times,
  the result is proved.
  % \vspace{2cm}
  % By Corollary \ref{lem::uniqueMaxMin} every fixed orbi-surface is a global extremum of $H$.
  % Moreover,	 by Corollary~\ref{cor::index2gradientsph}, we can
  % blow-down every interior $S^1$-invariant orbi-sphere. After this,
  % for each remaining non-extremal fixed point, $x_i$,  we can still
  % blow-down one additional $S^1$-invariant orbi-sphere connecting
  % $x_i$ to a fixed orbi-surface. Indeed,  consider an $S^1$-invariant
  % orbi-sphere $Z$ with a pole on a fixed orbi-surface $\Sigma$. If the
  % north pole of $Z$ is $x_i$ then, using the notation of
  % Lemma~\ref{lem::euler}, we have $m_-=0$, $m_+>0$  and so the degree
  % of the normal orbi-bundle of $Z$ is negative, implying that $Z$ can
  % be blown down (the case where $x_i$ is the south pole of $Z$ is
  % similar). In this way we obtain a Hamiltonian $S^1$-space with no
  % non-extremal fixed points.  
\end{proof}	

By Lemma \ref{lemma:existsurface1}, in order to understand the
Hamiltonian $S^1$-spaces with at least one fixed surface, it suffices
to determine which minimal spaces with this property occur. The
following result describes the labeled multigraphs of the latter spaces.

\begin{theorem}\label{fig::minimalfixedsurface}
  Let $(M,\omega,H)$ be a minimal Hamiltonian $S^1$-space with at
  least one fixed orbi-surface. Then the labeled multigraph of
  $(M,\omega,H)$ is one of the following:
  
  % \vspace{2cm}
  % Any Hamiltonian $S^1$-space $(M,\omega,H)$ with fixed
  % orbi-surfaces can be obtained through a sequence of
  % admissible equivariant weighted blow-ups from a
  % Hamiltonian $S^1$-space with one of the following
  % labelled multigraphs, 
  
  \vspace{.1cm}
  \resizebox{12cm}{!}{%
    \begin{tikzpicture}
      
      \coordinate (point) at  (-6.5,1);
      \fill (point) circle (2pt);
      \fill  (-6.5,3) circle (8pt);
      
      \coordinate (point) at  (-0.5,3);
      \fill (point) circle (2pt);
      \fill  (-0.5,1) circle (8pt);
      
      \fill  (5.5,3) circle (8pt);
      \fill  (5.5,1) circle (8pt);
      \coordinate[label=above:{$\frac{1}{qc}(1,l_{qc}),\color{red} H=\alpha + \frac{A}{\beta_\Sigma}$}] (c1) at (-0.5,3);
      \coordinate[label=right:{$\frac{1}{qc}(1,l_{qc}), \color{red}$}] (c2) at (-7.2,0.5);
      \coordinate[label=right:{$ \color{red} H=\alpha $}] (c2) at (-7.1,0.);
      \coordinate[label=below:{$(I)$}] (c21) at (-6.5,-0.5);
      
      \coordinate[label=below:{$\frac{1}{p_1c}(1,l_1),\frac{1}{p_2c}(1,l_2)$}] (ACC) at (-0.5,0.65); 
      
      \coordinate[label=below:{$g=0,A, \color{red} H=\alpha$}] (ACD) at (-0.5,0); 
      \coordinate[label=below:{$(II)$}] (c22) at (-0.5,-0.5);
      
      \coordinate[label=above:{$\frac{1}{p_1c}(1,l_1),\frac{1}{p_2c}(1,l_2)$}] (AAA)  at (-6.4,3.8);
      \coordinate[label=above:{$g=0,A , \color{red} H=\alpha + \frac{A}{\beta_\Sigma}$}] (ABB)  at (-6.5,3.3);
      \coordinate[label=above:{$\frac{1}{p_1}(1,\hat{l}_1),\dots,\frac{1}{p_n}(1,\hat{l}_n)$}] (VFD) at (5.7,0);
      \coordinate[label=above:{$g,A,\color{red} H=\alpha$}] (VFE) at (5.5,-0.4);
      
      \coordinate[label=below:{$\frac{1}{p_1}(1,l_1),\dots,\frac{1}{p_r}(1,l_n)$}] (VCC) at (5.6,4.4);
      \coordinate[label=below:{$g, A - \beta_{\Sigma_-}s, \color{red} H=\alpha + s$}] (VCC) at (5.5,3.8);
      
      \coordinate[label=below:{$(III)$}] (c22) at (5.5,-0.5);
      
      \draw[thick](-6.6,2.6) -- (-6.6,1.2);
      \draw[thick](-6.4,2.6) -- (-6.4,1.2);
      \coordinate[label=right:{$p_2c$}] (c1) at (-6.4,2);
      \coordinate[label=left:{$p_1c$}] (c1) at (-6.6,2);
      
      \draw[thick](-0.6,2.8) -- (-0.6,1.4);
      \draw[thick](-0.4,2.8) -- (-0.4,1.4);
      \coordinate[label=right:{$p_2c$}] (c1) at (-0.4,2);
      \coordinate[label=left:{$p_1c$}] (c1) at (-0.6,2);

      \draw[thick](5,2.6) -- (5,1.4);
      \draw[thick](6,2.6) -- (6,1.4);
      \node at ($(5,2)!.5!(6,2)$) {\ldots};
      \coordinate[label=right:{$p_n$}] (c1) at (6,2);
      \coordinate[label=left:{$p_1$}] (c1) at (5,2);		
    \end{tikzpicture}}
  
  \vspace{.1cm}
  where the labels satisfy the following compatibility conditions:
  \begin{itemize}[leftmargin=*]
  \item \textbf{Multigraph (I):}  assuming w.l.o.g. that $p_1\geq p_2$, we have 
    \begin{enumerate}[label=(\roman*)]
    \item $1\leq l_{qc} < q c,  \quad 1\leq l_{1} <  p_1c, \quad  1\leq l_2< p_2 c$, \label{cond0}
    \item $\gcd(l_{qc},qc)  =  \gcd(l_1,p_1c) =  \gcd(l_2,p_2c)=1,$
    \item $ p_1 l_{qc}  = p_2  \mod q, \quad \gcd{(p_1l_{qc} - p_2, c)}=1$,
    \item $q=c\, p_1p_2 \beta_\Sigma$,
    \item $q - l_1p_2-l_2p_1 = 0 \mod p_1p_2c$, \label{cond}
    \item $p_2 -l_{qc} p_1  - l^\prime_1 q  =0 \mod p_1qc$, where $1\leq l^\prime_{1} < p_1 c$ is such that $l_1 \cdot  l^\prime_1=1 \mod p_1c$. \label{condiv}
    \end{enumerate}
  \item \textbf{Multigraph (II):}  assuming w.l.o.g. that $p_1\leq p_2$, we have 
    \begin{enumerate}[label=(\roman*)]
    \item $1\leq l_{qc} < q c,  \quad 1\leq l_{1} <  p_1c, \quad  1\leq l_2< p_2 c,$
    \item  $\gcd(l_{qc},qc)  =  \gcd(l_1,p_1c) =  \gcd(l_2,p_2c)=1$,
    \item   $p_1 l_{qc}  = p_2  \mod q, \quad \gcd{(p_1l_{qc} - p_2, c)}=1$,
    \item $q=c\, p_1p_2 \beta_\Sigma$,
    \item   $q - l^\prime_1p_2-l^\prime_2p_1 = 0  \mod p_1p_2c$, with $1\leq l^\prime_i<p_i c$ s.t. $l_i \cdot l^\prime_i=1 \!\mod p_ic$,\footnote{i.e. $(p_i c, l^\prime_i)$  is the Seifert invariant of the circle orbi-bundle associated to the normal orbi-bundle of $\Sigma_-$ corresponding to the singular point of order $p_ic$.}
    \item $p_2 - l_1 q -l_{qc} p_1 =0 \mod p_1qc$.
    \end{enumerate}
  \item \textbf{Multigraph (III):}
    \begin{enumerate}[label=(\roman*)]
    \item $1\leq l_{i},\hat{l}_i < p_i$,  for $i=1,\ldots, n$,
    \item $\gcd (l_i,p_i)=1$, for $i=1,\ldots, n$,
    \item if  $g=0$ and $n\leq 2$ then $\beta_{\Sigma_-} = 0$, \label{condIII0}
    \item $l_i^\prime + \hat{l_i} = p_i$, for $i=1,\ldots, n$, with  $1\leq l^\prime_i<p_i$ s.t. $l_i \cdot l^\prime_i=1 \!\mod p_i$.\footnote{Or equivalently, $l_i+\hat{l_i}^\prime = p_i$, ($i=1,\ldots, n$)  with  $1\leq \hat{l}^\prime_i<p_i$ s.t. $\hat{l}_i \cdot \hat{l}^\prime_i=1 \!\mod p_i$,} \label{condIII}
    \item $\sum_{i=1}^n \frac{l_i}{p_i} + \beta_{\Sigma_-} \in \Z$.\label{condv}
    \end{enumerate}
  \end{itemize}
\end{theorem}

\begin{proof}%\vspace{-.3cm}
  The Hamiltonian $S^1$-space $(M,\omega,H)$ either has exactly one or
  two fixed orbi-surfaces and no non-extremal fixed point. 

  Suppose first that there is only one fixed orbi-surface $\Sigma$
  that, without loss of generality, is a maximum. We claim that
  $\Sigma$ has at most two singular points. Suppose not; since there
  are no non-extremal fixed points, there would be at least three
  isotropy orbi-spheres incident to the isolated fixed point at the
  minimum. This contradicts Theorem~\ref{cor::localformpt}.

  Let $p_1 c$ and $p_2 c$, with $\gcd(p_1,p_2)=1$, be the orders of
  the orbifold structure group of the two (possibly) singular points in $\Sigma$ and let $qc$ be the
  order of the orbifold structure group of the isolated fixed point. 
  By \eqref{eq:bsigma-}, we have that $\beta_\Sigma = \frac{q}{c p_1
    p_2} >0$, where $\beta_{\Sigma}$ is the self-intersection of
  $\Sigma$. Moreover, by applying Theorem~\ref{prop:localization} to
  the first equivariant Chern class $\alpha=c_1^{S^1}(TM)$ and using
  \eqref{eq:*} and \eqref{eq:**}, we have that
$$
0 =  \frac{p_1 + p_2}{c \, p_1\, p_2 \, y} + \int_\Sigma \frac{c_1(T\Sigma) + \beta_\Sigma\, \mathit{u} - y}{\beta_\Sigma\, \mathit{u} - y} 
= \frac{p_1 + p_2}{c p_1\, p_2 \, y} + \int_\Sigma \frac{\left(-2g + \frac{1}{c\,p_1} +   \frac{1}{cp_2}  \right) \mathit{u}}{\beta_\Sigma\, \mathit{u} - y},
$$
for  a generator $u\in H^2(\Sigma;\mathbb{Z})$. Hence, 
%the fact that 
%$$\iota_\Sigma^* c_1^{S^1}(TM)= c_1^{S^1}(T\Sigma) + b_\Sigma \mathit{u} -y $$ 
%and 
%$$
%c_1^{S^1}(T\Sigma) = c_1(T\Sigma) = \chi(\Sigma)\, \mathit{u} = \left(2-2g+\sum_{i=1}^2 (\frac{1}{cp_i} - 1)\right)\, \mathit{u},  
%$$
%where $\chi(\Sigma)$ is the orbifold Euler characteristic of $\Sigma$  and $g$ is the genus of the underlying topological surface $|\Sigma |$ (cf. Section~\ref{sec:loc}). 
%$$
%\frac{p_1 + p_2}{c\, p_1\, p_2 \, y} + \frac{1}{y}\left(2g -  \frac{1}{cp_1} - \frac{1}{cp_1} \right) \int_\Sigma \sum_{j=0}^\infty \left(\frac{b_\Sigma \mathit{u}}{y}\right)^j \mathit{u}
%$$
%and then
$$
\frac{p_1 + p_2}{c \, p_1\, p_2 } =  - 2g + \frac{1}{c\, p_1} + \frac{1}{c p_2},
$$
which yields that $g=0$. Moreover, by \eqref{eq:comp_cond2} and since $\beta_\Sigma = \frac{q}{c p_1
    p_2} >0$, we have that
%
%applying the localization formula to the equivariant symplectic form $\omega^\sharp$ gives
%$$
%0 = \int_M \omega^\sharp =   \frac{q \,\iota_F^*\omega^\sharp}{c\, p_1p_2\, y^2} +  \int_\Sigma  \frac{\iota_\Sigma^* \omega^\sharp}{b_\Sigma \, \mathit{u} -y } =
% -  \frac{q H_{min}y}{c\, p_1p_2\, y^2}  +  \int_\Sigma \frac{\omega - H_{max} \, y}{b_\Sigma \, \mathit{u} -y},
%$$
%where we denote by $H_{min}$ and $H_{max}$ respectively the minimum and  the maximum   values of $H$. We conclude that
%$$
% \frac{q H_{min}}{c\, p_1p_2\, y}  = \int_\Sigma \left(\frac{\omega}{y} - H_{max}\right)\, \sum_{j=0}^\infty (-1)^j \, \left(\frac{b_\Sigma \mathit{u}}{y}\right)^j
%$$
%and so we have 
$\mathrm{area}(\Sigma)=\beta_\Sigma \left(H_{max}- H_{min}
\right)$. We conclude that, in this case, the only possible labeled
multigraph is given by $(I)$. (The case of $\Sigma$ at the minimum
corresponds to $(II)$.)

Suppose that there are two fixed orbi-surfaces, $\Sigma_-$ and
$\Sigma_+$, that are respectively the minimum and maximum of
$H$. Since $(M,\omega,H)$ contains no isolated fixed points, by the
second equation in \eqref{eq:linear_system} we have that
%Using the localization formula with $\alpha=1$ yields 
%\begin{equation}
%\label{eq:loc3}
%0 = \int_M 1  =  \int_{\Sigma_-}  \frac{1}{b_{\Sigma_-} \, \mathit{u} + y} + \int_{\Sigma_+}  \frac{1}{b_{\Sigma_+} \, \mathit{u} - y},
%\end{equation}
$\beta_{\Sigma_-}=-\beta_{\Sigma_+}$. We claim that $\Sigma_-$ and
$\Sigma_+$ contain the same number of singular points. This holds if
there are no singular points, so we may assume that there is at least
one. By Corollary \ref{cor:action_preserves_structure_group}, any
singular point is fixed by the $S^1$-action. Hence, it must lie on
either $\Sigma_-$ or $\Sigma_+$. Let $x_-$ be a singular point on
$\Sigma_-$. By Theorem \ref{cor::localformpt} and Lemma
\ref{lem::isotropy}, $x_-$ is the south pole of an isotropy
orbi-sphere. Since $(M,\omega,H)$ contains no isolated fixed points,
the north pole of this isotropy orbi-sphere is necessarily a singular
point $x_+$ on $\Sigma_+$. Moreover, the orders of the orbifold
structure groups of $x_-$ and $x_+$ coincide This defines a function between the set of
singular points on $\Sigma_-$ and those on $\Sigma_+$ that preserves orbifold structure groups. This function
is injective by Theorem \ref{cor::localformpt} and Lemma
\ref{lem::isotropy}. Reversing the roles of $\Sigma_-$ and $\Sigma_+$,
this function is a bijection between the sets of singular points of
$\Sigma_-$ and $\Sigma_+$ that preserves orbifold structure
groups.

Let $n$ be the number of singular points on $\Sigma_{\pm}$. By applying Theorem~\ref{prop:localization} to
  the first equivariant Chern class $\alpha=c_1^{S^1}(TM)$, we have that  
$$
\int_{\Sigma_-}  \frac{\left( 2 - n - 2g_- +  \sum_{i=1}^{n} \frac{1}{p_{i}} \right) \mathit{u}} {\beta_{\Sigma_-} \, \mathit{u} + y} +
\int_{\Sigma_+}   \frac{\left( 2 - n - 2g_+ + \sum_{i=1}^{n} \frac{1}{p_{i}}\right)\mathit{u}}{\beta_{\Sigma_-} \, \mathit{u} - y}=0,
$$
where $g_\pm$ is  the genus of $\Sigma_\pm$. Hence, $g_+ =g_-$ so that, by Theorem \ref{thm::classorbisurface}, $\Sigma_-$ is diffeomorphic to $\Sigma_+$. By the first equation in \eqref{eq:linear_system}, we have 
%Applying the Localization Formula to the equivariant symplectic form $\omega^\sharp$ of $M$ yields 
%$$
%0=\int_M \omega^\sharp = \int_{\Sigma_-} \frac{\iota_{\Sigma_-}^*\omega^\sharp}{b_{\Sigma_-}\, \mathit{u} + y} + \int_{\Sigma_+} \frac{\iota_{\Sigma_+}^*\omega^\sharp}{b_{\Sigma_+} \, \mathit{u} - y}, 
%$$
%and so
%$$
%\int_{\Sigma_-} \frac{\omega - H_{min} y}{b_{\Sigma_-}\, \mathit{u} + y} =   \int_{\Sigma_+} \frac{\omega - H_{max} y}{b_{\Sigma_-} \, \mathit{u} + y}, 
%$$
%implying that 
%$$
%\int_{\Sigma_-} \left(\frac{\omega}{y} - H_{min}\right) \hspace{-.1cm}\sum_{j=0}^\infty (-1)^j \left(\frac{b_{\Sigma_-} \mathit{u}}{y}\right)^j  =\int_{\Sigma_+}  \left(\frac{\omega}{y} - H_{max}\right) \sum_{j=0}^\infty (-1)^j \left(\frac{b_{\Sigma_-} \mathit{u}}{y}\right)^j
%$$
%and consequently  
$$\mathrm{area}(\Sigma_+) - \mathrm{area}(\Sigma_-) =- b_{\Sigma_-} \left(H_{max} - H_{min}\right).$$
We conclude that, in this case, the only possible multigraph is $(III)$. \\

To finish the proof, we turn to the compatibility
conditions. Conditions $(i)$ and $(ii)$ on all multigraphs are
determined by Theorem~\ref{cor::localformpt} as well as condition
$(iii)$ in multigraphs $(I)$ and $(II)$. Condition $(iv)$ in
multigraphs $(I)$ and $(II)$ is a consequence of \eqref{eq:bsigma-} as
shown above. Condition $(v)$ in both multigraphs $(I)$ and
$(II)$ follows from Proposition~\ref{def::EulerOrbibundle} applied to
the principal $S^1$-orbi-bundle associated to the normal orbi-bundle of
$\Sigma$.
%
%	
%	Condition \ref{cond} for Multigraphs $(I)$ and $(II)$ is an immediate consequence of Proposition~\ref{def::EulerOrbibundle} applied to the circle orbibundle associated to the normal orbi-bundle of $\Sigma_+$ or $\Sigma_-$ respectively.  
%	
%	
%	by applying the localization formula (cf. Theorem~\ref{prop:localization}) to the equivariant  form 1. Indeed, for Multigraph $(I)$ we have
%\begin{align*}
%		0=\int_M1&=\int_{\Sigma_{+}}\frac{1}{b_{\Sigma_{+}}u-y}+ \frac{1}{qc}\cdot \frac{1}{e(\nu_{x})}
%		= - \frac{1}{y}\int_{\Sigma_{+}}\sum_{j=0}^\infty \left(\frac{b_{\Sigma_{+}}u}{y}\right)+ \frac{1}{qc}\cdot \frac{q^2}{p_1 p_2\,y^2}\\
%		&=\left(-b_{\Sigma_{+}}+ \frac{q}{c p_1 p_2}\right)\frac{1}{y^2} = \left( - \left(b + \frac{l_1}{p_1 c} + \frac{l_2}{p_2 c}\right) +\frac{q}{c p_1 p_2} \right) \frac{1}{y^2}  \\ & =  - \left(b +\frac{l_1 p_2  + l_2 p_1}{p_1 p_2 c}  - \frac{q}{c p_1 p_2} \right)\frac{1}{y^2}, 
%	\end{align*}
%	where $b\in \mathbb{Z}$ is such that $(b,0;(p_1c,l_1),(p_2c,l_2))$ is the Seifert invariant of the circle orbi-bundle associated to the normal orbi-bundle of $\Sigma_+$ inside $M$ (cf.  Proposition~\ref{def::EulerOrbibundle}), %and similarly for Multigraph $(II)$. 
Note that this condition, along with $(i)$, completely determine $l_1$
and $l_2$ since we must have $p_2 l_1 = q \mod p_1$ and $p_1 l_2= q
\mod p_2$ in multigraph $(I)$ and $p_2 l^\prime_1 = q \mod p_1$ and
$p_1 l^\prime_2= q \mod p_2$ in multigraph $(II)$.  
Condition \ref{condiv} for multigraphs $(I)$ and $(II)$ is obtained by applying Lemma~\ref{lem::euler} to the $\mathbb{Z}_{p_1c}$-isotropy orbi-sphere. Indeed, for Multigraph $(I)$, the orbi-weights at the fibers over the north and south poles of the $\mathbb{Z}_{p_1c}$-isotropy orbi-sphere are $0$ and $p_2/q$ and so, by Lemma~\ref{lem::euler}, the degree of the normal orbi-bundle  of this orbi-sphere is $\frac{p_2}{p_1 q c}$ and then, by Proposition~\ref{def::EulerOrbibundle}, we have that
$$
\frac{l^\prime_1}{p_1c} + \frac{l_q}{qc} - \frac{p_2}{p_1 q c} \in \mathbb{Z}.
$$
Condition $(vi)$ for multigraph $(II)$ is obtained similarly. Note
that, for given values of $l_1,l_2$ satisfying conditions \ref{cond0}-\ref{cond} for multigraphs $(I)$ and $(II)$, the value of $l_{qc}$ is uniquely determined by \ref{condiv}. In fact, for multigraph $(I)$, we have, by \ref{condiv}, that
$$
 \frac{p_2 - l^\prime_1 q}{p_1}\in \mathbb{Z},  
$$
and so $l_{qc}$ is the unique integer satisfying $1\leq l_{qc}< qc$ and $l_{qc}  =  \frac{p_2 - l^\prime_1 q}{p_1} \mod qc$ (similarly for multigraph $(II)$).

Condition \ref{condIII} in multigraph (III) is a consequence of
Lemma~\ref{lem::deg_zero}, since the Seifert invariant of the
principal $S^1$-orbi-bundle associated to the normal orbi-bundle of a  $\mathbb{Z}_{p_i}$-isotropy orbi-sphere  corresponding to a singular point on the north pole is $(p_i,l_i^\prime)$, and the one corresponding to the south pole is $(p_i,\hat{l}_i)$. 
Condition~\ref{condv} in multigraph (III) is a consequence of
Proposition~\ref{def::EulerOrbibundle}, since the Seifert invariant of
the principal $S^1$-orbi-bundle associated to the normal orbi-bundle
of $\Sigma_+$  is $(g; \beta_0, (p_1,l_1),\ldots (p_n,l_n))$ and $\beta_0\in
\Z$. 
Note that, if the normal orbi-bundle of one of the fixed orbi-surfaces
has negative degree, genus zero and at most two singular points, then
it can be blown down, resulting in either multigraph $(I)$ or
$(II)$. We exclude these cases with condition \ref{condIII0}. 
%	
%	The graph depicted on the right side in  Figure~\ref{fig::minimalfixedsurface}, has two diffeomorphic orbi-surfaces at the extrem. Applying the localization formula on this space by integrating the equivariant form $1$, shows that the self-intersections of these orbi-surfaces need to be either both zero or differ by sign.
%	If one of the surfaces has negative self-intersection, genus zero and at most two singularities, then we can blow it down, which results in one of the first two graphs.
%	In general, the orbi-spaces corresponding to these multigraphs are is given by $S^2$-orbi-bundles over a compact orbi-surface, where the circle acts by rotation of the fibers, while fixing the north and south poles. 
\end{proof}

\subsubsection{Minimal spaces with fixed orbi-surfaces}\label{MSF} In
this section, we provide explicit Hamiltonian $S^1$-spaces whose
labeled multigraphs are those listed in Theorem~\ref{fig::minimalfixedsurface}.

\vspace{.2cm}
\noindent $\bullet$ {\bf Multigraphs $(I)$ and $(II)$:}
Consider the quotient  
$$M=\mathbb{C}P^2(p_1,p_2,q)/\mathbb{Z}_c(a,b,bl_1+al_2),$$ 
of Example~\ref{ex::orbifold_quotient}, where we may assume that
$p_1\geq p_2$. The integers $a,b,l_{1},l_{2}$ are such that $1\leq
l_{1} <c\, p_1$, $1\leq l_{2} <c\, p_2$, $\gcd(c\, p_1, l_{1}) = \gcd(c\, p_2,l_{2})=1$, 
\begin{equation}\label{eq:qq} q - l_1p_2-l_2p_1 = 0 \mod p_1p_2\, c,\end{equation}
and $b p_1- a p_2=1$ (cf. the necessary conditions for multigraphs
$(I)$ and $(II)$ in Theorem \ref{fig::minimalfixedsurface}). We equip $M$ with the symplectic form inherited from
the standard symplectic form on $\mathbb{C}P^2(p_1,p_2,q)$
(see Example~\ref{exm:wps_as_symp}). We have that
\begin{align*}  [z_0:z_1:z_2]_{c}  & =[e^{\frac{2\pi i a}{c}} z_0: e^{\frac{2\pi i b }{c}} z_1:  e^{\frac{2\pi i (b l_1+ a l_2) }{c}}  z_2]_c \\& 
%\\ & = [z_0: e^{\frac{2\pi i b }{c }} e^{-\frac{2\pi i a p_2 }{cp_1 }} z_1:  e^{\frac{2\pi i (b l_1+ a l_2) }{c}} e^{-\frac{2\pi i a q }{cp_1 }}  z_2]_c 
=  [z_0: e^{\frac{2\pi i  }{c p_1 }}  z_1:  e^{\frac{2\pi i  l_1}{cp_1}}  z_2]_c  = [e^{- \frac{2\pi i  }{c p_2 }} z_0:   z_1:  e^{-\frac{2\pi i  l_2}{cp_2}}  z_2]_c \\ 
& = [e^{ \frac{2\pi i  (aq-al_2p_1-bl_1p_1)}{q c}} z_0:  e^{\frac{2\pi i  (bq - bl_1p_2 - a l_2p_2)}{q c}}  z_1:   z_2]_c \\ 
& = [\xi_{qc} z_0: \xi_{qc}^{l_{qc}} z_1: z_2]_c,
\end{align*}
for some primitive $qc$-th root of unit $\xi_{qc}$ and $l_{qc}\in \Z$
satisfying $1\leq l_{qc}< qc$ and $l_{qc}= \frac{p_2-l_1^\prime}{p_1}
\mod qc$ where, as usual, $1\leq l^\prime_1< p_1c$ is such that
$l_1\cdot l^\prime_1 = 1 \mod p_1c$ (note that from \eqref{eq:qq} we
have $p_2- l^\prime_1q=0\mod p_1$). We see that
$[1:0:0]_c$, $[0:1:0]_c$,  and $[0:0:1]_c$ are singular points of type
$\frac{1}{p_1c}(1,l_1)$, $\frac{1}{p_2c}(1,l_2)$ and
$\frac{1}{qc}(1,l_{qc})$ respectively.

The $S^1$-action on $M$ given by
\begin{align*}
\phi: S^1 \times M & \to M \\	
	(\lambda, [z_0:z_1:z_2]_c ) & \mapsto [z_0:z_1:\lambda z_2]_c.
\end{align*}
is Hamiltonian (cf. Example \ref{exm:weighted_ham}). It 
fixes the orbi-sphere
$$\Sigma = \{[z_0:z_1:0] \in M \} \simeq \mathbb{C}P^1(p_1,p_2)/\mathbb{Z}_c$$
as well as the point $[0:0:1]_c$. The labeled
multigraph of this Hamiltonian
$S^1$-space is of type $(I)$ in 
Theorem \ref{fig::minimalfixedsurface} (cf. Example \ref{exm:wps2}). A
multigraph of type $(II)$ in Theorem \ref{fig::minimalfixedsurface} is
obtained by reversing the action.

% Moreover, a
% generic point in the orbi-sphere
% $\{[z_0:0:z_2]\in M\}$ has stabilizer
% $\mathbb{Z}_{cp_1}$ and a generic point  in
% the orbi-sphere  $\{[0:z_1:z_2]\in M\}$ has
% stabilizer $\mathbb{Z}_{cp_2}$. The labelled
% multigraph corresponding to this Hamiltonian
% $S^1$-space is the one in
% Theorem~\ref{fig::minimalfixedsurface}-$(I)$. Note
% that $p_1,p_2,l_1,l_2,l_{qc}$ satisfy
% conditions $(i)$ through $(v)$ of
% Theorem~\ref{fig::minimalfixedsurface}-$(I)$.
% Reversing the circle action we obtain a
% Hamiltonian $S^1$-space with the labelled
% multigraph in
%  Theorem~\ref{fig::minimalfixedsurface}-$(II)$.  

\vspace{.2cm}			
\noindent $\bullet$ {\bf Multigraph $(III)$:}  This type of multigraph
is that of the Hamiltonian $S^1$-action on the projectivized plane
bundle $\mathbb{P}(L\oplus\C)\to \Sigma$  of
Example~\ref{exm:projectivized_suborbifolds2}, where $L$ is the
complex line orbi-bundle that has corresponding principal $S^1$-orbi-bundle with Seifert invariant 
$$
(g;\beta_0,(p_1,\hat{l}_1^\prime),\ldots,(p_n,\hat{l}_n^\prime)),
$$
where $\beta_0=b_{\Sigma_-}-\sum_{i=1}^n \frac{\hat{l}_i^\prime}{p_i}$ (see Figure~\ref{fig:proj}). \\

Lemma \ref{lemma:existsurface1}, 
Theorems \ref{thm:uniqueness} and \ref{fig::minimalfixedsurface}, and
the above constructions immediately imply the following
result.
\begin{theorem}\label{thm:existsurface}		
  Any Hamiltonian $S^1$-space that has at least one
  fixed orbi-surface can be obtained by applying finitely many
  $S^1$-equivariant symplectic weighted blow-ups to either
  \begin{enumerate}[label=(\arabic*),ref=(\arabic*),leftmargin=*]
  \item a weighted projective plane or a quotient  of a  weighted
    projective plane by a cyclic group, or
  \item a ruled orbi-surface (a projectivized plane bundle over a closed orbi-surface).
  \end{enumerate}
\end{theorem}

\subsubsection{Blowing down to a minimal space with only isolated fixed points}
\begin{lemma}\label{lemma:only_iso}
  Any Hamiltonian $S^1$-space with only isolated fixed
  points can be obtained from a minimal Hamiltonian $S^1$-space with
  only isolated fixed points by applying finitely many
  $S^1$-equivariant symplectic weighted blow-ups. Moreover, any such minimal space has at
  least one non-extremal fixed point.
\end{lemma}
\begin{proof}
  Since blowing down does not create fixed orbi-surfaces, the first
  statement follows. Suppose that $(M,\omega,H)$ is a minimal
  Hamiltonian $S^1$-space that has no non-extremal fixed points. Then
  there are precisely two isolated fixed points $F_1$ and $F_2$ at
  which $H$ attains its minimum and maximum respectively. By Theorem~\ref{prop:localization},
  $$
  0=\int_{M} 1 = 
  % \frac{1}{p_1} \cdot \frac{1}{\frac{a_1}{p_1} \cdot \frac{a_2}{p_1}} +  \frac{1}{p_2} \cdot \frac{1}{\frac{b_1}{p_2} \cdot \frac{b_2}{p_2}} = 
  \frac{p_1}{a_1 a_2} +  \frac{p_2}{b_1 b_2},
  $$
where $p_1$ and  $p_2$ are the orders of the orbifold structure groups
of $F_1$ and $F_2$ and $\frac{a_1}{p_1},\frac{a_2}{p_1}$ and
$-\frac{b_1}{p_2},-\frac{b_2}{p_2}$ (with $a_1,a_2,b_1,b_2>0$) are the
orbi-weights  at $F_1$ and $F_2$ respectively. Since $p_i, a_i,b_i>0$
for $i=1,2$, this is impossible.  
\end{proof}

By Lemma \ref{lemma:only_iso}, in order to understand the
Hamiltonian $S^1$-spaces with only isolated fixed points, it suffices
to determine which minimal spaces with this property occur. The
following result describes the labeled multigraphs of the latter spaces.

\begin{theorem}\label{thm:existisolated} Let $(M,\omega,H)$ be a
  minimal Hamiltonian $S^1$-space with only isolated fixed points. Then the labeled multigraph of
  $(M,\omega,H)$ is one of the following:

  % can be obtained
  % through a sequence of admissible equivariant weighted blow-ups at
  % fixed points from  a Hamiltonian $S^1$-space with one of the
  % following labelled multigraphs, 
	
	\vspace{.2cm}
		\resizebox{12cm}{!}{%
			\begin{tikzpicture}[scale=.7]
				
				%%%%%%%%%%%%%%%%%%%% space (2a)				
				\coordinate (1A) at (-7.5,3);
				\coordinate[label=above:$(A)$] (2a) at (-5,-2);
				\coordinate (PP1) at (-5,3);
				\fill (PP1) circle (4pt);
				%index 2 points
				\coordinate (PP2) at (-5,0);
				\fill (PP2) circle (4pt);		
				\coordinate (PP3) at (-6,1.5);
				\fill (PP3) circle (4pt);
				\coordinate (PP4) at (-4,1.5);
				\fill (PP4) circle (4pt);		
				\draw[very thick](PP1) -- (PP3);
				\draw[very thick](PP2) -- (PP3);
				\draw[very thick](PP1) -- (PP4);
				\draw[very thick](PP4) -- (PP2);
				\coordinate[label=right:$m$] (AB) at ($(PP1)!0.5!(PP4)$);
				\coordinate[label=left:$n$] (BC) at ($(PP1)!0.5!(PP3)$);
				\coordinate[label=right:$n$] (AB) at ($(PP2)!0.5!(PP4)$);
				\coordinate[label=left:$m$] (BC) at ($(PP2)!0.5!(PP3)$);
				\coordinate[label=above:{$\frac{1}{c}(1,c-l^\prime), \color{red} H= \alpha + m \beta + n \gamma$}] (AB) at ($(PP1)$);
				\coordinate[label=below:{$\frac{1}{c}(1,c-l), \color{red} H=\alpha $}] (AB) at ($(PP2)$);
				\coordinate[label=left:{$\frac{1}{c}(1,l^\prime),$}] (AB) at ($(PP3)$);	
				\coordinate[label=left:{$\color{red} H= \alpha + m \beta $}] (AB) at (-6.2,.8);	
				\coordinate[label=right:{$\frac{1}{c}(1,l),$}] (AB) at ($(PP4)$);
				\coordinate[label=right:{$\color{red} H= \alpha + n \gamma $}] (AB) at (-4,.8);
%%%%%%%%%%%%%%%%%%%%%%%%%%%%%%%%%%%%%%%%%%%%%%%%%%%%			
				\coordinate (2A) at (3.5,3);
				\coordinate[label=above:$(B)$] (1a) at (6,-2);
				
				\coordinate (P1) at (6,3);
				\fill (P1) circle (4pt);
				%index 2 points
				\coordinate (P2) at (6,0);
				\fill (P2) circle (4pt);		
				\coordinate (P3) at (7,1.5);
				\fill (P3) circle (4pt);		
				\draw[very thick](P1) -- (P2);
				\draw[very thick](P2) -- (P3);
				\draw[very thick](P1) -- (P3);
				
				\coordinate (Q0) at (4.5,1.5);		
				\fill (Q0) circle (4pt);
				\coordinate (Q00) at (1.5,1.5);
				\fill (Q00) circle (4pt);	
				\node at ($(Q0)!.5!(Q00)$) {\ldots};

				\coordinate[label=above:${\frac{1}{k_2}(1,k_2-1)}$] (BC) at ($(Q0)$);
				\coordinate[label=below:${\color{red} H=H_2}$] (BC) at ($(Q0)$);
				\coordinate[label=above:${\frac{1}{k_r}(1,k_r-1)}$] (BC) at ($(Q00)$);
				\coordinate[label=below:${\color{red} H=H_r}$] (BC) at ($(Q00)$);
				\coordinate[label=right:$s$] (AB) at ($(P1)!0.5!(P2)$);
				\coordinate[label=right:$m$] (BC) at ($(P1)!0.5!(P3)$);	
				\coordinate[label=right:$n$] (BC) at ($(P2)!0.5!(P3)$);	
				\coordinate[label=above:{$\frac{1}{p}(1,l_{\text{max}}),\color{red} H=\alpha + m a_2$}] (BC) at ($(P1)$);
				\coordinate[label=below:{$\frac{1}{q}(1,l_{\text{min}}),\color{red} H=\alpha - n a_1$}] (BC) at ($(P2)$);
				\coordinate[label=right:{$\frac{1}{k_1}(1,l_1)$}] (BC) at	
				($(7.1,1.5)$);
				\coordinate[label=right:{$\color{red} H=\alpha $}] (BC) at	
				($(7.1,.9)$);
				
				%%%%%%%%%%%%%%%%%%%%%%%% 
			\end{tikzpicture}}\\

		\hspace{3cm}
			\resizebox{6cm}{!}{%
		\begin{tikzpicture}[scale=.8]

\coordinate (Q1) at (6,3);
\fill (Q1) circle (4pt);
\coordinate (Q2) at (6,0);
\fill (Q2) circle (4pt);	
\coordinate (Q0) at (4.5,1.5);
\fill (Q0) circle (4pt);
\coordinate (Q00) at (1.5,1.5);
\fill (Q00) circle (4pt);	
\node at ($(Q0)!.5!(Q00)$) {\ldots};

\coordinate[label=left:$m_1c$] (BC) at ($(5.9,1.5)$);	
\coordinate[label=right:$m_2c$] (BC) at ($(6.1,1.5)$);

\coordinate[label=below:${\frac{1}{qc}(1,l_{min}), \color{red} H=\alpha}$] (BC) at ($(Q2)$);
\coordinate[label=above:${\frac{1}{pc}(1,l_{max}),\color{red} H=\alpha + c m_1 m_2 a}$] (BC) at ($(Q1)$);
\coordinate[label=above:${\frac{1}{k_1}(1,k_1-1)}$] (BC) at ($(4.5,1.5)$);
\coordinate[label=below:${\color{red}  H=H_1}$] (BC) at ($(Q0)$);
\coordinate[label=above:${\frac{1}{k_r}(1,k_r-1)}$] (BC) at ($(Q00)$);
\coordinate[label=below:${\color{red}  H=H_r}$] (BC) at ($(Q00)$);

\draw[very thick](6.1,2.8) -- (6.1,0.2);
\draw[very thick](5.9,2.8) -- (5.9,0.2);
\coordinate[label=above:$(C)$] (BCCC) at ($(6,-1.5)$);
	\end{tikzpicture}}

where the labels satisfy the following compatibility conditions:
\begin{itemize}[leftmargin=*]
	\item \textbf{Multigraph (A):} assuming, w.l.o.g. that $m\geq
          n$, we have that
	\vspace{.2cm}
	\begin{enumerate}[label=(\roman*)]
	\item $1\leq l,l^\prime < c$, $\gcd(l,c)=\gcd(l^\prime,c)=1$ and $l\cdot l^\prime=1 \mod c$;
	\item $c = 0 \mod \gcd(n,m)$ and $\gcd\left(\frac{ml+n}{ \gcd(n,m)},c\right)=\frac{c}{ \gcd(n,m)}$.\label{condAii}
	\end{enumerate}
	
	\vspace{.2cm}
	\item \textbf{Multigraph (B):}
	\vspace{.2cm}
	\begin{enumerate}[label=(\roman*)]
	\item $1\leq l_{\text{min}} <  q$, $1\leq l_{\text{max}} <  p$,  $1\leq l_{1} <  k_1$ and \\
	$\gcd (l_{\text{min}},q)=\gcd (l_{\text{max}},p)=\gcd (l_{1},k_1)=1$;
	\item $\gcd \left(  \frac{m l_1+n}{\gcd(m,n)}, k_1 \right) =\frac{k_1}{\gcd (m,n)}$;
	\item $\gcd \left(  \frac{m l -s}{\gcd(m,s)}, p  \right) =\frac{p }{\gcd(m,s)}$, where $1\leq l < p$ is such that\\ $l=l_{\text{max}}$ if $s\geq m$ and $l \cdot l_{\text{max}}=1 \mod p$  otherwise;
	\item $\gcd \left(  \frac{n l -s}{\gcd(n,s)}, q  \right) =\frac{ q }{\gcd(n,s)}$, where $1\leq l < q$ is such that\\ $l=l_{\text{min}}$ if $n\geq s$ and $l \cdot l_{\text{min}}=1 \mod q$ otherwise;
	\item $(p  -x m s)n +q m = k_1 s$, where $x=\sum_{i=2}^r k_i$, and  $p\geq x ms$, $q \geq x n s$;
	\item $p  a_2 = q  a_1 +  s \sum_{i=2}^r k_i (H_i-\alpha)$, and $\alpha-na_1<H_i< \alpha + m a_2$.\label{cc}
	\end{enumerate} 

\vspace{.2cm}
	\item \textbf{Multigraph (C):}
	\vspace{.2cm}
	\begin{enumerate}[label=(\roman*)]
		\item $\gcd(m_1,m_2)=\gcd(m_i,p)=\gcd(m_i,q)=1$, with $i=1,2$;
		\item $1\leq l_{\text{min}} < c q$, $1\leq l_{\text{max}} < c p$, and \\ $\gcd( l_{\text{min}} ,cq)=\gcd(l_{\text{max}},cp)=1$;
		\item $\gcd(m_i l_{\text{min} } - m_j ,q c) = q$ and $\gcd(m_j l_{\text{max} } - m_i ,p c) = p$, where $m_i=\max\{m_1,m_2\}$ and $m_j=\min\{m_1,m_2\}$;
	\item $p + q  =c\, m_1m_2(k_1+\dots+k_n)$;\label{cc2}
	\item $p  \, a =   \sum_{i=1}^r k_i (H_i -   \alpha )$, and $\alpha < H_i < \alpha + c \, m_1m_2 a$.\label{ccc}
\end{enumerate}
\end{itemize}  
\end{theorem}

\begin{proof}
  We fix a generic compatible metric on $(M,\omega,H)$. Since
  $(M,\omega,H)$ is minimal, by Corollary \ref{cor::index2gradientsph}
  any non-trivial chain of $S^1$-invariant gradient 
  orbi-spheres contains at most one non-extremal fixed point. By
  Theorem~\ref{cor::localformpt} and Lemma \ref{lem::isotropy}, every
  vertex of the labeled multigraph of $(M,\omega,H)$ is incident to at most two edges. Consequently,
  we may assume that it is given by one of the six
  possibilities depicted in Figure~\ref{fig:indextwo}. We remind the
  reader that an
  edge label could be $1$, i.e., that edge corresponds to a free
  $S^1$-invariant gradient orbi-sphere. In what follows, we use the
  localization formula 
  of Theorem~\ref{prop:localization} to analyze each case: we either
  rule it out, or determine the compatibility conditions stated
  above. Throughout we use the fact that $(M,\omega,H)$ is minimal,
  so that any fully embedded $S^1$-invariant orbi-sphere has
  non-negative self-intersection.

\renewcommand{\thefigure}{\thesection.\arabic{figure}}
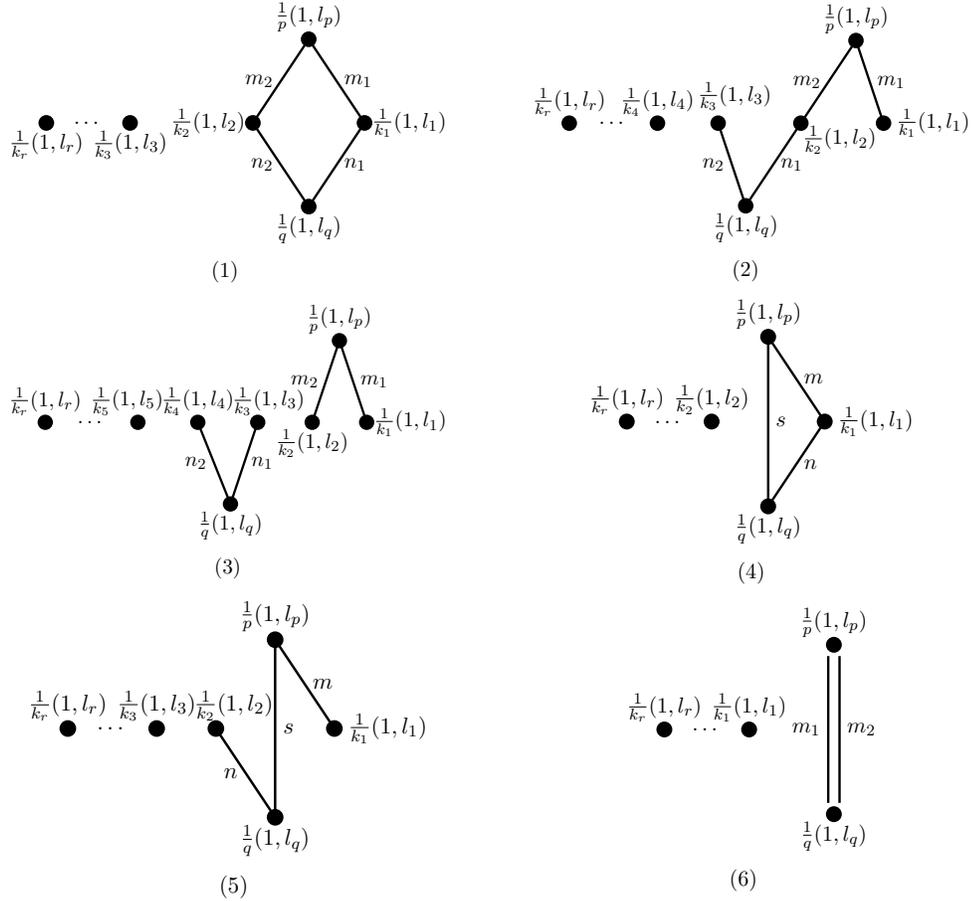
\begin{figure}[h!]
\noindent\begin{minipage}{.55\linewidth}
\centering
\resizebox{6cm}{!}{%
\begin{tikzpicture}
%%%%%%%
% (1)
%%%%%%%%			
			\coordinate (1A) at (-0.5,3);
			\coordinate (PP1) at (4,3);
			\fill (PP1) circle (4pt);
			%index 2 points
			\coordinate (PP2) at (4,0);
			\fill (PP2) circle (4pt);		
			\coordinate (PP3) at (3,1.5);
			\fill (PP3) circle (4pt);
			\coordinate (PP4) at (5,1.5);
			\fill (PP4) circle (4pt);		
			\draw[very thick](PP1) -- (PP3);
			\draw[very thick](PP2) -- (PP3);
			\draw[very thick](PP1) -- (PP4);
			\draw[very thick](PP4) -- (PP2);
			\coordinate[label=above:{$\frac{1}{p}(1,l_p)$}] (AB) at ($(PP1)$);
			\coordinate[label=below:{$\frac{1}{q}(1,l_q)$}] (AB) at ($(PP2)$);
			\coordinate[label=left:{$\frac{1}{k_2}(1,l_{2})$}] (AB) at ($(PP3)$);					
			\coordinate[label=right:{$\frac{1}{k_1}(1,l_{1})$}] (AB) at ($(PP4)$);
			\coordinate[label=right:$m_1$] (AB) at ($(PP1)!0.5!(PP4)$);
			\coordinate[label=left:$m_2$] (BC) at ($(PP1)!0.5!(PP3)$);
			\coordinate[label=right:$n_1$] (AB) at ($(PP2)!0.5!(PP4)$);
			\coordinate[label=left:$n_2$] (BC) at ($(PP2)!0.5!(PP3)$);
			
			\coordinate (Q0) at (-0.7,1.5);		
			\fill (Q0) circle (4pt);
			\coordinate (Q00) at (0.8,1.5);
			\fill (Q00) circle (4pt);	
			\node at ($(Q0)!.5!(Q00)$) {\ldots};
			\coordinate[label=below:${\frac{1}{k_r}(1,l_r)}$] (BC) at ($(Q0)$);
			\coordinate[label=below:${\vspace{.2 cm}\frac{1}{k_3}(1,l_3)}$] (BC) at ($(Q00)$);
			\coordinate[label=above:$(1)$] (2a) at (2.5,-1.5);					
\end{tikzpicture}}
\end{minipage}%					
\noindent\begin{minipage}{.55\linewidth}
\centering
\resizebox{6cm}{!}{%
\begin{tikzpicture}	
%%%%%%%
% (2)
%%%%%%%%									
					
			\coordinate (2A) at (7.5,3);			
			\coordinate (P1) at (12.5,3);
			\fill (P1) circle (4pt);
			%index 2 points
			\coordinate (P2) at (10.5,0);
			\fill (P2) circle (4pt);		
			\coordinate (P3) at (13,1.5);
			\fill (P3) circle (4pt);	
			\coordinate (P4) at (11.5,1.5);
			\fill (P4) circle (4pt);	
			\draw[very thick](P1) -- (P4);
			
			\draw[very thick](P1) -- (P3);
			\draw[very thick](P2) -- (P4);
			
			\coordinate (W0) at (10,1.5);		
			\fill (W0) circle (4pt);
			
			\coordinate (Q0) at (8.9,1.5);		
			\fill (Q0) circle (4pt);
			
			\coordinate (Q00) at (7.3,1.5);
			\fill (Q00) circle (4pt);
			
			\draw[very thick](P2) -- (W0);	
			\node at ($(Q0)!.5!(Q00)$) {\ldots};
			\coordinate[label=above:${\frac{1}{k_4}(1,l_4)}$] (BC) at ($(Q0)$);
			\coordinate[label=above:${\frac{1}{k_r}(1,l_r)}$] (BC) at ($(Q00)$);
			\coordinate[label=above:${\frac{1}{k_3}(1,l_3)}$] (BC) at ($(10.3,1.6)$);
			\coordinate[label=right:${\frac{1}{k_2}(1,l_2)}$] (BC) at ($(11.4,1.2)$);
			\coordinate[label=left:$m_2$] (AB) at ($(P1)!0.5!(P4)$);
			\coordinate[label=right:$n_1$] (AB) at ($(P2)!0.5!(P4)$);
			\coordinate[label=right:$m_1$] (BC) at ($(P1)!0.5!(P3)$);	
			\coordinate[label=left:$n_2$] (BC) at ($(P2)!0.5!(W0)$);	
			\coordinate[label=above:{$\frac{1}{p}(1,l_p)$}] (BC) at ($(P1)$);
			\coordinate[label=below:{$\frac{1}{q}(1,l_q)$}] (BC) at ($(P2)$);
			\coordinate[label=right:{$\frac{1}{k_1}(1,l_1)$}] (BC) at ($(13.1,1.5)$);
			\coordinate[label=above:$(2)$] (2A) at (10.5,-1.5);				
\end{tikzpicture}}
\end{minipage} \\  

\noindent\begin{minipage}{.55\linewidth}
\centering
%\vspace{7pt}
\resizebox{6cm}{!}{%
		\begin{tikzpicture}
%%%%%
% (3)
%%%%%			
			\coordinate (1A) at (-5.7,3);
					\coordinate[label=above:$(3)$] (BC) at ($(-2.55,-1.5)$);
					\coordinate (P1) at (-0.5,3);
					\fill (P1) circle (4pt);
					%index 2 points
					\coordinate (P2) at (-2.5,0);
					\fill (P2) circle (4pt);
					
					\coordinate (P3) at (-1,1.5);
					\fill (P3) circle (4pt);
					
					\coordinate (P4) at (-2,1.5);
					\fill (P4) circle (4pt);	
					\coordinate (P5) at (0,1.5);
					\fill (P5) circle (4pt);
					
					\draw[very thick](P1) -- (P3);
					\draw[very thick](P1) -- (P5);
					\coordinate (Q0) at (-3.1,1.5);		
					\fill (Q0) circle (4pt);
					
					\coordinate (W0) at (-4.2,1.5);		
					\fill (W0) circle (4pt);
					\coordinate (Q00) at (-5.9,1.5);
					\fill (Q00) circle (4pt);
					
					\draw[very thick](P2) -- (P4);
					\draw[very thick](P2) -- (Q0);	
					\node at ($(W0)!.5!(Q00)$) {\ldots};
					\coordinate[label=above:${\frac{1}{k_4}(1,l_4)}$] (BC) at ($(Q0)$);
					\coordinate[label=above:${\frac{1}{k_r}(1,l_r)}$] (BC) at ($(Q00)$);
					\coordinate[label=below:${\frac{1}{k_2}(1,l_2)}$] (BC) at ($(P3)$);
					\coordinate[label=above:${\frac{1}{k_3}(1,l_3)}$] (BC) at ($(-1.8,1.5)$);
					\coordinate[label=above:${\frac{1}{k_5}(1,l_5)}$] (BC) at ($(-4.4,1.5)$);
					\coordinate[label=left:$m_2$] (BC) at ($(P1)!0.5!(P3)$);
					\coordinate[label=right:$m_1$] (BC) at ($(P1)!0.5!(P5)$);	
					\coordinate[label=right:$n_1$] (BC) at ($(P2)!0.5!(P4)$);	
					\coordinate[label=left:$n_2$] (BC) at ($(P2)!0.5!(Q0)$);
					\coordinate[label=above:{$\frac{1}{p}(1,l_p)$}] (BC) at ($(P1)$);
					\coordinate[label=below:{$\frac{1}{q}(1,l_q)$}] (BC) at ($(P2)$);
					\coordinate[label=right:{$\frac{1}{k_1}(1,l_1)$}] (BC) at ($(0,1.5)$);
\end{tikzpicture}}
\end{minipage}%					
\noindent\begin{minipage}{.55\linewidth}
\centering
\resizebox{4.5cm}{!}{%
\begin{tikzpicture}	
%%%%%
% (4)
%%%%%			
			\coordinate (2A) at (2.5,3);
			\coordinate[label=above:$(4)$]  (BC) at (4.7,-1.5);
			\coordinate (P1) at (5,3);
			\fill (P1) circle (4pt);
			%index 2 points
			\coordinate (P2) at (5,0);
			\fill (P2) circle (4pt);		
			\coordinate (P3) at (6,1.5);
			\fill (P3) circle (4pt);		
			\draw[very thick](P1) -- (P2);
			\draw[very thick](P2) -- (P3);
			\draw[very thick](P1) -- (P3);
			
			\coordinate (Q0) at (4,1.5);		
			\fill (Q0) circle (4pt);
			\coordinate (Q00) at (2.5,1.5);
			\fill (Q00) circle (4pt);	
			\node at ($(Q0)!.5!(Q00)$) {\ldots};
			\coordinate[label=above:${\frac{1}{k_2}(1,l_2)}$] (BC) at ($(Q0)$);
			\coordinate[label=above:${\frac{1}{k_r}(1,l_r)}$] (BC) at ($(Q00)$);
			\coordinate[label=right:$s$] (AB) at ($(P1)!0.5!(P2)$);
			\coordinate[label=right:$m$] (BC) at ($(P1)!0.5!(P3)$);	
			\coordinate[label=right:$n$] (BC) at ($(P2)!0.5!(P3)$);	
			\coordinate[label=above:{$\frac{1}{p}(1,l_p)$}] (BC) at ($(P1)$);
			\coordinate[label=below:{$\frac{1}{q}(1,l_q)$}] (BC) at ($(P2)$);
			\coordinate[label=right:{$\frac{1}{k_1}(1,l_1)$}] (BC) at ($(6.1,1.5)$);		
\end{tikzpicture}}
\end{minipage} \\

\noindent\begin{minipage}{.55\linewidth}
\centering
\resizebox{5.5cm}{!}{%
\begin{tikzpicture}
%%%%%
% (4)
%%%%%				
\coordinate (2A) at (-3.5,3);
\coordinate[label=above:$(5)$] (Bc) at ($(-.7,-1.5)$);

\coordinate (P1) at (0,3);
\fill (P1) circle (4pt);
%index 2 points
\coordinate (P2) at (0,0);
\fill (P2) circle (4pt);		
\coordinate (P3) at (1,1.5);
\fill (P3) circle (4pt);		
\draw[very thick](P1) -- (P2);

\draw[very thick](P1) -- (P3);

\coordinate (Q0) at (-1,1.5);		
\fill (Q0) circle (4pt);

\coordinate (W0) at (-2,1.5);
\fill (W0) circle (4pt);
\coordinate (Q00) at (-3.5,1.5);
\fill (Q00) circle (4pt);

\draw[very thick](P2) -- (Q0);	
\node at ($(W0)!.5!(Q00)$) {\ldots};
\coordinate[label=above:${\frac{1}{k_3}(1,l_3)}$] (BC) at ($(W0)$);
\coordinate[label=above:${\frac{1}{k_r}(1,l_r)}$] (BC) at ($(Q00)$);
\coordinate[label=above:${\frac{1}{k_2}(1,l_2)}$] (BC) at ($(-0.7,1.5)$);
\coordinate[label=right:$s$] (AB) at ($(P1)!0.5!(P2)$);
\coordinate[label=right:$m$] (BC) at ($(P1)!0.5!(P3)$);	
\coordinate[label=left:$n$] (BC) at ($(P2)!0.5!(Q0)$);	
\coordinate[label=above:{$\frac{1}{p}(1,l_p)$}] (BC) at ($(P1)$);
\coordinate[label=below:{$\frac{1}{q}(1,l_q)$}] (BC) at ($(P2)$);
\coordinate[label=right:{$\frac{1}{k_1}(1,l_{1})$}] (BC) at ($(1.1,1.5)$);
\end{tikzpicture}}
\end{minipage}%
\noindent\begin{minipage}{.55\linewidth}
\centering
\resizebox{3.5cm}{!}{%
		\begin{tikzpicture}
%	
%	%%second graph
%		
\coordinate (3A) at (4,3);
\coordinate[label=above:$(6)$] (BC) at (5.4,-1.5);

\coordinate (Q1) at (7,3);
\fill (Q1) circle (4pt);
%index 2 points
\coordinate (Q2) at (7,0);
\fill (Q2) circle (4pt);	
\coordinate (Q0) at (5.5,1.5);
\fill (Q0) circle (4pt);
\coordinate (Q00) at (4,1.5);
\fill (Q00) circle (4pt);	
\node at ($(Q0)!.5!(Q00)$) {\ldots};

\coordinate[label=left:$m_1$] (BC) at ($(6.9,1.5)$);	
\coordinate[label=right:$m_2$] (BC) at ($(7.1,1.5)$);

\coordinate[label=below:${\frac{1}{q}(1,l_{q})}$] (BC) at ($(Q2)$);
\coordinate[label=above:${\frac{1}{p}(1,l_{p})}$] (BC) at ($(Q1)$);
\coordinate[label=above:${\frac{1}{k_1}(1,l_1)}$] (BC) at ($(Q0)$);
\coordinate[label=above:${\frac{1}{k_r}(1,l_r)}$] (BC) at ($(Q00)$);

\draw[very thick](7.1,2.8) -- (7.1,0.2);
\draw[very thick](6.9,2.8) -- (6.9,0.2);
\end{tikzpicture}}
\end{minipage}
\caption{Possible labeled multigraphs for a minimal Hamiltonian
  $S^1$-space with only interior fixed points.}
\label{fig:indextwo}
\end{figure}

%\begin{itemize}[leftmargin=*]
%%%%%%%%%%%%%%%%%%%%%%%%%%%%%%%%%%%%%%% Graph 1

\vspace{.2cm}
\noindent $\bullet$ {\bf Multigraph (1):} By
Theorem~\ref{prop:localization}, we have that
\begin{align}\label{local_gr1}
	0=\int_M1=\frac{p}{m_1m_2}+\frac{q}{n_1n_2}-\frac{k_1}{m_1n_1}-\frac{k_2}{m_2n_2}-k_3-\cdots-k_r
\end{align}
and so
\begin{equation}\label{eq:1.loc}
	0<\frac{p}{m_1m_2}+\frac{q}{n_1n_2} =  \frac{k_1}{m_1n_1}+\frac{k_2}{m_2n_2}+k_3+\cdots+k_r. 
\end{equation}
We claim that the $S^1$-invariant orbi-spheres corresponding to the
four edges are fully embedded. This follows immediately if all of the
labels are greater than one. If $m_1 =
1$, i.e., the edge corresponds to a free $S^1$-invariant gradient orbi-sphere, then there is at most one isotropy orbi-sphere that contains the
fixed point $x$ corresponding to the vertex at the top of the labeled
multigraph. Hence, by
Theorem \ref{cor::localformpt} and Lemma \ref{lem::isotropy}, one of
the orbi-weights of $x$ satisfies the hypotheses of Lemma
\ref{lemma:free_smooth} and so the free
$S^1$-invariant gradient orbi-sphere corresponding to the edge under
consideration is fully embedded. This
very argument applies {\em mutatis mutandis} to all edges, thus
yielding the desired result.

Since $(M,\omega,H)$ is minimal and the $S^1$-invariant orbi-spheres corresponding to the
four edges are fully embedded, the degree of the normal orbi-bundle of
each such orbi-sphere is non-negative. Hence, by
Lemma~\ref{lem::euler}, we have that
\begin{equation}\label{eq::1.1}
	 \frac{k_1}{n_1} \ge\frac{p}{m_2},  \quad
         \frac{k_2}{m_2}\ge\frac{q}{n_1}, \quad \frac{k_1}{m_1}
         \ge\frac{q}{n_2},   \quad \text{and} \quad
         \frac{k_2}{n_2}\ge\frac{p}{m_1}.
\end{equation}
By the first two inequalities in \eqref{eq::1.1}, we have that
	$$
	\frac{k_1}{m_1n_1}+\frac{k_2}{m_2n_2} \geq
        \frac{p}{m_1m_2}+\frac{q}{n_1n_2}.
	$$
Whence by \eqref{eq:1.loc} and the last two inequalities in
\eqref{eq::1.1},  we have that
$$
\frac{k_1}{m_1n_1}=\frac{k_2}{m_2n_2} = \frac{p}{m_1m_2} =
\frac{q}{n_1n_2} \, , \text{ and } k_i = 0 \text{ for } i\geq 3,
$$ 
i.e., there are at most two non-extremal fixed points. Then
\begin{align}\label{graph1}
	\frac{k_1}{m_1}=\frac{q}{n_2},\quad \quad\frac{k_1}{n_1}=\frac{p}{m_2},\quad \quad\frac{k_2}{m_2}=\frac{q}{n_1}\quad\text{and}\quad\frac{k_2}{n_2}=\frac{p}{m_1},
\end{align}
and so the normal orbi-bundles of all the $S^1$-invariant orbi-spheres
on the  two depicted chains have degree $0$. We conclude that
multigraph (1) should be as in
Figure~\ref{fig:minimal1}, where the labels satisfy \eqref{graph1}.

\renewcommand{\thefigure}{\thesection.\arabic{figure}}
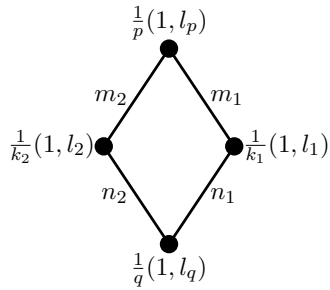
\begin{figure}[h!]
\centering
\resizebox{4.5cm}{!}{%		
\begin{tikzpicture}
		\coordinate (PP1) at (4,3);
		\fill (PP1) circle (4pt);
		%index 2 points
		\coordinate (PP2) at (4,0);
		\fill (PP2) circle (4pt);		
		\coordinate (PP3) at (3,1.5);
		\fill (PP3) circle (4pt);
		\coordinate (PP4) at (5,1.5);
		\fill (PP4) circle (4pt);		
		\draw[very thick](PP1) -- (PP3);
		\draw[very thick](PP2) -- (PP3);
		\draw[very thick](PP1) -- (PP4);
		\draw[very thick](PP4) -- (PP2);
		\coordinate[label=above:{$\frac{1}{p}(1,l_p)$}] (AB) at ($(PP1)$);
		\coordinate[label=below:{$\frac{1}{q}(1,l_q)$}] (AB) at ($(PP2)$);
		\coordinate[label=left:{$\frac{1}{k_2}(1,l_{2})$}] (AB) at ($(PP3)$);					
		\coordinate[label=right:{$\frac{1}{k_1}(1,l_{1})$}] (AB) at ($(PP4)$);
		\coordinate[label=right:$m_1$] (AB) at ($(PP1)!0.5!(PP4)$);
		\coordinate[label=left:$m_2$] (BC) at ($(PP1)!0.5!(PP3)$);
		\coordinate[label=right:$n_1$] (AB) at ($(PP2)!0.5!(PP4)$);
		\coordinate[label=left:$n_2$] (BC) at ($(PP2)!0.5!(PP3)$);
	\end{tikzpicture}}
\caption{The multigraph (1) is necessarily of this form.}
\label{fig:minimal1}
\end{figure}

We may assume that $m_1\geq m_2$. 
%Since the normal orbi-bundles of all the $S^1$-invariant orbi-spheres on the two chains have degree $0$, we have, 
By Proposition~\ref{def::EulerOrbibundle}, we have that
\begin{equation}\label{eq::ll}
l_p k_2 + l_2 p= 0 \mod k_2 \,p,
\end{equation}
whence $l_2 p = 0 \mod k_2$ and $ l_p k_2 = 0 \mod p$.
Thus, since
$$\gcd(k_2,l_2)= \gcd(p,l_p)=1,$$
we have that $p = 0\mod
k_2$ and $ k_2 = 0 \mod p$, i.e., $k_2=p$. Hence, by \eqref{graph1},
$n_2=m_1$ and $k_1=q$. Similarly, 
one can show that $k_2=q=p$ and $m_2=n_1$ and obtain a multigraph as
in $(A)$ with $c=p$, $m=m_1$ and $n=m_2$.
Moreover, by \eqref{eq::ll}, $l_p=-l_2 \mod p$, implying that
$l_p=p-l_2$. Similarly, we can obtain the labels for the other fixed
points. Note that, by Proposition~\ref{def::EulerOrbibundle}, we have
that
\begin{equation}\label{eq:labels}
l_2l_1 = 1 \mod p.
\end{equation}

Using the notation in  $(A)$, we denote by $F_1$ and $F_4$, the
minimal and maximal points respectively, by $F_2$ the fixed point with
orbi-weights  $-\frac{m}{c}$,  $\frac{n}{c}$, and by $F_3$ the
remaining fixed point. By Theorem~\ref{prop:localization} applied to
the equivariant symplectic form $\omega^{S^1}$, 
\begin{equation}\label{eq:area1}
  0=\int_M \omega^{S^1} = - \frac{c}{m n y}\left(H(F_1)+H(F_4)-H(F_2)-H(F_3)\right).
\end{equation}
On the other hand, by \eqref{eq:44}, we have that
\begin{equation*}
  \begin{split}
    & H(F_2)-H(F_1)=m \,a_1, \quad H(F_3)-H(F_1)=n \,a_2 , \\
    & H(F_4)-H(F_2)=n \,a_3, \quad H(F_4)-H(F_3)=m \,a_4,
  \end{split}
\end{equation*}
for some positive real numbers $a_1,a_2,a_3,a_4$; each of these is the symplectic area
of the $S^1$-invariant orbi-sphere connecting the corresponding pair of fixed
points.  By \eqref{eq:area1} we conclude that $a_1=a_4$ and $a_2=a_3$
and we obtain the moment map labels in Multigraph $(A)$, where
$\beta=a_1=a_4$ and $\gamma=a_2=a_3$. Conditions $(i)$ and $(ii)$ for
Multigraph $(A)$ are a consequence of Theorem~\ref{cor::localformpt}
and of \eqref{eq:labels}. 

\vspace{.2cm}
\noindent $\bullet$ {\bf Multigraphs (2) and (3):}
First we consider the multigraph $(2)$. By Theorem~\ref{prop:localization}, we have that
\begin{align*}
	0=\int_M1=\frac{p}{m_1m_2}+\frac{q}{n_1n_2}-\frac{k_1}{m_1}-\frac{k_2}{m_2n_1}-\frac{k_3}{n_2}-k_4-\cdots -k_r,
\end{align*}
so that
\begin{equation}
  \label{eq:63}
  0< \frac{p}{m_1m_2}+\frac{q}{n_1n_2} = \frac{k_1}{m_1}+\frac{k_2}{m_2n_1}+\frac{k_3}{n_2}+k_4+\cdots+k_r.
\end{equation}
As in the case of multigraph (1), each $S^1$-invariant orbi-sphere
corresponding to an edge is fully embedded. Hence, since
$(M,\omega,H)$ is minimal, Lemma~\ref{lem::euler} yields
\begin{align*}
	k_1\ge\frac{p}{m_2},\quad\frac{k_2}{n_1}\ge\frac{p}{m_1},\quad\frac{k_2}{m_2}\ge\frac{q}{n_2},\quad   k_3\ge\frac{q}{n_1} .
\end{align*}
Hence,
$$	
\frac{k_1}{m_1}+\frac{k_2}{m_2n_1}+\frac{k_3}{n_2} \geq
2\frac{p}{m_1m_2}+\frac{q}{n_1n_2} >   \frac{p}{m_1m_2}+\frac{q}{n_1n_2},
$$
which contradicts \eqref{loc_cp2}, since $k_i \geq 0$ for each $i \geq
4$. Hence, there is no minimal Hamiltonian $S^1$-space with labeled
multigraph as in $(2)$. A similar argument can be used to rule out the
existence of a minimal Hamiltonian $S^1$-space with labeled
multigraph as in $(3)$.

% \vspace{.2cm}
% \noindent $\bullet$ {\bf Multigraph (3):} The arguments used for
% multigraph $(2)$ can be replicated in this case to show that there is
% no minimal Hamiltonian $S^1$-space with labeled
% multigraph as in $(3)$.
%  no  localization formula  of Theorem~\ref{prop:localization}  gives 
% \begin{align*}
% 	0=\int_M1=\frac{p}{m_1m_2}+\frac{q}{n_1n_2}-\frac{k_1}{m_1}-\frac{k_2}{m_2}-\frac{k_3}{n_1}-\frac{k_4}{n_2}-k_5-\cdots -k_r
% \end{align*}
% and, by Lemma~\ref{lem::euler}, we have 
%  \begin{align*}
%  	 k_1\ge \frac{p}{m_2}, \quad k_2\ge \frac{p}{m_1},\quad k_3\ge \frac{q}{n_2}, \quad k_4\ge \frac{q}{n_1}.
%  \end{align*}
% Hence,
% $$
% 	0<\frac{p}{m_1m_2}+\frac{q}{n_1n_2}= \frac{k_1}{m_1}+\frac{k_2}{m_2}+\frac{k_3}{n_1}+\frac{k_4}{n_2}+k_5+\cdots+k_r
% $$
% and
% $$
% 	\frac{k_1}{m_1}+  \frac{k_2}{m_2}+\frac{k_3}{n_1}+\frac{k_4}{n_2}\geq 2\left( \frac{p}{m_1m_2}+\frac{q}{n_1n_2} \right),
% $$
% and so there are no multigraphs as in $(3)$.

\vspace{.2cm}
\noindent $\bullet$ {\bf Multigraph (4):} The same argument as in the
above cases yields that 
\begin{align}\label{loc_cp2}
	0=\int_M1&=\frac{p}{ms}+\frac{q}{ns}-\frac{k_1}{mn}-k_2-\cdots-k_r, 
\end{align}
and
\begin{align}\label{graph4}
	\frac{k_1}{n}\ge \frac{p}{s},\quad\frac{k_1}{m}\ge\frac{q}{s}.
\end{align}
Therefore,
\begin{align}\label{eq:loc}
	0<  \frac{p}{ms}+\frac{q}{ns}=  \frac{k_1}{mn}+k_2+\cdots+k_r \leq 2 \frac{k_1}{mn},
\end{align}
and so
\begin{align*}
	k_2+\cdots+k_r\le \frac{k_1}{mn}.
\end{align*}
Note that if $s=1$ the labeled multigraph is a particular case of
either $(1)$ or $(2)$. Hence, we may assume that $s > 1$.

If $r=1$, i.e., if there is only one non-extremal fixed point, we claim
that either $p,q,k_1$ are pairwise coprime or they all have a common
divisor. Set $c = \gcd(p,q)$. By \eqref{eq:loc}, $np+mq=k_1s$. Hence,
$c$ divides $k_1s$. If $c$ has no common divisor with $k_1$
then it must divide $s$ and then, by  Theorem~\ref{cor::localformpt},
we have that $c$ divides $m$ and $n$ since $s/q,n/q$ are the
orbi-weights at the minimum and $-s/p,-m/q$ are the orbi-weights at
the maximum. However, this is impossible by
Theorem~\ref{cor::localformpt}   as the orbi-weights at the singular
point of order $k_1$ are $\frac{m}{k_1},-\frac{n}{k_1}$. 

If $r>1$, set $x:=\sum_{i=2}^r k_i$. By \eqref{graph4} and
\eqref{eq:loc}, $p\geq xms$
and $q\geq xns$. Let $x$ be a non-extremal isolated fixed point of
order $k_i$ with $i=2,\ldots, r$. By the form of the labeled
multigraph, by Theorem \ref{cor::localformpt} and Lemma
\ref{lem::isotropy}, the
orbi-weights at the non-extremal isolated fixed point of order $k_i$ with $i=2,\ldots, r$
are $\frac{1}{k_i}, -\frac{1}{k_i}$. Hence, by
Theorem~\ref{cor::localformpt}, we have that
$l_i=k_i-1$. Additionally, note that if $m=1$, then the free $S^1$-invariant gradient orbi-sphere
that `joins' the non-extremal isolated fixed point with orbifold
structure group of order $k_i$, $i\geq 2$, and the maximum is fully embedded. This
follows from the same argument as in the case of multigraph
$(1)$. Hence, since $(M,\omega,H)$ is minimal, by
Lemma~\ref{lem::euler} $s k_i \geq p$ for each $i=2,\ldots,r$ and so  
$sx \geq p(r-1)$.
This implies that  $r=2$ and $p=k_2 s$. Therefore, by \eqref{eq:loc}, $q=k_1s$. Similarly, if $n=1$ then  $r=2$, $q=k_2 s$ and
$p=k_1s$. If  $m=n=1$ then $r=2$, $k_1=k_2$ and  $p=q=k_1s$.

For any $r\geq 1$,  let us denote the minimal and maximal point of
$(M,\omega,H)$ by $F_-$ and $F_+$ respectively, and by $F_i$
($i=1,\ldots, r$) the non-extremal fixed points. By Theorem~\ref{prop:localization},
\begin{equation}\label{eq:area2}
		0=\int_M \omega^{S^1} = - \frac{1}{y}\left(\frac{p}{sm}H(F_+)+\frac{q}{sn}H(F_-)- \frac{k_1}{mn}H(F_1)-\sum_{i=2}^r k_i H(F_i)\right).
\end{equation}
On the other hand, by \eqref{eq:44}, we have
$$
H(F_1)-H(F_-)=n a_1 \quad \text{and}  \quad  H(F_{+})-H(F_1)=m a_2, 
$$
for some positive real numbers $a_1,a_2$; each of these is the
symplectic area of the $S^1$-invariant orbi-sphere connecting $F_1$ to $F_-$ and $F_+$. By \eqref{eq:area2} and \eqref{eq:loc}, we have
\begin{equation}\label{eq:cond_v}
p a_2 - q a_1 - s \sum_{i=2}^r H(F_i) k_i + s H(F_1)  \sum_{i=2}^r k_i = 0,
\end{equation}
i.e., we obtain the labeled multigraph  (B). Conditions $(i)$ to
$(iii)$ are direct consequences of Theorem~\ref{cor::localformpt}, while
$(iv)$ and $(v)$ are just \eqref{eq:area2} and \eqref{eq:cond_v}.

\vspace{.2cm}
\noindent $\bullet$ {\bf Multigraph (5):} The same argument as in the
above cases yields that 
\begin{align}\label{equ:loc_graph5}
  0=\int_M1=\frac{p}{ms}+\frac{q}{ns}-\frac{k_1}{m}-\frac{k_2}{n}-k_3-\cdots-k_r
\end{align}
and
\begin{align}\label{equ:condition_graph5}
  k_1\ge\frac{p}{s},\quad k_2\ge\frac{q}{s}.
  % \quad \frac{q}{n}\ge\frac{p}{m}.
\end{align}
Then
\begin{align*}
  0< \frac{k_1}{m}+\frac{k_2}{n} +k_3+\cdots+k_r=\frac{p}{ms}+\frac{q}{ns}\le\frac{k_1}{m}+\frac{k_2}{n},
\end{align*}
and so $k_3=\cdots=k_r=0$, $p=k_1s$ and $q=k_2s$. Moreover, we may
assume that $s, m,n > 1$, as, otherwise, we fall in either case $(1)$
or $(4)$. 

Since the orbi-weights at the singular point of order $p=k_1s$ are
$-\frac{s}{k_1s},-\frac{m}{k_1s}$, by Theorem~\ref{cor::localformpt},
$m=m^\prime s$ for some positive integer $m^\prime < m$. Analogously,
$n=n^\prime s$ for some positive integer $n^\prime < n$. Hence, we obtain a multigraph as in Figure~\ref{mulit5}.	
\renewcommand{\thefigure}{\thesection.\arabic{figure}}
\begin{figure}[h!]
  \centering
\resizebox{3.8cm}{!}{%		
\begin{tikzpicture}		
		\coordinate (P1) at (2,3);
		\fill (P1) circle (4pt);
		%index 2 points
		\coordinate (P2) at (2,0);
		\fill (P2) circle (4pt);		
		\coordinate (P3) at (3,1.5);
		\fill (P3) circle (4pt);		
		\draw[very thick](P1) -- (P2);
		
		\draw[very thick](P1) -- (P3);
		
		\coordinate (Q0) at (1,1.5);		
		\fill (Q0) circle (4pt);
		
		\draw[very thick](P2) -- (Q0);	

		\coordinate[label=above:${\frac{1}{k_2}(1,l_2)}$] (BC) at ($(1,1.5)$);
		\coordinate[label=right:$s$] (AB) at ($(P1)!0.5!(P2)$);
		\coordinate[label=right:$s m^\prime$] (BC) at ($(P1)!0.5!(P3)$);	
		\coordinate[label=left:$s n^\prime$] (BC) at ($(P2)!0.5!(Q0)$);	
		\coordinate[label=above:{$\frac{1}{k_1s}(1,l_p)$}] (BC) at ($(P1)$);
		\coordinate[label=below:{$\frac{1}{k_2s}(1,l_q)$}] (BC) at ($(P2)$);
		\coordinate[label=right:{$\frac{1}{k_1}(1,l_{1})$}] (BC) at ($(3.1,1.5)$);
		\end{tikzpicture}}\caption{The multigraph in (5) is
                necessarily of this form.}\label{mulit5}
		\end{figure}
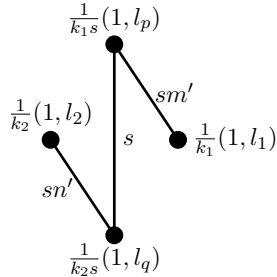
Since $s >1$, the edges with labels $sn'$ and $sm'$ correspond to
isotropy orbi-spheres and are, therefore, fully embedded. Hence, by
Lemma~\ref{lem::euler}, each of their normal orbi-bundles has degree
zero. By Proposition~\ref{def::EulerOrbibundle}, we have that
$$
l_q + l^\prime_2 s = 0 \mod k_2 s,
$$
where $1\leq l^\prime_2 < k_2$ is such that $l^\prime_2 \cdot l_2=1
\mod k_2$. Consequently, $l_q = 0 \mod s$, contradicting $\gcd(l_q,k_2
s)=1$. Hence, there is no minimal Hamiltonian $S^1$-space with labeled
multigraph as in $(5)$.

\vspace{.2cm}	
\noindent $\bullet$ {\bf Multigraph (6):} 
We may assume that $m_1m_2 >1$ as, otherwise, we fall in a particular
case of $(4)$. Theorem~\ref{prop:localization} applied to the equivariant form $1$ yields
\begin{align}\label{graph6}
	(k_1+\cdots+k_r)m_1m_2=p+q.
\end{align}
We set $c:=\gcd (m_1,m_2)$. By Theorem~\ref{cor::localformpt}, $c$
divides both $p$ and $q$, $\frac{m_i}{c}$ is coprime with both $p$ and
$q$ for $i=1,2$, and, since the orbi-weights at the fixed point of
order $k_i$ are $\frac{1}{k_i},-\frac{1}{k_i}$, then $l_i = k_i-1$ for any $i =1,\ldots, r$.

We write $m_i= m^\prime_i c$, $p=p^\prime c$ and $q=q^\prime c$ and
denote the minimal and maximal points of $(M,\omega,H)$ by $F_-$ and
$F_+$, and the other fixed points by $F_1,\ldots, F_r$. Theorem~\ref{prop:localization} yields
\begin{equation}\label{eq:area6}
0=\int_M \omega^{S^1} = - \frac{1}{y}\left(\frac{p^\prime}{m_1^\prime m_2^\prime c}H(F_+)+\frac{q^\prime}{m_1^\prime m_2^\prime c}H(F_-)- \sum_{i=1}^r k_i H(F_i)\right).
\end{equation}
On the other hand, by \eqref{eq:44}, we have
$$
H(F_+)-H(F_-)=m_1^\prime  c \,a_1 =   m_2^\prime  c \,a_2,
$$
where the positive real numbers $a_1,a_2$ are the symplectic areas of
the $\mathbb{Z}_{m_1}$- and $\mathbb{Z}_{m_2}$-isotropy orbi-spheres
respectively. By combining \eqref{graph6} and \eqref{eq:area6}, we conclude that 
\begin{equation}\label{eq:lastcond}
 p^\prime  \, a -   \sum_{i=1}^r H(F_i) k_i +   H(F_-)  \sum_{i=1}^r k_i = 0,
\end{equation}
where $a=\frac{a_1}{m_2^\prime} =\frac{a_2}{m_1^\prime}\in \Z$.  We
obtain multigraph $(C)$, where we write $m_i, p$ and $q$ instead of
$m_i^\prime,p^\prime$ and $q^\prime$ for ease of notation. Conditions
$(i)$ to $(iii)$ are  direct consequences of
Theorem~\ref{cor::localformpt} and $(iv)$ and $(v)$ are just
\eqref{graph6} and \eqref{eq:lastcond}.
\end{proof}	

\subsubsection{Minimal spaces without fixed orbi-surfaces}\label{MSI}
In this section, we provide explicit Hamiltonian $S^1$-spaces whose
labeled multigraphs are listed in Theorem \ref{thm:existisolated}. 

\vspace{.2cm}
\noindent $\bullet$ {\bf Multigraph $(A)$:} Let $M_c= \left(
  \mathbb{C} P^1 \times\mathbb{C}  P^1 \right) /\mathbb{Z}_c$ be as in
Example~\ref{ex:quotient_cp1xcp1_finite}. We equip it with the
symplectic form that descends from the symplectic form $2 c \beta \omega_0 \oplus
2 c \gamma \omega_0$ on $\mathbb{C} P^1 \times\mathbb{C}  P^1$, where
$\omega_0$  is the Fubini-Study form on $\mathbb{C}  P^1 $ and $\gamma, \beta\in \R^+$. 
Consider the Hamiltonian $S^1$-action on $M_c$ of Example~\ref{exm:circle_action_cp1cp1} with Hamiltonian function
$$
H([[z_0:z_1],[w_0:w_1]]_c)=\alpha + m\beta \frac{\lvert z_1\rvert^2}{\lvert z_0\rvert^2+\lvert z_1\rvert^2} +  n \gamma \frac{\lvert w_1\rvert^2}{\lvert w_0\rvert^2+\lvert w_1\rvert^2}, 
$$
with $\alpha \in \R$. We may assume that
$m\geq n$.

By \eqref{eq:30}, a coordinate system centered at $F_1=[[1:0],[1:0]]_c$ is given by 
$$
(z,w)\in \mathbb{C}^2 \longmapsto [[1:z],[1:w]]_c \in M_c,
$$
with $(z,w) \sim ( \xi_{c} \, z , \xi_{c}^{-l}\, w)$ and so the fixed
point $F_1$ has type $\frac{1}{c}(1,c-l)$. The $S^1$-action in these coordinates is given by
$$
\lambda \cdot ([[1: z],[1: w]]_c) = [[1 :\lambda^{\frac{m}{c}} z],[1 : \lambda^{ \frac{n}{c}} w]]_c,
$$
and so the orbi-weights at $F_1$ are $\frac{m}{c}$ and
$\frac{n}{c}$. Moreover, we have $H(F_1)=\alpha$. Similarly, the fixed
point $F_2=[[1:0],[0:1]]_c$ has type $\frac{1}{c}(1,l)$, and
orbi-weights $\frac{m}{c}, - \frac{n}{c}$ and $H(F_2)=\alpha +
n\gamma$. The fixed point  $F_3=[[0:1],[1:0]]_c$ has type
$\frac{1}{c}(1,l^\prime)$ with $1\leq l^\prime < c$ satisfying $
l^\prime \cdot l = 1 \mod c$, the orbi-weights at $F_3$ are
$-\frac{m}{c},  \frac{n}{c}$, and $H(F_3)=\alpha + m\beta$. The fixed
point $F_4=[[0:1],[0:1]]_c$  has type $\frac{1}{c}(1,c- l^\prime)$,
orbi-weights $-\frac{m}{c},-\frac{n}{c}$, and $H(F_4)=\alpha + m\beta
+n \gamma$. Hence, the labeled multigraph of this Hamiltonian
$S^1$-space is as in Theorem~\ref{thm:existisolated} (A). We remark
that, if $c=1$, then $M=\mathbb{C} P^1 \times\mathbb{C}  P^1$ and the
$S^1$-action is the standard diagonal action. \\

\noindent $\bullet$ {\bf Multigraph $(B)$ with $r=1$:} % (Quotients of weighted projective planes by finite cyclic groups):
Let $M=\mathbb{C} P^2(p, k_1, q)/\mathbb{Z}_c(a,a l_q + b l_p, b)$ be
as in Example~\ref{ex::orbifold_quotient}
for positive integers $a,b,c,p,q,k_1,l_p,l_q$ such that $p,q,k_1$ are
pairwise coprime, $pb- q a=1$ and $1\leq l_p<p$, $1\leq l_q < q$ are
such that $\gcd(1,l_p)=\gcd(1,l_q)=1$ and $l_qp+l_pq=k_1 \mod
pqc$. Consider the symplectic form on $M$ inherited from the the
standard symplectic form on $\mathbb{C} P^2(p, k_1, q)$ of Example~\ref{exm:wps_as_symp} and the $S^1$-action
\begin{align*}
\phi : S^1 \times M &\to M \\
(\lambda, [z_0:z_1:z_2]_c) & \mapsto[\lambda^{\frac{s}{q}}z_0:\lambda^{\frac{n}{q}}z_1:z_2]_c
\end{align*}
with $s,n\in \Z_+$ such that $k_1 s= mq + np$ for some $m\in \Z_+$ and
$$
s l_q = n \mod cq, \quad s l_p = m \mod cp \quad \text{and} \quad m l_1 = -n \mod{ck_1}.
$$ 
This action has three isolated fixed points  $F_1=[1:0:0]_c$,
$F_2=[0:1:0]_c$ and $F_3=[0:0:1]_c$ and has the following Hamiltonian function
$$
H([z_0:z_1:z_2]_c) = K +\frac{1}{c} \left( \frac{sk_1\lvert z_0\rvert^2 + n p \lvert z_1\rvert^2}{\lvert z_0\rvert^2 + \lvert z_1\rvert^2 +\lvert z_2\rvert^2 }\right),
$$
where $K \in \mathbb{R}$. 
A coordinate system centered at $F_1$ is given by 
$$
(z,w)\in \mathbb{C}^2 \longmapsto [1:z:w]_c\in M,
$$
with $(z,w) \sim (\xi_{cp}^{l_p} \, z , \xi_{cp}\, w)$ and so $F_1$ has type $\frac{1}{cp}(1,l_p)$. The $S^1$-action in these coordinates is given by
$$
\lambda \cdot  [1: z: w]_c  = [\lambda^{\frac{s}{cq}}: \lambda^{ \frac{n}{cq}}z: w]_c  = [1 : \lambda^{-\frac{m}{cp}} z : \lambda^{- \frac{s}{cp}}  w]_c,
$$
and so the  orbi-weights at $F_1$ are $-\frac{m}{cp}, -\frac{s}{cp}$. Moreover, we have $H(F_1)=K+\frac{sk_1}{2c}$.
Similarly, $F_2$ has type $\frac{1}{ck_1}(1,l_1)$ with $l_1 l_p + l_q = 0 \mod ck_1$, orbi-weights  $\frac{m}{ck_1},- \frac{n}{ck_1}$, and $H(F_2)=K + \frac{np}{2c}$. Finally, $F_3$ has type $\frac{1}{cq}(1,l_q)$, orbi-weights $\frac{s}{cq}, \frac{n}{cq}$, and $H(F_3)=K$.

If we set $\alpha := K + \frac{np}{2c}$, $a_1:=\frac{p}{2c}$ and  $a_2:=\frac{q}{2c}$, then
$$H(F_1)=\alpha + m a_2, \quad H(F_2)=\alpha \quad \text{and} \quad H(F_3)=\alpha - na_1.$$
%The multigraphs for these $S^1$-orbi-spaces are as in Theorem~\ref{thm:existisolated} $(B)$ with $r=1$, where 
Moreover, setting $p:=p^\prime c$, $q:=q^\prime c$, $k_1 := k_1^\prime c$, $m:=m^\prime c$, $n:=n^\prime c$, $s:=s^\prime$ and  
$$
l_{\max}:=\left\{ \begin{array}{ll} l_p & \text{if $m\leq s$} \\  \\ l^\prime_p & \text{otherwise}
\end{array}\right. \quad \text{and} \quad  l_{\min}:= \left\{ \begin{array}{ll} l_q & \text{if $s\geq n$}\\ \\  l^\prime_q &\text{otherwise},
\end{array}\right.
$$ 
with $ l^\prime_p\cdot l_p = 1 \mod cp $ and $ l^\prime_q\cdot l_q= 1
\mod cq$, the above Hamiltonian $S^1$-space has labeled multigraph as
in Theorem \ref{thm:existisolated} (B) with $r=1$. \\
%and 
%$$
%l_{\min}= \left\{ \begin{array}{ll} l_q & \text{if $s\geq n$}\\ \\  \hat{l}_q &\text{otherwise}
%\end{array}\right.
%$$ 
%with $\hat{l}_q\cdot l_q= 1 \mod cq$. 
%In this figure $\alpha$ and $\beta$ are positive real numbers, $\alpha$, $\alpha-\frac{n}{c k_1 q} \beta$ and $\alpha+\frac{m}{c k_1 p}\beta$ are moment map labels.

\noindent $\bullet$ {\bf Multigraph $(B)$  with $r>1$:}
Let $(M,\omega,H)$ be a Hamiltonian $S^1$-space with labeled
multigraph as in Theorem~\ref{fig::minimalfixedsurface} $(I)$. We set
$s:=p_1c$, $n:=p_2c$ and, abusing notation to denote $qc$ by $q$, we assume that  $s,n,q$ satisfy 
$$(p  -x m s)n +q m = k_1 s,$$ 
where  $x=\sum_{i=2}^r k_i$, for some $p,m,k_i\geq 1$. Note that $M$
is a quotient of a weighted projective plane by a finite cyclic group,
where the group is allowed to be trivial.
We claim that a Hamiltonian $S^1$-space with labeled multigraph as in Theorem~\ref{thm:existisolated}  (B) with $r>1$ can be
obtained from $(M,\omega,H)$ by performing finitely many
$S^1$-equivariant weighted blow-ups and
blow-downs (see Figure~\ref{fig:constrgraph4}). In what follows,
seeing as we may adjust the size of a blow-up appropriately so as
to satisfy the necessary conditions, we do not specify it, denoting it
with $\bullet$.

\renewcommand{\thefigure}{\thesection.\arabic{figure}}
\begin{figure}[h!]
\centering
\resizebox{9.5cm}{!}{%
\begin{tikzpicture}
%\begin{figure}[h!]
%\begin{center}
%	\resizebox{12cm}{!}{%
%	\begin{tikzpicture}			
			\coordinate (point) at  (-6.5,1);
			\fill (point) circle (2pt);
			\fill  (-6.5,3) circle (8pt);
			
			\draw[thick](-6.6,2.6) -- (-6.6,1.2);
			\draw[thick](-6.4,2.6) -- (-6.4,1.2);
			\coordinate[label=right:{$n$}] (c1) at (-6.4,2);
			\coordinate[label=left:{$s$}] (c1) at (-6.6,2);
				\coordinate[label=right:{$\frac{1}{q}(1,l_{q}) \color{red}$}] (c2) at (-7.2,0.5);
				\coordinate[label=above:{$\frac{1}{s}(1,l_1),\frac{1}{n}(1,l_2)$}] (AAA)  at (-6.4,3.8);
			\coordinate[label=above:{$g=0,A , \color{red} H=\alpha + \frac{A sn }{q}$}] (ABB)  at (-6.5,3.3);
			\draw[->, thick] (-4,2.2) to (-3.2,2.2);	
	\end{tikzpicture}	\hspace{1cm}
\begin{tikzpicture}
		\coordinate (P1) at (-4,3);
		\fill (P1) circle (8pt);
		%index 2 points
		\coordinate (P2) at (-4,0);
		\fill (P2) circle (4pt);		
		\coordinate (P3) at (-3,1.5);
		\fill (P3) circle (4pt);		
		\draw[very thick](P1) -- (P2);
		\draw[very thick](P2) -- (P3);
		\draw[very thick](P1) -- (P3);

		\coordinate[label=right:$s$] (AB) at ($(P1)!0.5!(P2)$);
		\coordinate[label=right:$m$] (BC) at ($(P1)!0.5!(P3)$);	
		\coordinate[label=right:$n$] (BC) at ($(P2)!0.5!(P3)$);	
		\coordinate[label=above:{$\frac{1}{s}(1,l_1),\frac{1}{m}(1,l_m)$}] (BC) at (-4,3.8);
		\coordinate[label=above:{$g=0$}] (ABB)  at (-4,3.2);
		\coordinate[label=below:{$\frac{1}{q}(1,l_q)$}] (BC) at ($(P2)$);
		\coordinate[label=right:{$\frac{1}{k_1}(1,l_1)$}] (BC) at ($(-2.9,1.5)$);
		\draw[->, thick] (-.5,1.4) to (0.5,1.4); 
\end{tikzpicture}}	
			\resizebox{9.5cm}{!}{%
\begin{tikzpicture}		
		%%%%%%%%%%%%%%%%%%%%%%%%%%%%%%%%%%%%%%				
		\coordinate (P1) at (-4,3);
		\fill (P1) circle (8pt);
		%index 2 points
		\coordinate (P2) at (-4,0);
		\fill (P2) circle (4pt);		
		\coordinate (P3) at (-3,1.5);
		\fill (P3) circle (4pt);		
		\draw[very thick](P1) -- (P2);
		\draw[very thick](P2) -- (P3);
		\draw[very thick](P1) -- (P3);
		
		\coordinate[label=right:$s$] (AB) at ($(P1)!0.5!(P2)$);
		\coordinate[label=right:$m$] (BC) at ($(P1)!0.5!(P3)$);	
		\coordinate[label=right:$n$] (BC) at ($(P2)!0.5!(P3)$);	
		\coordinate[label=above:{$\frac{1}{s}(1,l_1),\frac{1}{m}(1,l_m)$}] (BC) at (-4,3.8);
		\coordinate[label=above:{$g=0$}] (BC) at ($(-4,3.2)$);
		\coordinate[label=below:{$\frac{1}{q}(1,l_q)$}] (BC) at ($(P2)$);
		\coordinate[label=right:{$\frac{1}{k_1}(1,l_1)$}] (BC) at ($(-3,1.5)$);
						
		\coordinate (W0) at (-4.7,1.5);
		\fill (W0) circle (4pt);
		\coordinate (W00) at (-6,1.5);
		\fill (W00) circle (4pt);	
		\node at ($(W0)!.5!(W00)$) {\ldots};
		
		\coordinate[label=above:${\frac{1}{k_2}(1,l_2)}$] (BC) at ($(-4.7,1.6)$);
		\coordinate[label=above:${\frac{1}{k_r}(1,l_r)}$] (BC) at ($(-6,1.6)$);
				\draw[->, thick] (-7.5,1.5) to (-7,1.5);
				\draw[->, thick] (-1,1.5) to (-0.5,1.5);
\end{tikzpicture}	\hspace{.5cm}	
\begin{tikzpicture}		
		%%%%%%%%%%%%%%%%%%%%%%%%%%%%%%%%%%%%%			
		
		\coordinate (P1) at (-4,3);
		\fill (P1) circle (4pt);
		%index 2 points
		\coordinate (P2) at (-4,0);
		\fill (P2) circle (4pt);		
		\coordinate (P3) at (-3,1.5);
		\fill (P3) circle (4pt);		
		\draw[very thick](P1) -- (P2);
		\draw[very thick](P2) -- (P3);
		\draw[very thick](P1) -- (P3);
		
		\coordinate[label=right:$s$] (AB) at ($(P1)!0.5!(P2)$);
		\coordinate[label=right:$m$] (BC) at ($(P1)!0.5!(P3)$);	
		\coordinate[label=right:$n$] (BC) at ($(P2)!0.5!(P3)$);	
		\coordinate[label=above:{$\frac{1}{p}(1,l_p)$}] (BC) at ($(-4,3)$);
		\coordinate[label=below:{$\frac{1}{q}(1,l_q)$}] (BC) at ($(P2)$);
		\coordinate[label=right:{$\frac{1}{k_1}(1,l_1)$}] (BC) at ($(-3,1.5)$);
						
		\coordinate (W0) at (-4.7,1.5);
		\fill (W0) circle (4pt);
		\coordinate (W00) at (-6,1.5);
		\fill (W00) circle (4pt);	
		\node at ($(W0)!.5!(W00)$) {\ldots};
		
		\coordinate[label=above:${\frac{1}{k_2}(1,l_2)}$] (BC) at ($(-4.7,1.6)$);
		\coordinate[label=above:${\frac{1}{k_r}(1,l_r)}$] (BC) at  ($(-6,1.6)$);
	\end{tikzpicture}}
%	\end{center}
\caption{Constructing a Hamiltonian $S^1$-space with labeled
  multigraph as in Theorem \ref{thm:existisolated} (B).}
\label{fig:constrgraph4}
\end{figure}
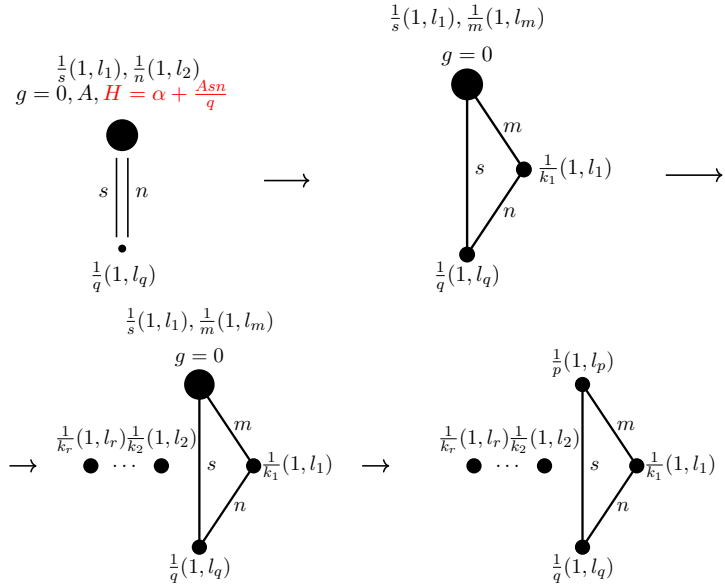

To this end, we first perform a weighted blow-up
$\text{WBU}(n,0,-1,\frac{m}{n},\frac{k_1}{n},\bullet)$ at the
singular point of order $n$ on the maximal orbi-sphere, as in
Proposition~\ref{prop::desing}-(IV). This generates an exceptional divisor
that is an $S^1$-invariant orbi-sphere with north pole a singular
point of order $m$ on the maximal orbi-surface and south pole a
singular point of order $k_1$ (see the labeled multigraph on the right
in the first row of Figure \ref{fig:constrgraph4}). By
Corollary~\ref{lem::blowupchange0} and \eqref{eq:bsigma-}, the normal
orbi-bundle of the maximal orbi-sphere $\Sigma$ in the blown-up
Hamiltonian $S^1$-space has degree
$$\frac{q}{sn}-\frac{k_1}{nm}=x-\frac{p}{ms}.$$
Then we perform $r-1$ weighted blow-ups $\text{WBU}(1,0,-1,1,k_i,\bullet)$  at $r-1$ regular points in $\Sigma$, 
%using  circle actions with weights $b_{i,1}=1$ and $b_{i,2}=k_i$, where $b_{i,1}$ is the weight on the direction tangent to $\Sigma$
and obtain a space with  $r$ interior fixed points of orders
$k_1,\ldots, k_r$ (see the labeled multigraph on the left in the
second row of Figure~\ref{fig:constrgraph4}).  By
Corollary~\ref{lem::blowupchange0}, this reduces the degree of the
normal orbi-bundle of the maximal orbi-sphere to
$-\frac{p}{ms}$. Since this degree is negative we can blow-down the
maximal orbi-sphere and obtain a singular point of order $p$ and  the
multigraph (B) of Theorem~\ref{thm:existisolated}. Note that  we can
choose the appropriate blow-up sizes so that the moment map labels
satisfy condition  \ref{cc}.\\

\noindent $\bullet$ {\bf Multigraph $(C)$:} For positive integers $m_1,m_2,p,c$ such that $m_1,m_2,p$ are pairwise coprime,
consider  the quotient
$M:=\mathbb{C}P^2(m_1,m_2,p)/\mathbb{Z}_c(l,1,0)$ as in
Example~\ref{ex::orbifold_quotient}, with an $S^1$-action given by  
\begin{align*}
\phi : S^1 \times M &\to M \\
(\lambda, [z_0:z_1:z_2]_c)  & \mapsto [z_0:z_1:\lambda z_2]_c= [\lambda^{-\frac{m_1}{p}}z_0:\lambda^{-\frac{m_2}{p}}z_1:z_2],
\end{align*}
We endow $M$ with the symplectic form of
Example~\ref{exm:cyclic_quotients_symp}. The above action is
Hamiltonian and the labeled multigraph for the resulting Hamiltonian
$S^1$-space $(M,\omega,H)$ is depicted on the left of the first row of
Figure~\ref{multigraph5}. We claim that a Hamiltonian $S^1$-space with labeled multigraph as in Theorem~\ref{thm:existisolated}  (C) can be
obtained from $(M,\omega,H)$ by performing finitely many
$S^1$-equivariant weighted blow-ups and
blow-downs (see Figure \ref{multigraph5}). As above, we do not specify
the size of each blow-up.

\renewcommand{\thefigure}{\thesection.\arabic{figure}}
\begin{figure}[h!]
\centering
\resizebox{8cm}{!}{%
\begin{tikzpicture}
			\coordinate (Q1) at (-4,3);
			\fill (Q1) circle (4pt);
			%index 2 points
			\coordinate (Q2) at (-4,0);
			\fill (Q2) circle (8pt);

			\coordinate[label=left:$m_1c$] (BC) at ($(-4.1,1.5)$);	
			\coordinate[label=right:$m_2c$] (BC) at ($(-3.9,1.5)$);
			
			\coordinate[label=below:${\frac{1}{m_1c}(1,l_1),\frac{1}{m_2c}(1,l_2)}$] (BC) at ($(-4,-0.2)$);
			\coordinate[label=below:${g=0,A}$] (BC) at ($(-4,-0.8)$);
			
			\coordinate[label=above:${\frac{1}{pc}(1,l_{max})}$] (BC) at ($(Q1)$);

			\draw[very thick](-3.9,2.8) -- (-3.9,0.4);
			\draw[very thick](-4.1,2.8) -- (-4.1,0.4);
			
			%%%%%%%%%%%%%%%%%%%%%%%%%%%%%%%%%%%%%%
			\draw[->, thick] (-2.5,1.5) to (-2,1.5);	
			
			\coordinate (Q1) at (0.5,3);
			\coordinate (Q2) at (0.5,0);
	
\end{tikzpicture}
\hspace{.5cm}
\begin{tikzpicture}										
			\coordinate (Q1) at (2,3);
			\fill (Q1) circle (4pt);
			%index 2 points
			\coordinate (Q2) at (2,0);
			\fill (Q2) circle (8pt);				
			\coordinate (W0) at (0.8,1.5);
			\fill (W0) circle (4pt);
			\coordinate (W00) at (-1.2,1.5);
			\fill (W00) circle (4pt);	
			\node at ($(W0)!.5!(W00)$) {\ldots};

			\coordinate[label=left:$m_1c$] (BC) at ($(1.9,1.5)$);	
			\coordinate[label=right:$m_2c$] (BC) at ($(2.1,1.5)$);
			
			\coordinate[label=below:${\frac{1}{m_1c}(1,l_1),\frac{1}{m_2c}(1,l_2)}$] (BC) at ($(2,-0.5)$);				
			\coordinate[label=below:${}$] (BC) at ($(2,-0.7)$);

			\coordinate[label=above:${\frac{1}{pc}(1,l_{max})}$] (BC) at ($(Q1)$);
			
			\draw[very thick](2.1,2.8) -- (2.1,0.4);
			\draw[very thick](1.9,2.8) -- (1.9,0.4);
			\coordinate[label=above:${\frac{1}{k_1}(1,k_1-1)}$] (BC) at ($(.8,1.5)$);
			\coordinate[label=above:${\frac{1}{k_r}(1,k_r-1)}$] (BC) at ($(-1.2,1.5)$);

%			%%%%%%%%%%%%%%%%%%%%%%%%%%%%%%%%%%%%%%
			\draw[->, thick] (3.5,1.6) to (4,1.6);		
\end{tikzpicture}}
\hspace{-.5cm}
	\resizebox{6.5cm}{!}{%
\begin{tikzpicture}
\coordinate (Q1) at (6.5,3.5);
\fill (Q1) circle (4pt);
%index 2 points
\coordinate (Q2) at (6.5,0.5);
\fill (Q2) circle (4pt);	

\coordinate (Q0) at (5.2,2);
\fill (Q0) circle (4pt);
\coordinate (Q00) at (3,2);
\fill (Q00) circle (4pt);	
\node at ($(Q0)!.5!(Q00)$) {\ldots};

\draw[->, thick] (1.5,1.8) to (2,1.8);	

\coordinate[label=left:$m_1c$] (BC) at ($(6.4,2)$);	
\coordinate[label=right:$m_2c$] (BC) at ($(6.6,2)$);

\coordinate[label=below:${\frac{1}{c^2m_1m_2(k_1+\dots+k_n)-pc}(1,l_{min})}$] (BC) at ($(Q2)$);
\coordinate[label=above:${\frac{1}{pc}(1,l_{max})}$] (BC) at ($(Q1)$);
\coordinate[label=above:${\frac{1}{k_1}(1,k_1-1)}$] (BC) at ($(5.3,2)$);
\coordinate[label=above:${\frac{1}{k_r}(1,k_r-1)}$] (BC) at ($(Q00)$);

\draw[very thick](6.6,3.3) -- (6.6,0.7);
\draw[very thick](6.4,3.3) -- (6.4,0.7);
\end{tikzpicture}}\caption{Construction of a Hamiltonian $S^1$-space with Multigraph $(C)$.}\label{multigraph5}
\end{figure}
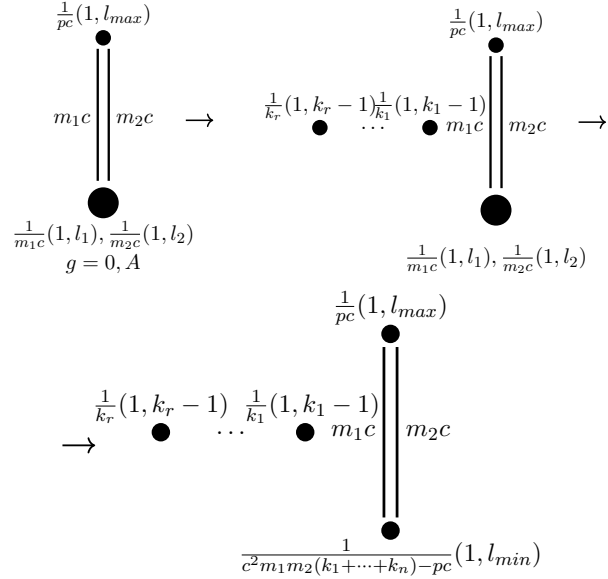

The space $(M,\omega,H)$ has a fixed orbi-sphere
$\Sigma_-\simeq \mathbb{C}P^1(m_1,m_2)/\mathbb{Z}_c$ with  normal
orbi-bundle that has degree $\beta=\frac{p}{cm_1m_2}$ (see
\eqref{eq:bsigma-}). We consider integers $k_1, \ldots, k_r \geq 1$
such that $m_1m_2 \,c \,(k_1+\cdots +k_r) >p$ and perform $r$ weighted
blow-ups, $\text{WBU}(1,1,0,k_i,1,\bullet)$  of  $X$ at $r$ regular
points in $\Sigma_-$ (see the labeled multigraph on the right of the
first row of Figure \ref{multigraph5}). We obtain $r$ non-extremal fixed points with cyclic orbifold structure
groups of orders $k_1,\ldots, k_r$.  By Lemma~\ref{lem::blowupchange0} 
the degree of the normal orbi-bundle of the blown-up minimal orbi-sphere $\widetilde{\Sigma}_-$ becomes 
 $$\widetilde{\beta}=\frac{p}{cm_1m_2} - k_1-\cdots-k_r<0.$$
Hence we can blow-down $\widetilde{\Sigma}_-$, using the reverse
procedure of a weighted blow-up
$\text{WBU}(qc,\frac{m_1}{q},\frac{m_2}{q},\frac{m_1}{q},\frac{m_2}{q},\bullet)$,
obtaining a singular point  of order
$qc:=c^2m_1m_2(k_1+\dots+k_n)-pc$. This yields a Hamiltonian
$S^1$-space with labeled multigraph as in the bottom of Figure
\ref{multigraph5}, i.e., as in Theorem \ref{thm:existisolated}
(C). It can be checked that the labels satisfy condition \ref{cc2}. Moreover, we can choose the appropriate
blow-up sizes so that the moment map labels satisfy condition
\ref{ccc}.

\begin{remark}
  At present, we do not know whether there is an even more explicit
  description of a general minimal Hamiltonian $S^1$-space with 
  labeled multigraph as in Theorem \ref{thm:existisolated} (B) with
  $r > 1$, or (C). In some specific cases, we know the answer: For
  instance, the example studied in \cite{counterex} is a global
  quotient of $\C P^1 \times \C P^1$ and has labeled multigraph of
  type (C) -- see Example \ref{tswEx}. 
\end{remark}

Lemma \ref{lemma:only_iso}, 
Theorems \ref{thm:uniqueness} and \ref{thm:existisolated}, and
the above constructions immediately imply the following
result.
\begin{theorem}\label{thm:fixed}
  Any Hamiltonian $S^1$-space with only isolated fixed points can be obtained by applying finitely many
  $S^1$-equivariant symplectic weighted blow-ups to either
  \begin{enumerate}[label=(\Alph*),ref=(\Alph*),leftmargin=*]
  \item a quotient of $\mathbb{C}P^1 \times \mathbb{C}P^1$ by a finite
    cyclic group;
  \item a quotient of a weighted projective plane by a finite
    cyclic group or a weighted blow-up of one of these spaces at a
    singular point and at a certain number of regular points on the
    same orbi-sphere, followed by a weighted blow-down of this
    orbi-sphere; or
  \item a weighted blow-up of a quotient of a weighted
    projective plane at a certain number of regular points on an
    orbi-sphere followed by an equivariant weighted blow-down of this
    orbi-sphere.
  \end{enumerate}
\end{theorem}

To conclude this section, we point out an important difference between
Hamiltonian $S^1$-manifolds and Hamiltonian $S^1$-manifolds: while in
the former, the $S^1$-action of any minimal spaces with only isolated
fixed points can be extended to a toric one (cf. \cite[Theorem
5.1]{karshon}), this no longer holds in the latter. More specifically,
the following result holds. 

\begin{lemma}\label{lemma:notextendTorus}
  The $S^1$-action on a minimal Hamiltonian $S^1$-space with labeled
  multigraph as in Theorem \ref{thm:existisolated} (C) does not extend
  to a toric action. 
\end{lemma}
\begin{proof}
  Let $(M,\omega,H)$ be as in the statement. Since $m_1c, m_2c >1$ and
  since there is at least one non-extremal isolated fixed point, there
  is a level set of $H$ that contains at least three orbits with
  non-trivial stabilizer. By the classification of symplectic toric
  orbifolds in \cite{LermanTolman}, this cannot hold if the
  $S^1$-action can be extended to a toric one. 
\end{proof}

\begin{remark}
  There is an alternative proof of Lemma \ref{lemma:notextendTorus}
  that does not use \cite{LermanTolman}. Let $(M,\omega,H)$ be as in
  the statement of the lemma and suppose that the $S^1$-action extends
  to a toric one. The desingularization procedure of Theorem
  \ref{thm:desingularization} works {\em mutatis mutandis} for toric
  actions. Hence, we may apply finitely many toric weighted blow-ups
  to obtain a symplectic toric manifold. Moreover, we may choose the
  sizes of the blow-up so that the symplectic toric manifold has the
  property that there is a level set of $H$ contains three orbits with
  non-trivial stabilizer. This is because the corresponding property
  holds for the original space $(M,\omega,H)$. This contradicts
  \cite[Proposition 5.21]{karshon}.
\end{remark}

% In particular we obtain the following result which contrasts with the fact  that a Hamiltonian circle action on a compact $4$-dimensional manifold with isolated fixed points always extends to a torus action \cite[Proposition 5.21]{karshon}.
% \begin{proposition}
% 	The circle actions on the minimal models $(B)$  with $r>1$ and $(C)$  in Theorem~\ref{thm:existisolated} do not extend to  torus actions.
% \end{proposition}
% \begin{proof}
% Desingularizing the orbi-spaces corresponding to these minimal models, we obtain smooth Hamiltonian $S^1$-spaces that have more than two non-trivial chains of gradient spheres. By  \cite[Proposition 5.21]{karshon}, the corresponding circle actions do not extend to  Hamiltonian torus actions. Consequently, the circle actions of the original  Hamiltonian $S^1$-spaces also do not extend to torus actions.
% \end{proof}

\section{Final remarks}\label{sec:final-remarks}

\subsection{Hamiltonian $S^1$-spaces are K\"{a}hler}\label{sec:hamilt-s1-spac}
\begin{theorem}\label{thm:kahler}
  Every Hamiltonian $S^1$-space is K\"{a}hler, i.e., it admits a
  complex structure for which the symplectic form is K\"{a}hler and
  the $S^1$-action is holomorphic.
\end{theorem}
\begin{proof}
  Let $(M,\omega,H)$ be a Hamiltonian $S^1$-space. By Theorem
  \ref{thm:desingularization}, it can be desingularized after applying finitely many
  $S^1$-equivariant weighted blow-ups. Let $k$ be the number of
  blow-ups needed in the desingularization procedure given by the
  proof of Theorem \ref{thm:desingularization}. We argue by induction
  on $k$. If $k=0$, the result follows by \cite[Theorem
  7.1]{karshon}. Suppose that the result holds for $k-1$. After
  applying the first blow-up in the desingularization procedure, we
  obtain a Hamiltonian $S^1$-space $(M',\omega',H')$ that requires
  $k-1$ blow-ups to be desingularized. Hence, by the inductive
  hypothesis, $(M',\omega',H')$ is Kähler. Since $(M,\omega,H)$ is
  (isomorphic to a space) obtained by performing a weighted blow-up on
  $(M',\omega',H')$, the result follows by Remark \ref{rmk:symp_weighted_blow_down}.
%   Since every Hamiltonian $S^1$-space can be desingularized through a sequence of weighted blow ups (cf. Theorem~\ref{thm:desingularization}) and, by \cite[Theorem 7.1]{karshon}, every smooth Hamiltonian $S^1$-space is K\"{a}hler, the result follows from the fact that if  a Hamiltonian $S^1$-space $(M,\omega,H)$ is  K\"{a}hler, then there exists an invariant  K\"{a}hler form on any of its equivariant weighted blow-downs along an invariant orbi-sphere.
% Indeed, let $\Sigma$ be an invariant orbi-sphere in $M$ whose normal orbi-bundle has negative degree and assume that $M$ admits a complex structure $J$ and an invariant  K\"{a}hler form $\omega$ for which the action is Hamiltonian. By removing a neightborhood of $\Sigma$ and gluing back a suitable orbi-ball as in Section~\ref{sec:wbd}, we get an invariant K\"{a}hler form $\omega^\prime$ on the equivariant weighted blow-down $M^\prime$  for which the resulting $S^1$-action on  $(M^\prime,\omega^\prime)$ is Hamiltonian. Note that by reversing the weighted blow down through an appropriate admissible weighted blow-up, the resulting $S^1$-space would be isomorphic to $(M, \omega,J)$ with the original $S^1$-action.
\end{proof}

\subsection{Birational equivalence of minimal Hamiltonian $S^1$-spaces}\label{sec:birat-equiv-minim}
The aim of this section is to prove that some of the minimal Hamiltonian $S^1$-spaces obtained in Sections~\ref{MSF} and \ref{MSI} are $S^1$-equivariantly birationally equivalent. % In particular we have the following result.

\begin{theorem}\label{thm:birationalequiv}
  Let $(M,\omega,H)$ be a minimal Hamiltonian $S^1$-space that has either fixed orbi-spheres or no
  fixed orbi-surfaces. Then
  $(M,\omega,H)$ is birationally equivalent to a minimal
  Hamiltonian $S^1$-space with multigraph as in $(B)$ of
  Theorem~\ref{thm:existisolated}.
\end{theorem}

\begin{proof}
  We assume first that any fixed orbi-surface of $(M,\omega,H)$ has
  genus zero. By Section \ref{MSI}, the result holds if the labeled multigraph
  of $(M,\omega,H)$ is either as in (I) or (II) in Theorem
  \ref{fig::minimalfixedsurface} (see Figure
  \ref{fig:constrgraph4}). Hence, we may assume that the labeled
  multigraph is as in (III) in Theorem
  \ref{fig::minimalfixedsurface} with $g=0$. Let $\Sigma_{\pm}$ denote
  the orbi-sphere lying at the maximum/minimum of $H$ and let $\beta_{\pm}$ be the degree of the normal
  orbi-bundle of $\Sigma_{\pm}$. Since there are no interior fixed
  points in $M$, it follows that $\beta_+ + \beta_- = 0$ (cf. equation
  \eqref{eq:48}). Without loss of generality, we assume that $\beta_- \leq 0$.   We fix notation so that the labeled multigraph of $(M,\omega,H)$ is
  as in the leftmost in Figure \ref{fig:proj2}, where $n$ is the
  number of singular points on $\Sigma_{\pm}$ and we allow $n=0$. The Seifert invariant of the circle orbi-bundle
  associated to the normal orbi-bundle of $\Sigma_-$ is 
  $$
  (g=0,\beta_0, (p_1,\hat{l}_1^\prime), \cdots, (p_n,\hat{l}_n^\prime)), 
  $$
  where, as usual, $1\leq \hat{l}_i^\prime<p_i$ is such that
  $\hat{l}_i^\prime \cdot \hat{l}_i=1 \mod p_i$. 

  % From Section~\ref{MSI}  we know that any minimal Hamiltonian $S^1$-space with a multigraph as in Theorem~\ref{fig::minimalfixedsurface}-(I) is birational equivalent to a Hamiltonian $S^1$-space with a multigraph as in $(B)$ of Theorem~\ref{thm:existisolated} and vice versa (cf. Figure~\ref{fig:constrgraph4}) and similarly for multigraphs as in Theorem~\ref{fig::minimalfixedsurface}-(II).
  If $n > 0$, for each
  $i=1,\ldots, n$, we perform a weighted blow up
  $\text{WBU}(p_i,1,0,\frac{\hat{l}_i^\prime}{p_i},
  \frac{1}{p_i},\varepsilon_i)$ at each of the singular points in
  $\Sigma_-$, where $\varepsilon_i >0$ is sufficiently
  small. As a
  result, for each $i$, after the blow-up we obtain an interior
  singular point of type $\frac{1}{\hat{l}_i^\prime}(1,\hat{s}_i)$,
  where $1\leq \hat{s}_i< \hat{l}_1^\prime$ is such that $\hat{s}_i\,
  p_i=-1 \mod \hat{l}_i^\prime$ (see the second multigraph from the
  left in Figure \ref{fig:proj2}).

% We  can  perform a weighted blow up $\text{WBU}(p_i,1,0,\frac{\hat{l}_i^\prime}{p_i}, \frac{1}{p_i},\bullet)$  at each of the singular points in $\Sigma_-$,  originating new singular points of type  $\frac{1}{\hat{l}_i^\prime}(1,\hat{s}_i)$ in the $\Z_{p_i}$-isotropy orbi-spheres,  with $1\leq \hat{s}_i< \hat{l}_1^\prime$ s.t. $\hat{s}_i\, p_i=-1 \mod \hat{l}_i^\prime$. We obtain the Hamiltonian $S^1$-space with the second  multigraph in Figure~\ref{fig:proj2}. 
    
\renewcommand{\thefigure}{\thesection.\arabic{figure}}
\begin{figure}[h!]
  \centering
  \begin{tikzpicture}[scale=.7]	
    \coordinate (point) at  (-0.5,3);
    \fill  (5.5,4) circle (8pt);
    \fill  (5.5,1) circle (8pt);
    \coordinate[label=above:{$\frac{1}{p_1}(1,\hat{l}_1) \cdots \frac{1}{p_n}(1,\hat{l}_n)$}] (VFD) at (5.5,-.2);
    \coordinate[label=above:{$g=0$}] (VFE) at (5.6,-0.8);			
    \coordinate[label=below:{$\frac{1}{p_1}(1,l_1)\cdots \frac{1}{p_n}(1,l_n)$}] (VCC) at (5.5,5);
    \coordinate[label=below:{$g=0$}] (VCC) at (5.5,5.7);			
    \draw[thick](5,3.6) -- (5,1.4);
    \draw[thick](6,3.6) -- (6,1.4);
    \node at ($(5,2.5)!.5!(6,2.5)$) {\ldots};
    \coordinate[label=right:{$p_n$}] (c1) at (6,2.5);
    \coordinate[label=left:{$p_1$}] (c1) at (5,2.5);		
    
    %%%%%%%%%%%%%%%%%%%%%%%%%%
    \draw[->, thick] (7.5,2.7) to (7.9,2.7);	
    %%%%%%%%%%%%%%%%%%%%%%%%%%

    \fill  (10.7,4) circle (8pt);
    \fill  (10.2,2.5) circle (3pt);
    \fill  (11.2,2.5) circle (3pt);
    \fill  (10.7,1) circle (8pt);
\coordinate[label=below:{$g=0$}] (VFE) at (10.7,.8);		
\coordinate[label=below:{$g=0,$}] (VCC) at (10.8,5.7);	
\coordinate[label=below:{$\frac{1}{p_1}(1,\hat{l}_1) \cdots \frac{1}{p_n}(1,\hat{l}_n)$}] (VCC) at (10.8,5);
\coordinate[label=below:{$\frac{1}{\hat{l}_1^\prime}(1,\hat{s}_1)$}] (VCC) at (9.4,2.5);
\coordinate[label=below:{$\cdots$}] (VCC) at (10.8,2.3);
\coordinate[label=below:{$\frac{1}{\hat{l}_n^\prime}(1,\hat{s}_n)$}] (VCC) at (12.2,2.5);

\draw[thick](10.2,3.6) -- (10.2,2.5);
\node at ($(10.2,3)!.5!(11.2,3)$) {\ldots};
\draw[thick](11.2,3.6) -- (11.2,2.5);

\coordinate[label=right:{$p_n$}] (c1) at (11.2,3);
\coordinate[label=left:{$p_1$}] (c1) at (10.3,3);	
    
%%%%%%%%%%%%%%%%%%%%%%%%%%
 \draw[->, thick] (13.4,2.7) to (13.8,2.7);	
 %%%%%%%%%%%%%%%%%%%%%%%%%%
    
\fill  (16.3,4) circle (8pt);
\fill  (15.9,2.5) circle (3pt);
\fill  (16.7,2.5) circle (3pt);
\fill  (16.3,1) circle (3pt);
       
\coordinate[label=below:{$\frac{1}{\lvert \beta_0\rvert }(1,1)$}] (VFE) at (16.4,.8);		
\coordinate[label=below:{$g=0$}] (VCC) at (16.3,5.7);	
\coordinate[label=below:{$\frac{1}{p_1}(1,\hat{l}_1) \cdots \frac{1}{p_n}(1,\hat{l}_n)$}] (VCC) at (16.3,5);
\coordinate[label=below:{$\frac{1}{\hat{l}_1^\prime}(1,\hat{s}_1)$}] (VCC) at (15,2.5);
\coordinate[label=below:{$\cdots$}] (VCC) at (16.4,2.3);
\coordinate[label=below:{$\frac{1}{\hat{l}_n^\prime}(1,\hat{s}_n)$}] (VCC) at (17.7,2.5);

\draw[thick](15.9,3.6) -- (15.9,2.5);
\node at ($(15.9,3)!.5!(16.7,3)$) {\ldots};
\draw[thick](16.7,3.6) -- (16.7,2.5);

\coordinate[label=right:{$p_n$}] (c1) at (16.7,3);
\coordinate[label=left:{$p_1$}] (c1) at (15.9,3);	
    
%%%%%%%%%%%%%%%%%%%%%%%%%
 \draw[->, thick] (18.5,2.7) to (18.9,2.7);	
%%%%%%%%%%%%%%%%%%%%%%%%%%
    
\fill  (20,4) circle (8pt);
\fill  (20,1) circle (3pt);
       
\coordinate[label=below:{$\frac{1}{\lvert \beta_0\rvert }(1,1)$}] (VFE) at (20,.8);	
\coordinate[label=below:{$g=0$}] (VCC) at (20,5);	

\end{tikzpicture}\caption{Birational equivalence of $\mathbb{P}(L
  \oplus \C)$ to a weighted projective plane, for an  orbi-bundle $L$
 over an orbi-sphere.} \label{fig:proj2}
\end{figure}
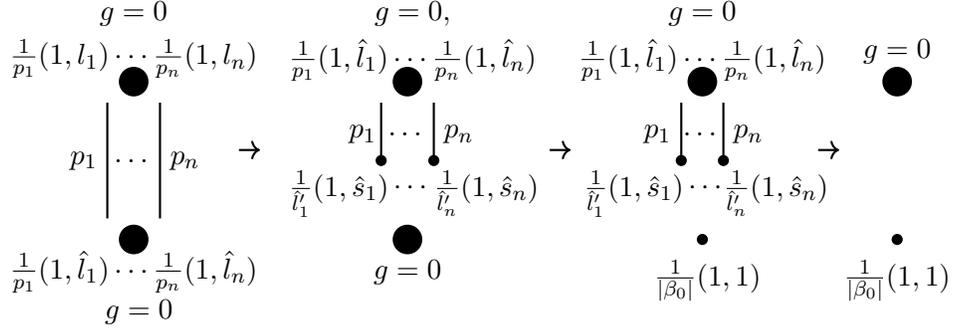

By Corollary~\ref{lem::blowupchange0} and
Proposition~\ref{def::EulerOrbibundle}, the degree of the normal
orbi-bundle of the resulting minimal sphere  is $\beta_- -
\sum_{i=1}^n \frac{\hat{l}_i^\prime}{p_i} = \beta_0 <0$  and so we can
blow down this sphere,  obtaining a singular point of order $\lvert
\beta_0\rvert$. For this we use the reverse construction of the
weighted blow up
$\text{WBU}\left(\beta_0,\frac{1}{\beta_0},\frac{1}{\beta_0},\frac{1}{\beta_0},\frac{1}{\beta_0},
\varepsilon\right)$ for an appropriate $\varepsilon > 0$. We obtain the third multigraph in Figure~\ref{fig:proj2}. By
Lemma \ref{lemma:existsurface1} and Theorem \ref{fig::minimalfixedsurface}, by applying finitely many
$S^1$-equivariant weighted blow-downs, we obtain a minimal Hamiltonian
$S^1$-space with no interior fixed points that has precisely one fixed
orbi-sphere at the maximum, i.e., with labeled multigraph given by (I) in
Theorem \ref{fig::minimalfixedsurface} (see the rightmost labeled
graph in Figure \ref{fig:proj2}). Hence, the result follows in
this case. If $n=0$, we argue as above starting with a standard blow-up
to the smooth sphere $\Sigma_-$.

Next we assume that $(M,\omega,H)$ only has isolated fixed points. If
the multigraph is as in $(A)$ of Theorem~\ref{thm:existisolated}, we
can, for instance,  blow up the singular point of type
$\frac{1}{c}(1,l)$, using a
$\text{WBU}(c,\frac{m}{c},-\frac{n}{c},\frac{1}{c},\frac{l}{c}, \bullet)$,
resulting in an exceptional divisor with one singular point of order
$l$ (see the second multigraph in Figure~\ref{fig:proj3}). By
Corollary~\ref{lem::blowupchange0}, the  normal orbi-bundles of the
blown-up $\Z_m$ and $\Z_n$-isotropy orbi-spheres in the resulting
space have degrees $-\frac{l}{c}$ and $-\frac{1}{l c}$
respectively. Blowing down these orbi-spheres we obtain the
Hamiltonian $S^1$-space with the last multigraph in
Figure~\ref{fig:proj3} (we leave the details to the reader). 

\renewcommand{\thefigure}{\thesection.\arabic{figure}}
\begin{figure}[h!]
  \centering
  \begin{tikzpicture}[scale=.7]	
  	\coordinate (1A) at (-7.5,3);
				%\coordinate[label=above:$(A)$] (2a) at (-5,-2);
				\coordinate (PP1) at (-5,3);
				\fill (PP1) circle (4pt);
				%index 2 points
				\coordinate (PP2) at (-5,0);
				\fill (PP2) circle (4pt);		
				\coordinate (PP3) at (-6,1.5);
				\fill (PP3) circle (4pt);
				\coordinate (PP4) at (-4,1.5);
				\fill (PP4) circle (4pt);		
				\draw[very thick](PP1) -- (PP3);
				\draw[very thick](PP2) -- (PP3);
				\draw[very thick](PP1) -- (PP4);
				\draw[very thick](PP4) -- (PP2);
				\coordinate[label=right:$m$] (AB) at ($(PP1)!0.5!(PP4)$);
				\coordinate[label=left:$n$] (BC) at ($(PP1)!0.5!(PP3)$);
				\coordinate[label=right:$n$] (AB) at ($(PP2)!0.5!(PP4)$);
				\coordinate[label=left:$m$] (BC) at ($(PP2)!0.5!(PP3)$);
				\coordinate[label=above:{$\frac{1}{c}(1,c-l^\prime)$}] (AB) at ($(PP1)$);
				\coordinate[label=below:{$\frac{1}{c}(1,c-l)$}] (AB) at ($(PP2)$);
				\coordinate[label=left:{$\frac{1}{c}(1,l^\prime)$}] (AB) at ($(PP3)$);	
		%		\coordinate[label=left:{$\color{red} H= \alpha + m \beta $}] (AB) at (-6.2,.8);	
				\coordinate[label=right:{$\frac{1}{c}(1,l)$}] (AB) at ($(PP4)$);
		%		\coordinate[label=right:{$\color{red} H= \alpha + n \gamma $}] (AB) at (-4,.8);

  %%%%%%%%%%%%%%%%%%%%%%%%%%%
      \draw[->, thick] (-2.2,1.5) to (-1.8,1.5);	
    %%%%%%%%%%%%%%%%%%%%%%%%%%
    
%\coordinate (1A) at (-7.5,3);
				%\coordinate[label=above:$(A)$] (2a) at (1,-2);
				\coordinate (PP1) at (1.2,3);
				\fill (PP1) circle (4pt);
				%index 2 points
				\coordinate (PP2) at (1.2,0);
				\fill (PP2) circle (4pt);		
				\coordinate (PP3) at (0.2,1.5);
				\fill (PP3) circle (4pt);
				\coordinate (PP4) at (1.7,2.2);
				\fill (PP4) circle (4pt);		
				\coordinate (PP5) at (1.7,.8);
				\fill (PP5) circle (4pt);		
				\draw[very thick](PP1) -- (PP3);
				\draw[very thick](PP2) -- (PP3);
				\draw[very thick](PP2) -- (PP5);
				\draw[very thick](PP1) -- (PP4);
				\draw[very thick](PP4) -- (PP5);
				\coordinate[label=right:$m$] (AB) at ($(PP1)!0.3!(PP4)$);
				\coordinate[label=right:$\quad \quad \quad  \frac{m l +n}{c}$] (AB) at ($(PP3)!0!(PP5)$);
				\coordinate[label=left:$n$] (BC) at ($(PP1)!0.5!(PP3)$);
				\coordinate[label=right:$n$] (AB) at ($(PP2)!0.3!(PP5)$);
				\coordinate[label=left:$m$] (BC) at ($(PP2)!0.5!(PP3)$);
				\coordinate[label=above:{$\frac{1}{c}(1,c-l^\prime)$}] (AB) at ($(PP1)$);
				\coordinate[label=below:{$\frac{1}{c}(1,c-l)$}] (AB) at ($(PP2)$);
				\coordinate[label=left:{$\frac{1}{c}(1,l^\prime)$}] (AB) at ($(PP3)$);	
	%			\coordinate[label=left:{$\color{red} H= \alpha + m \beta $}] (AB) at (-.2,.8);	
				\coordinate[label=right:{$\frac{1}{l}(1,s)$}] (AB) at ($(PP5)$);
		%		\coordinate[label=right:{$\color{red} H= \alpha + n \gamma $}] (AB) at (2,.8);	
    
  %%%%%%%%%%%%%%%%%%%%%%%%%%%
      \draw[->, thick] (3.6,1.5) to (4,1.5);	
    %%%%%%%%%%%%%%%%%%%%%%%%%%  
%\coordinate (1A) at (-7.5,3);
				%\coordinate[label=above:$(A)$] (2a) at (1,-2);
				\coordinate (PP1) at (7.3,3);
				\fill (PP1) circle (4pt);
				%index 2 points
				\coordinate (PP2) at (7.3,0);
				\fill (PP2) circle (4pt);		
				\coordinate (PP3) at (6,1.5);
				\fill (PP3) circle (4pt);
				\coordinate (PP4) at (7.3,2.2);
				%\fill (PP4) circle (4pt);		
				\coordinate (PP5) at (7.3,.8);
				% \fill (PP5) circle (4pt);		
				\draw[very thick](PP1) -- (PP3);
				\draw[very thick](PP2) -- (PP3);
				\draw[very thick](PP2) -- (PP5);
				\draw[very thick](PP1) -- (PP4);
				\draw[very thick](PP4) -- (PP5);
	%			\coordinate[label=right:$m$] (AB) at ($(PP1)!0.3!(PP4)$);
				\coordinate[label=right:$\quad \quad \quad  \frac{m l +n}{c}$] (AB) at ($(PP3)!0!(PP5)$);
				\coordinate[label=left:$n$] (BC) at ($(PP1)!0.4!(PP3)$);
		%		\coordinate[label=right:$n$] (AB) at ($(PP2)!0.3!(PP5)$);
				\coordinate[label=left:$m$] (BC) at ($(PP2)!0.4!(PP3)$);
				\coordinate[label=above:{$\frac{1}{l}(1,t)$}] (AB) at ($(PP1)$);
		%		\coordinate[label=below:{$\frac{1}{c}(1,c-l)$}] (AB) at ($(PP2)$);
				\coordinate[label=left:{$\frac{1}{c}(1,l^\prime)$}] (AB) at ($(PP3)$);	
	%			\coordinate[label=left:{$\color{red} H= \alpha + m \beta $}] (AB) at (-.2,.8);	
				%\coordinate[label=right:{$\frac{1}{l}(1,s)$}] (AB) at ($(PP5)$);
  \end{tikzpicture}\caption{Birational equivalence of  $\left( \mathbb{C} P^1 \times\mathbb{C}  P^1 \right) /\mathbb{Z}_c$ and a weighted projective plane.} \label{fig:proj3}
\end{figure}
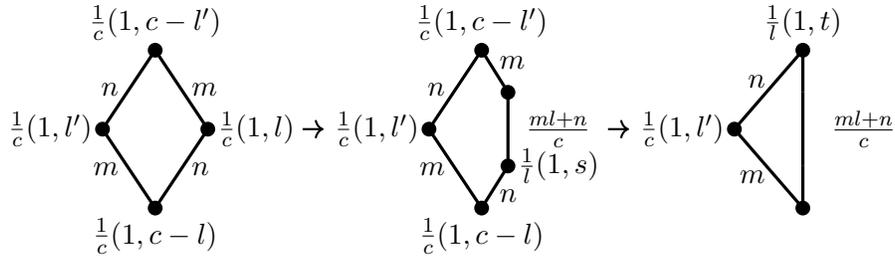

Finally, by Section~\ref{MSI}, if the labeled multigraph is as in
$(C)$ in Theorem~\ref{thm:existisolated}, then $(M,\omega,H)$ is
$S^1$-equivariantly birationally equivalent to a Hamiltonian
$S^1$-space with labeled multigraph as in
Theorem~\ref{fig::minimalfixedsurface}-(I) or (II) (see
Figure~\ref{multigraph5}). This completes the proof.
\end{proof}

\begin{example} Let us consider a minimal Hamiltonian $S^1$-space with
  multigraph as in Theorem~\ref{fig::minimalfixedsurface}-(III), with
  $g=0$ and such that each fixed orbi-sphere has exactly two singular
  points (see Figure \ref{fig:proj2b}). Since there are no interior
  fixed points, the degree $\beta_\pm$ of the normal orbi-bundle of
  the fixed orbi-sphere  $\Sigma_\pm$ is zero. Moreover, by $(iv)$ and
  $(v)$ of Theorem~\ref{fig::minimalfixedsurface}-(III), if the type
  of one of the singular points  in $\Sigma_+$ is $\frac{1}{p}(1,l)$,
  then the type of the other singular point in $\Sigma_+$ must be
  $\frac{1}{p}(1,p-l)$ and the multigraph is the one in
  Figure~\ref{fig:proj2b}, where $1\leq l^\prime<p$ is such that $l \cdot
  l^\prime = 1 \mod p$.
\renewcommand{\thefigure}{\thesection.\arabic{figure}}
  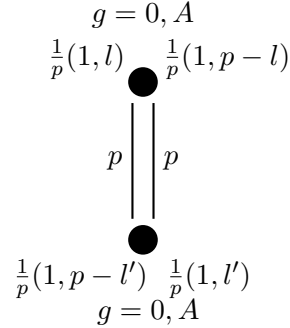
\begin{figure}[h!]
  \centering
  \begin{tikzpicture}[scale=.7]	
    \coordinate (point) at  (-0.5,3);
    \fill  (5.5,4) circle (8pt);
    \fill  (5.5,1) circle (8pt);
    \coordinate[label=above:{$\frac{1}{p}(1,p-l^\prime)\,\,\,\,\frac{1}{p}(1,l^\prime)$}] (VFD) at (5.3,-.3);
    \coordinate[label=above:{$g=0,A$}] (VFE) at (5.6,-0.8);			
    \coordinate[label=below:{$\frac{1}{p}(1,l)\quad \,\, \frac{1}{p}(1,p-l)$}] (VCC) at (6,5);
    \coordinate[label=below:{$g=0, A$}] (VCC) at (5.5,5.7);			
    \draw[thick](5.3,3.6) -- (5.3,1.4);
    \draw[thick](5.7,3.6) -- (5.7,1.4);

    \coordinate[label=right:{$p$}] (c1) at (5.7,2.5);
    \coordinate[label=left:{$p$}] (c1) at (5.3,2.5);
     \end{tikzpicture}\caption{Multigraph of a minimal Hamiltonian $S^1$-space with two orbi-spheres with trivial normal orbi-bundle.} \label{fig:proj2b}
\end{figure}
Since the Seifert invariant of the circle orbi-bundle associated to the normal orbi-bundle of $\Sigma_-$ is 
$$
(g=0, \beta_0=-1, (p,l^\prime), (p,p-l^\prime)),
$$
the proof of Theorem~\ref{thm:birationalequiv} shows that this space
is $S^1$-equivariantly birationally equivalent to $S^2\times S^2$ with the standard diagonal circle action.
%\begin{equation}\label{eq:birational1}
%\hat{l}_i + l_i^\prime= \hat{l}_i^\prime + l_i= p_i,\,\,\, (\text{for $i=1,2$}) \quad \text{and} \quad \frac{p_2 l_1 + p_1 l_2}{p_1p_2} \in \Z. 
%\end{equation}
%By the last condition of \eqref{eq:birational1} we have that $p_1 l_2=0 \mod p_2$ and then, since $\gcd(l_2,p_2)=1$, we have that $p_1=0 \mod p_2$. Similarly, we can conclude that $p_2=0 \mod p_1$ and so $p_1=p_2$.
\end{example}

\color{black}

\end{document}